\documentclass[11pt,a4paper,draft]{amsart}
\usepackage{amssymb,amscd,stmaryrd}
\usepackage[top=3cm,bottom=3cm,left=2.5cm,right=2.5cm,footskip=17mm]{geometry}
\usepackage[mathcal]{eucal}
\usepackage{array,float}

\usepackage[all]{xy}
\def\arraystretch{1.2}
\theoremstyle{plain}
\newtheorem{theorem}{Theorem}[section]
\newtheorem{thm}[theorem]{Theorem}
\newtheorem{lem}[theorem]{Lemma}
\newtheorem{prop}[theorem]{Proposition}
\newtheorem{cor}[theorem]{Corollary}

\newtheorem{conj}[theorem]{Conjecture}

\newtheorem{thmABC}{Theorem}

\newtheorem{corABC}[thmABC]{Corollary}

\newtheorem*{property}{Property}

\theoremstyle{remark}
\newtheorem{remark}[theorem]{Remark}
\newtheorem{remark.star}[theorem]{Remark$\,{}^*$}

\theoremstyle{definition}
\newtheorem{defn}[theorem]{Definition}

\numberwithin{equation}{section}
\numberwithin{table}{section}
\numberwithin{figure}{section}

\renewcommand{\epsilon}{\varepsilon}
\renewcommand{\phi}{\varphi}
\renewcommand{\theta}{\vartheta}

\DeclareMathOperator{\Div}{Div}
\DeclareMathOperator{\simi}{sim}

\DeclareMathOperator{\consub}{CS}
\DeclareMathOperator{\gen}{gen}
\DeclareMathOperator{\cha}{char}

\DeclareMathOperator{\im}{im}
\DeclareMathOperator{\Ima}{Im}
\DeclareMathOperator{\cd}{cd}
\DeclareMathOperator{\Irr}{Irr}
\DeclareMathOperator{\In}{I}
\DeclareMathOperator{\Stab}{Stab}
\DeclareMathOperator{\stab}{stab}
\DeclareMathOperator{\rad}{rad}
\DeclareMathOperator{\Ad}{Ad}
\DeclareMathOperator{\ad}{ad}
\DeclareMathOperator{\tr}{tr}

\DeclareMathOperator{\Res}{Res}
\DeclareMathOperator{\Id}{Id}
\DeclareMathOperator{\diag}{diag}

\DeclareMathOperator{\Cay}{Cay}
\DeclareMathOperator{\cay}{cay}
\newcommand{\coloneqq}{\mathrel{\mathop:}=}

\newcommand{\N}{\mathbb{N}}
\newcommand{\Z}{\mathbb{Z}}
\newcommand{\R}{\mathbb{R}}
\newcommand{\Q}{\mathbb{Q}}
\newcommand{\C}{\mathbb{C}}
\newcommand{\F}{\mathbb{F}}
\newcommand{\T}{\mathbb{T}}
\newcommand{\lri}{\mathfrak{o}}
\newcommand{\Lri}{\mathfrak{O}}
\newcommand{\lfi}{\mathfrak{f}}
\newcommand{\Lfi}{\mathfrak{F}}
\newcommand{\gri}{\ensuremath{{\scriptstyle \mathcal{O}}}}

\newcommand{\Gri}{\ensuremath{\mathcal{O}}}

\newcommand{\ee}{\epsilon}     
\newcommand{\fp}{\mathfrak{p}} 
\newcommand{\fP}{\mathfrak{P}} 
\newcommand{\kk}{\mathbf{k}}   
\newcommand{\KK}{\mathbf{K}}  
\newcommand{\len}{\ell}
\newcommand{\gl}{\mathsf{gl}}
\newcommand{\fsl}{\mathsf{sl}}
\newcommand{\gu}{\mathsf{gu}}
\newcommand{\su}{\mathsf{su}}

\newcommand{\ag}{\mathsf{g}}
\newcommand{\ah}{\mathsf{h}}
\newcommand{\ab}{\mathsf{b}}

\newcommand{\as}{\mathsf{s}}

\newcommand{\UU}{A_{q,\len}(s)}
\newcommand{\VV}{B_{q,\len}(s)}
\newcommand{\WW}{C_{q,\len}(s)}
\newcommand{\XX}{D_{q,\len}(s)}

\newcommand{\aH}{\mathsf{H}}
\newcommand{\aG}{\mathsf{G}}

\newcommand{\aI}{\mathsf{I}}
\newcommand{\Heis}{\mathsf{Heis}}

\newcommand{\GL}{\mathsf{GL}}
\newcommand{\SL}{\mathsf{SL}}

\newcommand{\GU}{\mathsf{GU}}
\newcommand{\SU}{\mathsf{SU}}

\newcommand{\Cen}{\mathrm{C}}
\newcommand{\wt}{\widetilde}
\newcommand{\wh}{\widehat}
\newcommand{\Qgl}{\mathcal{Q}^{\gl_n}}
\newcommand{\Qgu}{\mathcal{Q}^{\gu_n}}

\newcommand{\real}{\textup{Re}}

\newcommand{\hlf}{\genfrac{}{}{0.1pt}{1}{1}{2}}
\newcommand{\third}{\genfrac{}{}{0.1pt}{1}{1}{3}}
\newcommand{\sixth}{\genfrac{}{}{0.1pt}{1}{1}{6}}

\newcommand{\Gsha}{\mathcal{G}}
\newcommand{\Lsha}{\mathcal{L}}
\newcommand{\Tasha}{{\mathcal{T}_1}}
\newcommand{\Tbsha}{{\mathcal{T}_2}}
\newcommand{\Tcsha}{{\mathcal{T}_3}}
\newcommand{\Nsha}{\mathcal{N}}
\newcommand{\Msha}{\mathcal{M}}
\newcommand{\Jsha}{\mathcal{J}}
\newcommand{\Kasha}{{\mathcal{K}_0}}
\newcommand{\Kbsha}{{\mathcal{K}_\infty}}

\newcommand{\Sh}{\mathfrak{Sh}}
\newcommand{\ShGn}{\Sh_{\GL_n(\lri)}}
\newcommand{\ShUn}{\Sh_{\GU_n(\lri)}}

\newcommand{\sh}{\mathrm{sh}}

\newcommand{\shg}{\sh_\gl}
\newcommand{\shG}{\sh_\GL}
\newcommand{\shGC}{\sh_\GL} 
\newcommand{\shu}{\sh_\gu}
\newcommand{\shU}{\sh_\GU}
\newcommand{\shUC}{\sh_\GU} 

\newcommand{\cA}{\mathcal A}
\newcommand{\cC}{\mathcal C}
\newcommand{\cCtilde}{\wt{\mathcal{C}}}

\newcommand{\cQ}{\mathcal Q}
\newcommand{\cS}{\mathcal S}
\newcommand{\cT}{\mathcal T}
\newcommand{\cV}{\mathcal V}

\newcommand{\fg}{\mathfrak{g}}
\newcommand{\fh}{\mathfrak{h}}

\newcommand{\fr}{\mathfrak{r}}
\newcommand{\fs}{\mathfrak{s}}
\newcommand{\fn}{\mathfrak{n}}

\newcommand{\Mat}[9]{
    \begin{bmatrix}
      #1 & #2 & #3 \\
      #4 & #5 & #6 \\
      #7 & #8 & #9
    \end{bmatrix}
}
\newcommand{\vI}{\ensuremath{\mathrm{v}_\mathrm{I}}}
\newcommand{\vII}{\ensuremath{\mathrm{v}_\mathrm{II}}}
\newcommand{\vIII}{\ensuremath{\mathrm{v}_\mathrm{III}}}

\newcommand{\artin}[3]{\left(\frac{#1 \,\vert\, #2}{#3}\right)}

\title[Similarity classes and representation zeta functions of groups
of type $\mathsf{A}_2$]{Similarity classes of integral $\fp$-adic
  matrices and representation zeta functions of groups of type
  $\mathsf{A}_2$}

\author{Nir Avni}
\address{Department of Mathematics, Harvard University, One Oxford
  Street, Cambridge MA 02138, USA} \email{avni.nir@gmail.com}

\author{Benjamin Klopsch} \address{Mathematisches Institut der
  Heinrich-Heine-Universit\"at, Universit\"atsstr.\ 1, 40225
  D\"usseldorf, Germany}\email{klopsch@math.uni-duesseldorf.de}

\author{Uri Onn}
\address{Department of Mathematics, Ben Gurion University of the
  Negev, Beer-Sheva 84105, Israel} \email{urionn@math.bgu.ac.il}

\author{Christopher Voll} \address{Fakult\"at f\"ur Mathematik,
  Universit\"at Bielefeld, 33501 Bielefeld, Germany}
\email{C.Voll.98@cantab.net}

\thanks{Avni was supported by NSF grant DMS-0901638.  Onn was
  supported by ISF grant 382/11.}



\begin{document}

\maketitle

\centerline{\today}

\begin{abstract}
  We compute explicitly Dirichlet generating functions enumerating
  finite-dimen\-sio\-nal irreducible complex representations of
  various $p$-adic analytic and ad\`elic profinite groups of
  type~$\mathsf{A}_2$.  This has consequences for the representation
  zeta functions of arithmetic groups $\Gamma \subset \mathbf{H}(k)$,
  where $k$ is a number field and $\mathbf{H}$ a $k$-form of~$\SL_3$:
  assuming that $\Gamma$ possesses the strong Congruence Subgroup
  Property, we obtain precise, uniform estimates for the
  representation growth of~$\Gamma$.  Our results are based on
  explicit, uniform formulae for the representation zeta functions of
  the $p$-adic analytic groups $\SL_3(\lri)$ and $\SU_3(\lri)$, where
  $\lri$ is a compact discrete valuation ring of characteristic~$0$.
  These formulae build on our classification of similarity classes of
  integral $\fp$-adic $3\times3$ matrices in $\gl_3(\lri)$ and
  $\gu_3(\lri)$, where $\lri$ is a compact discrete valuation ring of
  arbitrary characteristic.  Organising the similarity classes by
  invariants which we call their \emph{shadows} allows us to combine
  the Kirillov orbit method with Clifford theory to obtain explicit
  formulae for representation zeta functions.  In a different
  direction we introduce and compute certain similarity class zeta
  functions.

  Our methods also yield formulae for representation zeta functions of
  various finite subquotients of groups of the form $\SL_3(\lri)$,
  $\SU_3(\lri)$, $\GL_3(\lri)$, and $\GU_3(\lri)$, arising from the
  respective congruence filtrations; these formulae are valid in case
  that the characteristic of $\lri$ is either $0$ or sufficiently
  large.  Analysis of some of these formulae leads us to observe
  $p$-adic analogues of `Ennola duality'.
\end{abstract}

\setcounter{tocdepth}{1}

\thispagestyle{empty}
\tableofcontents



\section{Introduction and discussion of main results}
Let $G$ be a group. For $n \in \N$, let $r_n(G)$ denote the number of
$n$-dimensional irreducible complex representations of $G$ up to
equivalence. If $G$ is a topological or an algebraic group, the
representations are assumed to be continuous or algebraic,
respectively. Following~\cite{BLMM}, we say that $G$ is
(\emph{representation}) \emph{rigid} if $r_n(G)$ is finite for
all~$n$. In this case one takes interest in the arithmetic function $n
\mapsto r_n(G)$.  Examples of such rigid groups include most
`semisimple' arithmetic and compact $p$-adic analytic groups. The
\emph{representation zeta function} of a rigid group $G$ is the
Dirichlet generating function
\begin{equation} \label{equ:defn.zeta} \zeta_G(s) =
  \sum_{n=1}^{\infty}r_n(G) n^{-s} \quad (s\in\C).
\end{equation}

The group $G$ is said to have \emph{polynomial representation growth}
if the growth of $N \mapsto R_N(G) = \sum_{n=1}^{N} r_n(G)$ is bounded
by some polynomial in~$N$.  In this case the Dirichlet series
$\zeta_G(s)$ converges absolutely in a complex half-plane of the form
$\{s \in \C \mid \real(s)>\alpha\}$ for some $\alpha \in \R$.  The
infimum of such numbers $\alpha$ is called the \emph{abscissa of
  convergence} of $\zeta_G(s)$ and denoted by~$\alpha(G)$.  If $G$
admits only finitely many irreducible complex representations, then
$\alpha(G)=-\infty$ and $\zeta_G(s)$ is holomorphic on the entire
complex plane.  Otherwise the abscissa of convergence $\alpha(G)$
satisfies
\[
\alpha(G) = \limsup_{N \to \infty} \frac{\log \left(R_N(G)
  \right)}{\log N}
\]
and thus gives the degree of polynomial growth.  For a range of
results on representation growth and zeta functions of arithmetic and
profinite groups see, for instance, \cite{LM,Avni,
  JZ,LL,S12,AizenbudAvni/13}.  The current paper forms part of a
series of papers on representation growth;
see~\cite{cr_AKOV/09,AKOVIII/11,AKOV1,AKOV4}.  For surveys,
see~\cite{Klopsch/13, Voll/14}.

\subsection{Analytic properties of zeta functions of ad\`elic and
  arithmetic groups}\label{subsec:main.results}
The arithmetic groups considered in this paper are of type
$\mathsf{A}_2$ and defined in characteristic~$0$.  Let $k$ be a number
field with ring of integers~$\Gri$.  Let $\mathbf{H}$ be a connected,
simply-connected absolutely almost simple algebraic group defined
over~$k$, with a fixed embedding into~$\GL_d$ for some~$d \in \N$.
For a place $v$ of $k$, we write $k_v$ for the completion of $k$
at~$v$ and, if $v$ is non-archimedean, $\Gri_v$ for the completion of
$\Gri$ at~$v$.  Let $S$ be a finite set of places of $k$, including
all the archimedean ones, and let $\Gri_S = \{x \in k \mid \forall v
\not\in S: x \in \Gri_v\}$ denote the ring of $S$-integers in~$k$.
The arithmetic group $\mathbf{H}(\Gri_S) = \mathbf{H}(k) \cap
\GL_d(\Gri_S)$ embeds diagonally into the $S$-ad\`ele group
$\mathbf{H}(\mathcal{A}_{k,S}) = \{ (g_v) \in \prod_{v \not\in S}
\mathbf{H}(k_v) \mid g_v \in \mathbf{H}(\Gri_v) \text{ for almost all
} v \}$.  By the Strong Approximation Theorem, the congruence
completion of $\mathbf{H}(\Gri_S)$ coincides with the open compact
subgroup $\mathbf{H}(\widehat{\Gri_S}) \simeq \prod_{v \not\in S}
\mathbf{H}(\Gri_v)$ of~$\mathbf{H}(\mathcal{A}_{k,S})$.

It was shown in~\cite{LM} that the congruence
completion~$\mathbf{H}(\widehat{\Gri_S})$ has polynomial
representation growth.  In~\cite[Theorem~C]{AKOV1} we quantified this
result for groups $\mathbf{H}$ of type~$\mathsf{A}_2$: in this case,
$\alpha(\mathbf{H}(\widehat{\Gri_S})) = 1$, in other words the
representation growth of $\mathbf{H}(\widehat{\Gri_S})$ is linear.
Our first main result establishes finer asymptotic properties of
$\zeta_{\mathbf{H}(\widehat{\Gri_S})}(s)$ for groups $\mathbf{H}$ of
type~$\mathsf{A}_2$.

\begin{thmABC} \label{thm:A} Let $\mathbf{H}(\widehat{\Gri_S})$ be an
  ad\`elic profinite group as above, where the algebraic group
  $\mathbf{H}$ is connected, simply-connected absolutely almost simple
  of type~$\mathsf{A}_2$.
  \begin{enumerate}
  \item The zeta function $\zeta_{\mathbf{H}(\widehat{\Gri_S})}(s)$
    can be meromorphically continued to the complex half-plane $\{s \in \C
    \mid \real(s) > 5/6 \}$.  The only pole of
    $\zeta_{\mathbf{H}(\widehat{\Gri_S})}(s)$ in this domain is a
    double pole at~$s=1$.
  \item There exists an invariant $c(\mathbf{H}(\widehat{\Gri_S})) \in
    \R_{>0}$ such that
    \[
    c(\mathbf{H}(\widehat{\Gri_S})) =
    \lim_{N\to\infty}\frac{R_N(\mathbf{H}(\widehat{\Gri_S}))}{N\log
      N}.
    \]
  \end{enumerate}
\end{thmABC}

\begin{remark}
  In fact, the proof of Theorem~\ref{thm:A} works for a somewhat
  larger class of profinite groups, including groups of the form $H =
  \prod_{v \not\in S} H_v$, where $H_v$ is commensurable to a compact
  open subgroup of an absolutely almost simple $k_v$-algebraic group
  $\mathbf{H}_v(k_v)$ of type~$\mathsf{A}_2$ for each place $v$ and
  such that $H_v$ is equal to either $\SL_3(\Gri_v)$ or
  $\SU_3(\Gri_v)$ for almost all~$v$; see Section~\ref{sec:abscissa}
  for details.  A precise definition of the standard unitary group
  $\SU_3(\lri)$ over a discrete valuation ring $\lri$ is given in
  Section~\ref{sec:sim.gu}.
\end{remark}

The arithmetic group $\mathbf{H}(\Gri_S)$ has the \emph{weak
  Congruence Subgroup Property} (wCSP) if the congruence kernel
$\ker(\widehat{\mathbf{H}(\Gri_S)} \to \mathbf{H}(\widehat{\Gri_S}))$
of the natural projection from the profinite completion onto the
congruence completion is finite.  We say that $\mathbf{H}(\Gri_S)$ has
the \emph{strong Congruence Subgroup Property} (sCSP) if this
congruence kernel is trivial.  For instance, the solution of the
congruence subgroup problem for the Chevalley group schemes $\SL_n$,
$n \geq 3$, by Bass, Milnor, and Serre~\cite{BMS/67} implies that the
group $\SL_n(\Gri)$ always has the wCSP and that it fails to have the
sCSP if and only if $k$ is totally imaginary.  Theorem~\ref{thm:A}
leads to the following corollary.

\begin{corABC} \label{cor:B} Let $\mathbf{H}(\Gri_S)$ be an arithmetic
  group as above, where the algebraic group $\mathbf{H}$ is connected,
  simply-connected absolutely almost simple of type~$\mathsf{A}_2$,
  and suppose that $\mathbf{H}(\Gri_S)$ has the~wCSP.  Then
  $\mathbf{H}(\Gri_S)$ contains a finite index subgroup $\Gamma$ such
  that the following is true.
  \begin{enumerate}
  \item The zeta function $\zeta_{\Gamma}(s)$ can be meromorphically
    continued to the complex half-plane $\{s \in \C \mid \real(s) >
    5/6 \}$.  The only pole of $\zeta_{\Gamma}(s)$ in this domain is a
    double pole at~$s=1$.
  \item \label{item:second} There exists $c(\Gamma) \in \R_{>0}$ such
    that
    \[
    c(\Gamma) = \lim_{N\to\infty}\frac{R_N(\Gamma)}{N\log N}.
    \]
  \end{enumerate}
  Moreover, if $\mathbf{H}(\Gri_S)$ has the sCSP then one may take
  $\Gamma = \mathbf{H}(\Gri_S)$.
\end{corABC}

\begin{remark}\label{rem:corB}
  Within the special class of groups that it covers,
  Corollary~\ref{cor:B} goes beyond a conjecture of Larsen and
  Lubotzky on the degrees of polynomial representation growth of
  arithmetic lattices in higher rank semisimple groups;
  see~\cite[Conjecture~1.5]{LL}.  In the same way, it refines the
  variant of this conjecture that was proved
  in~\cite[Theorem~1.2]{AKOV4}.  Indeed, Corollary~\ref{cor:B} asserts
  that, for the relevant arithmetic groups~$\Gamma$, not only the
  degree of representation growth but also the order of the pole of
  the meromorphically continued function at $s = \alpha(\Gamma)$, and
  thus the exponent of the log-$N$-term in \eqref{item:second}, are
  invariants of the type~$\mathsf{A}_2$.  Likewise meromorphic
  continuation can be achieved uniformly in a strip of width at
  least~$1/6$.  The value of the constant $c(\Gamma)$, in contrast,
  depends subtly on the specific group $\Gamma$; see
  Section~\ref{sec:abscissa} for details.

  Furthermore, it is not difficult to extend Theorem~\ref{thm:A} and
  Corollary~\ref{cor:B} to cover ad\`elic profinite groups arising
  from semisimple algebraic groups that are not absolutely almost
  simple, by using the multiplicativity of the representation zeta
  function, i.e.\ $\zeta_{H_1 \times H_2}(s) = \zeta_{H_1}(s)
  \zeta_{H_2}(s)$, for groups whose categories of finite-dimensional
  complex representations are semisimple.
\end{remark}

For simplicity, consider an arithmetic group of the form
$\mathbf{H}(\Gri_S)$ with the sCSP.  A key role in the study of the
representation zeta function $\zeta_{\mathbf{H}(\Gri_S)}(s)$ plays the
fact that it admits an Euler product decomposition.  Indeed, the
triviality of the congruence kernel implies that
\begin{equation}\label{equ:euler}
  \zeta_{\mathbf{H}(\Gri_S)}(s) =
  \zeta_{\mathbf{H}(\C)}(s)^{\lvert k:\Q \rvert} \, \prod_{v\not\in
    S}\zeta_{\mathbf{H}(\Gri_v)}(s);
\end{equation}
see \cite[Proposition~1.3]{LL}.  Here, each archimedean local factor
$\zeta_{\mathbf{H}(\C)}(s)$, known as the Witten zeta function of the
algebraic group $\mathbf{H}(\C)$, enumerates the irreducible rational
representations of~$\mathbf{H}(\C)$; see~\cite{Witten}.  The
non-archimedean local factors $\zeta_{\mathbf{H}(\Gri_v)}(s)$ are the
zeta functions of the rigid compact $p$-adic analytic groups
$\mathbf{H}(\Gri_v)$.  For places $v$ not dividing the prime~$2$, the
representation zeta functions of these groups are given by rational
functions; see~\cite{JZ}.  If $\mathbf{H}$ is absolutely almost simple
of type $\mathsf{A}_2$, then for all but finitely many places~$v$, the
groups $\mathbf{H}(\Gri_v)$ are isomorphic to $\SL_3(\Gri_v)$ or
$\SU_3(\Gri_v)$; cf.~\cite[Appendix~A]{AKOV1}.  Our second main
result, Theorem~\ref{thm:C} below, describes the rational functions
$\zeta_{\SL_3(\Gri_v)}(s)$ and $\zeta_{\SU_3(\Gri_v)}(s)$ explicitly,
apart from finitely many exceptions; see Corollary~\ref{cor:D} for a
concrete, self-contained formula.  Our proof of Theorem~\ref{thm:A} is
based on an analysis of this formula.

\subsection{Character shadows}\label{subsec:shadows.types}
The explicit formulae in Theorem~\ref{thm:C} below, giving the
representation zeta functions of various $p$-adic analytic groups of
type $\mathsf{A}_2$ and finite subquotients thereof, are organised in
terms of representation-theoretic invariants, which we call shadows, a
concept that we now explain.

Let $G$ be a group with a normal subgroup~$N$.  Suppose that the
category of finite-dimensional complex representations of $G$,
respectively $N$, is semisimple and that equivalence classes of
irreducible finite-dimensional complex representations are
parametrised by the corresponding characters.  Writing $\Irr(G)$ for
the set of irreducible complex characters of~$G$ and $\consub(G/N)$
for the set of conjugacy classes of subgroups of $G/N$, consider the
map
\[
\sh_{G,N} \colon \Irr(G) \rightarrow \consub(G/N), \quad \chi \mapsto
\{ \textrm{I}_G(\varphi)/N \mid \varphi \text{ an irred.\ constituent
  of } \Res^G_N(\chi) \},
\]
where $\In_G(\varphi)$ denotes the inertia subgroup of~$\varphi$
in~$G$.  We call $\sh_{G,N}(\chi)$ the (\emph{character})
\emph{shadow} of $\chi \in \Irr(G)$ with respect to~$N$.  If $G$ is
rigid then the map $\sh_{G,N}$ gives rise to a decomposition of the
representation zeta function:
\begin{equation}\label{equ:dec.shadows}
  \zeta_{G}(s) = \sum_{\sigma \in \,\im(\sh_{G,N})}\zeta^\sigma_G(s),
  \qquad \text{where $\zeta^\sigma_G(s) =
    \sum_{\substack{\chi\in\Irr(G) \\ \sh_{G,N}(\chi) = \sigma}} \chi(1)^{-s}$.}
\end{equation}
Albeit arguably too general to be of interest for rigid groups at
large, this decomposition allows us to give explicit, uniform formulae
for the zeta functions of selected classes of groups.

Specifically we consider $p$-adic analytic groups of the form $G =
\aG(\lri)$, respectively $G = \aH(\lri)$, where $\lri$ is a compact
discrete valuation ring of characteristic $0$ with (finite) residue
field $\kk$ and $\aG$ is one of the $\lri$-group schemes $\GL_3$ or
$\GU_3$, and $\aH$ is one of $\SL_3$ or $\SU_3$.  We set $N =
\aG^1(\lri)$, respectively $N = \aH^1(\lri)$, the $1$st principal
congruence subgroup, so that $G/N$ is isomorphic to $\aG(\kk)$
or~$\aH(\kk)$.  Here, the standard unitary $\lri$-group schemes
$\GU_3,\SU_3$ are defined with respect to the standard involution
based on the non-trivial Galois automorphism of an unramified
quadratic extension of~$\lri$; see \eqref{def:unitary} for details.
We define
\begin{equation} \label{equ:def.epsilon} \ee = \ee_\aG =\ee_\aH=
  \begin{cases}
    +1 & \text{if $\aG = \GL_3$ and $\aH = \SL_3$}, \\
    -1 & \text{if $\aG = \GU_3$ and $\aH = \SU_3$}.
  \end{cases}
\end{equation}
In this setup we describe in an explicit and uniform manner
\begin{itemize} \renewcommand{\labelitemi}{$\circ$}
\item the images of the maps $\sh_{G,N}$, i.e.\ the conjugacy classes
  of subgroups of the finite groups $\aG(\kk)$ and $\aH(\kk)$ arising
  as shadows and
\item their fibres, i.e.\ the sets of characters that have a given
  shadow;
\end{itemize}
this description is the key to our proof of Theorem~\ref{thm:C} below.
More precisely, let
\begin{equation}\label{equ:types-def}
  \T = \T_{\mathsf{A}_2} = \{\Gsha, \Lsha, \Jsha, \Tasha, \Tbsha, \Tcsha,
  \Msha, \Nsha, \Kasha, \Kbsha\}
\end{equation}
be a set of ten distinct labels, from now referred to as \emph{shadow
  types}, and set $\T^{(1)} = \T$ and $\T^{(-1)} =\T \smallsetminus
\{\Kasha,\Kbsha\}$.  It turns out that there exist $\kk$-forms
\begin{equation}\label{equ:shadows}
  \aI^{\cS}_\ee = \aI^{\cS}_{\mathsf{A}_2,\epsilon}, \quad \cS\in\T^{(\ee)},
\end{equation}
of algebraic subgroups of $\GL_3$ such that the following hold: if $p
\coloneqq \cha(\kk) > 3e+3$, where $e=e(\lri,\Z_p)$ denotes the
absolute ramification index of~$\lri$, then
\begin{itemize}
  \renewcommand{\labelitemi}{$\circ$} 
\item the image of $\sh_{\aG(\lri),\aG^1(\lri)}$ is represented by the
  groups $\aI^{\cS}_\ee(\kk)$, $\cS\in\T^{(\ee)}$,
\item the image of $\sh_{\aH(\lri),\aH^1(\lri)}$ is represented by the
  groups $\aH(\kk) \cap \aI^{\cS}_\ee(\kk)$, $\cS\in\T^{(\ee)}$.
\end{itemize}
In fact, for $\ee = 1$ the groups $\aI^{\cS}_\ee$ can be defined
uniformly over~$\Z$, while for $\ee = -1$ the groups $\aI^{\cS}_\ee$
can be defined uniformly over $\kk$ using the Galois automorphism of
the quadratic extension $\kk_2$ of $\kk$.  For simplicity, we call the
groups $\aI^\cS_\ee(\kk)$, rather than their respective conjugacy
classes, the (\emph{character}) \emph{shadows} of $\aG(\lri)$.  The
character shadows $\aI^{\cS}_\ee(\kk)$, $\cS\in \T^{(\ee)}
\smallsetminus \{\Kasha,\Kbsha\}$, turn out to be centralisers of
elements of~$\aG(\kk)$.  Table~\ref{tab:shadows} lists the isomorphism
types of the groups $\aI^\cS_\ee(\KK)$, where $\KK$ denotes an
algebraic closure of~$\kk$.  In the table, $\Heis$ stands for the
Heisenberg group of upper uni-triangular $3\times3$ matrices and
$\aG_\text{a}$ denotes the additive group.  Further details of the
character shadows $\aI^{\cS}_\ee(\kk)$ and their intersections with
$\aH(\kk)$, including their isomorphism types and orders, are compiled
in Tables~\ref{tab:shadows.GLGU} and~\ref{tab:shadows.SLSU}, where the
notation $\aI^{\cS}_\ee(\kk) = \sigma(\kk)$ and $\aH(\kk) \cap
\aI^{\cS}_\ee(\kk) = \sigma'(\kk)$ is used; compare
Section~\ref{subsec:sim.shadows}.  The representation zeta functions
of the finite groups $\sigma(\kk)$ and $\sigma'(\kk)$ are recorded in
Proposition~\ref{pro:zeta.shadows}.

\begin{table}[htb!]
  \centering
  \caption{Algebraic groups giving rise to shadows in $\aG(\kk)$
    of type $\cS$}
  \label{tab:shadows}
  \begin{tabular}{|c||l|l|l|l|}
    \hline
    Type &  Isomorphism type of
    $\aI^{\cS}_\epsilon(\mathbf{K})$ for an algebraic closure $\mathbf{K}$ of $\kk$ \\
    \hline
    $\Gsha$ & $\GL_3(\mathbf{K})$ \\
    $\Lsha$ & $\GL_1(\mathbf{K}) \times \GL_2(\mathbf{K})$ \\
    $\Jsha$ & $\Heis(\mathbf{K}) \rtimes
    (\GL_1(\mathbf{K}) \times \GL_1(\mathbf{K}))$ \\
    $\Tasha, \Tbsha, \Tcsha$ &    $\GL_1(\mathbf{K}) \times
    \GL_1(\mathbf{K}) \times \GL_1(\mathbf{K})$ \\
    $\Msha$ & $\GL_1(\mathbf{K}) \times
    \GL_1(\mathbf{K})\times
    \aG_\text{a}(\mathbf{K})$\\
    $\Nsha$ & $\GL_1(\mathbf{K}) \times
    \aG_\text{a}(\mathbf{K}) \times
    \aG_\text{a}(\mathbf{K})$ \\  \hline
    $\Kasha, \Kbsha$  & $\GL_1(\mathbf{K}) \times
    \aG_\text{a}(\mathbf{K}) \times
    \aG_\text{a}(\mathbf{K})$ \\\hline
\end{tabular}
\end{table}

The same analysis applies, mutatis mutandis, to the character shadows
of groups of the form $\aG(\lri_\len)$ and $\aH(\lri_\len)$,
$\len\in\N$, where $\lri_\len \coloneqq \lri / \fp^{\len}$ with $\fp$
denoting the valuation ideal of~$\lri$.  Whilst groups of the form
$\aG(\lri)$ are clearly not rigid, the representation zeta functions
of their finite quotients $\aG(\lri_\len)$ are of considerable
interest.  In this situation we may even assume that $\lri$ has
positive characteristic, provided that $p = \cha(\kk)$ is large
compared to~$\len$.

In Proposition~\ref{pro:xi.sigma.len} we introduce, for
$\cS\in\T^{(\ee)}$ and $\len\in\N_0$, certain Dirichlet polynomials
$\Xi^{\cS}_{\ee,q,\len}(s) = \Xi^{\cS}_{\mathsf{A}_2,\ee,q,\len}(s)
\in \Z[\sixth][q,q^{-s}]$.  Their formal limits $\Xi^{\cS}_{\ee,q}(s)$
as $\len \rightarrow \infty$ are given in
Corollary~\ref{cor:infinite.xi}.  Deferring precise definitions and
motivations of these functions for the moment, we now state our second
main result.

\begin{thmABC}\label{thm:C}
  Let $\lri$ be a compact discrete valuation ring of residue
  characteristic $p = \cha(\kk)$.  Let $\aG,\aH$ be either
  $\GL_3,\SL_3$ or $\GU_3,\SU_3$ as above and $\len\in\N$.  Assume
  that $p \geq \min\{3\len,3e+3\}$ if $\cha(\lri) = 0$, and $p \geq
  3\len$ if $\cha(\lri) = p$. Then the following hold:
  \begin{align}
    \label{equ:zeta.G.fin}
    \zeta_{\aG(\lri_\len)}(s) &= q^{\len-1} \sum_{\cS \in \T^{(\ee)}}
    [\aG(\kk) : \aI^\cS_\ee(\kk)]^{-1-s} \;
    \zeta_{\aI^\cS_\ee(\kk)}(s) \; \Xi^\cS_{\ee,q,\len-1}(s),\\
    \label{equ:zeta.H.fin} \zeta_{\aH(\lri_\len)}(s) &=
    \sum_{\cS\in\T^{(\ee)}} [\aH(\kk) : (\aH(\kk)\cap
    \aI^{\cS}_\ee(\kk))]^{-1-s} \; \zeta_{\aH(\kk) \cap
      \aI^{\cS}_\ee(\kk)}(s) \; \Xi_{\ee,q,\len-1}^{\cS}(s).
  \end{align}
  Moreover, if $\cha(\lri) = 0$ and $p > 3e+3$, then
  \begin{equation}\label{equ:zeta.H.infin}
    \zeta_{\aH(\lri)}(s) = \sum_{\cS\in\T^{(\ee)}} [\aH(\kk) :
    (\aH(\kk)\cap \aI^{\cS}_\ee(\kk))]^{-1-s} \;
    \zeta_{\aH(\kk)\cap \aI^{\cS}_\ee(\kk)}(s) \;
    \Xi_{\ee,q}^{\cS}(s).
  \end{equation}
\end{thmABC}
For $\cha(\lri) = 0$, we record a self-contained formula for the zeta
functions $\zeta_{\aH(\lri)}(s)$ that can be read off from the
structural formulation in~\eqref{equ:zeta.H.infin}.  For this purpose
we define
\begin{equation}\label{def:iota}
  \iota(\aH,\kk)  = \iota(\ee,q) = \gcd(q-\ee,3), \qquad
  \text{where $q = \lvert \kk \rvert$.}
\end{equation}

\begin{corABC}\label{cor:D}
  Let $\lri$, $q$, $\aH$, and $\ee$ be as above.  Suppose that
  $\cha(\lri) = 0$ and $p > 3e+3$.  Then
  \begin{equation*}
    \zeta_{\aH(\lri)}(s) = \zeta_{\aH(\kk)}(s) + \psi_{\ee,q}(s),
  \end{equation*}
  where $\zeta_{\aH(\kk)}(s)$ is the zeta function of the finite group
  of Lie type $\SL_3(\kk)$ for $\ee=1$, respectively $\SU_3(\kk)$ for
  $\ee=-1$, given by the uniform formula
  \begin{equation}\label{equ:zeta.Lie.type.H}
    \begin{split}
      \zeta_{\aH(\kk)}(s) & = 1 + (q^2+\ee q)^{-s} + (q-1-\ee)
      (q^2 + \ee q+1)^{-s} \\
      & \quad + \hlf (q^2-q-1+\ee) (q^3-\ee)^{-s}
      + q^{-3s} + (q-1-\ee) (q^3+\ee q^2+q)^{-s} \\
      & \quad + \third (q^2+\ee q -2) \left( (q+\ee)(q-\ee)^2
      \right)^{-s}+ \genfrac{}{}{0.1pt}{1}{2}{3}\iota(\ee,q)^2
      \left( (q+\ee)(q-\ee)^2/\iota(\ee,q) \right)^{-s} \\
      & \quad + \sixth (q-\ee)(q-3-\ee) \left( (q^2+\ee q+1)(q+ \ee)
      \right)^{-s} \\
      & \quad + \third \iota(\ee,q)^2 \left( (q^2 + \ee
        q+1)(q+\ee)/\iota(\ee,q) \right)^{-s},
    \end{split}
  \end{equation}
  and
  \[
  \begin{split}
    \psi_{\ee,q}(s) & = \tfrac{1}{2} \frac{(q-1)(q - \ee) \left(2 + 2
        q^{-s} + (q-2) (q+1)^{-s} + q(q-1)^{-s} \right)}{1-q^{1-2s}} %
    \left(q^2(q^2 + \ee q + 1) \right)^{-s}
    \\
    & \quad + \frac{((q-\ee) + (q+\ee) \iota(\ee,q)^2 \left(
        (q-\ee)/\iota(\ee,q)\right)^{-s} + (q-1)(q-\ee)
      q^{-s})}{1-q^{1-2s}} \left((q^3-\ee)(q+\ee)\right)^{-s}
    \\
    & \quad + \tfrac{1}{6} \frac{(q-1) (q-\ee)^2 \left( q-2 + 2 q^{2-2s} - q^{1-2s}
      \right)}{(1-q^{1-2s})(1-q^{2-3s})}
    \left(q^3(q^2 + \ee q + 1)(q + \ee) \right)^{-s}
    \\
    &  \quad + \tfrac{1}{2} \frac{(q-1) (q^2-1)q
      (1-q^{-2s})}{(1-q^{1-2s})(1-q^{2-3s})} %
    \left(q^3(q^3 - \ee) \right)^{-s}
    \\
    &  \quad + \tfrac{1}{3} \frac{(q^2-1)(q^2 + \ee q + 1)}{1-q^{2-3s}} %
     \left( q^3(q^2-1) (q - \ee))\right)^{-s}
    \\
    &  \quad + \frac{(q-1) (q -\ee) q (1+q^{1-2s})}{(1-q^{1-2s})(1-q^{2-3s})} %
     \left(q^2 (q^3-\ee) (q+\ee) \right)^{-s}
    \\
    & \quad + \frac{(\iota(\ee,q)q)^2(1-q^{-2s})}{(1-q^{1-2s})(1-q^{2-3s})} %
     \left(q(q^3-\ee) (q^2-1)/\iota(\ee,q) \right)^{-s}
    \\
    & \quad + (\ee+1) \frac{(\iota(\ee,q)q^{1-s})^2}{(1-q^{1-2s})(1-q^{2-3s})} %
     \left((q^3-1) (q^2-1) q/\iota(\ee,q) \right)^{-s}.
  \end{split}
  \]
  Here the order of summation follows the ordering of the shadow types
  $\Lsha$, $\Jsha$, $\Tasha$, $\Tbsha$, $\Tcsha$, $\Msha$, $\Nsha$,
  and $\Kasha,\Kbsha$ in Table~\ref{tab:shadows}.
\end{corABC}

\begin{remark}\label{rem:zeros.special.A2}
  Assume that $\cha(\lri) = 0$. For general reasons, the zeta
  functions $\zeta_{\aH(\lri)}(s)$ vanish at $s=-2$ for $p>2$;
  cf.~\cite[Corollary~2]{GonzalezJaikinKlopsch/14}. Assume further
  that $p$ is as in Corollary~\ref{cor:D}. Computations with the
  explicit formulae in Corollary~\ref{cor:D} suggest that then
  $\zeta_{\SL_3(\lri)}(s)$ has no further integral zeros. In contrast,
  $\zeta_{\SU_3(\lri)}(s)$ also vanishes at~$s=0$. In addition, it
  vanishes at $s=-1$ if and only if~$\iota(-1,q)= \gcd(q+1,3) = 1$.

  We further remark that the special values of the zeta functions of
  the finite groups $\zeta_{\GU_3(\lri_\len)}(s)$ -- as far as they
  are given by \eqref{equ:zeta.G.fin} -- at $s=-1$, i.e.\ the sum of
  the character degrees of these finite groups, yield the number of
  invertible symmetric matrices in $\GU_3(\lri_\len)$, viz.\
$$\zeta_{\GU_3(\lri_\len)}(-1) = (1+q^{-1})(1+q^{-3})q^{6\len}.$$
The corresponding assertion for groups of the form $\GL_3(\lri_\len)$
seems to hold only for $\len=1$, i.e.\ for $\GL_3(\kk)$. In this case,
the phenomenon is a special case of \cite[Corollary~5.2]{TV07},
concerning the sums of character degrees of unitary groups of the form
$\GU_d(\kk)$. This result, in turn, is a unitary analogue of results
of Gow and Klyachko for groups of the form $\GL_d(\kk)$; see
\cite[Section~5.2.2]{TV07}.
\end{remark}

Formulae for the representation zeta functions of principal congruence
subgroups of the groups considered in Theorem~\ref{thm:C} are provided
in Theorem~\ref{thm:J} below.

A key tool in the analysis of zeta functions of groups is the Kirillov
orbit method, describing the irreducible characters of suitable
pro-$p$ subgroups of $p$-adic analytic groups such as $\aG(\lri)$ in
terms of co-adjoint orbits in the duals of the corresponding $\Z_p$-Lie
lattices; see Section~\ref{sec:kom.clifford} for details.  This
approach leads naturally to the study of similarity classes of
$\fp$-adic matrices, where invariants called \emph{similarity class
  shadows} -- very much analogous to the character shadows of the
matrix groups considered in the present section -- play an important
role, as we explain next.

\subsection{Similarity classes and their
  shadows} \label{subsec:sim.shadows} Let $\lri$ be a compact discrete
valuation ring with valuation ideal~$\fp$ and residue field $\kk$ of
cardinality~$q$.  We impose no restriction on the characteristic
of~$\lri$. Recall that $\lri_\len = \lri / \fp^{\len}$ for
$\len\in\N$.  The problem of classifying and enumerating
\emph{similarity classes} in $\mathsf{Mat}_n(\lri_\len)$, or
equivalently orbits of the adjoint action of $\GL_n(\lri)$ on
$\gl_n(\lri_\len)$, has attracted much attention over the years.  In
the field case, i.e.\ for $\len=1$, a classification is achieved, for
instance, by the Frobenius normal form.  For the case $\len=2$ see,
for example, \cite{JamborPlesken/12, PrasadSinglaSpallone/13}. In
Theorem~\ref{thm:irredundant.list.for.M3x3} we give a complete and
irredundant list of representatives of the similarity classes in
$\gl_3(\lri_\len)$, for any~$\len\in\N$, building on and refining
results from \cite{APOV}.  We also study the analogous problem of
classifying and enumerating similarity classes of anti-hermitian
integral $\fp$-adic matrices, i.e.\ orbits of the adjoint action of
the unitary group $\GU_n(\lri)$ on the unitary Lie lattices
$\gu_n(\lri_\len)$. Here the residue characteristic of $\lri$ is
assumed to be odd, and the relevant objects are defined by means of
the non-trivial Galois automorphism of an unramified quadratic
extension of~$\lri$; see~\eqref{def:unitary} for details. In
Theorem~\ref{thm:representatives.for.GU3.action.gu3} we provide an
explicit list of matrices parametrising $\GU_3(\lri)$-similarity
classes in $\gu_3(\lri_\len)$, for any~$\len\in\N$.

\subsubsection{Similarity class shadows}
A fundamental idea of the current paper is to organise similarity
classes by invariants called shadows, which we now explain. Given
$\len\in\N$ and $A\in\gl_n(\lri_\len)$, the \emph{group centraliser
  shadow} $\sh_{\GL}(A)$ of $A$ is the image
$\overline{\Cen_{\GL_n(\lri)}(A)} \leq \GL_n(\kk)$ of the centraliser
$\Cen_{\GL_n(\lri)}(A)$ under reduction modulo~$\fp$.  Evidently,
similar matrices have conjugate group centraliser shadows.  Roughly
speaking, the \emph{shadow $\shGC(\cC)$ of a similarity class} $\cC
\subset \gl_n(\lri_\len)$ is the conjugacy class of $\sh_{\GL}(A)$
in $\GL_n(\kk)$, for any $A\in\cC$.  (More precisely, we also keep
track of Lie objects associated to the shadows; see
Definition~\ref{def.of.shadows} and the discussion following it.)  We
write $\Sh_{\GL_n(\lri)}$ for the set of shadows arising.  Similar to
the definitions for $\gl_n(\lri)$, we define shadows of similarity
classes of anti-hermitian integral $\fp$-adic matrices.  Broadly
speaking, these may be thought of as conjugacy classes of subgroups
in~$\GU_n(\kk)$; see Definition~\ref{def.of.unitary.shadows}.  We
write $\Sh_{\GU_n(\lri)}$ for the set of shadows in the unitary
setting.

In order to discuss the general linear and unitary scenarios for type
$\mathsf{A}_{n-1}$ in parallel, let $\aG$ be one of the $\lri$-group
schemes $\GL_n, \GU_n$ and, accordingly, let $\ag$ be one of the
$\lri$-Lie lattice schemes $\gl_n, \gu_n$.  As above, $\GU_n$ and
$\gu_n$ are defined over $\lri$ using the non-trivial Galois
automorphism of an unramified quadratic extension of $\lri$.  We
continue to use the parameter $\ee = \ee_\aG \in \{1,-1\}$ defined
in~\eqref{equ:def.epsilon}.  The following questions naturally present
themselves:
\begin{enumerate}
\item Describe the set $\Sh_{\aG(\lri)}$ of shadows. How does it vary
  with the ring $\lri$?
\item Let $\len\in\N$ and let $\cC$ be a similarity class in
  $\ag(\lri_\len)$.  Which shadows arise among the similarity classes
  $\wt{\cC}$ in $\ag(\lri_{\len+1})$ lifting $\cC$, and with what
  multiplicities?  What can be said about the cardinalities $\lvert
  \cC \rvert$ and $\lvert \wt{\cC} \rvert$?
\end{enumerate}
For $n=3$, i.e.\ groups and Lie lattices of type $\mathsf{A}_2$, we
answer these questions completely.  Let us restrict to this setting.
Theorems~\ref{thm:G.shad.graph} and \ref{thm:U.shad.graph}, two of the
paper's main technical results, yield:
\begin{enumerate}
\item The elements of $\Sh_{\aG(\lri)}$ are represented by the groups
  $\aI^{\cS}_\ee(\kk)$, $\cS\in\T^{(\ee)}$; cf.~\eqref{equ:shadows}.
\item Let $\sigma, \tau \in \Sh_{\aG(\lri)}$.  Given a similarity
  class $\cC$ in $\ag(\lri_\len)$ of shadow $\sigma$, the number of
  similarity classes $\wt{\cC}$ in $\ag(\lri_{\len+1})$ of shadow
  $\tau$ lifting $\cC$ is given by a rational polynomial in~$q$,
  depending only on the types of $\sigma$ and $\tau$, but not
  on~$\len$ or $\ee$.  The quotients $\lvert \wt{\cC} \rvert / \lvert
  \cC \rvert$ are given by integral polynomials in~$q$, depending only
  on the types of $\sigma$ and $\tau$ and mildly on~$\ee$, but not
  on~$\len$.
\end{enumerate}
Theorem~\ref{thm:G.shad.graph} and \ref{thm:U.shad.graph} deliver
these groups and polynomials explicitly; cf.\
Tables~\ref{tab:shadows.iso.GL} and \ref{tab:shadows.iso.GU} for the
shadows' isomorphism types and Table~\ref{tab:branch.rules.A2} for the
polynomial data.

Our results on shadows unveil a remarkable recursive structure on the
collection $\cQ^\ag_\lri \coloneqq \coprod_{\len\in\N_0}
\Ad(\aG(\lri)) \backslash \ag(\lri_\len)$ of similarity classes over
all $\len \in \N_0$.  Indeed, informally speaking we may view
$(\Ad(\aG(\lri)) \backslash \ag(\lri_\len))_{\len\in\N_0}$ as a
memory-less stochastic process with finite state space
$\Sh_{\aG(\lri)}$, indexed by $\len\in\N_0$: in order to enumerate,
for instance, similarity classes in $\ag(\lri_{\len+1})$ it suffices
to enumerate similarity classes in $\ag(\lri_\len)$, sorted by their
shadows, and process the `transition data' provided by
Table~\ref{tab:branch.rules.A2}.  Formally, we define on
$\cQ^\ag_\lri$ the structure of an infinite rooted \emph{similarity
  class tree}; see Definitions~\ref{def:graph.Qgl}
and~\ref{def:graph.Qgu}.  Remarkably, the tree's structure is
completely determined by local branching rules, given by the data
provided by Theorems~\ref{thm:G.shad.graph}
and~\ref{thm:U.shad.graph}.  This data may also be organised in a
finite \emph{shadow graph} $\Gamma^{(\ee)}$ with vertex set
$\T^{(\ee)}$; cf.\ Figure~\ref{fig:shadow.graph.A2}.

\subsubsection{Enumerating similarity classes}
Our first application of the concept of similarity class shadows is to
the enumeration of similarity classes of integral $\fp$-adic
$3\times3$ matrices.  As above, let $\aG$ be one of the $\lri$-group
schemes $\GL_3,\GU_3$ and accordingly $\ag$ one of the Lie lattice
schemes $\gl_3,\gu_3$; let $\ee = \ee_\aG \in \{1,-1\}$ as
in~\eqref{equ:def.epsilon}.  We write $\Sh$ for the shadow set
$\Sh_{\aG(\lri)}$.  In Proposition~\ref{pro:sim.shadow.fin} we give
explicit formulae for the \emph{partial similarity class zeta
  functions}
\[
\gamma^{\sigma}_\len(s) = \sum_{\stackrel{\cC \in \Ad(\aG(\lri))
    \backslash \ag(\lri_\len)}{\sh(\cC)=\sigma}} \lvert \cC
\rvert^{-s} \quad \text{for $\sigma \in \Sh$ and $\len \in \N_0$,}
\]
enumerating similarity classes in $\ag(\lri_\len)$ of shadow~$\sigma$;
cf.\ Definition~\ref{def:sim.class.zeta}.

These formulae and variants thereof appear throughout the paper.
Indeed, for our applications to representation zeta functions it is
useful to consider the related Dirichlet polynomials
\[
\xi_\len^\sigma(s) = [\aG(\kk) : \aI^{\cS}_\ee(\kk)]^{1+s/2} \;
q^{-\len} \; \gamma^\sigma_\len(s/2) \quad \text{for $\sigma\in\Sh$ of
  type $\cS$ and $\len\in\N_0$;}
\]
cf.\ Definition~\ref{def:xi}.  In Proposition~\ref{pro:xi.sigma.len}
we establish that the Dirichlet polynomials $\xi^\sigma_\len(s)$ are,
in fact, equal to the functions $\Xi^\cS_{\ee,q,\len}(s)$ featuring in
Theorem~\ref{thm:C}; the proposition provides explicit formulae for
these functions.

Our first application, however, of the \emph{similarity class zeta
  functions} 
\[
\gamma_\len(s) \coloneqq \sum_{\sigma\in\Sh}\gamma^\sigma_\len(s)
\]
is based on the observation that for, all $\len \in \N_0$,
\[
s_\len(\ag(\lri)) \coloneqq \gamma_\len(0) = \lvert \Ad(\aG(\lri))
\backslash \ag(\lri_\len) \rvert
\]
is just the total number of similarity classes in~$\ag(\lri_\len)$.

\begin{thmABC}\label{thm:sim.zeta.local}
  Let $\lri$, $\aG$, $\ag$, and $\ee = \ee_\aG$ be as above; if
  $\ee=-1$ suppose that $\lri$ has odd residue characteristic.  Then
  \begin{equation}\label{equ:sim.local}
    \zeta^{\mathrm{sc}}_{\ag(\lri)}(s) \coloneqq \sum_{\len=0}^\infty
    s_\len(\ag(\lri)) q^{-\len s} = \frac{1+\ee
      q^{2-2s}}{(1-q^{1-s})(1-q^{2-s})(1-q^{3-s})}.
  \end{equation}
\end{thmABC}

For $\ee=1$, this confirms the relevant part of
\cite[Theorem~5.2]{APOV}; for $\ee=-1$ the formula is new.  In any
case, the local results may be put in an ad\`elic context as
follows. Let $k$ be a number field with ring of integers $\Gri$.  Let
$\mathbf{G}$ be one of the $k$-algebraic groups $\GL_3$ or
$\GU_3(K,f)$, where the unitary group $\GU_3(K,f)$ is defined with
respect to the standard hermitian form $f$ associated to the
non-trivial Galois automorphism of a quadratic extension $K$ of~$k$.
Accordingly, let $\fg$ be one of the Lie algebra schemes $\gl_3$ or
$\gu_3(K,f)$.  Put $\ee_\mathbf{G} = 1$ if $\mathbf{G} = \GL_3$, and
$\ee_\mathbf{G} = -1$ if $\mathbf{G} = \GU_3(K,f)$.  For the ring of
$S$-integers $\Gri_S$, where $S$ is a finite set of places of $k$
including all the archimedean ones, $\cV^\infty_k \subset S$, we
consider the Dirichlet series
\begin{equation}\label{def:sim.global}
  \zeta^{\mathrm{sc}}_{\fg(\Gri_S)}(s)  = \sum_{n=1}^\infty
  s_n(\fg(\Gri_S)) n^{-s} \coloneqq \sum_{I \triangleleft \Gri_S} \lvert
  \Ad(\mathbf{G}(\wh{\Gri_S})) \backslash \fg(\Gri_S/I)
  \rvert \; [\Gri_S:I]^{-s},
\end{equation}
where -- in absence of the strong approximation property for
$\mathbf{G}$ -- we count adjoint orbits of the congruence completion
$\mathbf{G}(\wh{\Gri_S}) = \varprojlim_{I \triangleleft \Gri_S}
\mathbf{G}(\Gri_S/I)$ rather than $\mathbf{G}(\Gri_S)$.  As $\Gri_S$
is a Dedekind domain, this Dirichlet series admits the Euler product
\begin{equation}\label{equ:euler.zeta.sc}
  \zeta^{\textup{sc}}_{\fg(\Gri_S)}(s) = \prod_{v\not\in S}
  \zeta^{\textup{sc}}_{\fg(\Gri_v)}(s).
\end{equation}
Writing $\zeta_{k}(s) = \prod_{v \not\in \cV^\infty_k}
(1-\textrm{N}(\fp_v)^{-s})^{-1}$ for the Dedekind zeta function of the
number field~$k$, and $\zeta_{k,S}(s) = \prod_{v \not\in S}
(1-\textrm{N}(\fp_v)^{-s})^{-1}$ for the same product with the factors
indexed by non-archimedean places in $S$ omitted, we obtain the
following corollary.

\begin{corABC}\label{cor:F}
  Let $\Gri_S \subset k$ and $\mathbf{G}$, $\fg$,
  $\ee_\mathbf{G}$ be as above; if $\ee_\mathbf{G} = -1$ suppose that
  $S$ includes all dyadic places of $k$ as well as those places which
  ramify in the quadratic extension $K$ of $k$ defining $\mathbf{G} =
  \GU_3(K,f)$.  Then
  \begin{equation*}
    \zeta^{\mathrm{sc}}_{\fg(\Gri_S)}(s) =
    \begin{cases}
      \hfill \zeta_{k,S}(2s-2) \; \zeta_{k,S}(4s-4)^{-1}
      \prod_{i=1}^3 \zeta_{k,S}(s-i) & \text{if $\ee_\mathbf{G} =1$,}\\
      \zeta_{k,S}(2s-2)^{-1}\zeta_{K,S}(2s-2) \;
      \zeta_{k,S}(4s-4)^{-1} \prod_{i=1} ^3\zeta_{k,S}(s-i) & \text{if
        $\ee_\mathbf{G} = -1$.}
    \end{cases}
  \end{equation*}
  In particular, there exists an invariant
  $\delta(\ee_\mathbf{G},\Gri_S) \in \R_{>0}$ such that
  \[
  \delta(\ee_\mathbf{G},\Gri_S) = \lim_{N\rightarrow\infty}
  \frac{\sum_{n=1}^N s_n(\fg(\Gri_S))}{N^4}.
  \]
  For instance, if $\ee_\mathbf{G} = 1$ and $S = \cV_k^\infty$
  comprises just the archimedean places of~$k$, then
  \[
  \delta(1,\Gri) = \frac{\zeta_k(6)\zeta_k(3)\zeta_k(2)}{4
    \zeta_k(12)}.
  \]
\end{corABC}

We briefly return to the local setting.  Evaluating $\gamma_\len(s)$
in $s=0$ as above means, of course, to disregard most of the
information encoded in the similarity class zeta functions.  In our
second application, we retain this information and define suitable
limits (as $\len \rightarrow \infty$) and Euler products, which we now
explain.  In Proposition~\ref{pro:limit.gamma} we verify that the
normalised polynomials $q^{-\len} \gamma_\len^\sigma(s)$ converge
coefficientwise. Apart from the exceptional case that $p=\cha(\kk)$
divides $n$, the presence of scalar matrices implies that the
coefficients of the Dirichlet polynomials $\gamma_\len^\sigma(s)$ are
actually integers divisible by~$q^{\len}$, whence the normalised
polynomials $q^{-\len} \gamma_\len^\sigma(s)$ have integral
coefficients.  The limit functions $\lim_{\len\rightarrow\infty}
q^{-\len}\gamma_\len^\sigma(s)$ are recorded in
Corollary~\ref{cor:infinite.gamma}.  They, too, may be put in an
ad\`elic context, as follows.

Let $k$ be a number field with ring of integers $\Gri$.  As above, let
$\mathbf{G}$ be one of the $k$-algebraic groups $\GL_3$ or
$\GU_3(K,f)$ and, accordingly, let $\fg$ be one of the Lie algebra
schemes $\gl_3$ or $\gu_3(K,f)$.  Let $S$ be a finite set of places of
$k$ including all the archimedean ones, $\cV^\infty_k \subset S$.  For
any non-zero ideal $I \triangleleft \Gri_S$, consider the normalised
Dirichlet generating polynomial
\[
Z_{\fg(\Gri_S/I)}(s) \coloneqq [\Gri_S/I]^{-1} \sum_{\cC\in
  \Ad(\aG(\wh{\Gri_S})) \backslash \fg(\Gri_S/I)} \lvert \cC
\rvert^{-s}
\]
enumerating similarity classes in $\fg(\Gri_S/I)$ by their
cardinality.  We consider the Dirichlet series
\begin{equation}\label{equ:def.Z.global}
  Z_{\fg(\Gri_S)}(s) = \sum_{n=1}^\infty
  \simi_n(\fg(\Gri_S)) n^{-s} \coloneqq \lim_{I \triangleleft \Gri_S}
  Z_{\fg(\Gri_S/I)}(s).
\end{equation}
Informally speaking, $\simi_n(\fg(\Gri_S))$ is the number of
similarity classes of cardinality $n$ in $\fg(\Gri_S/I)$ modulo
scalars, for ideals $I$ such that the index $[\Gri_S:I]$ is divisible
by a `relatively large' power of~$n$.  By construction, the Dirichlet
generating function $Z_{\fg(\Gri_S)}(s)$ satisfies an Euler product
decomposition of the form
\begin{equation}\label{equ:euler.zeta.sim}
  Z_{\fg(\Gri_S)}(s) = \prod_{v\not\in S}
  \lim_{\len\rightarrow\infty} q^{-\len}\gamma_\len(s).
\end{equation}

\begin{thmABC}\label{thm:G}
  Let $\Gri_S \subset k$ and $\mathbf{G}$, $\fg$,
  $\ee_\mathbf{G}$ be as above; if $\ee_\mathbf{G} = -1$ suppose that
  $S$ includes all dyadic places of $k$ as well as those places which
  ramify in the quadratic extension $K$ of $k$ defining $\mathbf{G} =
  \GU_3(K,f)$.  Then the following hold:
  \begin{enumerate}
  \item The abscissa of convergence of $Z_{\fg(\Gri_S)}(s)$
    is equal to $1/2$.
  \item The zeta function $Z_{\fg(\Gri_S)}(s)$ has meromorphic
    continuation to the complex half-plane $\{s\in\C \mid \real(s) >
    2/5 \}$. The only pole of $Z_{\fg(\Gri_S)}(s)$ in this domain is a
    double pole at $s=1/2$.
  \item There exists an invariant
    $\delta'(\ee_\mathbf{G},\Gri_S)\in\R_{>0}$ such that
    \[
    \delta'(\ee_\mathbf{G},\Gri_S) =
    \lim_{N\rightarrow\infty}\frac{\sum_{n=1}^N
      \simi_n(\fg(\Gri_S))}{N^{1/2}\log N}.
    \]
  \end{enumerate}
\end{thmABC}

To put Theorem~\ref{thm:G} into perspective, we remark that
$Z_{\fg(\Gri_S)}(s/2)$ can be regarded as an `approximation' of the
non-achimedean part $\prod_{v \not\in S}
\zeta_{\mathbf{H}(\Gri_v)}(s)$ of the representation zeta
function~$\zeta_{\mathbf{H}(\Gri_S)}(s)$ in~\eqref{equ:euler}; cf.\
\eqref{equ:zeta.H.infin}.  The zeta functions $Z_{\fg(\Gri_S)}(s)$ may
well turn out to be more tractable than representation zeta functions
and thus serve as a tool for studying the latter.

\subsection{Character degrees: `Ennola duality' and estimates}
Our results -- or sometimes rather their proofs -- have a number of
consequences regarding the finer asymptotic and arithmetic properties
of character degrees of the groups under consideration.

Given a group $G$, we denote the collection of its \emph{irreducible
  character degrees} by
\[
\cd(G) = \{\chi(1) \mid \chi\in\Irr(G)\},
\]
and, for any prime $p$, we write $\cd(G)_{p'} = \{\chi(1)_{p'} \mid
\chi\in\Irr(G)\}$ for the prime-to-$p$ parts of the irreducible
character degrees of~$G$.  Let~$n\in\N$.  In the 1960s Ennola observed
an intriguing duality between the character tables of the finite
groups $\GL_n(\F_q)$ and~$\GU_n(\F_q)$.  In particular, he noted that
there exist a finite index set $I=I(n)$ and polynomials $g_i \in
\Z[t]$, $i\in I$, such that
\[
\cd(\GL_n(\F_q)) = \{g_i(q)\mid i\in I \} \qquad \text{and} \qquad
\cd(\GU_n(\F_q)) = \{(-1)^{\deg g_i}g_i(-q)\mid i\in I \};
\]
cf.~\cite{Green, Ennola/63} and \cite[Chapter~IV,
  Section~6]{Macdonald/95}.  This phenomenon, known as `Ennola
duality', was explained only later, culminating in work by Kawanaka;
cf.~\cite{K85} and also~\cite{TV07}.  While we cannot offer an
analogous theory for the character degrees of compact $p$-adic Lie
groups $\GL_n(\lri)$ and~$\GU_n(\lri)$, our approach allows us to
generalise Ennola's observation as follows.

Let $\lri$ be a compact discrete valuation ring of residue
characteristic~$p$ and residue cardinality~$q$.  If $\cha(\lri) = 0$,
let $e=e(\lri,\Z_p)$ denote the absolute ramification index of~$\lri$.

\begin{thmABC} \label{thm:H} Let $\lri$ be as above.  Let $\aG$ be one
  of the $\lri$-group schemes $\GL_3,\GU_3$, and let $\ee = \ee_\aG
  \in \{1,-1\}$ as in~\eqref{equ:def.epsilon}.  Let $\len\in\N$.
  Suppose that $p \geq \min\{3\len, 3e+3 \}$ if $\cha(\lri) = 0$, and
  $p \geq 3\len$ if $\cha(\lri) = p$.  Then the prime-to-$p$ parts of
  the character degrees of $\aG(\lri_\len)$ are as follows.
  \begin{equation*}
    \cd(\aG(\lri_\len))_{p'} = \begin{cases}\left\{1,q+\ee, q^2+\ee q
      +1, (q+\ee)(q^2+\ee q +1),q^3-\ee, (q-\ee)^2(q+\ee) \right\} &
      \textup{ for $\len=1$,}\\ \cd(\aG(\lri_1))_{p'} \cup \{
      (q^3-\ee)(q^2-1), (q^3-\ee)(q+\ee) \} & \textup{ for $\len\geq
        2$}.\end{cases}
  \end{equation*}
  Furthermore, for all $g \in \Z[t]$ and $\len\in\N$, 
  \begin{equation}\label{equ:ennola}
  g(q) \in \cd(\GL(\lri_\len)) \quad \text{if and only if} \quad
  (-1)^{\deg g}g(-q) \in \cd(\GU(\lri_\len)).
  \end{equation}
\end{thmABC}
It is of great interest to determine the precise scope of this
phenomenon, in the first place for the groups $\GL_n$ and $\GU_n$ for
$n>3$; see Section~\ref{subsec:ennola}.  We remark that whilst the
theorem addresses the character degrees' prime-to-$p$ parts, the
explicit formulae underpinning its proof would also allow for a
uniform, albeit somewhat technical description of the powers of $q$
entering into the character degrees.

Our next main result concerns the character degrees of groups of the
form $\SL_3(\lri)$ and $\SU_3(\lri)$.  Let $\aH$ denote one of the
$\lri$-group schemes $\SL_3,\SU_3$.  By convention, the \emph{level}
of an irreducible character $\chi \in \Irr(\aH(\lri))$ is equal to
$\len-1$, where $\len \in \N$ is minimal such that $\chi$ is trivial
on the $\len$-th principal congruence subgroup~$\aH^\len(\lri)$.  The
following theorem relates the degree of an irreducible character in
$\Irr(\aH(\lri))$ to its level.
\begin{thmABC} \label{thm:I} There exist absolute constants $C_1, C_2
  \in \R_{>0}$ such that the following holds.  Let $\lri$ be as above.
  Let $\aH$ be one of the $\lri$-group schemes $\SL_3,\SU_3$, and let
  $\ee = \ee_\aH \in \{1,-1\}$ as in~\eqref{equ:def.epsilon}.  Let
  $\len \in \N$.  Suppose that $p \geq \min\{3 \len, 3e+3 \}$ if
  $\cha(\lri) = 0$, and $p \geq 3\len$ if $\cha(\lri) = p$.  For every
  non-trivial $\chi \in \Irr(\aH(\lri))$ of level $\len-1$ the degree
  of $\chi$ is bounded by the inequalities:
  \[
  C_1 q^{2\len} < \chi(1) < C_2q^{3\len}.
  \]
\end{thmABC}
In fact, our proof of Theorem~\ref{thm:I} yields slightly more precise
estimates. The constants $C_1$ and $C_2$, for instance, may be taken
arbitrarily close to $1$ at the cost of excluding finitely many values
of~$q$.  Note that the groups $\GL_3(\lri)$ and $\GU_3(\lri)$ have
$1$-dimensional representations of arbitrary level, namely those
factoring through the determinant map.  Therefore there is no
non-trivial lower bound for the irreducible character degrees of these
groups in relation to the level.  However, similar considerations as
in the proof of Theorem~\ref{thm:I} apply to these groups so that the
upper bound holds for them as well.  Bounds as in Theorem~\ref{thm:I}
are of interest, for instance, in the study of the `Gelfand-Kirillov
dimensions' of admissible smooth complex representations of the
locally compact group $\aH(\lfi)$, where $\lfi$ denotes the fraction
field of~$\lri$; cf.\ \cite[Remark~1.19]{CalegariEmerton/12}.

\subsection{Principal congruence subgroups}
Finally we record applications to principal congruence subgroups and
subquotients defined in terms of the congruence filtration.  As above,
let $\lri$ denote a compact discrete valuation ring of residue
characteristic~$p$ and residue cardinality~$q$.  If $\cha(\lri) = 0$,
let $e=e(\lri,\Z_p)$ denote the absolute ramification index of~$\lri$.
Let $\aG$ be one of the $\lri$-group schemes $\GL_3,\GU_3$ and,
accordingly, let $\aH$ be one of the $\lri$-group schemes
$\SL_3,\SU_3$, the choice being reflected in the value of the
parameter $\ee = \ee_\aG = \ee_\aH \in \{1,-1\}$;
see~\eqref{equ:def.epsilon}.  For $m \in \N$, let $\aG^m(\lri)$ and
$\aH^m(\lri)$ denote the $m$th principal congruence subgroups of
$\aG(\lri)$ and $\aH(\lri)$.  We put
\begin{equation*}\label{def:u}
  u_\ee(t) = \ee t^3 + t^2 - t -\ee - t^{-1} \in \Z[t,t^{-1}].
\end{equation*}
Our last main result generalises and yields an different approach
to~\cite[Theorem~E]{AKOV1} which, for $\cha(\lri)=0$ and $p > 3$,
implies that for all $m\in\N$ with $m \geq e/(p-2)$,
\begin{equation}\label{equ:duke}
\zeta_{\aH^m(\lri)}(s) = q^{8m} \frac{1 + u_\ee(q)q^{-3-2s} +
  u_\ee(q^{-1})q^{-2-3s} + q^{-5-5s}}{ (1 - q^{1-2s})(1 - q^{2-3s})}.
\end{equation}
Recall the notation introduced in Section~\ref{subsec:shadows.types}
, in particular the Dirichlet polynomials
$\Xi^{\cS}_{\ee,q,\len}(s)$ and their limits $\Xi^{\cS}_{\ee,q}(s)$,
first mentioned just before Theorem~\ref{thm:C}.

\begin{thmABC} \label{thm:J} Let $\lri$ and $\aG$, $\aH$, $\ee =
  \ee_\aG = \ee_\aH$ be as above.  Let $\len,m \in \N$ with $\len \geq
  m$.  Suppose that $p>3$; suppose further that $m \geq \min \{
  \len/p, e/(p-2) \}$ if $\cha(\lri) = 0$, and $m \geq \len/p$ if
  $\cha(\lri) = p$.  Then
  \begin{align}
    \zeta_{\aH^m(\lri)/\aH^\len(\lri)}(s) & =
    \begin{cases} q^{8(\len-m)}
      & \text{if  $\len \leq 2m$,} \\
      q^{8(m-1)}\sum_{\cS\in\T^{(\ee)}} \Xi^{\cS}_{\ee,q,\len-2m+1}(s)
      & \text{if $\len > 2m$,}
    \end{cases} \\
    \zeta_{\aG^m(\lri)/\aG^\len(\lri)}(s) & =
    q^{\len-m}\zeta_{\aH^m(\lri)/\aH^\len(\lri)}(s).
  \end{align}
  Moreover, if $\cha(\lri) = 0$ and $m \geq e/(p-2)$ then
  \begin{equation}\label{equ:zeta.princ.inf}
    \zeta_{\aH^m(\lri)}(s) = q^{8(m-1)}
    \sum_{\cS\in\T^{(\ee)}} \Xi^\cS_{\ee,q}(s).
  \end{equation}
\end{thmABC}

\begin{remark}
  The zeta functions $\zeta_{\aH^m(\lri)}(s)$ vanish at $s=-2$ for
  $p>2$; cf.~\cite[Corollary~2]{GonzalezJaikinKlopsch/14}. Inspection
  of the right hand side of \eqref{equ:duke} shows that it vanishes,
  in addition, at $s=-1$ if $\ee=1$ and at $s=0$ if $\ee=-1$, but not
  vice versa.
\end{remark}

\subsection{Outlook and conjectures}  The results discussed above
raise many interesting questions.  We highlight and discuss some of these.

\subsubsection{Analytic properties of zeta functions of arithmetic
  groups} It is of interest to investigate whether the assertions in
Corollary~\ref{cor:B} for $\Gamma$ hold more generally also for
arithmetic groups of type $\mathsf{A}_2$ satisfying just the~wCSP.
Let $\mathbf{H}(\Gri_S)$ be such a group and $\Gamma \leq
\mathbf{H}(\Gri_S)$ as in the corollary. That the abscissae of
convergence of $\zeta_\Gamma(s)$ and $\zeta_{\mathbf{H}(\Gri_S)}(s)$
coincide is well known (see, for instance, \cite[Corollary~4.5]{LL}),
but we do not know whether they also share the finer analytic
properties described in Corollary~\ref{cor:B}, such as meromorphic
continuation, pole order et cetera.  Note that subgroups of arithmetic
groups satisfying the sCSP also satisfy this property.

As we mentioned in Remark~\ref{rem:corB}, Corollary~\ref{cor:B}
transcends -- for the groups it covers -- general results for
arithmetic groups under base extension. It is interesting to decide
whether such uniformity also governs the analytic behaviour of
representation zeta functions of arithmetic groups of other types.

\subsubsection{Similarity classes of matrices}
Our results exhibit similarity class shadows and associated
combinatorial structures as an effective tool to analyse and enumerate
similarity classes of integral $\fp$-adic $3\times3$ matrices,
uniformly in the linear and unitary setting. Of particular relevance
is the shadows' capacity to uniformly describe the lifting behaviour
of similarity classes in Lie lattices such as $\gl_3(\lri_\len)$. It
is of great interest to investigate whether shadows of similarity
classes in more general Lie lattices, say of type $\mathsf{A}_{n-1}$
or the other classical types, also share this feature.  The (simpler)
case of type $\mathsf{A}_1$ is treated in Appendix~\ref{sec:A1}.

Another remarkable fact in type $\mathsf{A}_2$ is that the shadows are
represented by a finite number of algebraic subgroups of $\GL_3$; in
particular, their number is uniformly bounded independently of the
residue cardinality~$q$.  We do not know whether this is a general
phenomenon, even in type~$\mathsf{A}_{n-1}$; for type $\mathsf{A}_1$
see Appendix~\ref{sec:A1}.  It is worth exploring potential
connections between shadows in `semisimple' Lie lattices and
decomposition classes; cf.\ \cite{BorhoKraft/79, Broer/98}.

The ad\`elic results Corollary~\ref{cor:F} and Theorem~\ref{thm:G} are
phrased in such a way that the relevant Euler products
\eqref{equ:euler.zeta.sc} and \eqref{equ:euler.zeta.sim} extend over
places for which our results give precise formulae for the involved
Euler factors.  It seems reasonable to expect that global features
such as the ad\`elic zeta functions' abscissae of convergence,
meromorphic continuation, pole order et cetera remain unchanged in the
general case, in which the Euler products are enlarged by finitely
many `exceptional' factors.  In particular, it would be of interest to
set up a universal $\fp$-adic integration formalism that covers these
factors, too; cf.\ \cite[Theorem~B]{AKOV1} and~\cite{BDOP12}.

\subsubsection{Positive characteristic}
All our local results assume that the discrete valuation ring $\lri$
has characteristic $0$ or -- in the case of finite groups over rings
of the form $\lri_\len$ -- residue characteristic large in comparison
to~$\len$.  This restriction is owed to the limitations of the
linearisation techniques we use, which allow us to employ the Kirillov
orbit method. It is natural to ask for results in the remaining cases,
i.e.\ in `small' positive characteristic.  There are some indications
that the formulae we obtain could -- to a large extent -- be
characteristic-independent, just depending on the residue field.

This is, for instance, the case for the zeta functions of groups of
the form $\SL_2(\lri)$, where $\lri$ is an arbitrary compact discrete
valuation ring of odd residue characteristic.  In \cite[Section
7]{JZ}, Jaikin-Zapirain computed a uniform formula for the zeta
functions of such groups, which only depends on the residue field
of~$\lri$; see~\cite[Section~3.4]{AKOVIII/11} for a discussion of the
case of even residue characteristic.  In light of this, it would be
interesting to compute, for instance, the zeta functions of groups of
the form $\SL_3(\kk[\![x]\!])$ and $\SU_3(\kk[\![x]\!])$, where $\kk$
is a finite field with $\cha(\kk) \neq 3$, as well as their principal
congruence subquotients.  We expect that the resulting formulae
coincide with those given in Theorem~\ref{thm:C}.

The results in~\cite[Theorem C]{BDOP12} on `conjugacy class zeta
functions' -- enumerating the total numbers of irreducible characters
of principal congruence quotients, such as $\aH(\lri_\len)$, as
opposed to enumerating them by their degrees -- also point towards a
very general `characteristic independence' of representation zeta
functions associated to suitable group schemes.

\subsubsection{Uniformity}
All the explicit formulae of zeta functions for $p$-adic analytic
groups provided in this paper -- notably in Theorems~\ref{thm:C} and
\ref{thm:J} -- display a high degree of uniformity in the residue
field of the underlying compact discrete valuation ring: the character
degrees and their multiplicities for the groups in question are given
by (quasi-)polynomials in $q$, the residue field's cardinality, whose
coefficients only depend on the residue class of $q$ modulo some
small, well-understood modulus and, possibly, the splitting behavior
of the place determined by the local ring in some quadratic
extension. We speculate that these features are not specific to
type~$\mathsf{A}_2$.

Let $k$ be a number field with ring of integers $\Gri$, and
$\mathbf{H}$ a connected, simply-connected semisimple algebraic group
defined over $k$, with a fixed embedding into~$\GL_d$ for some $d \in
\N$.  It is natural to ask under which conditions on $\mathbf{H}$ the
following uniformity property holds.

\begin{property}
  There exist $N\in\N$, finite index sets $I$ and $J$, polynomials
  $f_{\tau, i},g_{\tau, i}\in \Q[t]$ for $(\tau,i) \in \{1,\ldots,N\}
  \times I$, non-negative integers $A_j, B_j$ for $j\in J$, and a
  finite set $S$ of places of~$k$, containing all archimedean ones, all
  depending on~$\mathbf{H}$, such that the following holds.

  If $v$ is a place of $k$ not in $S$ and the residue cardinality
  $q_v$ satisfies $q_v \equiv_N \tau$, then
  \begin{equation}\label{equ:conj.uniform}
    \zeta_{\mathbf{H}(\Gri_v)}(s) = \frac{\sum_{i \in
        I}f_{\tau,i}(q_v) \; g_{\tau,i}(q_v)^{-s}}{\prod_{j \in J}(1 -
      q_v^{A_{j}-B_{j}s})}.
  \end{equation}
\end{property}

Theorem~\ref{thm:C} establishes that $\mathbf{H} = \SL_3$ has this
property. Finite groups of Lie type -- giving rise to representations
of `level $0$'~ of~$\mathbf{H}(\Gri_v)$ -- satisfy an analogous
property; see~\cite[Theorem~1.7]{LS}.

In the following we formulate a more specific conjecture on the shape
of almost all local factors of an arithmetic group of type
$\mathsf{A}_{n-1}$, generalizing Theorem~\ref{thm:C}.  Suppose that
the group $\mathbf{H}$ is absolutely almost simple of
type~$\mathsf{A}_{n-1}$.  Then $\mathbf{H}$ is either an inner form,
i.e.\ of type~${}^1\!\mathsf{A}_{n-1}$, arising from a matrix algebra
over a central division algebra over $k$, or an outer form, i.e.\ of
type~${}^2\!\mathsf{A}_{n-1}$, arising from a matrix algebra over a
central division algebra over a quadratic extension $K$ of~$k$,
equipped with an involution and with reference to a suitable hermitian
form; see \cite[Propositions~2.17 and 2.18]{PlatonovRapinchuk/94}. For
almost all non-archimedean places $v$ of~$k$, the completed group
$\mathbf{H}(\Gri_v)$ is of the form $\SL_n(\Gri_v)$ or
$\SU_n(\Gri_v)$; compare \cite[Appendix~A]{AKOV1}.  The latter case
distinction -- which occurs infinitely many often if and only if
$\mathbf{H}$ is an outer form -- is, for all but finitely many places
$v$ of~$k$, described by the \emph{Artin symbol} $\ee(v) =
\artin{K}{k}{v} \in \{1,-1\}$, which dictates whether or not $v$ is
decomposed in $K \,\vert\, k$.  For each non-archimedean place $v$
of~$k$, we set $\artin{k}{k}{v}=1$ and we write $\iota(v) \coloneqq
\gcd(q_v -\ee(v),n)$.  For $\ee(v)=1$ the latter gives the number of
$n$th roots of unity in the residue field $\kk_v$ of $\Gri_v$; for
$\ee(v)=-1$ it gives the number of norm-$1$ elements in the residue
field extension $\KK_w \,\vert\, \kk_v$, associated to the induced
quadratic extension~$K_w \,\vert\, k_v$, whose order divides~$n$.  We
write $\Div(n) = \{m\in\N \mid m \!\mid\! n\}$ for the set of divisors
of~$n$.

\begin{conj}
  Let $n \in \N_{\geq 2}$.  There exist finite index sets $I$ and $J$,
  polynomials $f_{\iota,\ee,i},g_{\iota,\ee,i}\in \Q[t]$ for
  $(\iota,\ee,i) \in \Div(n) \times \{1,-1\} \times I$ and
  non-negative integers $A_j, B_j$ for $j\in J$, such that the
  following holds.

  Let $k$ be a number field with ring of integers $\Gri$, and
  $\mathbf{H}$ a connected, simply-connected absolutely almost simple
  $k$-algebraic group of type~$\mathsf{A}_{n-1}$.  If $\mathbf{H}$ is
  an outer form, let $K$ denote the quadratic extension of $k$
  appearing in the definition of $\mathbf{H}$; if $\mathbf{H}$ is an
  inner form, put $K=k$.  Then there exists a finite set of places $S$
  of $k$, containing all archimedean ones and depending on
  $\mathbf{H}$, such that for every place $v$ of $k$ not in $S$,
  \begin{equation}\label{equ:conj.uniform.An}
    \zeta_{\mathbf{H}(\gri_v)}(s) =
    \frac{\sum_{i \in I} f_{\iota(v),\ee(v),i}(q_v) \;
      g_{\iota(v),\ee(v),i}(q_v)^{-s}}{\prod_{j \in
        J}(1 - q_v^{A_j-B_{j}s})},
  \end{equation}
  where $q_v$ denotes the residue cardinality of $\Gri_v$.
\end{conj}

\subsubsection{Ennola duality}\label{subsec:ennola}
Let $n \in \N$, and let $\lri$ be a compact discrete valuation ring
with residue cardinality~$q$.  Examples suggest that a dependence of
the representation zeta function on the residue class of $q$ modulo
$N$ as in~\eqref{equ:conj.uniform} does not occur for general linear
groups.  Although the groups $\GL_n(\lri)$ do not have convergent zeta
functions one may consider the zeta functions of the finite principal
congruence quotients~$\GL_n(\lri_\len)$, $\len \in \N$.  In all known
cases these zeta functions are uniform in~$q$, i.e.\ both the
occurring character degrees and their multiplicities are given by
polynomials in $q$ with constant coefficients.  In the case $(n \in
\N,\len=1)$ this follows from \cite{Green}, the case $(n \in
\N,\len=2)$ appears in~\cite{Singla}, and the case $(n=2,\len \in \N)$
in~\cite{Onn}.  Theorem~\ref{thm:C} confirms that this is the case
also for $(n=3,\len \in \N)$, at least for sufficiently large~$p$; we
expect that these restrictions on the primes covered are limitations
of the methods we use rather than genuine exceptions.  We phrase the
following general conjecture.

\begin{conj}\label{con:ennola}
  Let $n \in \N$.  There exist
  \begin{itemize} \renewcommand{\labelitemi}{$\circ$}
  \item a finite index set $I$,
  \item polynomials $f^{(\ee)}_i \in \Z[\frac{1}{n!}][t]$ and
    $g^{(\ee)}_i \in \Z[t]$ for $i \in I$ and $\ee \in\{-1,1\}$,
  \item ascending chains of finite sets 
    $\mathcal{B}_{i,1} \subset \mathcal{B}_{i,2} \subset \ldots
    \subset \N$ for $i \in I$,
  \item non-negative integers $A^{(\ee)}_{ij}, B_{ij}$ for $(i,j) \in
    I \times \mathcal{B}_{i,\len}$ and $\ee \in \{-1,1\}$
  \end{itemize}
  such that the following hold.
  \begin{enumerate}
  \item Let $\lri$ be a compact discrete valuation ring with residue
    cardinality $q$, let $\aG$ be one of the $\lri$-group schemes
    $\GL_n,\GU_n$, and $\ee = \ee_\aG \in \{-1,1\}$ accordingly.  For
    every $\len\in\N$, the character degrees of the finite group
    $\aG(\lri_\len)$ are given by
    \begin{align*}
      \cd(\aG(\lri_\len)) & = \{g^{(\ee)}_{i}(q) \; q^{B_{ij}} \mid i \in
      I, j \in \mathcal{B}_{i,\len}\},
    \end{align*}
    and its representation zeta function is given by
    \begin{align*}
      \zeta_{\aG(\lri_\len)}(s) & = \sum_{i \in I} \sum _{j \in
        \mathcal{B}_{i,\len}} f^{(\ee)}_{i}(q) \; q^{A^{(\ee)}_{ij}}
      \left(g_i^{(\ee)}(q) \; q^{B_{ij}} \right)^{-s}.
    \end{align*}
  \item Ennola duality holds for character degrees: for all $i \in I$,
    \[
    g_i^{(-1)}(t) = (-1)^{\deg(g^{(1)}_i)}g^{(1)}_i(-t).
    \]
  \end{enumerate}
\end{conj}

\subsection{Notation and organisation}
Throughout, $\lri$ denotes a compact discrete valuation ring, with
valuation ideal $\fp$ and residue field~$\kk$.  We write $q = \lvert
\kk \rvert$ and $p = \cha(\kk)$.  Often, but not always, we assume
$\cha(\lri) = 0$.  In this case we write $e=e(\lri,\Z_p)$ for the
absolute ramification index of~$\lri$. Superscripts usually denote
cartesian powers. In other contexts they form part of our notation for
principal congruence subgroups, powers of maximal ideals, and
subgroups generated by powers. Table~\ref{tab:notation} summarises
some further frequently used notation.

\begin{table}[htb!]
  \caption{Some frequently used notation.}
  \label{tab:notation}
  \begin{tabular}{l|l}Notation & Description\\
    \hline $\aG$ &
    group scheme $\GL_n$ or $\GU_n$, depending on $\ee \in \{1,-1\}$ \\
    $\aH$ & group scheme $\SL_n$ or
    $\SU_n$, depending on $\ee \in \{1,-1\}$ \\  
    $G_\len$, $H_\len$ & $\lri_\len$-rational points $\aG(\lri_\len)$,
    $\aH(\lri_\len)$\\ 
    $G^m$, $H^m$ & principal congruence subgroups $\aG^m(\lri)$,
    $\aH^m(\lri)$ \\  
    $G^m_\len$, $H^m_\len$ & principal congruence subquotients
    $\aG^m(\lri)/\aG^\len(\lri)$, $\aH^m(\lri)/\aH^\len(\lri)$\\ 
    \hline  
    $\ag$ & Lie lattice scheme $\gl_n$ or $\gu_n$, depending on $\ee
    \in \{1,-1\}$ \\ 
    $\fg = \ag(\lri_\len)$ & $\lri_\len$-rational points $ \gl_n(\lri_\len)$,
    $\gu_n(\lri_\len)$\\ 
    $\fg^m = \ag^m(\lri)$ & principal congruence sublattices $\fp^m
    \gl_n(\lri)$, $\fp^m \gu_n(\lri)$\\ 
    $\fg^m_\len$ & principal congruence
    subquotient $\fg^m/\fg^\len$ \\ 
    \hline 
    $\Sh$ & set of shadows $\Sh_{\GL_3(\lri)}$ or $\Sh_{\GU_3(\lri)}$  \\ 
    $\T^{(\ee)}$ & set of shadow types for $\ee \in \{1,-1\}$ \\
    $\sigma(\kk) = \aI^{\cS}_{\ee}(\kk)$ & $\kk$-rational points of
    shadow $\sigma$ of type $\cS$ for $\ee \in \{1,-1\}$ \\
    $\shG(A)$, $\shU(A)$ & group centraliser shadows \\
    $\shg(A)$, $\shu(A)$ &  Lie centraliser shadows \\
    $\Gamma^{(\ee)}$ & shadow graph for $\ee \in \{1,-1\}$ \\ 
    \hline 
    $\Qgl_\lri$, $\Qgu_\lri$ & similarity class trees 
  \end{tabular}
\end{table}

The paper's broad organisation is as follows. Part~\ref{part:1}
relates to questions about the classification and enumeration of
similarity classes of integral $\fp$-adic $n \times n$ matrices, with
a focus on~$n=3$.  The main results are
Theorems~\ref{thm:G.shad.graph} and \ref{thm:U.shad.graph},
classifying shadows of similarity classes in $\gl_3(\lri_\len)$ and
$\gu_3(\lri_\len)$, respectively. In Section~\ref{sec:sim.zeta} these
two results are applied to compute similarity class zeta functions in
type $\mathsf{A}_2$.

In Part~\ref{part:2} we apply our results on similarity classes of
$3\times3$ matrices to representation zeta functions of groups of
type~$\mathsf{A}_2$.  To this end, methodology from $p$-adic Lie
theory, the Kirillov orbit method, and Clifford theory are prepared in
Section~\ref{sec:kom.clifford}.  Section~\ref{sec:zeta.A2} contains
explicit computations of (local) representation zeta functions of
various groups of type $\mathsf{A}_2$.  Results on global zeta
functions, obtained as Euler products of local ones, are proved in
Section~\ref{sec:abscissa}.

Complementing the paper's main ideas, we collect, in
Appendix~\ref{sec:A1}, a number of results in type $\mathsf{A}_1$ that
are analogous to and easier to derive than those obtained in type
$\mathsf{A}_2$.

Table~\ref{tab:roadmap} collects the locations of the proofs of the
main results stated in the introduction.

\begin{table}[htb!]
  \caption{Location of proofs.}
  \label{tab:roadmap}
  \begin{tabular}{l|c}
    Result & Proved in Section\\ \hline Theorem~\ref{thm:A},
    Corollary~\ref{cor:B} &
    \ref{subsec:rep.global}\\ Theorem~\ref{thm:C} &
    \ref{subsec:van.coh.gps}\\ Corollary~\ref{cor:D} &
    \ref{subsec:zeta.shadows}\\ Theorem~\ref{thm:sim.zeta.local},
    Corollary~\ref{cor:F} &
    \ref{subsec:sim.zeta.A2}\\ Theorem~\ref{thm:G} &
    \ref{subsec:sim.global}\\ Theorem~\ref{thm:H}, Theorem~\ref{thm:I}
    &\ref{subsec:ennola.duality} \\ Theorem~\ref{thm:J} &
    \ref{subsec:principal}
  \end{tabular}
\end{table}

\subsection{Acknowledgments}
We thank the Batsheva de Rothschild Fund, the DFG, the EPSRC, the
Mathematisches Forschungsinstitut Oberwolfach, the NSF, the ISF, and
the Nuffield Foundation.  We are grateful to Matthew Levy and
Alexander Stasinski for pointing out inaccuracies in a previous
version of the paper.


\part{Similarity classes of $\fp$-adic matrices}\label{part:1}


\section{Similarity classes of integral $\fp$-adic
  matrices} \label{sec:sim.gl}

Let $\lri$ be a compact discrete valuation ring, with valuation ideal
$\fp$ and finite residue field~$\kk$.  Put $p = \cha(\kk)$ and $q =
\lvert \kk \rvert$.  Let $\pi$ be a fixed uniformiser of $\lri$ so
that $\fp = \pi \lri$, and let $v \colon \lri \to \Z \cup \{\infty\}$
denote the valuation map on $\lri$.  In this section there is no
restriction on either $\cha(\lri)$ or $\cha(\kk)$.  In the simplest
cases, $\lri$ is the ring $\mathrm{Witt}(\kk)$ of Witt vectors over
$\kk$, i.e.\ the unique unramified extension of the $p$-adic integers
$\Z_p$ with residue field $\kk$, or the ring $\kk[\![x]\!]$ of formal
power series over~$\kk$.

For $\len \in \N_0$ let $\lri_\len$ denote the finite quotient ring
$\lri/\fp^\len$.  Let $n \in \N$, and let $\gl_n(\lri_\len)$ denote the
collection of $n \times n$ matrices over $\lri_\len$, with the
standard structure as an $\lri_\len$-Lie ring.

\begin{defn} \label{def:graph.Qgl} Let $\Qgl_{\lri,\len} =
  \Ad(\GL_n(\lri)) \backslash \gl_n(\lri_\len)$ denote the set of
  similarity classes in $\gl_n(\lri_\len)$, that is, orbits in
  $\gl_n(\lri_\len)$ under the adjoint action of~$\GL_n(\lri)$.  We
  endow
  \[
  \Qgl_\lri = \coprod_{\len=0}^\infty \Qgl_{\lri,\len}
  \]
  with the structure of a directed graph, induced by reduction modulo
  powers of~$\fp$, as follows.  Vertices $\cC \in \Qgl_{\lri,\len}$
  and $\cCtilde \in \Qgl_{\lri,\len+1}$ are connected by a directed
  edge $(\cC,\cCtilde)$ if the reduction of $\cCtilde$ modulo
  $\fp^\len$ is equal to~$\cC$, and we say that $\cCtilde$ lies
  above~$\cC$.  In this way $\Qgl_\lri$ becomes an infinite rooted
  tree, its root being the single element $\{0\}$ of~$\Qgl_{\lri,0}$.
  We refer to $\Qgl_\lri$ as the \emph{similarity class tree} in
  degree $n$ over~$\lri$.
\end{defn}

The aim in this section is to provide a framework for analysing the
structure of~$\Qgl_\lri$ and to apply this method in the concrete case
$n=3$.  In Section~\ref{subsubsec:shadows} we study centralisers and
introduce the concept of similarity class shadows for arbitrary
degree~$n$.  In Section~\ref{subsec:branching} we specialise to $n=3$
and state Theorem~\ref{thm:G.shad.graph}, concerning shadows and
branching rules, which plays a crucial role in the computation of
similarity class and representation zeta functions in the following
sections. The purpose of Section~\ref{appendix:similarity.new} is to
produce a complete set of representatives for the similarity classes
of $3\times3$ matrices over a discrete valuation ring; we emphasise
that the ring in question may be of positive characteristic.  The
proofs are technically involved and may be skipped at first reading.
Section~\ref{subsec:proof.main.SL} contains a proof of our main
result, Theorem~\ref{thm:G.shad.graph}.

\subsection{Centralisers and shadows}\label{subsubsec:shadows}
Let $\len \in \N_0$ and let $A \in \gl_n(\lri_\len)$.  The centraliser
$\Cen_{\GL_n(\lri)}(A)$ of $A$ in the group $\GL_n(\lri)$ is the
stabiliser of $A$ under the adjoint action of $\GL_n(\lri)$.  The
centraliser $\Cen_{\gl_n(\lri)}(A)$ of $A$ in the $\lri$-Lie lattice
$\gl_n(\lri)$ is the stabiliser of $A$ under the adjoint action of
$\gl_n(\lri)$.

\begin{defn}\label{def.of.shadows} Let $\len \in \N_0$.  The
  \emph{group centraliser shadow} $\shG(A)$ of an element $A \in
  \gl_n(\lri_\len)$ is the image $\overline{\Cen_{\GL_n(\lri)}(A)} \leq
  \GL_n(\kk)$ of $\Cen_{\GL_n(\lri)}(A)$ under reduction modulo~$\fp$.
  The \emph{Lie centraliser shadow} $\shg(A)$ of an element $A \in
  \gl_n(\lri_\len)$ is the image $\overline{\Cen_{\gl_n(\lri)}(A)} \leq
  \gl_n(\kk)$ of $\Cen_{\gl_n(\lri)}(A)$ under reduction modulo~$\fp$.

  For each similarity class $\cC$ in $\gl_n(\lri_\len)$ we define the
  (\emph{similarity class}) \emph{shadow}
  \[
  \shGC(\cC) = \{ (\shG(A),\shg(A)) \mid A \in\cC \},
  \]
  of $\cC$, and we denote the collection of all shadows by
  \begin{equation}\label{equ:Sh.GL}
    \ShGn = \{ \shGC(\cC) \mid \cC \in \Qgl_{\lri,\len}
    \text{ for some $\len \in \N_0$}\}.
  \end{equation}
  For $\sigma \in \ShGn$ we set
  \[
  \| \sigma \| = \lvert \shG(A) \rvert \qquad \text{and} \qquad
  \dim(\sigma) = \dim_\kk(\shg(A)),
  \]
  where $A \in \cC \in \Qgl_{\lri,\len}$, for some $\len \in \N_0$,
  with $\sigma = \shGC(\cC)$; furthermore, it is convenient to select
  one group centraliser shadow $\shG(A)$, where $A \in \cC \in
  \Qgl_{\lri,\len}$, for some $\len \in \N_0$, with $\sigma =
  \shGC(\cC)$, and to denote it by $\sigma(\kk)$.  We only use
  properties of $\sigma(\kk)$ that are independent of the arbitrary
  choice involved in its definition.
\end{defn}

The definition of a shadow reflects the idea that a shadow is a group
object together with an associated Lie structure, independently of the
choices for the parameters $\len$, $\cC$ and~$A$.  Indeed one may
think of a shadow $\sigma$ as a conjugacy class of subgroups of
$\GL_n(\kk)$, represented by~$\sigma(\kk)$, together with
corresponding Lie subalgebras of $\gl_n(\kk)$.  Formally, there is
some built-in redundancy.  Of course, every shadow $\sigma$ is
completely determined by any of its `representatives'
$(\shG(A),\shg(A))$.  Furthermore, the second coordinate~$\shg(A)$
suffices to pin down $\sigma$, as $\shG(A)$ simply consists of the
units of the ring~$\shg(A)$.  Similarly, the first coordinate
$\shG(A)$ determines $\shg(A)$, at least for $q>2$, by the next lemma.

\begin{lem}\label{lem:shadow-well-defined}
  Suppose that $q > 2$.  Let $\len \in \N_0$ and $A \in
  \gl_n(\lri_\len)$.  Then the Lie centraliser shadow $\shg(A)$ is
  equal to the additive span of the group centraliser shadow
  $\shG(A)$.
\end{lem}

\begin{proof}
  We only need to prove that $\shg(A)$ is contained in the additive
  span of $\shG(A)$, because the other inclusion is clear.  Let
  $\overline{X} \in \shg(A)$ be the image of $X \in
  \Cen_{\gl_n(\lri)}(A)$.  Illustrating the idea for the general case
  treated below, we observe that under the extra assumption $q>n$ we
  may choose $a \in \lri$ so that $\overline{X - a \Id_n}$ has no zero
  eigenvalues and therefore $\overline{X} = \overline{X - a \Id_n} +
  \overline{a \Id_n}$ lies in the additive span of $\shG(A)$.

  In general we may assume, by the Primary Decomposition Theorem, that
  \[
  \overline{X} = \diag(\overline{Y_1},\ldots,\overline{Y_r})
  =\overline{X_1} + \ldots + \overline{X_r},
  \]
  where the $\overline{Y_i}$ are block matrices and $\overline{X_i} =
  \diag(0,\ldots,0,\overline{Y_i},0,\ldots,0)$, for $1 \leq i \leq r$.
  The blocks correspond to the generalised eigenspaces of
  $\overline{X}$ or, equivalently, the factors in the factorization
  $f_{\overline{X}}(t) = \overline{f_X}(t) = \prod_{i=1}^r
  f_{\overline{X_i}}(t)^{e_i}$ of the characteristic polynomial into a
  product of pairwise coprime powers of irreducible
  polynomials. Applying Hensel's Lemma (cf.~\cite[Theorem~8.3]{Mat})
  to this factorization, we may lift the decomposition to
  $\gl_n(\lri)$, and hence assume that
  \[
  X = \diag(Y_1,\ldots,Y_r) = X_1 + \ldots + X_r,
  \]
  where the block matrices $Y_i \in \gl_n(\lri)$ are lifts of the
  block matrices $\overline{Y_i}$, and each matrix $X_i=
  \diag(0,\ldots,0,Y_i,0,\ldots,0)$ is a lift of $\overline{X_i}$, such
  that
  \begin{itemize} \renewcommand{\labelitemi}{$\circ$}
  \item each $X_i$ is a polynomial expression in~$X$, and
    hence $\overline{X_i} \in \shg(A)$,
  \item the minimal polynomial of $\overline{X_i}$ over $\kk$
    is a power of an irreducible polynomial.
  \end{itemize}

  For each $i \in \{1,\ldots,r\}$ the matrix $\overline{Y_i}$ has at
  most one eigenvalue in $\kk$.  Since $q>2$, we may choose
  $\overline{a_i} \in \kk \smallsetminus \{0\}$ so that
  $\overline{X_i} = \overline{X_i - a_i \Id_n} + \overline{a_i \Id_n}
  \in \shG(A) + \shG(A)$.  This shows that $\overline{X}$ can be
  expressed as the sum of at most $2n$ elements from $\shG(A)$.
\end{proof}

The next proposition shows that group centralisers can naturally be
identified as groups of $\kk$-rational points of certain algebraic
groups.  For $\cha(\lri)= 0$, the Greenberg transform of level~$\len$
associates to an $\lri_\len$-scheme $X$ of finite type a $\kk$-scheme
$\mathcal{X}$ of finite type in such a way that $X(\Lri_\len) \simeq
\mathcal{X}(\KK)$ for unramified finite extensions $\Lri$ of $\lri$
with residue field $\KK$; the construction makes use of Witt vectors;
see~\cite{Greenberg}.  For $\cha(\lri)>0$, the residue field $\kk$ can
be regarded as a subfield of $\lri_\len$ and Weil restriction
associates, in a similar but simpler way, to any
$\lri_\len$-scheme~$X$ a $\kk$-scheme $\mathcal{X}$ by `restriction of
scalars'.

\begin{prop}\label{prop:centralisers.connected}
  Let $\mathcal{X}_{n,\len}^\gl$ be the Greenberg transform of level
  $\len$ \textup{(}for $\cha(\lri)= 0$\textup{)} or the Weil
  restriction \textup{(}for $\cha(\lri)>0$\textup{)} of the
  $\lri_\len$-scheme $\gl_n$ to $\kk$-schemes so that $\gl_n(\lri_\len)
  \simeq \mathcal{X}_{n,\len}^\gl(\kk)$.  For $A \in \gl_n(\lri_\len)$
  there is a linear subvariety $\mathbf{V} \subset
  \mathcal{X}_{n,\len}^\gl$ such that
  \[
  \Cen_{\gl_n(\lri_\len)}(A) \simeq \mathbf{V}(\kk) \subset
  \mathcal{X}_{n,\len}^\gl(\kk).
  \]
  Furthermore, $\mathbf{V}$ contains a Zariski-open connected
  algebraic group $\mathbf{C}$ such that
  \[
  \Cen_{\GL_n(\lri_\len)}(A) \simeq \mathbf{C}(\kk).
  \]
\end{prop}

\begin{proof}
  Upon choosing an $\lri_\len$-basis for the $\lri_\len$-module
  $\gl_n(\lri_\len)$ we may identify $\gl_n(\lri_\len)$
  with~$\lri_\len^{\, N}$, where $N = n^2$, and there exists a matrix
  $B\in\gl_N(\lri_\len)$ representing the $\lri_\len$-linear map
  $\ad(A)$. Then we may identify $\Cen_{\gl_n(\lri_\len)}(A)$ with the
  set
  \[
  C(A) = \{ \mathbf{x} \in \lri_\len^{\, N} \mid B
  \mathbf{x} \equiv_\fp 0 \}.
  \]
  After performing an $\lri_\len$-linear change of bases, if
  necessary, we may assume that $B$ is in standard elementary divisor
  form, that is
  \[
  B = \diag(\pi^{e_1},\ldots,\pi^{e_N})
  \]
  for suitable integers $0 \leq e_1 \leq e_2 \leq \ldots \leq e_N \leq
  \len$. Then
  \[
  C(A) = \{ (x_1,\ldots,x_N)^\mathrm{tr} \in \lri_\len^{\, N} \mid
  \forall i: v(x_i) \geq \len - e_i \}.
  \]
  Clearly, this corresponds to the set of $\kk$-rational points of a
  linear subvariety $\mathbf{V} \subset \mathcal{X}_{n,\len}^\gl$.
  The elements of $C(A)$ in bijection to elements of
  $\Cen_{\GL_n(\lri_\len)}(A)$ are determined by imposing the open
  condition that $\mathbf{x}$ corresponds to an invertible element
  in~$\gl_n(\lri_\len)$.  This defines a Zariski-open algebraic group
  $\mathbf{C}$ in~$\mathbf{V}$.  Being a linear variety, $\mathbf{V}$
  is irreducible and hence $\mathbf{C}$ is connected.
\end{proof}

\begin{prop}\label{pro:class.quot.GL}
  Let $\sigma, \tau \in \ShGn$.  Let $\len \in \N_0$ and
  suppose that $\cCtilde \in \Qgl_{\lri,\len+1}$ is a class with
  $\shGC(\cCtilde) = \tau$ which lies above a class $\cC
  \in \Qgl_{\lri,\len}$ with $\shGC(\cC) = \sigma$.  Then
  \[
  \frac{\lvert \cCtilde \rvert}{\lvert \cC \rvert} = q^{\dim \gl_n -
    \dim(\sigma)} \frac{\| \sigma \|}{\| \tau \|}.
  \]
  In particular, the ratio $\lvert \cCtilde \rvert/ \lvert \cC \rvert$
  depends only on the shadows $\sigma, \tau$ and not on $\len$, $\cC$
  or $\cCtilde$.
\end{prop}

\begin{proof}
  Let $A \in \cC$, $\wt{A} \in \cCtilde$ such that $A \equiv
  \wt{A}$ modulo $\fp^\len$.  If $\len = 0$, then $\cC = \{0\}$
  and $\| \sigma \| = \lvert \GL_n(\kk) \rvert$ so that
  \[
  \frac{ \lvert \cCtilde \rvert}{ \lvert \cC \rvert} = \lvert \cCtilde
  \rvert = [\GL_n(\kk) : \Cen_{\GL_n(\kk)}(\wt{A})] = \frac{\|
    \sigma \|}{\| \tau \|}.
  \]

  Now suppose that $\len \geq 1$, and put $A_{\len+1} = \wt{A}$
  and $A_\len = A$.  For $i \in \{\len, \len+1 \}$ set $C_i =
  \Cen_{\GL_n(\lri_i)}(A_i)$.  Then $\lvert \GL_n(\lri_{\len+1})
  \rvert / \lvert \GL_n(\lri_\len) \rvert = [\GL_n^\len(\lri) :
  \GL_n^{\len+1}(\lri)] = q^{\dim \gl_n}$ implies that
  \begin{equation}\label{equ:reduction.to.C}
    \frac{\lvert \cCtilde \rvert}{\lvert \cC \rvert} =
    \frac{[\GL_n(\lri_{\len+1}) : C_{\len+1}]}{[\GL_n(\lri_\len) :
      C_\len]} = q^{\dim \gl_n} \frac{\lvert C_\len \rvert}{\lvert
      C_{\len+1} \rvert}.
  \end{equation}
  Writing $A_{\len-1}$ for $A$ modulo $\fp^{\len-1}$ and setting
  $Z_{i-1} = \Cen_{\gl_n(\lri_{i-1})}(A_{i-1})$ for $i \in \{\len,
  \len+1 \}$, we observe that the reduction map modulo $\fp$ yields
  exact sequences
  \begin{equation*}
    0 \rightarrow  Z_{i-1}  \cap \gl_n^1(\lri_{i-1}) \rightarrow
    Z_{i-1} \rightarrow \shg(A_{i-1}) \rightarrow 0,
  \end{equation*}
  \begin{equation*}
    1 \rightarrow C_i \cap \GL_n^1(\lri_i) \rightarrow C_i
    \rightarrow \shG(A_i) \rightarrow 1,
  \end{equation*}
  and that, furthermore, the maps
  \begin{equation} \label{equ:map}
    \begin{split}
      \gl_n(\lri) \rightarrow \gl_n^1(\lri) &, \quad X \mapsto \pi X,
      \\ \gl_n(\lri) \rightarrow \GL_n^1(\lri) &, \quad X \mapsto
      \Id_n + \pi X
    \end{split}
  \end{equation}
  induce bijections (of sets)
  \[
  Z_{i-1} \simeq Z_i \cap \gl_n^1(\lri_i) \qquad \text{and} \qquad
  Z_{i-1} \simeq C_i \cap \GL_n^1(\lri_i).
  \]
  Thus we conclude from \eqref{equ:reduction.to.C} that
  \begin{align*}
    \frac{ \lvert \cCtilde \rvert}{ \lvert \cC \rvert} & = q^{\dim
      \gl_n} \frac{ \lvert C_\len \cap \GL_n^1(\lri_\len) \rvert}{
      \lvert C_{\len+1} \cap \GL_n^1(\lri_{\len+1}) \rvert}
    \frac{\lvert \shG(A_\len) \rvert}{\lvert \shG(A_{\len+1}) \rvert} \\
    & = q^{\dim \gl_n} \frac{ \lvert Z_{\len-1} \rvert}{ \lvert
      Z_\len \rvert} \frac{\| \sigma \|}{\| \tau \|} \\
    & = q^{\dim \gl_n} \frac{ \lvert Z_\len \cap \gl_n^1(\lri_\len)
      \rvert}{ \lvert Z_\len \cap \gl_n^1(\lri_\len) \rvert \lvert
      \shg(A_\len) \rvert} \frac{\| \sigma \|}{\| \tau
      \|} \\
    & = q^{\dim \gl_n - \dim(\sigma)} \frac{\| \sigma \|}{\|
      \tau \|}. \qedhere
  \end{align*}
\end{proof}

Proposition~\ref{pro:class.quot.GL} highlights the relevance of the
shadows to the computation of the sizes of similarity classes, and
motivates the following definition.

\begin{defn}\label{def:b.GL} For $\sigma, \tau \in
  \ShGn$ let
  \begin{equation} \label{equ:def.b.GL} b^{(1)}_{\sigma,
      \tau}(q) = q^{\dim \gl_n - \dim(\sigma)} \frac{\|
      \sigma \|}{\| \tau \|}.
  \end{equation}
\end{defn}

\begin{remark}\label{rem:bees}
  The quantities $b^{(1)}_{\sigma, \tau}(q)$ which are relevant for us
  are the ones where $\sigma = \shGC(\cC)$ and $\tau =
  \shGC(\cCtilde)$ with $\cCtilde \in \Qgl_{\lri,\len+1}$ lying above
  $\cC \in \Qgl_{\lri,\len}$.  In the case $n=3$ these turn out to be
  integral polynomials in $q$, independent of $\lri$; see
  Table~\ref{tab:branch.rules.A2}.  It is an interesting open question whether
  this is true more generally.

  For a uniform treatment of integral $\fp$-adic matrices and
  anti-hermitian integral $\fp$-adic matrices, leading to a uniform
  description of the representation zeta functions of general/special
  linear groups and general/special unitary groups, it is convenient
  to write $b_{\sigma,\tau}^{(1)}$ and to use $b_{\sigma,\tau}^{(-1)}$
  to denote similar polynomials occurring in Section~\ref{sec:sim.gu},
  where we treat similarity classes of anti-hermitian matrices.
  Table~\ref{tab:branch.rules.A2} already incorporates the
  complementary information from Section~\ref{sec:sim.gu}.
\end{remark}


\subsection{Shadows and branching rules for $\gl_3(\lri_\len)$}
\label{subsec:branching}
In Table~\ref{tab:shadows.iso.GL} we list ten shadows
in~$\Sh_{\GL_3(\lri)}$, classified by (shadow) types;
compare~\eqref{equ:types-def}.  For $q>2$, all but the last two of
these types already arise from~$\len = 1$: they are conjugacy
classes of centralisers of elements $A \in \gl_3(\kk)$ and are
therefore classified by the shape of the minimal polynomials of such
$A$ over~$\kk$.  The last two shadows, of types $\Kasha$ and $\Kbsha$,
come from~$\len = 2$ and higher: they are the conjugacy classes
of the reductions modulo $\fp$ of the centralisers of the matrices
$A_0 = \left[ \begin{smallmatrix} 0 & \pi & 0 \\ 0 & 0 & 1 \\ 0 & 0 &
    0 \end{smallmatrix} \right]$ and $A_\infty =
\left[ \begin{smallmatrix} 0 & 0 & 0 \\ 0 & 0 & 1 \\ \pi & 0 &
    0 \end{smallmatrix} \right]$ in $\gl_3(\lri_2)$, as explained
below.  The types $\Kasha$ and $\Kbsha$ can be thought of as `shears'
of the type~$\Nsha$; see the proof of Theorem~\ref{thm:G.shad.graph}.
The third column of Table~\ref{tab:shadows.iso.GL} describes the
isomorphism types of the group centraliser shadows as algebraic
groups; compare Table~\ref{tab:shadows}.  In the table, we write
$\kk_2$ and $\kk_3$ for the quadratic and cubic extensions of the
finite field~$\kk$.  Furthermore, $\Heis$ stands for the Heisenberg
group of upper uni-triangular $3\times3$ matrices.

Let us take a closer look at the types $\Kasha$ and $\Kbsha$ that do
not arise for $\len = 1$.  The group centraliser shadows of the
matrices $A_0, A_\infty \in \gl_3(\lri_2)$ are equal to
\begin{equation*}
  H_0 = \left\{\left[\begin{smallmatrix}
        {t}  & 0&  s_1 \\
        0 & {t} & {s_3} \\
        0 & 0 & {t}
      \end{smallmatrix} \right] \mid s_1, s_3 \in \kk \text{ and } t
    \in \kk^\times \right\} \text{ and } H_\infty =
  \left\{\left[\begin{smallmatrix}
        {t}  & 0&  0 \\
        s_2 & {t} & {s_3} \\
        0 & 0 & {t}
      \end{smallmatrix} \right] \mid s_2, s_3 \in \kk \text{ and } t
    \in \kk^\times \right\};
\end{equation*}
see Proposition~\ref{shadows.case.v.new}, where the notation $A_0 =
E_2(1,0,0,0,0)$ and $A_\infty = E_2(\infty,\pi,0,0,0)$ is employed.
Comparing orders of groups, we see that the corresponding shadows
$\sigma_0$ and $\sigma_\infty$ cannot be of type~$\Gsha$, $\Lsha$,
$\Jsha$, $\Tasha$, $\Tbsha$, $\Tcsha$ or $\Msha$.  To rule out type
$\Nsha$, we observe that $(h - t\Id_3)^2 = 0$ for any $h \in H_0 \cup
H_\infty$ with diagonal entries $t$.  Thus neither $H_0$ nor
$H_\infty$ contains a matrix whose minimal polynomial is of degree~$3$
over~$\kk$.  Consequently, $\sigma_0$ and $\sigma_\infty$ cannot be of
type~$\Nsha$.

Finally, we verify that the subgroups $H_0$ and $H_\infty$ are not
conjugate in $\GL_3(\kk)$ so that $\sigma_0 \neq \sigma_\infty$.
Assume, for a contradiction, that $g = (g_{ij}) \in \GL_3(\kk)$
satisfies $g H_0 = H_\infty g$.  As scalar matrices are invariant
under conjugation, it suffices to compare matrices
\[
g \left[\begin{smallmatrix}
    {0}  & 0&  s_1 \\
    0 & {0} & {s_3} \\
    0 & 0 & {0}
  \end{smallmatrix} \right] \quad \text{and} \quad
\left[\begin{smallmatrix}
    {0}  & 0 & 0 \\
    \tilde s_2 & 0 & \tilde s_3 \\
    0 & 0 & {0}
  \end{smallmatrix} \right] g
\]
for $s_1, s_3, \tilde s_2, \tilde s_3 \in \kk$.  Inspecting the
$(1,3)$- and $(3,3)$-entries, we deduce that for all $s_1,s_3 \in
\kk$,
\[
g_{11}s_1+g_{12}s_3 = 0 \qquad \text{and} \qquad g_{31}s_1+g_{32}s_3 =
0.
\]
This implies that $g_{11} = g_{12} = g_{31} = g_{32} = 0$ and hence
$g$ cannot be invertible.  This completes the discussion of
Table~\ref{tab:shadows.iso.GL}.

\begin{table}[htb!]
  \centering
  \caption{Shadows $\sigma$ in $\GL_3(\kk)$}
  \label{tab:shadows.iso.GL}
  \begin{tabular}{|c||l|l|c|}
    \hline
    Type & Minimal polynomial in $\kk[t]$ & Isomorphism type of $\sigma(\kk)$
    & $\dim(\sigma)$ \\
    \hline
    $\Gsha$ & ~$t-\alpha$ \hfill $\alpha \in \kk$ & $\GL_3(\kk)$ & 9 \\
    $\Lsha$ & $(t-\alpha_1)(t-\alpha_2)$ \hfill $\alpha_1, \alpha_2
    \in \kk$ distinct & $\GL_1(\kk) \times \GL_2(\kk)$ & 5 \\
    $\Jsha$  &  $(t-\alpha)^2$ \hfill $\alpha \in \kk$ &
    $\mathsf{Heis}(\kk) \rtimes (\GL_1(\kk) \times \GL_1(\kk))$ & 5 \\
    $\Tasha$ &  $\prod_{i=1}^3 (t-\alpha_i)$ \hfill
    $\alpha_1,\alpha_2,\alpha_3 \in \kk$ distinct &
    $\GL_1(\kk) \times \GL_1(\kk) \times \GL_1(\kk)  $ & 3 \\
    $\Tbsha$ &  $(t-\alpha)f(t)$ \quad $\alpha \in \kk$, $f$ irred.\
    quadratic & $\GL_1(\kk) \times \GL_1(\kk_2)$ & 3 \\
    $\Tcsha$ &  $f(t)$ \hfill $f$ irred.\ cubic& $\GL_1(\kk_3)$ & 3 \\
    $\Msha$ &  $(t-\alpha_1)(t-\alpha_2)^2$ \hfill $\alpha_1, \alpha_2
    \in \kk$ distinct & $\GL_1(\kk)\times \GL_1(\kk[t]/(t^2))$ & 3 \\
    $\Nsha$ &  $(t-\alpha)^3$ \hfill $\alpha \in \kk$ &
    $\GL_1(\kk[t]/(t^3))$ & 3 \\
    \hline
    $\Kasha$ &  \qquad not applicable   & $\GL_1(\kk) \times \aG_\text{a}(\kk)
    \times \aG_\text{a}(\kk)$ & 3 \\
    $\Kbsha$ &  \qquad not applicable   & $\GL_1(\kk) \times \aG_\text{a}(\kk)
    \times \aG_\text{a}(\kk)$ & 3 \\
    \hline
  \end{tabular}
\end{table}

\begin{thm}[Classification of shadows and branching
  rules] \label{thm:G.shad.graph} \quad
  \begin{enumerate}
  \item \label{item:G.conj.graph.1} The set of shadows
    $\Sh_{\GL_3(\lri)}$ consists of ten elements, classified by the
    types
    \[
    \Gsha, \, \Lsha, \, \Jsha, \, \Tasha, \, \Tbsha, \, \Tcsha, \,
    \Msha, \, \Nsha, \, \Kasha, \, \Kbsha
    \]
    described in Table~\textup{\ref{tab:shadows.iso.GL}}.
  \item \label{item:G.conj.graph.2} For all $\sigma, \tau \in
    \Sh_{\GL_3(\lri)}$ there exists a polynomial $a_{\sigma, \tau} \in
    \Z[\sixth][t]$ such that the following holds: for every $\len \in
    \N$ and every $\cC \in \cQ_{\lri,\len}^{\gl_3}$ with $\shG(\cC) =
    \sigma$ the number of classes $\cCtilde \in
    \cQ_{\lri,\len+1}^{\gl_3}$ with $\shG(\cCtilde) = \tau$ lying
    above $\cC$ is equal to $a_{\sigma, \tau}(q)$.
  \end{enumerate}
\end{thm}

\begin{table}[hbt!]
  \centering
  \caption{Branching rules for $\cQ^{\gl_3}_\lri$ ($\ee = 1$)
    and for $\cQ^{\gu_3}_\lri$ ($\ee = -1$)}
  \label{tab:branch.rules.A2}
  \renewcommand*\arraystretch{1.4}
  \begin{tabular}{|c|c|c||l|l|}
    \hline
    $\#$ & Type of $\sigma$ & Type of $\tau$ &
    $a_{\sigma,\tau}(q)$  & $b^{(\ee)}_{\sigma,\tau}(q)$ \\
    \hline
    $1$ & $\Gsha$ & $\Gsha$  & $q$               & $1$ \\
    $2$ & $\Gsha$ & $\Lsha$  & $(q-1)q$          & $(q^2 + \ee q + 1) q^2$ \\
    $3$ & $\Gsha$ & $\Jsha$  & $q$               & $(q^3-\ee)(q+\ee)$ \\
    $4$ & $\Gsha$ & $\Tasha$ & $\sixth (q-1)(q-2)q$
    & $(q^2 + \ee q + 1)(q + \ee)q^3$ \\
    $5$ & $\Gsha$ & $\Tbsha$ & $ \hlf (q-1)q^2$  & $(q^3-\ee)q^3$ \\
    $6$ & $\Gsha$ & $\Tcsha$ & $\third (q^2-1)q$ & $(q+\ee)(q-\ee)^2q^3$ \\
    $7$ & $\Gsha$ & $\Msha$  & $(q-1)q$          & $(q^3-\ee)(q+\ee)q^2$ \\
    $8$ & $\Gsha$ & $\Nsha$  & $q$               & $(q^3-\ee)(q^2-1)q$ \\
    \hline
    $9$  & $\Lsha$ & $\Lsha$  & $q^2$           & $q^4$ \\
    $10\phantom{^*}$ & $\Lsha$ & $\Tasha$ & $\hlf (q-1)q^2$ & $(q+\ee)q^5$ \\
    $11\phantom{^*}$ & $\Lsha$ & $\Tbsha$ & $\hlf (q-1)q^2$ & $(q-\ee)q^5$ \\
    $12\phantom{^*}$ & $\Lsha$ & $\Msha$  & $q^2$           & $(q^2-1)q^4$ \\
    \hline
    $13\phantom{^*}$ & $\Jsha$ & $\Jsha$ & $q^2$      & $q^4$ \\
    $14\phantom{^*}$ & $\Jsha$ & $\Msha$ & $(q-1)q^2$ & $q^6$ \\
    $15\phantom{^*}$ & $\Jsha$ & $\Nsha$ & $(q-1)q$   & $(q-\ee)q^5$ \\
    \hline
    \hline
    $16^*$ & $\Jsha$ & $\Kasha$  & $q$             & $(q-1)q^5$ \\
    $17^*$ & $\Jsha$ & $\Kbsha$  & $q$             & $(q-1)q^5$ \\
    \hline
    \hline
    $18^*$ & other & same as $\sigma$ & $q^3$ & $q^6$ \\
    \hline
    \multicolumn{5}{@{} c @{}}{${}^*$ Rows involving types
      $\Kasha$ and $\Kbsha$ only apply if $\ee = 1$.}
  \end{tabular}
\end{table}

\begin{remark}
  Of course, many of the polynomials $a_{\sigma,\tau}$ are simply
  zero.  The non-zero $a_{\sigma, \tau}$ give the local branching
  behaviour of the directed graph $\Qgl_\lri$, which we introduced in
  Definition~\ref{def:graph.Qgl}, and thus determine $\Qgl_\lri$
  completely.  Moreover, together with the corresponding polynomials
  $b^{(1)}_{\sigma, \tau}$, defined in~\eqref{equ:def.b.GL}, the
  $a_{\sigma,\tau}$ determine recursively the numbers and sizes of
  similarity classes in $\gl_3(\lri_\len)$ for all~$\len \in \N$; cf.\
  Section~\ref{subsec:sim.zeta.A2}.  We refer to these two sets of
  polynomials as \emph{branching rules} and record them in
  Table~\ref{tab:branch.rules.A2}.

  While the polynomials $a_{\sigma,\tau}$ are determined in the course
  of the proof of Theorem~\ref{thm:G.shad.graph}, the polynomials
  $b^{(1)}_{\sigma,\tau}$ can already be easily computed with the aid
  of Table~\ref{tab:shadows.iso.GL}.
\end{remark}

\begin{remark}
  One reads off rows (2) and (3) in Table~\ref{tab:branch.rules.A2}
  that exactly
  \[
  (q-1)q \cdot (q^2 + q + 1)q^2 + q\cdot (q^3-1)(q+1) = q(q-1) \lvert
  \mathbb{P}^2(\F_q) \rvert^2
  \]
  of the $q^9$ elements of $\gl_3(\F_q)$ have adjoint orbits of
  dimension $4$. The other elements are either scalar or have
  $6$-dimensional adjoint orbits.  This reflects the fact that $(q-1)
  \lvert \mathbb{P}^2(\F_q) \rvert^2$ of the $q^8-1$ non-zero elements
  of $\fsl_3(\F_q)$ elements are irregular in the sense
  of~\cite[Section~6.1]{AKOV1}.  Similar considerations hold for
  $\gu_3(\F_q)$.  Sorting matrices in $\gl_3(\F_q)$, respectively
  $\gu_3(\F_q)$, by their shadows therefore yields a partition
  refining the stratification by centraliser dimension or (in type
  $\mathsf{A}_2$ equivalently) by sheets; cf.\
  Remark~\ref{rem:hensel.lifts}.
\end{remark}

The proof of Theorem~\ref{thm:G.shad.graph} is given in
Section~\ref{subsec:proof.main.SL}.

\subsection{Similarity classes of $3\times3$
  matrices} \label{appendix:similarity.new} Let $\len \in \N$ be
fixed.  In preparation for the proof of Theorem~\ref{thm:G.shad.graph}
we introduce some notation and refine several results
from~\cite{APOV}.  For elements $a,b,c,d$ in $\lri$, or its finite
quotient $\lri_\len$, and $m \in \N \cup \{\infty\}$, let
\begin{equation*}
  D(a,b,c)=
  \begin{bmatrix}
    a & 0 & 0 \\
    0 & b & 0 \\
    0 & 0 & c
  \end{bmatrix}
  \quad \text{and} \quad
  E = E(m,a,b,c,d) =
  \begin{bmatrix}
    d & \pi^m & 0 \\
    0 & d & 1 \\
    a & b & c+d
  \end{bmatrix}.
\end{equation*}

A matrix $C \in \gl_n(\lri_\len)$ is called \emph{cyclic} if
$\lri_\len^{\, n}$ is cyclic as an $\lri_\len[C]$-module, i.e.\ if
there exists $v \in \lri_\len^{\, n}$ such that the $\lri_\len$-span
of $\{ C^iv \mid 0 \leq i < n \}$ is equal to $\lri_\len^{\, n}$.
Equivalently, a matrix is cyclic if it is similar to the companion
matrix of its characteristic polynomial.

For every $\nu \in \N_0$, fix a set of representatives for $\lri_\nu =
\lri / \fp^\nu$ in $\lri$,
\[
\varsigma(\lri_\nu) = \varsigma_\nu(\lri_\nu) \subset \lri,
\]
including $\pi^\nu$ as a representative for~$0$, so that $0 \leq
v(\varsigma(a)) \leq \nu$ for all $a \in \lri_\nu$.  The particular
choice $\varsigma(0) = \pi^\nu$, rather than the less pretentious
convention $\varsigma(0)=0$, plays a role in the subcase
($\text{III}_\infty$) in Theorem~\ref{thm:irredundant.list.for.M3x3}
below; otherwise it has no significance.  The valuation map naturally
extends to $\lri_\nu$ via $v(a) \coloneqq v(\varsigma(a))$.

In the formulation of Theorem~\ref{thm:irredundant.list.for.M3x3} and
some of the proofs below, we slightly abuse notation in two ways.
Firstly, we write $\varsigma(\lri_\nu) \subset \lri_\mu$ for $\nu <
\mu$ to denote the reduction of $\varsigma(\lri_\nu)$
modulo~$\fp^\mu$.  Secondly, we write $\pi^\nu \lri_{\len-\nu}
\subset \lri_\len$ for the reduction of $\pi^\nu
\varsigma(\lri_{\len-\nu})$ modulo $\pi^\len$.  These conventions are
also applied in the obvious way to expressions involving matrices.

The next theorem gives a complete description of the similarity
classes in $\gl_3(\lri_\len)$.  In addition it describes the resulting
shadows.

\begin{thm} \label{thm:irredundant.list.for.M3x3} Let $\mathcal{C}
  \subset \gl_3(\lri_\len)$ be a similarity class.  Then $\mathcal{C}$
  contains exactly one of the following matrices.
  \begin{enumerate}
  \item[(i)] $d\Id_3$, where $d \in \lri_\len$; the shadow $\shGC(\cC)$
    has type~$\Gsha$.
  \item[(ii)] $d\Id_3 + \pi^{i}D(a,0,0)$, where $0 \leq i<\len$, $\; d
    \in \lri_\len$ and $a \in \lri_{\len-i}^\times$; the shadow
    $\shGC(\cC)$ has type~$\Lsha$.
  \item[(iii)]  $\phantom{x}$ \vspace*{-0.7\baselineskip}
    \[
    d\Id_3 + \pi^i D(a,0,0) +\pi^j
    \begin{bmatrix}
      0 & 0 \\
      0 & C
    \end{bmatrix},
    \]
    where $0 \leq i < j < \len$, $\; d \in \varsigma(\lri_{j})$, $\; a
    \in \lri_{\len-i}^\times$ and $C \in \gl_2(\lri_{\len-j})$ a
    companion matrix; the shadow $\shGC(\cC)$ has type $\Tasha$,
    $\Tbsha$ or $\Msha$, depending on $C$.
  \item[(iv)] $d\Id_3+\pi^iC$, where $0\leq i<\len$, $\; d \in
    \varsigma(\lri_{i})$ and $C \in \gl_3(\lri_{\len-i})$ a companion
    matrix; the shadow $\shGC(\cC)$ has type $\Tasha$, $\Tbsha$,
    $\Tcsha$, $\Msha$ or $\Nsha$, depending on $C$.
  \item[(v)] $d'\Id_3+\pi^i E$, where $0\leq i<\len$, $\; d' \in
    \varsigma(\lri_{i})$ and $E \in \gl_3(\lri_{\len-i})$ is one of
    the following matrices:
    \begin{enumerate}
    \item[(I)] $E(\len-i,0,0,c,d)$, where $c,d \in \lri_{\len-i}$ with
      $v(c)>0$,
    \item[(II)] $E(\mu,a,b,c,d)$, where $1 \leq \mu < \len-i$,
      $\; a,b \in \lri_{\len-i}$ with $\mu = v(b) \leq v(a)$, $\; c
      \in \lri_{\len-i}$ with $v(c)>0$ and $d \in
      \varsigma(\lri_\mu)$,
    \item[($\text{III}_1$)] $E(\mu,a,b,c,d)$, where $1 \leq \mu <
      \len-i$, $\; a,b \in \lri_{\len-i}$ with $\mu = v(a) < v(b)$,
      $\; c \in \lri_{\len-i}$ with $v(c)>0$ and $d \in
      \varsigma(\lri_\mu)$,
    \item[($\text{III}_0$)] $E(\mu,a,b,c,d)$, where $1 \leq \mu <
      \len-i$, $\; a,b \in \lri_{\len-i}$ with $\mu < v(a)$ and $\mu <
      v(b)$, $\; c \in \lri_{\len-i}$ with $v(c)>0$ and $d\in
      \varsigma(\lri_\mu)$,
    \item[($\text{III}_\infty$)] $E(m,a,b,c,d)$, where $1 \leq \mu <
      \len-i$, $\; \mu < m \leq \len-i$, $\; a \in
      \varsigma(\lri_{\len-i-m+\mu})$ and $b \in \lri_{\len-i}$ with
      $\mu = v(a) < v(b)$, $\; c \in \lri_{\len-i}$ with $v(c)>0$ and
      $d \in \varsigma(\lri_\mu)$;
    \end{enumerate}
    the shadow $\shGC(\cC)$ in these subcases has type $\Jsha$,
    $\Msha$, $\Nsha$, $\Kasha$ and~$\Kbsha$, respectively.
  \end{enumerate}
\end{thm}

Similarity classes in $\gl_3(\lri_\len)$ were already studied
in~\cite{APOV}, where the following is proved.

\begin{thm}[{\cite[\S3 \S4]{APOV}}]\label{thm:sim.class.ring} Every matrix
  $A \in \gl_3(\lri_\len)$ is $\Ad(\GL_3(\lri))$-conjugate to the
  reduction modulo $\fp^\len$ of at least one matrix of one of five
  types described below.  Matrices of different types are not
  $\Ad(\GL_3(\lri))$-conjugate.

  The types \textup{(i)--(v)} consist of matrices of the following
  form, where $a,b,c,d \in \lri$:
  \begin{enumerate}
  \item[(i)] $d\Id_3$,
  \item[(ii)] $d\Id_3 + \pi^{i}D(a,b,b)$, where $0\leq i<\len$ and $a
    \not\equiv_\fp b$,
  \item[(iii)] $\phantom{x}$ \vspace*{-0.7\baselineskip}
    \[
    d\Id_3 + \pi^i D(a,b,b) +\pi^j
    \begin{bmatrix}
      0 & 0 \\
      0 & \varsigma(C)
    \end{bmatrix},
    \]
    where $0 \leq i < j < \len$, $\; a \not\equiv_\fp b$ and
    $C\in\gl_2(\lri_{\len-j})$ is cyclic,
  \item[(iv)] $d\Id_3+\pi^i \varsigma(C)$, where $0\leq i<\len$ and
    $C\in\gl_3(\lri_{\len-i})$ is cyclic,
  \item[(v)] $d\Id_3+\pi^iE(m,a,b,c,0)$, where $0 \leq i < \len$ and $m,
    v(a), v(b), v(c) > 0$.
  \end{enumerate}
\end{thm}

In each of the families (i)--(iv) it is easy to specify a set of
pairwise non-conjugate representatives, using
Lemma~\ref{lem:diagonal-form-shadow} below.  In contrast, it is less
clear how to manufacture a set of pairwise non-conjugate
representatives of the family~(v), and the problem remained unsolved
in~\cite{APOV}.  Theorem~\ref{substitute.for.AOPV.new} below removes
this stumbling block and leads to a proof of
Theorem~\ref{thm:irredundant.list.for.M3x3}.

\begin{lem} \label{lem:diagonal-form-shadow} Let $(n_1,\ldots,n_r)$ be
  a composition of $n \in \N$.  Let
  \[
  A = \diag(A_1,\ldots,A_r), \quad A' =
  \diag(A_1',\ldots,A_r') \in \gl_n(\lri_\len)
  \]
  be block diagonal matrices with blocks $A_i,A_i' \in
  \gl_{n_i}(\lri_\len)$ such that $A_i \equiv_\fp A'_i \equiv_\fp
  a_i\Id_{n_i}$, where $a_1, \ldots, a_r \in \lri_\len$ with $a_i
  \not\equiv_\fp a_j$ if $i \ne j$.  Let $X \in \gl_n(\lri_\len)$.

  Then $XA = A'X$ if and only if $X = \diag(X_1,\ldots,X_r)$ is block
  diagonal with blocks $X_i \in \gl_{n_i}(\lri_\len)$ satisfying $X_i
  A_i=A_i'X_i$ for $1 \leq i \leq r$.  In particular, $X \in
  \Cen_{\gl_n(\lri_\len)}(A)$ if and only if $X =
  \diag(X_1,\ldots,X_r)$ with $X_i \in
  \Cen_{\gl_{n_i}(\lri_\len)}(A_i)$ for $1 \leq i \leq r$.
\end{lem}

\begin{proof} The \emph{if} part is clear.  For the \emph{only if}
  part, write $X$ as a block matrix $(X_{ij})_{ij}$ with $X_{ij} \in
  \gl_{n_i,n_j}(\lri_\len)$ for $1 \leq i,j \leq r$.  Reducing the
  equation $XA = A'X$ modulo $\fp$, and looking at the $(i,j)$-th
  block for $i \ne j$, we get $(a_j - a_i) X_{ij} \equiv_\fp 0$.  It
  follows that $X_{ij} \equiv_\fp 0$.  Repeating the same argument
  with a sequence of reductions modulo $\fp^m$ for $2 \leq m \leq
  \len$ gives $X_{ij}=0$.  The equality $XA=A'X$ now implies
  $X_{ii}A_i=A_i' X_{ii}$.
\end{proof}

\begin{proof}[Proof of Theorem~\textup{\ref{thm:irredundant.list.for.M3x3}}
  \textup{(}modulo Proposition~\textup{\ref{shadows.case.v.new}} and
  Theorem~\textup{\ref{substitute.for.AOPV.new})}]
  Applying Theorem~\ref{thm:sim.class.ring}, we single out unique
  representatives for similarity classes in each of the cases
  (i)--(v).  To determine the shadow types it suffices to pin down the
  Lie centraliser shadows of these representatives.

  \smallskip

  \noindent
  (i) The assertion for scalar matrices is immediate.

  \smallskip

  \noindent
  (ii) Any matrix of the form $d\Id_3+\pi^i D(a,b,b)$ with $0 \leq i <
  \len$ and $a \not\equiv_\fp b$ can be written as $d'\Id_3 +
  \pi^iD(a',0,0)$, where $d' \in \lri_\len$ and $a \in
  \lri_{\len-i}^\times$.  Moreover, two matrices of the latter form
  are conjugate if and only if they have the same eigenvalues, or
  equivalently, have the same parameters.  By
  Lemma~\ref{lem:diagonal-form-shadow}, the centraliser of such a
  matrix is block diagonal with blocks of sizes $1\times1$ and
  $2\times2$, modulo $\fp^{\len - i}$.  Hence the corresponding shadow
  is of type $\Lsha$; cf.\ Table~\ref{tab:shadows.iso.GL}.

  \smallskip

  \noindent
  (iii) Clearly, any matrix of the form
  \begin{equation}\label{eq:case.iii.red}
    d\Id_3 + \pi^i D(a,b,b) +\pi^j
    \begin{bmatrix}
      0 & 0 \\
      0 & C
    \end{bmatrix},
  \end{equation}
  where $0 \leq i < j < \len$, $\; a \not\equiv_\fp b$ and $C \in
  \gl_2(\lri_{\len-j})$ is cyclic, can be written as
  \begin{equation}\label{eq:case.iii}
    A = d'\Id_3 + \pi^i D(a',0,0) +\pi^j
    \begin{bmatrix}
      0 & 0 \\
      0 & C'
    \end{bmatrix}
  \end{equation}
  with $d' \in \varsigma(\lri_j)$, $\; a' \in \lri_{\len-i}^\times$ and
  $C' \in \gl_2(\lri_{\len-j})$ cyclic.  After a further conjugation by
  an appropriate block diagonal matrix, with blocks $1\times1$ and
  $2\times2$, we may assume that $C'$ is a companion matrix.  It
  follows that matrices of the form~\eqref{eq:case.iii}, with $C'$ a
  companion matrix, represent as many classes as matrices of the
  form~\eqref{eq:case.iii.red}.  At the same time they are pairwise
  non-conjugate by Lemma~\ref{lem:diagonal-form-shadow}.  The same
  lemma implies that the centraliser of any matrix of the
  form~\eqref{eq:case.iii} is block diagonal, modulo~$\fp^{\len-i}$.
  Combined with the cyclicity of the $2\times2$-block $C'$, we deduce
  that the Lie centraliser shadow $\shg(A)$ is of the form $\gl_1(\kk)
  \times \kk[\overline{C'}]$, because the algebra $\kk[\overline{C'}]$
  generated by the reduction $\overline{C'}$ modulo~$\fp$ is equal to
  the reduction modulo $\fp$ of the centraliser of $C'$ in
  $\gl_2(\lri)$.  Consequently, the shadow $\shGC(\cC)$ has type
  $\Tasha$, $\Tbsha$ or $\Msha$ according to $\overline{C'}$ being
  split semisimple, non-split semisimple or a scalar translate of a
  nilpotent matrix.

  \smallskip

  \noindent
  (iv) Any matrix of the form $d\Id_3+\pi^iC \in \gl_3(\lri_\len)$, with
  $0\leq i<\len$ and $C \in \gl_3(\lri_{\len-i})$ cyclic, can be
  written as $d'\Id_3 + \pi^i C'$ with $d' \in \varsigma(\lri_{i})$ and
  $C' \in \gl_3(\lri_{\len-i})$ again cyclic.  Such a matrix has a
  unique conjugate $A$ of the same form where the cyclic matrix is a
  companion matrix, proving the first part of the assertion.  The
  image of the centraliser of $A$ in $\gl_3(\lri_{\len-i})$ is the
  same as the centraliser of $C$.  Hence the shadow $\shGC(C)$ has
  type equal to one of $\Tasha$, $\Tbsha$, $\Tcsha$, $\Nsha$, $\Msha$,
  depending on $C$; for instance, see~\cite[Lemma~7.5]{AKOV1}.

  \smallskip

  \noindent
  (v) The assertion follows from Proposition~\ref{shadows.case.v.new}
  and Theorem~\ref{substitute.for.AOPV.new} below.
\end{proof}

It remains to establish Proposition~\ref{shadows.case.v.new} and
Theorem~\ref{substitute.for.AOPV.new} below: our goal is to produce a
complete and irredundant list of representatives for the similarity
classes in case (v) of Theorem~\ref{thm:sim.class.ring} and to compute
their shadows.  Interestingly, shadows will play a crucial role in
producing the list of representatives in the first place.

For $m \in \N \cup \{\infty\}$ and $a,b,c,d \in \lri$, we write
$E_\len(m,a,b,c,d)$ for the reduction of $E(m,a,b,c,d)$
modulo~$\fp^\len$.  Furthermore, it will be convenient to keep track
of the parameter~$\len$ by writing $E_\len$, or more generally
$A_\len$, for elements of $\gl_3(\lri_\len)$.  It follows from
\cite[Proposition~3.5]{APOV} that
\[
\mathcal{E}_\len \coloneqq \left\{E_\len(m,a,b,c,d) \mid m \in \N \text{ and }
  a,b,c,d \in \lri \text{ with } v(a), v(b), v(c) > 0 \right\} \subset
\gl_3(\lri_\len)
\]
is an exhaustive, but redundant set of representatives for the
similarity classes of all matrices in $\gl_3(\lri_\len)$ such that the
minimal polynomial of their reduction modulo $\fp$ takes the form $(X
- \alpha)^2$, $\alpha \in \kk$.  The parameter
\[
\mu_\len(m,a,b) = \min\{m,v(a),v(b),\len\} \in \{1,\ldots,\len\}
\]
allows us to partition the set $\mathcal{E}_\len$ into disjoint
subsets
\begin{align*}
  \mathcal{E}_\len^{\mathrm{I}} & = \{ E_\len(m,a,b,c,d) \in
  \mathcal{E}_\len \mid \mu_\len(m,a,b) = \len \}, \\
  \mathcal{E}_\len^{\mathrm{II}} & = \{ E_\len(m,a,b,c,d) \in
  \mathcal{E}_\len \mid \mu_\len(m,a,b) < \len \text{ and }
  \mu_\len(m,a,b) = v(b)
  \leq \min\{m,v(a)\} \}, \\
  \mathcal{E}_\len^{\mathrm{III}} & = \{ E_\len(m,a,b,c,d) \in
  \mathcal{E}_\len \mid \mu_\len(m,a,b) < \len \text{ and }
  \mu_\len(m,a,b) = \min\{m,v(a)\} < v(b) \}.
\end{align*}
The third set can be divided further into three disjoint subsets
\begin{align*}
  \mathcal{E}_\len^{\mathrm{III, 1}} & = \{ E_\len(m,a,b,c,d) \in
  \mathcal{E}_\len^{\mathrm{III}} \mid \mu_\len(m,a,b) = m = v(a) \}, \\
  \mathcal{E}_\len^{\mathrm{III, 0}} & = \{ E_\len(m,a,b,c,d) \in
  \mathcal{E}_\len^{\mathrm{III}} \mid \mu_\len(m,a,b) = m < v(a) \}, \\
  \mathcal{E}_\len^{\mathrm{III, \infty}} & = \{ E_\len(m,a,b,c,d) \in
  \mathcal{E}_\len^{\mathrm{III}} \mid \mu_\len(m,a,b) = v(a) < m \}.
\end{align*}
The division of $\mathcal{E}_\len$ into subsets according to the
parameter $\mu_\len(m,a,b)$ is motivated by the following observation.

\begin{lem} \label{lem:mu-invariant} Let $E_\len = E_\len(m,a,b,c,d)
  \in \mathcal{E}_\len$.  Then $\mu \coloneqq \mu_\len(m,a,b)$ and the
  reduction of~$d$ modulo $\fp^\mu$ are invariants of the similarity
  class of $E_\len$, i.e.\ if $E_\len(m,a,b,c,d)$ is similar to
  $E_\len(m',a',b',c',d')$ then $\mu_\len(m,a,b) = \mu_\len(m',a',b')$
  and $d \equiv_{\fp^\mu} d'$.
\end{lem}

\begin{proof}
  The first claim follows from the fact that
  \begin{multline*}
    \mu_\len(m,a,b) = \max \big\{ \min \{ m, v(a), v(b), v(\tilde d),
    v(\tilde d(c+ \tilde d)-b), \len \} \mid \tilde d \in \lri \big\} \\
    = \max \big\{ \min \big(\{ v(\det(M)) \mid M \text{ a
      $2\times2$-submatrix of } E(m,a,b,c,\tilde d) \} \cup
    \{\len\}\big) \mid \tilde d \in \lri \big\}.
  \end{multline*}
  As for the second claim, the reduction of $E_\len$ modulo $\fp^\mu$
  has eigenvalues congruent to $d$, $d$ and $c+d$ modulo $\fp^\mu$.
  Hence $d$ modulo $\fp^\mu$ is the unique eigenvalue of multiplicity
  at least $2$ of this matrix and consequently an invariant of the
  similarity class of~$E_\len$.
\end{proof}

In Section~\ref{sec:centrs.shadows} we determine the centralisers and
shadows of matrices in each of the sets
$\mathcal{E}_\len^{\mathrm{I}}$, $\mathcal{E}_\len^{\mathrm{II}}$,
$\mathcal{E}_\len^{\mathrm{III,1}}$,
$\mathcal{E}_\len^{\mathrm{III,0}}$ and
$\mathcal{E}_\len^{\mathrm{III,\infty}}$.
Corollary~\ref{cor:no.overlap} shows that the five sets cover disjoint
sets of similarity classes in~$\gl_3(\lri_\len)$.  In
Section~\ref{sec:reps} we extend Lemma~\ref{lem:mu-invariant} and
determine, in Theorem~\ref{substitute.for.AOPV.new}, for each of the
five sets, explicit representatives for the similarity classes covered
by that set.

\subsubsection{Centralisers and shadows} \label{sec:centrs.shadows}

\begin{prop} \label{prop:centraliser.of.E.new} Let $E_\len =
  E_\len(m,a,b,c,d) \in \mathcal{E}_\len$.  Then the centraliser
  $\Cen_{\gl_3(\lri)}(E_\len)$ consists of all matrices in
  $\gl_3(\lri)$ which are congruent modulo $\fp^\len$ to a matrix of
  the form
  \[
  F_{m,a,b,c}(t_1,t_2,s_1,s_2,s_3) = \Mat {t_1}{\pi^m {s_3}
    -c{s_1}}{s_1}{s_2}{t_2}{s_3}{a{s_3}}{\pi^m{s_2}+b{s_3}}{{t_2}+c{s_3}},
  \]
  where ${t_1},{t_2},{s_1},{s_2},{s_3} \in \lri$ satisfy the
  congruences
  \begin{equation}\label{eq:cong.new}
    \begin{split}
      a{s_1}  & \equiv_{\fp^\len} \pi^m {s_2},   \\
      b{s_1}  & \equiv_{\fp^\len}  \pi^m ({t_2}-{t_1}), \\
      b{s_2} & \equiv_{\fp^\len} a ({t_2}-{t_1}).
    \end{split}
  \end{equation}

  The centraliser $\Cen_{\GL_3(\lri)}(E_\len)$ consists of the same
  matrices subject to the additional condition that ${t_1},{t_2} \in
  \lri^\times$.
\end{prop}

\begin{proof}
  This is a straightforward computation; see~\cite[\S4.1]{APOV}.  The
  additional condition for invertible matrices is obtained by
  considering $F_{m,a,b,c}(t_1,t_2,s_1,s_2,s_3)$ modulo~$\fp$.
\end{proof}

For the following proposition recall the shadow types listed in
Table~\ref{tab:shadows.iso.GL}.

\begin{prop}\label{shadows.case.v.new} Let $\cC$ be the similarity class
  of a matrix $E_\len = E_\len(m,a,b,c,d) \in \mathcal{E}_\len$.  Then
  the shadow $\shGC(\cC)$ is classified as follows.
  \begin{enumerate}
  \item If $E_\len \in \mathcal{E}_\len^{\mathrm{I}}$, then
    $\shGC(\cC)$ has type~$\Jsha$.
  \item If $E_\len \in \mathcal{E}_\len^{\mathrm{II}}$, then
    $\shGC(\cC)$ has type~$\Msha$.
  \item If $E_\len \in \mathcal{E}_\len^{\mathrm{III,1}}$, then
    $\shGC(\cC)$ has type~$\Nsha$.
  \item If $E_\len \in \mathcal{E}_\len^{\mathrm{III,0}}$, then
    $\shGC(\cC)$ has type~$\Kasha$.
  \item If $E_\len \in \mathcal{E}_\len^{\mathrm{III,\infty}}$, then
    $\shGC(\cC)$ has type~$\Kbsha$.
  \end{enumerate}
\end{prop}

\begin{proof}
  We put $\mu = \mu_\len(m,a,b)$, and throughout we use
  Proposition~\ref{prop:centraliser.of.E.new}.

  \smallskip

  (1) Suppose that $E_\len \in \mathcal{E}_\len^{\mathrm{I}}$.  Then
  the congruences~\eqref{eq:cong.new} hold trivially and
  $\Cen_{\gl_3(\lri)}(E_\len)$ consists of all the matrices which are
  congruent modulo $\fp^\len$ to matrices of the form
  \[
  F_{\infty,0,0,c}(t_1,t_2,s_1,s_2,s_3) = \Mat
  {t_1}{-c{s_1}}{s_1}{s_2}{t_2}{s_3}{0}{0}{{t_2}+c{s_3}}.
  \]
  Since $v(c) > 0$, we deduce that the collection of the reductions
  modulo $\fp$ of these matrices, i.e.\ the Lie centraliser shadow
  $\shg(E_\len)$, is equal to the centraliser of the matrix
  $\left[\begin{smallmatrix}
      0 & 0 & 0 \\
      0 & 0 & 1 \\
      0 & 0 & 0
    \end{smallmatrix} \right] \in \gl_3(\kk)$.  Hence $\shGC(\cC)$ has
  type~$\Jsha$.

  \smallskip

  (2) Suppose that $E_\len \in \mathcal{E}_\len^{\mathrm{II}}$.
  By~\cite[Lemma 4.3]{APOV}, we may assume that $\mu = m = v(b) \leq
  v(a)$.  Setting $\alpha = a/\pi^\mu$ and $\beta = b/\pi^\mu$, we
  have $v(\alpha) \geq 0$ and $v(\beta)=0$.  Moreover, the
  congruences~\eqref{eq:cong.new} are equivalent to the conditions
  $\alpha s_1 \equiv_{\fp^{\len - \mu}} s_2$ and $\beta s_1 \equiv
  _{\fp^{\len - \mu}} t_2-t_1$.  From this we deduce that the Lie
  centraliser shadow of $E_\len$ is
  \[
  \shg(E_\len) = \left\{ \left[
      \begin{smallmatrix}
        t_1  & 0 & (t_2 - t_1)/ \bar{\beta} \\
        \bar\alpha (t_2 - t_1)/ \bar\beta & \, t_2 & s_3 \\
        0 & 0 & t_2
      \end{smallmatrix} \right] \mid t_1,t_2,s_3 \in \kk  \right\}.
  \]
  Left-conjugation by
  $\left[\begin{smallmatrix}
      1  & 0& -1 / \bar{\beta} \\
      \bar\alpha/ \bar\beta & 1 & 0 \\
      0 & 0 & 1
    \end{smallmatrix} \right]$ maps $\shg(E_\len)$ onto the
  centraliser of $\left[\begin{smallmatrix}
      0 & 0 & 0 \\
      0 & 1 & 1 \\
      0 & 0 & 1
    \end{smallmatrix} \right] \in \gl_3(\kk)$. Hence $\shGC(\cC)$
  has type~$\Msha$.

  \smallskip

  (3),(4),(5) Suppose $E_\len \in \mathcal{E}_\len^{\mathrm{III}}$ so
  that $\mu = \min \{m, v(a) \} < v(b)$.  Setting $\alpha =
  a/\pi^\mu$, $\beta = b/\pi^\mu$ and $\rho = \pi^m/\pi^\mu$, we have
  $\min \{ v(\alpha), v(\rho) \} = 0$ and $v(\beta)>0$.

  The congruences~\eqref{eq:cong.new} are equivalent to the conditions
  $\alpha s_1 \equiv_{\fp^{\len - \mu}} \rho s_2$, $\beta s_1
  \equiv_{\fp^{\len - \mu}} \rho (t_2-t_1)$ and $\beta s_2
  \equiv_{\fp^{\len - \mu}} \alpha(t_2-t_1)$.  From this we deduce
  that
  \[
  \shg(E_\len) = \left\{\left[
      \begin{smallmatrix}
        t  & 0 &  s_1 \\
        s_2 & t & {s_3} \\
        0 & 0 & t
      \end{smallmatrix} \right] \mid t, s_1, s_2, s_3 \in \kk \text{
      such that } \bar{\alpha} s_1 = \bar{\rho} s_2  \right\}.
  \]
  Consideration of $(\bar\alpha : \bar\rho) \in \mathbb{P}^1(\kk)$
  leads naturally to the distinction into subcases $E_\len \in
  \mathcal{E}_\len^{\mathrm{III,1}}$, $E_\len \in
  \mathcal{E}_\len^{\mathrm{III,0}}$ and $E_\len \in
  \mathcal{E}_\len^{\mathrm{III,\infty}}$.

  First suppose that $E_\len \in \mathcal{E}_\len^{\mathrm{III,1}}$.
  Then $(\bar{\alpha} : \bar{\rho}) \not\in \{ (0:1), (1:0) \}$ and
  left-conjugation by $\left[\begin{smallmatrix}
      0 & \bar{\alpha} / \bar{\rho} & 0 \\
      \bar{\alpha} / \bar{\rho} & 0 & 0 \\
      0 & 0 & 1
    \end{smallmatrix} \right]$ maps $\shg(E_\len)$ onto the
  centraliser of $\left[\begin{smallmatrix}
      0 & 1 & 0 \\
      0 & 0 & 1 \\
      0 & 0 & 0
    \end{smallmatrix} \right] \in \gl_3(\kk)$.  Hence $\shGC(\cC)$ has
  type~$\Nsha$.  Now suppose that $E_\len \in
  \mathcal{E}_\len^{\mathrm{III,0}} \cup
  \mathcal{E}_\len^{\mathrm{III,\infty}}$.  Then $(\bar{\alpha} :
  \bar{\rho}) \in \{ (0:1), (1:0) \}$, where $(0:1)$ matches the
  subscript $0$ and $(1:0)$ matches $\infty$, and $\shg(E_\len)$ is
  equal to
  \begin{equation*}
    \left\{\left[\begin{smallmatrix}
          {t}  & 0&  s_1 \\
          0 & {t} & {s_3} \\
          0 & 0 & {t}
        \end{smallmatrix} \right] \mid t, s_1, s_3 \in \kk \right\}
    \quad \text{or} \quad \left\{\left[\begin{smallmatrix}
          {t}  & 0&  0 \\
          s_2 & {t} & {s_3} \\
          0 & 0 & {t}
        \end{smallmatrix} \right] \mid t, s_2, s_3 \in \kk \right\}
  \end{equation*}
  accordingly.  The corresponding shadows are of types $\Kasha$
  and~$\Kbsha$; see Section~\ref{subsec:branching}.
\end{proof}

\begin{cor} \label{cor:no.overlap} Each similarity class of a matrix
  in $\mathcal{E}_\len$ intersects precisely one of the five sets
  $\mathcal{E}_\len^{\mathrm{I}}$, $\mathcal{E}_\len^{\mathrm{II}}$,
  $\mathcal{E}_\len^{\mathrm{III,1}}$,
  $\mathcal{E}_\len^{\mathrm{III,0}}$ and
  $\mathcal{E}_\len^{\mathrm{III,\infty}}$.
\end{cor}

\subsubsection{Representatives} \label{sec:reps} We extend
Lemma~\ref{lem:mu-invariant} as follows.

\begin{prop}\label{action.on.parameters.J.new} Let $E_\len =
  E_\len(m,a,b,c,d) \in \mathcal{E}_\len$ and $\mu = \mu_\len(m,a,b)$.
  Then $\mu$ and the reduction of $d$ modulo $\pi^\mu$ are invariants
  of the similarity class of $E_\len$.  Moreover, $E_\len$ is similar
  to a matrix $E_\len' = E_\len(m',a',b',c',d')$ with $d' \in
  \varsigma(\lri_{\mu})$, and $m' = \mu$ if $E_\len \in
  \mathcal{E}_\len^{\mathrm{I}} \cup \mathcal{E}_\len^{\mathrm{II}}$.
\end{prop}

\begin{proof}
  The first part was proved in Lemma~\ref{lem:mu-invariant}.  It
  remains to show that $E_\len$ is similar to a matrix $E_\len' =
  E_\len(m',a',b',c',d')$ with $d' \in \varsigma(\lri_{\mu})$.  By
  Corollary~\ref{cor:no.overlap}, we may treat the cases $E_\len \in
  \mathcal{E}_\len^{\mathrm{I}}$, $E_\len \in
  \mathcal{E}_\len^{\mathrm{II}}$ and $E_\len \in
  \mathcal{E}_\len^{\mathrm{III}}$ one by one.

  \smallskip

  (1) Suppose that $E_\len \in \mathcal{E}_\len^{\mathrm{I}}$.  Then
  $\mu = \len$ and clearly we may assume that $d \in
  \varsigma(\lri_\mu)$.

  \smallskip

  (2) Suppose that $E_\len \in
  \mathcal{E}_\len^{\mathrm{II}}$. By~\cite[Lemma 4.3]{APOV}, we may
  assume that $\mu = m = v(b)$.  Consider the one-parameter
  subgroup
  \[
  X \colon \lri \rightarrow \GL_3(\lri), \quad x \mapsto X(x) =
  \begin{bmatrix}
    1  & 0 & 0 \\
    x & 1 & 0 \\
    \pi^mx^2 & 2\pi^mx & 1
  \end{bmatrix}
  \]
  A straightforward computation shows that for any $x \in \lri$ we
  have
  \[
  X(x) E(m,a,b,c,d) = E(m',a',b',c',d') X(x),
  \]
  where
  \[
  \begin{array}{llll}
    (1) \quad &a'= a-bx+\pi^{2m}x^3+c\pi^mx^2, & \quad (4) \quad &d'=
    d-\pi^{m}x, \\
    (2) \quad &b'= b-3\pi^{2m}x^2-2\pi^m c x, & \quad (5) \quad &m'= m. \\
    (3) \quad &c'= c+3\pi^mx,
  \end{array}
  \]
  Since $\mu = m$, we can choose $x$ such that $d' \in
  \varsigma(\lri_\mu)$ as claimed.

  \smallskip

  (3) Suppose that $E_\len \in \mathcal{E}_\len^{\mathrm{III}}$ so
  that $\mu = \min \{m,v(a)\}$. If $\mu = m$, then we can argue as
  in~(2).  Now suppose that $\mu = v(a) < m$.  Consider the
  one-parameter subgroup
  \[
  Y \colon \lri \rightarrow \GL_3(\lri), \quad y \mapsto Y(y) =
  \begin{bmatrix}
    e_0(y)^{-1}  & e_0(y)^{-1}(ay^2-cy) & e_0(y)^{-1} y \\
    0 & 1 & 0 \\
    0 & -ay & 1
  \end{bmatrix},
  \]
  where $e_0(y) = e(y) / \pi^{v(e(y))}$ with $e(y) = \pi^m + by +
  acy^2 - a^2 y^3$.  A straightforward computation shows that for any
  $y \in \lri$ we have
  \[
  Y(y) E(m,a,b,c,d) = E(m',a',b',c',d') Y(y),
  \]
  where
  \[
  \begin{array}{llll}
    (1) \quad &a'= a e_0(y), & \quad (4) \quad &d'=d+ay, \\
    (2) \quad &b'= b+2acy-3a^2y^2, & \quad (5) \quad &m'= v(e(y)). \\
    (3) \quad &c'= c-3ay,
  \end{array}
  \]
  Since $\mu = v(a)$, we can choose $y$ such that $d' \in
  \varsigma(\lri_\mu)$ as claimed.
\end{proof}

Proposition~\ref{action.on.parameters.J.new} has the following
immediate consequence.

\begin{cor} \label{cor:reduce.to.d.equal.0} Let $1 \leq \mu \leq
  \len$, and suppose that $\mathcal{R}$ is a complete and irredundant
  set of representatives for similarity classes intersected with $\{
  E_\len(m,a,b,c,0) \in \mathcal{E}_\len \mid \mu_\len(m,a,b) = \mu
  \}$.  Then
  \[
  \{ d' \Id_3 + R \mid d' \in \varsigma(\lri_\mu) \text{ and } R \in
  \mathcal{R} \}
  \]
  is a complete and irredundant set of representatives for
  similarity classes intersected with the set $\{ E_\len(m,a,b,c,d) \in
  \mathcal{E}_\len \mid \mu_\len(m,a,b) = \mu \}$.
\end{cor}

\begin{thm}\label{substitute.for.AOPV.new}
  A complete and irredundant set of representatives for the similarity
  classes intersected with the set $\mathcal{E}_\len$ is obtained as
  follows.
  \begin{enumerate}
  \item Every $E_\len \in \mathcal{E}_\len^{\mathrm{I}}$ is similar to
    a unique matrix of the form $E_\len' = E_\len(\len,0,0,c',d')$,
    where $c' \in \varsigma(\lri_\len)$ with $v(c')>0$ and $d' \in
    \varsigma(\lri_\len)$.
  \item Every $E_\len \in \mathcal{E}_\len^{\mathrm{II}}$ is similar
    to a unique matrix of the form $E_\len' =
    E_\len(\mu,a',b',c',d')$, where $1 \leq \mu < \len$, $\; a',b' \in
    \varsigma(\lri_\len)$ with $\mu = v(b') \leq v(a')$, $\; c' \in
    \varsigma(\lri_\len)$ with $v(c')>0$ and $d' \in
    \varsigma(\lri_\mu)$.
  \item Every $E_\len \in \mathcal{E}_\len^{\mathrm{III,1}}$ is
    similar to a unique matrix of the form $E_\len' =
    E_\len(\mu,a',b',c',d')$, where $1\leq \mu < \len$, $\; a',b' \in
    \varsigma(\lri_\len)$ with $\mu = v(a') < v(b')$, $\; c' \in
    \varsigma(\lri_\len)$ with $v(c')>0$ and $d' \in
    \varsigma(\lri_\mu)$.
  \item Every $E_\len \in \mathcal{E}_\len^{\mathrm{III,0}}$ is
    similar to a unique matrix of the form $E_\len' =
    E_\len(\mu,a',b',c',d')$, where $1 \leq \mu < \len$, $\; a',b' \in
    \varsigma(\lri_\len)$ with $\mu < v(a')$ and $\mu < v(b')$, $\; c'
    \in \varsigma(\lri_\len)$ with $v(c')>0$ and $d' \in
    \varsigma(\lri_\mu)$.
  \item Every $E_\len \in \mathcal{E}_\len^{\mathrm{III,\infty}}$ is
    similar to a unique matrix of the form $E_\len' =
    E_\len(m,a',b',c',d')$, where $1 \leq \mu < \len$, $\; \mu < m
    \leq \len$, $\; a' \in \varsigma(\lri_{\len-m+\mu})$ and $b' \in
    \varsigma(\lri_\len)$ with $\mu = v(a') < v(b')$, $\; c' \in
    \varsigma(\lri_\len)$ with $v(c')>0$ and $d' \in
    \varsigma(\lri_\mu)$.
  \end{enumerate}
\end{thm}

\begin{proof}
  By Corollary~\ref{cor:reduce.to.d.equal.0} it suffices to consider
  matrices $E_\len(m,a,b,c,d) \in \mathcal{E}_\len$ with $d=0$.  Fix
  $E_\len = E_\len(m,a,b,c,0) \in \mathcal{E}_\len$ and set $\mu =
  \mu_\len(m,a,b)$.  Without loss of generality we may assume that $m,
  v(a), v(b), v(c) \leq \len$.  We need to understand for which
  $m',a',b',c'$ the given matrix $E_\len$ is similar to $E_\len' =
  E_\len(m',a',b',c',0)$.  Lemma~\ref{lem:mu-invariant} shows that
  $\mu$ is an invariant of the similarity class of~ $E_\len$ and, by
  Corollary~\ref{cor:no.overlap}, we can investigate each of the five
  subsets $\mathcal{E}_\len^{\mathrm{I}}$,
  $\mathcal{E}_\len^{\mathrm{II}}$,
  $\mathcal{E}_\len^{\mathrm{III,1}}$,
  $\mathcal{E}_\len^{\mathrm{III,0}}$ and
  $\mathcal{E}_\len^{\mathrm{III,\infty}}$ separately.  The
  characteristic polynomial of $E_\len$ is $t^3 - c t^2 - b t - a
  \pi^m \in \lri_\len[t]$.  Hence, modulo~$\fp^\len$, the parameters
  $b$, $c$ and $a \pi^m$ are invariants of the similarity class
  of~$E_\len$.

  \smallskip

  (1) Suppose that $E_\len \in \mathcal{E}_\len^{\mathrm{I}}$.  Since
  $\mu = \min \{ m, v(a), v(b), \len \}$ and $c$ modulo $\fp^\len$ are
  invariants of the similarity class of $E_\len$, the claim follows.

  \smallskip

  (2) Suppose that $E_\len \in \mathcal{E}_\len^{\mathrm{II}}$.
  By~\cite[Lemma 4.3]{APOV}, we may assume that $\mu = m = v(b) \leq
  v(a)$ and we may restrict our attention to possible conjugates
  $E_\len' = E_\len(\mu,a',b,c,0)$, where $a' \in \lri$ with $v(a')
  \geq \mu$ and $v(a' - a) \geq \len -m$.  Part of the analysis for
  cases~(3),(4) below, which only requires $\mu = m \leq v(a)$, shows
  that --~in the present situation~-- for $E_\len'$ to be similar to
  $E_\len$ it is necessary that $a' \equiv_{\fp^\len} a$ and the claim
  follows.

  \smallskip

  (3),(4),(5) Suppose that $E_\len \in
  \mathcal{E}_\len^{\mathrm{III}}$.  We already observed that we can
  investigate the subsets $\mathcal{E}_\len^{\mathrm{III,1}}$,
  $\mathcal{E}_\len^{\mathrm{III,0}}$ and
  $\mathcal{E}_\len^{\mathrm{III,\infty}}$ separately.  Inspection
  shows that the elementary divisors of $E_\len$ are $1$, $\pi^m$ and
  $\pi^{v(a)}$.  Hence not only $\mu$, but even $m$ and $v(a)$ are
  invariants of the similarity class of $E_\len$.

  Thus we may restrict our attention to possible conjugates $E_\len' =
  E_\len(m,a',b,c,0)$, where $a' \in \lri$ with $v(a') = v(a)$ and
  $v(a' - a) \geq \len -m$.  Writing $a' = a + y \pi^k$ with $y \in
  \lri^\times$ and $k \geq \max \{0,\len-m\}$, we study the equation
  \begin{equation} \label{equ:EX.equals.XE} E_\len X = X E_\len',
    \qquad \text{for $X =(x_{ij}) \in \GL_3(\lri)$.}
  \end{equation}
  During the rest of the proof we abbreviate $\equiv_{\fp^\len}$ to
  $\equiv$; in all other congruences we continue to write the modulus
  explicitly.  Comparing individual matrix entries, as
  in~\cite[Section~4]{APOV}, we obtain the following collection of
  necessary and sufficient conditions for \eqref{equ:EX.equals.XE} to
  hold:
  \begin{equation*}
    \begin{array}{llll}
      (2,1) \quad & x_{31} \equiv (a + y \pi^k) x_{23}, & \quad (3,1) \quad & a
      x_{11} + b x_{21} + c x_{31} - (a + y \pi^k) x_{33} \equiv 0, \\
      (2,3) \quad & x_{33} \equiv x_{22} + c x_{23}, & \quad (1,2) \quad & \pi^m
      x_{11} \equiv \pi^m x_{22} - b x_{13}, \\
      (2,2) \quad & x_{32} \equiv \pi^m x_{21} + b x_{23}, & \quad (1,1) \quad &
      \pi^m x_{21} \equiv (a + y\pi^k) x_{13}, \\
      (3,3) \quad & x_{32} \equiv a x_{13} + b x_{23}, & \quad (3,2) \quad & a
      x_{12} + b x_{22} + c x_{32} - \pi^m x_{31} - b x_{33} \equiv 0, \\
      (1,3) \quad & x_{12} \equiv \pi^m x_{23} - c x_{13}, & \quad
      (X & \hspace*{-.6cm} \text{invertible}) \quad x_{11}, x_{22} \in
      \lri^\times.
    \end{array}
  \end{equation*}


  We claim that these conditions are equivalent to the modified
  conditions~\ref{equ:matrix.entries.2} below.  Indeed, using $(2,2)$,
  we can replace $(3,3)$ by $(3,3)'$.  Using $(3,3)'$, we can replace
  $(1,1)$ by $(1,1)'$.  Using $(2,1)$ and $(2,3)$, we can replace
  $(3,1)$ by $(3,1)'$.  Using $(2,1)$, $(2,2)$, $(2,3)$, $(1,3)$ and
  $(3,3)'$, we see that $(3,2)$ can replaced by $\pi^m y \pi^k x_{23}
  \equiv 0$ which holds automatically due to $k \geq \len - m$.
  \begin{equation} \label{equ:matrix.entries.2}
    \begin{array}{llll}
      (2,1) \quad & x_{31} \equiv (a + y \pi^k) x_{23}, & \quad (3,1)' \!\!\quad
      & a x_{11} \equiv (a + y \pi^k) x_{22} - b x_{21}, \\
      (2,3) \quad & x_{33} \equiv x_{22} + c x_{23}, & \quad (1,2) \quad & \pi^m
      x_{11} \equiv \pi^m x_{22} - b x_{13}, \\
      (2,2) \quad & x_{32} \equiv \pi^m x_{21} + b x_{23}, & \quad (1,1)'
      \!\!\quad & y \pi^k x_{13} \equiv 0, \\
      (3,3)' \!\!\quad & \pi^m x_{21} \equiv a x_{13}, &  \\
      (1,3) \quad & x_{12} \equiv \pi^m x_{23} - c x_{13}, & \quad
      (X & \hspace*{-.6cm} \text{invertible}) \quad x_{11}, x_{22} \in
      \lri^\times.
    \end{array}
  \end{equation}


  First suppose that $E_\len \in \mathcal{E}_\len^{\mathrm{III,1}}
  \cup \mathcal{E}_\len^{\mathrm{III,0}}$, corresponding to
  cases~(3),(4).  Then $\mu = m \leq \min \{v(a),v(b)\}$ and $(3,3)'$ is
  equivalent to
  \[
  (3,3)'' \quad x_{21} \equiv_{\fp^{\len-\mu}} a \pi^{-m} x_{13}.
  \]
  Multiplying $(1,2)$ by $a \pi^{-m}$ and using $(3,3)''$, we obtain
  \[
  a x_{11} \equiv a x_{22} - a \pi^{-m} b x_{13} \equiv a x_{22} - b x_{21}.
  \]
  Comparing with $(3,1)'$, we obtain the necessary condition $y \pi^k
  \equiv 0$, hence $k \geq \len$.  This means that $E_\len$ is similar
  to $E_\len' = E_\len(m,a',b,c,0)$ if and only if $a' \equiv a$.

  \smallskip

  Finally suppose that $E_\len \in
  \mathcal{E}_\len^{\mathrm{III,\infty}}$, corresponding to case~(5).
  Then $\mu = v(a) < \min \{ m,v(b) \}$ and $(3,3)'$ is equivalent to
  \[
  (3,3)''' \quad x_{13} \equiv_{\fp^{\len-\mu}} \pi^m a^{-1} x_{21}.
  \]
  Multiplying $(3,1)'$ by $\pi^m a^{-1}$ and using $(3,3)'''$, we
  obtain
  \[
  \pi^m x_{11} \equiv \pi^m x_{22} + \pi^m a^{-1} y \pi^k x_{22} - \pi^m
  a^{-1} b x_{21} \equiv \pi^m x_{22} - b x_{13} + \pi^m a^{-1} y \pi^k
  x_{22}.
  \]
  Comparing with $(1,2)$, we obtain the necessary condition $\pi^m
  a^{-1} y \pi^k \equiv 0$, hence
  \[
  k \geq \len - m + \mu.
  \]
  Conversely, if this inequality holds, then $(3,1)'$ implies $(1,2)$
  and $(3,3)'''$ implies $(1,1)'$.  The remaining conditions can
  easily be satisfied.  This means that $E_\len$ is similar to
  $E_\len' = E_\len(m,a',b,c,0)$ if and only if $v(a') = v(a) = \mu$
  and $a' \equiv_{\fp^{\len - m + \mu}} a$.
\end{proof}


\subsection{Proof of
  Theorem~\ref{thm:G.shad.graph}} \label{subsec:proof.main.SL}
Part~\eqref{item:G.conj.graph.1} of Theorem~\ref{thm:G.shad.graph}
follows from collecting the shadow types in
Theorem~\ref{thm:irredundant.list.for.M3x3}.  To prove
part~\eqref{item:G.conj.graph.2}, we consider a matrix $A_\len \in
\gl_3(\lri_\len)$ of the form given in
Theorem~\ref{thm:irredundant.list.for.M3x3}.  Let $\cC$ denote the
similarity class of $A_\len$.  Starting from the shadow $\sigma$ of
$\cC$, we determine the shadows $\tau$ associated to similarity
classes $\cCtilde$ of lifts of $A_\len$ to matrices
$\wt{A}_{\len+1} \in \gl_3(\lri_{\len+1})$.  We also keep track
of the multiplicities of such lifts.  The claim then follows from the
observation that in all cases the multiplicities depend only on the
shadows involved and are as listed in Table~\ref{tab:branch.rules.A2}.

\begin{list}{}{\setlength{\leftmargin}{0pt}
    \setlength{\labelwidth}{-10pt} \setlength{\itemsep}{2pt}
    \setlength{\parsep}{1pt}}
\item[($\Gsha$)] Suppose that $\sigma$ has type~$\Gsha$.  Then $A_\len
  = d\Id_3$ with $d \in \lri_\len$.  Let $\wt{A}_{\len+1} =
  \varsigma(d) \Id_3 +\pi^\len X \in \gl_3(\lri_{\len+1})$ be a lift
  of~$A_\len$ with $X \in \gl_3(\kk)$.  Then
  $\shG(\wt{A}_{\len+1}) = \shG(X)$ implies that the type of
  $\tau = \shGC(\cCtilde)$ is not $\Kasha$ or $\Kbsha$.  Indeed, it is
  one of $\Gsha$, $\Lsha$, $\Tasha$, $\Tbsha$, $\Tcsha$, $\Msha$,
  $\Nsha$, $\Jsha$, according to the shape of the minimal polynomial
  of~$X$.  For $\tau$ not of type $\Lsha$ the number
  $a_{\sigma,\tau}(q)$ of distinct lifts with shadow $\tau$ is the
  number of distinct minimal polynomials of the corresponding shape.
  For $\tau$ of type $\Lsha$ the number $a_{\sigma,\tau}(q)$ is the
  number of distinct minimal polynomials of the prescribed shape,
  paired with a compatible characteristic polynomial.  The numbers
  $a_{\sigma,\tau}(q)$ are easily computed from
  Table~\ref{tab:shadows.iso.GL} and can be found in
  Table~\ref{tab:branch.rules.A2}.
\item[($\Lsha$)] Suppose that $\sigma$ has type~$\Lsha$.  By
  Theorem~\ref{thm:irredundant.list.for.M3x3} we may assume that
  $A_\len$ is of the form $A_\len = d\Id_3 + \pi^i D(a,0,0)$ with $0
  \leq i < \len$, $\; d \in \lri_\len$ and $a \in
  \lri_{\len-i}^\times$.  Any lift of $A_\len$ is conjugate to a
  matrix of the form described in parts (ii) or (iii) of
  Theorem~\ref{thm:irredundant.list.for.M3x3}, that is, conjugate to a
  matrix of the form
  \[
  \wt{A}_{\len+1} = \varsigma(d)\Id_3 + \pi^i
  D(\varsigma(a),0,0) + \pi^{\len}
  \begin{bmatrix}
    f & 0 \\
    0 & F
  \end{bmatrix} \quad \text{with $f \in \kk$, $F \in \gl_2(\kk)$},
  \]
  where $F$ scalar corresponds to case (ii) and $F$ a companion matrix
  corresponds to case~(iii).  We classify the similarity classes
  depending on the form that $F$ takes.  The shadow $\tau =
  \shGC(\cCtilde)$ has one of four types:
  \begin{itemize} \renewcommand{\labelitemi}{$\circ$}
  \item $\tau$ has type $\Lsha$ if and only if $F$ is scalar.  There
    are $a_{\sigma,\tau}(q) = q^2$ choices for $(f,F)$.
  \item $\tau$ has type $\Tasha$ if and only if the characteristic
    polynomial of $F$ is separable and reducible over~$\kk$.  There
    are $a_{\sigma,\tau}(q) = \hlf (q-1)q^2$ choices for $(f,F)$.
  \item $\tau$ has type $\Tbsha$ if and only if the characteristic
    polynomial of $F$ is irreducible over~$\kk$.  There are
    $a_{\sigma,\tau}(q) = \hlf (q-1)q^2$ choices for $(f,F)$.
  \item $\tau$ has type $\Msha$ if and only if the minimal polynomial
    of $F$ is of the form $(x-\alpha)^2$ for some $\alpha \in \kk$.
    There are $a_{\sigma,\tau}(q) =q^2$ choices for $(f,F)$.
  \end{itemize}
\item[($\Jsha$)] Suppose that $\sigma$ has type~$\Jsha$.  In this case
  we may assume, by Theorem~\ref{thm:irredundant.list.for.M3x3}, that
  \[
  A_\len = d \Id_3 + \pi^i
  \begin{bmatrix}
    0 & 0 & 0 \\
    0 & 0 & 1 \\
    0 & 0 & c
  \end{bmatrix}, \quad \text{where $d \in \lri_\len$, $\; 0 \leq i <
    \len$ and $c \in \lri_{\len-i}$ with $v(c) > 0$.}
  \]
  Theorem~\ref{thm:irredundant.list.for.M3x3} yields a complete list
  of representatives for the similarity classes $\cCtilde$ of matrices
  in $\gl_3(\lri_{\len+1})$ lying above $A_\len$.  The shadow $\tau =
  \shGC(\cCtilde)$ has one of five types:
  \begin{itemize} \renewcommand{\labelitemi}{$\circ$}
  \item $\tau$ is of type $\Jsha$ if and only if the lift of $A_\len$
    is conjugate to
    \[
    \wt{A}_{\len+1} = d' \Id_3 + \pi^i
    \begin{bmatrix}
      0 & 0 & 0 \\
      0 & 0 & 1\\
      0 & 0 & c'
    \end{bmatrix},
    \]
    where $c' \in \lri_{\len-i+1}$, $\; d' \in \lri_{\len+1}$ are
    arbitrary lifts of $c$, $d$.  Consequently there are $a_{\sigma,
      \tau}(q) = q^2$ choices.
  \item $\tau$ is of type $\Msha$ if and only if the lift of $A_\len$
    is conjugate to
    \[
    \wt{A}_{\len+1} = \varsigma(d) \Id_3 + \pi^i
    \begin{bmatrix}
      0 & \pi^{\len-i} & 0 \\
      0 & 0 & 1\\
      a'\pi^{\len-i} & b'\pi^{\len-i} & c'
    \end{bmatrix},
    \]
    where $a' \in \kk$, $b' \in \kk^\times$, and $c' \in
    \lri_{\len-i+1}$ is an arbitrary lift of $c$.  There are
    $a_{\sigma, \tau}(q) = (q-1)q^2$ choices.
  \item $\tau$ is of type~$\Nsha$ if and only if the lift of $A_\len$
    is conjugate to
    \[
    \wt{A}_{\len+1} = \varsigma(d) \Id_3 + \pi^i
    \begin{bmatrix}
      0 & \pi^{\len-i} & 0 \\
      0 & 0 & 1 \\
      a'\pi^{\len-i} & 0 & c'
    \end{bmatrix},
    \]
    where $a' \in \kk^\times$ and $c' \in \lri_{\len-i+1}$ is an
    arbitrary lift of~$c$.  There are $a_{\sigma, \tau}(q) = (q-1)q$
    choices.
  \item $\tau$ is of type $\Kasha$ or $\Kbsha$ if and only if the lift
    $\wt{A}_{\len+1}$ of $A_\len$ is conjugate to a matrix of
    the form
    \[
    \wt{A}_{\len+1}(0) = \varsigma(d)\Id_3 + \pi^i
    \begin{bmatrix}
      0 & \pi^{\len-i} & 0 \\
      0 & 0 & 1 \\
      0 & 0 & c'
    \end{bmatrix}
    \quad \text{or} \quad \wt{A}_{\len+1}(\infty) =
    \varsigma(d)\Id_3 + \pi^i
    \begin{bmatrix}
      0 & 0 & 0 \\
      0 & 0 & 1\\
      \pi^{\len-i} & 0 & c'
    \end{bmatrix},
    \]
    where $c' \in \lri_{\len-i+1}$ is a lift of $c$; recall that
    $\pi^{\len-i} = \varsigma(0)$ for $0 \in \lri_{\len-i}$.  Matrices
    of the forms $\wt{A}_{\len+1}(0)$ and
    $\wt{A}_{\len+1}(\infty)$ are never conjugate and we have
    $a_{\sigma, \tau}(q) = q$; see Section~\ref{subsec:branching}.
  \end{itemize}
\item[($\Tasha,\Tbsha,\Tcsha,\Msha,\Nsha,\Kasha,\Kbsha$)] Suppose that
  $\sigma$ has type equal to one of $\Tasha$, $\Tbsha$, $\Tcsha$,
  $\Msha$, $\Nsha$, $\Kasha$, $\Kbsha$. From
  Table~\ref{tab:shadows.iso.GL} we observe that all these cases are
  minimal in the sense that $\shg(A_\len)$ cannot properly contain the
  Lie centraliser shadow of any other type.  This implies that the
  shadow associated to any lift $\wt{A}_{\len+1}$ of the matrix
  $A_\len$ satisfies $\shGC(\cCtilde) = \shGC(\cC)$.  Therefore, in
  all the cases under consideration
  Proposition~\ref{pro:class.quot.GL} and Definition~\ref{def:b.GL}
  yield
  \[
  a_{\sigma,\sigma}(q)=q^{\dim
    \gl_3}/b^{(1)}_{\sigma,\sigma}(q)=q^{\dim(\sigma)}=q^3.
  \]
\end{list}


\section{Similarity classes of anti-hermitian integral $\fp$-adic
  matrices}\label{sec:sim.gu}

Let $\lri$ be a compact discrete valuation ring, with valuation ideal
$\fp$ and finite residue field $\kk$ such that $p \coloneqq \cha(\kk)
\neq 2$.  Otherwise we impose no restriction on the characteristic of
$\lri$ or $\kk$.  Put $q = \lvert \kk \rvert$ and fix a uniformiser
$\pi$ of~$\lri$; we observe that~$q>2$.

Let $\Lri$ be an unramified quadratic extension of~$\lri$, with
valuation ideal $\fP$ and residue field~$\kk_2$, a quadratic extension
of~$\kk$.  Then $\Lri = \lri[\delta]$, where $\delta = \sqrt{\rho}$
for an element $\rho \in \lri$ whose reduction modulo $\fp$ is a
non-square in $\kk$, and $\fP = \pi \Lri$.  For $\len \in \N_0$, we
write $\delta_\len$, or even $\delta$, for the image of $\delta$
modulo~$\fp^\len$.  Let $\boldsymbol{\Lri}$ denote the integral
closure of $\Lri$ in some fixed algebraic closure of its fraction
field, and choose an $\lri$-automorphism $\circ$ of
$\boldsymbol{\Lri}$ restricting to the non-trivial Galois automorphism
of the quadratic extension $\Lri \,\vert\, \lri$.  In particular, $(a +
b\delta)^\circ = a - b \delta$ for all $a,b \in \lri$.

Let $n \in \N$.  We extend $\circ$ to obtain the \emph{standard
  $(\Lri,\lri)$-involution} `conjugate transpose' on the
$\Lri$-algebra $\mathsf{Mat}_n(\Lri)$, i.e.\
\begin{equation}\label{equ:def-std-inv}
A^\circ = \big( (a^{\; \circ}_{ij}) \big)^\text{tr} = (a^{\;
  \circ}_{ji})_{ij} \qquad \text{for $A = (a_{ij}) \in \gl_n(\Lri)$.}
\end{equation}
A matrix $A \in \gl_n(\Lri)$ is \emph{hermitian} if $A^\circ = A$ and
\emph{anti-hermitian} if $A^\circ = -A$.  The \emph{standard unitary
  group} over $\lri$ and the corresponding \emph{standard unitary
  $\lri$-Lie lattice} are
\begin{equation}\label{def:unitary}
  \GU_n(\lri) = \{A \in \GL_n(\Lri) \mid A^\circ
  A = \Id_n \} \quad \text{and} \quad
  \gu_n(\lri) = \{ A \in \gl_n(\Lri) \mid A^\circ + A = 0 \}.
\end{equation}
The associated \emph{special unitary group} and \emph{special unitary
  Lie lattice} are
\begin{equation*}
  \SU_n(\lri) = \GU_n(\lri) \cap \SL_n(\Lri) \quad \text{and} \quad
  \su_n(\lri) = \gu_n(\lri) \cap \fsl_n(\Lri).
\end{equation*}
For $\len \in \N$, we write $\lri_\len = \lri / \fp^\len$, $\Lri_\len
= \Lri / \fP^\len$ and correspondingly $\GU_n(\lri_\len)$,
$\gu_n(\lri_\len)$ et cetera.  A matrix $A \in \gl_n(\Lri_\len)$ is
called \emph{hermitian}, respectively \emph{anti-hermitian}, if it is
the image of a hermitian, respectively anti-hermitian matrix, modulo
$\fP^\len$.

Eigenvalues of matrices $A \in \gl_n(\Lri)$ are taken in the fixed
extension $\boldsymbol{\Lri}$ so that $\circ$ can be applied to them.
Throughout we shall also use $\circ$ to denote the induced action on
finite quotients $\gl_n(\Lri_\len)$ obtained by reduction modulo~$\fP^\len$.

\subsection{Preliminaries} We collect some auxiliary results regarding
hermitian and anti-hermitian matrices over discrete valuation rings,
starting with an analogue of
Proposition~\ref{prop:centralisers.connected}.

\begin{prop}\label{pro:unitary.connected}
  Let $\mathcal{X}_{n,\len}^\GU$ be the Greenberg transform of level
  $\len$ or the Weil restriction of the $\lri_\len$-scheme $\GU_n$ to
  $\kk$-schemes, depending on whether $\cha(\lri) = 0$ or $\cha(\lri)
  >0$, so that $\mathcal{X}_{n,\len}^\GU(\kk) \simeq
  \GU_n(\lri_\len)$.  Let $A \in \gu_n(\lri_\len) \subset
  \gl_n(\Lri_\len)$.  Then there exists a $\kk$-defined connected
  algebraic subgroup $\mathbf{C}$ of $\mathcal{X}_{n,\len}^\GU$ such
  that
  \[
  \mathbf{C}(\kk) \simeq \Cen_{\GU_n(\lri_\len)}(A).
  \]
\end{prop}

\begin{proof}
  Let $\mathcal{X}^{\GL}_{n,\len}$ and $\mathbf{C}^\GL$ denote the
  connected $\kk_2$-algebraic groups, supplied by the Greenberg
  functor (respectively Weil restriction) from $\Lri$-schemes to
  $\kk_2$-schemes, such that $\mathcal{X}^{\GL}_{n,\len}(\kk_2) \simeq
  \GL_n(\Lri_\len)$ and $\mathbf{C}^\GL(\kk_2) \simeq
  \Cen_{\GL_n(\Lri_\len)}(A)$; compare
  Proposition~\ref{prop:centralisers.connected}.

  The existence of a $\kk$-algebraic group $\mathbf{C}$ such that
  $\mathbf{C}(\kk) \simeq \Cen_{\GU_n(\lri_\len)}(A)$ is guaranteed by
  the general properties of the Greenberg transform (respectively Weil
  restriction) from $\lri$-schemes to $\kk$-schemes.  To see that
  $\mathbf{C}$ is connected, it suffices to observe that $\mathbf{C}
  \simeq \mathbf{C}^\GL$ over $\kk_2$.
\end{proof}

An important consequence of Proposition~\ref{pro:unitary.connected}
is the following.

\begin{prop} \label{pro:GL.conj.GU.conj} Let $A,B \in \gu_n(\lri)$ be
  similar, i.e.\ $\Ad(\GL_n(\Lri))$-conjugate.  Then $A,B$ are already
  $\Ad(\GU_n(\lri))$-conjugate.
\end{prop}

\begin{proof}
  It suffices to prove the claim modulo $\fP^\len$, i.e.\ that the
  images $A_\len, B_\len \in \gu_n(\lri_\len)$ of $A,B$ are
  $\Ad(\GU_n(\lri))$-conjugate, for $\len \in \N$.  Let $\mathbf{G} =
  \mathcal{X}_{n,\len}^\GU$ and $\mathbf{C}$ be as in
  Proposition~\ref{pro:unitary.connected} so that $\mathbf{G}(\kk)
  \simeq \GU_n(\lri_\len)$ and $\mathbf{C}(\kk) \simeq
  \Cen_{\GU_n(\lri_\len)}(A)$.  Furthermore, let
  $\mathcal{X}_{n,\len}^\gu$ be the Greenberg transform of level
  $\len$ (for $\cha(\lri)= 0$) or the Weil restriction (for
  $\cha(\lri)>0$) of the $\lri_\len$-scheme $\gu_n$ to $\kk$-schemes
  so that $\mathcal{X}_{n,\len}^\gu(\kk) \simeq \gu_n(\lri_\len)$.
  Write $\mathsf{A}, \mathsf{B} \in \mathcal{X}_{n,\len}^\gu(\kk)$ for
  the elements corresponding to $A_\len, B_\len \in \gu_n(\lri_\len)$,
  and let $\lri^\text{unr}$ denote the maximal unramified extension of
  $\lri$.

  Let $\KK$ denote an algebraic closure of $\kk$. By definition,
  $\mathbf{G}(\KK) \simeq \GL_n(\lri^\text{unr}_\len)$ acts
  transitively via the adjoint action on the orbit
  $\Ad(\mathbf{G}(\KK)) \mathsf{A}$ in $\mathcal{X}_{n,\len}^\gu(\KK)
  \simeq \gl_n(\lri^\text{unr}_\len)$.  Furthermore, $\mathsf{B} \in
  \Ad(\mathbf{G}(\KK)) \mathsf{A} \cap \mathcal{X}_{n,\len}^\gu(\kk)$.
  As the stabiliser $\mathbf{C}$ of $\mathsf{A}$ is connected, the
  Lang--Steinberg theorem implies that $\Ad(\mathbf{G}(\kk)) \simeq
  \Ad(\GU_n(\lri_\len))$ acts transitively on $\Ad(\mathbf{G}(\KK))
  \mathsf{A} \cap \mathcal{X}_{n,\len}^\gu(\kk)$; see
  \cite[Proposition~4.3.2]{Geck}.  Whence there is $g \in
  \GU_n(\lri_\len)$ such that $\Ad(g) A_\len = B_\len$.
\end{proof}

\begin{lem}\label{Dieudonne1} Let $\Gamma \in \GL_n(\Lri)$. Then $\Gamma$ is
  hermitian if and only if there exists $g \in \GL_n(\Lri)$ such that
  $\Gamma = g^\circ g$.
\end{lem}

\begin{proof} If $\Gamma = g^\circ g$ for $g \in \GL_n(\Lri)$ then
  clearly $\Gamma^\circ = \Gamma$.  For the converse direction,
  suppose that $\Gamma$ is hermitian.  It suffices to construct
  recursively a sequence $g_\len \in \GL_n(\Lri)$, $\len \in \N$, such
  that $\Gamma \equiv_{\fP^\len} g_\len^\circ g_\len$ for every $\len
  \in \N$.  The existence of $g_1$ is guaranteed by the theory of
  hermitian matrices over finite fields; e.g.\
  see~\cite[p.~16]{Dieudonne/55}.  Now suppose that $g_\len$ has been
  constructed for some $\len \in \N$.  Then $\Gamma - g_\len^{\,
    \circ} g_\len = \pi^\len \Delta$, where $\Delta \in \gl_n(\Lri)$
  is hermitian.  Thus $g_{\len+1} = g_\len + \hlf \pi^\len (g_\len^{\,
    \circ})^{-1}\Delta$ satisfies
  \[
  \begin{split}
    g_{\len+1}^{\, \circ} g_{\len+1} & = \left(g_\len^{\, \circ} +
      \hlf \pi^\len \Delta g_\len^{-1} \right) \left(g_\len+\hlf
      \pi^\len (g_\len^{\, \circ})^{-1} \Delta \right) \\
    & \equiv_{\fP^{\len+1}} g_\len^{\, \circ} g_\len +\pi^\len \Delta
    = \Gamma.  \qedhere
  \end{split}
  \]
\end{proof}

\begin{prop}\label{Dieudonne2} Let $A \in \gl_n(\Lri_\len)$. Then $A$ is
  $\Ad(\GL_n(\Lri))$-conjugate to an anti-hermitian matrix if and only
  if there exists $\Gamma \in \GL_n(\Lri_\len)$ such that
  $\Gamma^\circ=\Gamma$ and $A^\circ\Gamma+\Gamma A=0$.
\end{prop}

\begin{proof}
  First suppose that $g \in \GL_n(\Lri_\len)$ is such that $B = \Ad(g)
  A = gAg^{-1}$ is anti-hermitian.  Then $\Gamma = g^\circ g$ is
  hermitian and
  \[
  A^\circ\Gamma+\Gamma A=A^\circ g^\circ g+g^\circ g A=g^\circ B^\circ
  g+ g^\circ Bg=g^\circ (B^\circ + B)g=0.
  \]

  For the reverse implication, suppose that $\Gamma \in
  \GL_n(\Lri_\len)$ satisfies $\Gamma^\circ=\Gamma$ and
  $A^\circ\Gamma+\Gamma A=0$.  Then by Lemma~\ref{Dieudonne1} there
  exists $g \in \GL_n(\Lri_\len)$ such that $\Gamma = g^\circ g$, and
  hence $B = gAg^{-1}$ satisfies
  \[
  B^\circ + B = (g^\circ)^{-1}A^\circ g^\circ+gAg^{-1} =
  (g^\circ)^{-1} ( A^\circ\Gamma+\Gamma A ) g^{-1} = 0.
  \]
  Thus $B$ is anti-hermitian.
\end{proof}

\begin{lem}\label{cyclic.u} Let $A \in \gl_n(\Lri_\len)$ with
  characteristic polynomial $f_A = t^n + \sum_{i=0}^{n-1} c_i t^i \in
  \Lri_\len[t]$.  If $A$ is $\Ad(\GL_n(\Lri))$-conjugate to an
  anti-hermitian matrix, then $c_i^{\, \circ} = (-1)^{n-i} c_i$ for $0
  \le i < n$.  Conversely, if $A$ is cyclic then the latter condition
  on the coefficients of $f_A$ implies that $A$ is
  $\Ad(\GL_n(\Lri))$-conjugate to an anti-hermitian matrix.
\end{lem}

\begin{proof} If $A$ is $\Ad(\GL_n(\Lri))$-conjugate to an
  anti-hermitian matrix $B$ then, denoting the characteristic
  polynomial of $B$ by $f_B$, we deduce from $f_A = f_B$ and $B^\circ
  + B =0$ that
  \[
  t^n + \sum_{i=0}^{n-1} c_i^{\, \circ} t^i = (f_B)^{\circ} =
  f_{B^\circ} = f_{-B} = (-1)^n f_B(-t) = t^n + \sum_{i=0}^{n-1}
  (-1)^{n-i} c_i t^i.
  \]

  Now suppose that $A$ is cyclic and that $c_i^\circ = (-1)^{n-i} c_i$
  for $0 \le i < n$.  Without loss of generality $A = (a_{ij})$ is a
  companion matrix for $f_A$, i.e.\ $a_{ij} = 1$ if $i=j+1$, $a_{ij} =
  -c_{i-1}$ if $j=n$, and $a_{ij} = 0$ in all other cases.  Define
  $\Gamma = (\gamma_{ij}) \in \gl_n(\Lri_\len)$ as follows:
  $\gamma_{ij}$ is the coefficient of ${\bar t}^{\, n-1}$ in the
  expression of $(- {\bar t})^{i-1} {\bar t}^{\, j-1}$ as an
  $\Lri_\len$-linear combination of ${\bar 1}, {\bar t}, \ldots, {\bar
    t}^{\, n-1}$ modulo~$f_A$.  A short computation shows that $\Gamma
  \in \GL_n(\Lri_\len)$ with $\Gamma^\circ = \Gamma$ and $A^\circ
  \Gamma + \Gamma A = 0$; thus $A$ is $\GL_n(\Lri)$-conjugate to an
  anti-hermitian matrix by Proposition~\ref{Dieudonne2}.

  Indeed, the free $\Lri_\len$-module $\Lri_\len[t] / f_A
  \Lri_\len[t]$ with $(\Lri_\len,\lri_\len)$-involution $\circ$ admits
  the non-degenerate $\circ$-hermitian form $\langle \cdot, \cdot
  \rangle$, where $\langle g, h \rangle$ is the coefficient of ${\bar
    t}^{\, n-1}$ in the expression of $g^\circ(-{\bar t}) h({\bar t})$
  as an $\Lri_\len$-linear combination of the basis ${\bar 1}, {\bar
    t}, \ldots, {\bar t}^{\, n-1}$.  The matrix $\Gamma$ is the
  structure matrix of this hermitian form and $A$ is the coordinate
  matrix of the endomorphism given by multiplication by~${\bar t}$,
  with respect to the basis ${\bar 1}, {\bar t}, \ldots, {\bar t}^{\,
    n-1}$.
\end{proof}

\subsection{Similarity class tree, centralisers, and unitary
  shadows} \label{subsec:unitary-shadows} The following concepts are
analogous to the ones introduced in Definitions~\ref{def:graph.Qgl}
and~\ref{def.of.shadows}.

\begin{defn} \label{def:graph.Qgu} For $\len \in \N_0$ let
  $\Qgu_{\lri,\len} = \Ad(\GU_n(\lri)) \backslash \gu_n(\lri_\len)$
  denote the set of $\Ad(\GU_n(\lri))$-orbits in $\gu_n(\lri_\len)$;
  by Proposition~\ref{pro:GL.conj.GU.conj}, this is the same as the
  collection of $\GL_n(\Lri)$-similarity classes in~$\gu_n(\lri_\len)
  \subset \gl_n(\Lri)$, obtained by intersection.  We endow
  \[
  \Qgu_\lri = \coprod_{\len=0}^\infty \Qgu_{\lri,\len}
  \]
  with the structure of a directed graph, induced by reduction modulo
  powers of~$\fP$: vertices $\cC \in \Qgu_{\lri,\len}$ and $\cCtilde
  \in \Qgu_{\lri,\len+1}$ are connected by a directed edge
  $(\cC,\cCtilde)$ if the reduction of $\cCtilde$ modulo $\fP^\len$ is
  equal to~$\cC$.  In this way $\Qgu_\lri$ becomes an infinite rooted
  subtree of $\Qgl_\Lri$.  We refer to $\Qgu_\lri$ as the
  \emph{anti-hermitian similarity class tree} in degree $n$
  over~$\lri$.
\end{defn}

Let $\len \in \N_0$ and let $A \in \gu_n(\lri_\len)$.  The centraliser
$\Cen_{\GU_n(\lri)}(A)$ of $A$ in the group $\GU_n(\lri)$ is the
stabiliser of $A$ under the adjoint action of $\GU_n(\lri)$.  The
centraliser $\Cen_{\gu_n(\lri)}(A)$ of $A$ in the $\lri$-Lie lattice
$\gu_n(\lri)$ is the stabiliser of $A$ under the adjoint action of
$\gu_n(\lri)$.

\begin{defn}\label{def.of.unitary.shadows} Let $\len \in \N_0$.  The
  \emph{group centraliser shadow} $\shU(A)$ of an element $A \in
  \gu_n(\lri_\len)$ is the image $\overline{\Cen_{\GU_n(\lri)}(A)}
  \leq \GU_n(\kk)$ of $\Cen_{\GU_n(\lri)}(A)$ under reduction
  modulo~$\fP$.  The \emph{Lie centraliser shadow} $\shu(A)$ of an
  element $A \in \gu_n(\lri_\len)$ is the image
  $\overline{\Cen_{\gu_n(\lri)}(A)} \leq \gu_n(\kk)$ of
  $\Cen_{\gu_n(\lri)}(A)$ under reduction modulo~$\fP$.

  For each similarity class $\cC$ in $\gu_n(\lri_\len)$ we define the
  \emph{unitary} (\emph{similarity class}) \emph{shadow}
  \[
  \shUC(\cC) = \{ (\shU(A),\shu(A)) \mid A \in\cC \},
  \]
  of $\cC$ and we denote the collection of all unitary shadows by
  \begin{equation}\label{equ:Sh.GU}
  \ShUn = \{ \shUC(\cC) \mid \cC \in \Qgu_{\lri,\len} ~\text{for some
    $\len \in \N_0$}\}.
  \end{equation}
  For $\sigma \in \ShUn$ we set
  \[
  \| \sigma \| = \lvert \shU(A) \rvert \qquad \text{and} \qquad
  \dim(\sigma) = \dim_\kk(\shu(A)),
  \]
  where $A \in \cC \in \Qgu_{\lri,\len}$ for some $\len \in \N_0$ with
  $\sigma = \shUC(\cC)$; furthermore, it is convenient to select one
  group centraliser shadow $\shU(A)$, where $A \in \cC \in
  \Qgl_{\lri,\len}$ with $\sigma = \shUC(\cC)$, and to denote it by
  $\sigma(\kk)$.  We only use properties of $\sigma(\kk)$ that are
  independent of the arbitrary choice involved in its definition.
\end{defn}

Comments similar to those in connection with
Definition~\ref{def.of.shadows} apply.  In order to see that a unitary
shadow $\sigma$ is, in fact, completely determined by the
$\GU_n(\kk)$-conjugacy class of $\sigma(\kk)$, or alternatively the
associated Lie algebra, we recall the Cayley maps.

\begin{defn}\label{def:cayley.map}
  For any subset $Y \subset \gl_n(\Lri)$, we denote by
  $Y_\textup{gen} \subset Y$ the set of elements that do not have an
  eigenvalue congruent to $-1$ modulo~$\fP$.  The \emph{Cayley maps}
  \[
  \cay \colon \GU_n(\lri)_\textup{gen} \rightarrow
  \gu_n(\lri)_\textup{gen} \qquad \text{and} \qquad \Cay \colon
  \gu_n(\lri)_\textup{gen} \rightarrow \GU_n(\lri)_\textup{gen}
  \]
  are both defined by the mapping rule
  \begin{equation}\label{equ:cayley}
    y \mapsto (\Id_n - y)(\Id_n + y)^{-1} = (\Id_n + y)^{-1} (\Id_n - y).
  \end{equation}
  The Cayley maps are easily seen to be mutual inverses of each other;
  see~\cite[II.10 and VI.2]{We97}.  Furthermore, they commute with the
  adjoint action and preserve congruence levels.  Thus they induce
  Cayley maps between the finite quotients
  $\GU_n(\lri_\len)_\textup{gen}$ and $\gu_n(\lri_\len)_\textup{gen}$
  for each $\len \in \N$.
\end{defn}

The next lemma can be regarded as a `unitary' version of
Lemma~\ref{lem:shadow-well-defined}; recall that throughout $\lri$
does not have residue characteristic $2$ so that, in particular, $q >
2$.

\begin{lem}\label{lem:unitary-shadow-well-defined}
  Let $\len \in \N_0$ and $A \in \gu_n(\lri_\len)$.  Then the group
  centraliser shadow $\shU(A)$ and the Lie centraliser shadow
  $\shu(A)$ determine one another in the following way:
  \begin{align}
    \shu(A) & = \langle \cay(\shU(A)_{\gen}) \cup  \{
    \overline{a \Id_n} \mid a \in \gu_1(\lri) \}
    \rangle_{\textup{$+$-span}}, \label{inclusion1}\\
    \shU(A) & = \langle \Cay(\shu(A)_{\gen}) \cup \{
    \overline{a \Id_n} \mid a \in \GU_1(\lri) \}
    \rangle. \label{inclusion2}
  \end{align}
\end{lem}

\begin{proof}
  A direct computation yields $\cay(\shU(A)_{\gen}) \subset \shu(A)$
  and $\Cay(\shu(A)_{\gen}) \subset \shU(A)$.  Thus the left-hand
  side contains the right-hand side in \eqref{inclusion1}
  and~\eqref{inclusion2}, and it suffices to prove the reverse
  inclusions.  First consider~\eqref{inclusion1}.  Let $\overline{X}
  \in \shu(A)$ be the image of $X \in \Cen_{\gu_n(\lri)}(A)$.  We
  argue below that, as in the proof of
  Lemma~\ref{lem:shadow-well-defined}, it suffices to consider the
  situation
  \begin{equation} \label{equ:x=x1+x2}
    X = \diag(Y_1,Y_2) = X_1 + X_2,
  \end{equation}
  where
  \begin{itemize} \renewcommand{\labelitemi}{$\circ$}
  \item $X_1 = \diag(Y_1,0), X_2 = \diag(0,Y_2) \in
    \Cen_{\gu_n(\lri)}(A)$ are anti-hermitian,
  \item the eigenvalues of $\overline{X_1} \in \shu(A)$ are in
    $\{0,1,-1\}$ and $\overline{X_2}$ does not have eigenvalue $-1$.
  \end{itemize}
  As $\delta \in \gu_1(\lri) \smallsetminus \{0\}$, we obtain
  \[
  \overline{X} = \overline{\cay(\Cay(X_1 - \delta \Id_n))} +
  \overline{\delta \Id_n} + \overline{\cay(\Cay(X_2))}.
  \]

  To justify \eqref{equ:x=x1+x2}, observe that $X$ acts, by left
  multiplication, as an anti-hermitian operator on $V = \Lri^n$,
  equipped with the standard hermitian form $\langle v_1,v_2 \rangle =
  v_1^{\, \circ} v_2$.  By Hensel's Lemma, the characteristic
  polynomial $f_X \in \Lri[t]$ factorises as a product $f_X = f_1 f_2$
  of coprime monic polynomials so that $\overline{f_1} =
  \gcd(\overline{f_X}, \overline{(t^2-1)^n})$.  The roots $\lambda_1,
  \ldots, \lambda_n$ of $f_X$ are, up to permutation, equal to
  $-\lambda_1^{\, \circ}, \ldots, -\lambda_n^{\, \circ}$.
  Consequently, the roots of $f_1$ come in pairs $\mu, -\mu^\circ$ so
  that $f_1(X)^\circ + f_1(X) = 0$, and $f_1(X)$ is skew-adjoint as an
  operator on~$V$.  Hence $V$ decomposes as a direct orthogonal sum of
  the $X$-invariant spaces $U = \ker f_1(X)$ and $W = \Ima f_1(X)$.
  The standard form $\langle \cdot, \cdot \rangle$ restricts to
  non-degenerate, hence standard forms on $U$ and $W$; see
  Lemma~\ref{Dieudonne1}.  Concatenating suitable bases for $U$ and
  $W$, we may assume that $X = X_1 + X_2$, where $X_1 = \diag(Y_1,0)$,
  $X_2 = \diag(0,Y_2)$ and the anti-hermitian matrices $Y_1, Y_2$
  describe the restrictions of $X$ to $U, W$.  Since $X_1,X_2$ can be
  expressed as polynomials in $X$ we deduce that $X_1, X_2 \in
  \Cen_{\gu_n(\lri)}(A)$.

  Next consider the pending inclusion in~\eqref{inclusion2}.  Let
  $\overline{B} \in \shU(A)$ be the image of $B \in
  \Cen_{\GU_n(\lri)}(A)$.  We argue below that it suffices to consider
  the situation
  \begin{equation} \label{equ:b=b1xb2}
    B = \diag(C_1,C_2) = B_1 B_2,
  \end{equation}
  where
  \begin{itemize} \renewcommand{\labelitemi}{$\circ$}
  \item $B_1 = \diag(C_1,\Id_{n-m}), B_2 = \diag(\Id_m,C_2) \in
    \Cen_{\GU_n(\lri)}(A)$ are unitary,
  \item the eigenvalues of $\overline{B_1} \in \shU(A)$ are in
    $\{1,-1\}$ and $\overline{B_2}$ does not have eigenvalue $-1$.
  \end{itemize}
  Choosing $a \in \GU_1(\lri)$ such that $a \not\equiv_\fP \pm 1$, we
  obtain
  \[
  \overline{B} = \overline{\Cay(\cay(B_1 \cdot a^\circ \Id_n))} \;
  \overline{(a \Id_n)} \; \overline{\Cay(\cay(B_2))}.
  \]

  It remains to justify~\eqref{equ:b=b1xb2}.  Observe that $B$ acts,
  by left multiplication, as a unitary operator on $V = \Lri^n$,
  equipped with the standard hermitian form $\langle v_1,v_2 \rangle =
  v_1^{\, \circ} v_2$.  By Hensel's Lemma, the characteristic
  polynomial $f_B \in \Lri[t]$ factorises as a product $f_B = f_1 f_2$
  of coprime monic polynomials so that $\overline{f_1} =
  \gcd(\overline{f_B}, \overline{(t+1)^n})$.  Suppose that $f_B =
  \prod_{i=1}^n (t-\lambda_i)$ and $f_1 = \prod_{i=1}^m
  (t-\lambda_i)$.  Then $\lambda_1^{\, \circ}, \ldots, \lambda_n^{\,
    \circ}$ are, up to permutation, equal to $\lambda_1^{-1}, \ldots,
  \lambda_n^{-1}$, and thus $\lambda_1^{\, \circ}, \ldots,
  \lambda_m^{\, \circ}$ are, up to permutation, equal to
  $\lambda_1^{-1}, \ldots, \lambda_m^{-1}$.  Putting $D =
  \prod_{i=1}^m \lambda_i^{\, \circ} \Id_n \in \GU_n(\lri)$, we deduce
  that
  \[
  f_1(B)^\circ = \prod\nolimits_{i=1}^m (B^\circ - \lambda_i^{\,
    \circ} \Id_n) = ( (-B^\circ)^m D) \prod\nolimits_{i=1}^m (B-\lambda_i) =
  ((-B^\circ)^m D) f_1(B).
  \]
  Hence $V$ decomposes as a direct orthogonal sum of the $B$-invariant
  spaces $U = \ker f_1(B) = \ker f_1(B)^\circ$ and $W = \Ima f_1(B)$.
  The standard form $\langle \cdot, \cdot \rangle$ restricts to
  non-degenerate, hence standard forms on $U$ and $W$; see
  Lemma~\ref{Dieudonne1}.  Concatenating suitable bases for $U$ and
  $W$, we may assume that $B = B_1 B_2$, where $B_1 =
  \diag(C_1,\Id_{n-m})$, $B_2 = \diag(\Id_m,C_2)$ and the unitary
  matrices $C_1, C_2$ describe the restrictions of $B$ to $U, W$.
  Since $B_1,B_2$ can be expressed as polynomials in $B$ we deduce
  that $B_1, B_2 \in \Cen_{\GU_n(\lri)}(A)$.
\end{proof}

The following result is analogous to
Proposition~\ref{pro:class.quot.GL}.

\begin{prop}\label{pro:class.quot.GU}
  Let $\sigma, \tau \in \ShUn$.  Let $\len \in \N_0$ and
  suppose that $\cCtilde \in \Qgu_{\lri,\len+1}$ is a class with
  $\shUC(\cCtilde) = \tau$ which lies above a class $\cC
  \in \Qgu_{\lri,\len}$ with $\shUC(\cC) = \sigma$.  Then
  \[
  \frac{ \lvert \cCtilde \rvert}{ \lvert \cC \rvert}=q^{\dim \gu_n -
    \dim(\sigma)}\frac{\|\sigma\|}{\|\tau\|}.
  \]
  In particular, the ratio $ \lvert \cCtilde \rvert/ \lvert \cC
  \rvert$ depends only on the shadows $\sigma, \tau$ and not on
  $\len$, $\cC$ or $\cCtilde$.
\end{prop}

\begin{proof} The proof proceeds along the same lines as the proof of
  Proposition~\ref{pro:class.quot.GL}.  The second map
  in~\eqref{equ:map} is replaced by
  \begin{equation*}
    \gu_n(\lri) \rightarrow \GU_n^1(\lri), \quad X \mapsto
    \Cay(\pi X) = (\Id_n - \pi X)/(\Id_n + \pi X).
  \end{equation*}
  The other necessary translations are straightforward.
\end{proof}

In analogy with Definition~\ref{def:b.GL}, we introduce the
following functions.

\begin{defn}\label{def:b.GU} For $\sigma,\tau \in
  \ShUn$ let
  \begin{equation} \label{equ:def.b.GU} b^{(-1)}_{\sigma,\tau}(q) =
    q^{\dim \gu_n - \dim(\sigma)} \frac{\|\sigma\|}{\|\tau\|}.
  \end{equation}
\end{defn}

This variation of the earlier defined functions
$b^{(1)}_{\sigma,\tau}(q)$ was already hinted at in
Remark~\ref{rem:bees}.  Table~\ref{tab:branch.rules.A2} gives the
explicit values of~$b^{(-1)}_{\sigma,\tau}$ in the case~$n=3$, which
can be computed with the aid of Table~\ref{tab:shadows.iso.GU}.

\subsection{Unitary shadows and branching rules for $\gu_3(\lri_\len)$}
We list eight shadows in~$\Sh_{\GU_3(\lri)}$, classified by types;
compare~\eqref{equ:types-def}.  Recalling that $q>2$, one sees that
these are all unitary shadows arising from~$\len=1$, i.e.\ arising
from the centralisers of elements $A \in \gu_3(\kk)$.  These shadows
$\sigma$ and the isomorphism types of $\sigma(\kk)$ are easily
extracted from~\cite[Appendix~C]{AKOV1}.

\begin{table}[htb!]
  \centering
  \caption{Shadows in $\GU_3(\kk)$}
  \label{tab:shadows.iso.GU}
  \begin{tabular}{|c||l|l|c|}
    \hline
    Type & Minimal polynomial in $\kk_2[t]$ & Isomorphism type of
    $\sigma(\kk)$ & $\dim(\sigma)$ \\
    \hline
    $\Gsha$ & ~$t-\alpha$ \hfill $\alpha \in \gu_1(\kk)$ &
    $\GU_3(\kk)$ & 9 \\
    $\Lsha$ & $(t-\alpha_1)(t-\alpha_2)$ \hfill $\alpha_1,\alpha_2 \in
    \gu_1(\kk)$ distinct & $\GU_1(\kk) \times \GU_2(\kk)$ & 5 \\
    $\Jsha$ & $(t-\alpha)^2$ \hfill $\alpha \in \gu_1(\kk)$ &
    $\mathsf{Heis}(\kk) \rtimes (\GU_1(\kk) \times \GU_1(\kk))$ & 5 \\
    $\Tasha$ &  $\prod_{i=1}^{3}(t-\alpha_i)$ \hfill
    $\alpha_1,\alpha_2,\alpha_3 \in \gu_1(\kk)$ distinct &
    $\GU_1(\kk) \times \GU_1(\kk) \times \GU_1(\kk)$ & 3 \\
    $\Tbsha$ &  $\prod_{i=1}^{3} (t-\alpha_i)\,$ $\alpha_1 \in
    \gu_1(\kk), \alpha_2 = -\alpha_3^\circ$ distinct & $\GU_1(\kk)
    \times \kk_2^\times$ & 3 \\
    $\Tcsha$ &  $f(t)$ \hfill $f$ a suitable irreducible cubic$^*$
    & $\GU_1(\kk_3)$ & 3 \\
    $\Msha$ & $(t-\alpha_1)(t-\alpha_2)^2$ \hfill $\alpha_1,\alpha_2
    \in \gu_1(\kk)$ distinct & $\GU_1(\kk) \times \GU_1(\kk) \times
    \aG_\text{a}(\kk)$ & 3 \\
    $\Nsha$ & $(t-\alpha)^3$ \hfill $\alpha \in \gu_1(\kk)$ &
    $\GU_1(\kk) \times \aG_\text{a}(\kk) \times \aG_\text{a}(\kk)$ & 3 \\
    \hline
    \multicolumn{4}{@{} p{\textwidth} @{}}{${}^*$ We require: $f = t^3 +
      \sum_{i=0}^2 c_i t^i \in \kk_2[t]$ satisfies $c_i^{\, \circ} =
      (-1)^{i+1} c_i$ for $0 \leq i < 3$; cf.\ Lemma~\ref{cyclic.u}.
      The number of such polynomials is equal to $ \tfrac{1}{3} \lvert
      \gu_1(\kk_3) \smallsetminus \gu_1(\kk) \rvert = \tfrac{1}{3}(q^2-1)q$.}
  \end{tabular}
\end{table}

The following theorem is the counterpart of
Theorem~\ref{thm:G.shad.graph} for anti-hermitian matrices.

\newpage
\begin{thm}[Classification of unitary shadows and branching
  rules] \label{thm:U.shad.graph}\quad
  \begin{enumerate}
  \item \label{part-a} The set of shadows $\Sh_{\GU_3(\lri)}$ consists
    of eight elements, classified by the types
    \[
    \Gsha, \, \Lsha, \, \Jsha, \, \Tasha, \, \Tbsha, \, \Tcsha, \,
    \Msha, \, \Nsha
    \]
    described in Table~\textup{\ref{tab:shadows.iso.GU}}.
  \item \label{part-b} For all $\sigma, \tau \in \Sh_{\GU_3(\lri)}$
    there exists a polynomial ${a}_{\sigma,\tau} \in \Z[\sixth][t]$
    such that the following holds: for every $\len \in \N$ and every
    $\cC \in \cQ_{\lri,\len}^{\gl_3}$ with $\shU(\cC) = \sigma$ the
    number of classes $\cCtilde \in \cQ_{\lri,\len+1}^{\gl_3}$ with
    $\shU(\cCtilde) = \tau$ lying above $\cC$ is equal to
    ${a}_{\sigma,\tau}(q)$.
  \end{enumerate}
\end{thm}

\begin{remark}
  We emphasise that the types $\Kasha$ and $\Kbsha$ do not occur in
  the unitary setting and we refer to Table~\ref{tab:branch.rules.A2}
  for the explicit values of the polynomials~$a_{\sigma,\tau}$ which
  turn out to be the same as in the general linear case.
\end{remark}

Similar to the procedure in Section~\ref{sec:sim.gl}, we first
produce a complete parametrisation for the
$\Ad(\GU_3(\lri_\len))$-orbits in $\gu_3(\lri_\len)$.  The proof of
Theorem~\ref{thm:U.shad.graph} is given in
Section~\ref{subsec:proof.main.SU}.

\subsection{Similarity classes of anti-hermitian $3\times3$ matrices}
For $\len \in \N$, let $\mathcal{R}_3(\Lri_\len) \subset
\gl_3(\Lri_\len)$ denote the set of representatives for similarity
classes in $\gl_3(\Lri_\len)$ provided by
Theorem~\ref{thm:irredundant.list.for.M3x3}.  We use
Proposition~\ref{Dieudonne2} to check for each $A \in
\mathcal{R}_3(\Lri_\len)$ whether $A$ is $\GL_3(\Lri_\len)$-conjugate
to an anti-hermitian matrix.  In this way we obtain a parametrisation
$\mathcal{R}_3'(\Lri_\len) \subset \mathcal{R}_3(\Lri_\len)$ of the
set of similarity classes of anti-hermitian $3\times3$ matrices.  This
is enough for our purposes, but note that only in some cases
representatives $A \in \mathcal{R}_3'(\Lri_\len)$ are themselves
anti-hermitian.

For every $\nu \in \N_0$ construct a `strongly $\circ$-compatible' set
of representatives
\[
\varsigma(\Lri_\nu) \subset \Lri \qquad \text{for $\Lri_\nu = \Lri /
  \fP^\nu$}
\]
in the following way.  First choose a set of representatives
$\varsigma_0(\lri_\nu) \subset \lri$ for $\lri_\nu = \lri / \fp^\nu$
such that $\varsigma_0(-a) = - \varsigma_0(a)$ for $a \in \lri_\nu$.
Then extend $\varsigma_0$, by setting $\varsigma(a_1 + a_2 \delta_\nu)
= \varsigma_0(a_1) + \varsigma_0(a_2) \delta$ for $a_1, a_2 \in
\lri_\nu$.  In particular, this ensures that $\varsigma(a^\circ) =
\varsigma(a)^\circ$ for all $a \in \Lri_\nu$ and
$\varsigma(\gu_1(\lri_\nu)) \subset \gu_1(\lri)$.

In contrast to the convention favoured in
Section~\ref{appendix:similarity.new}, this means that $\varsigma(0)
=0$; as we will see, the types $\Kasha$ and $\Kbsha$ do not occur in
the unitary setting so that we do not run into any conflicts.  Similar
to the custom in Section~\ref{appendix:similarity.new}, we employ the
notation in a flexible way; e.g.\ we write $\varsigma(\Lri_\nu)
\subset \Lri_\mu$ for $\nu < \mu$ to denote the reduction of
$\varsigma(\Lri_\nu)$ modulo $\fP^\mu$, and we sometimes write
$\pi^\nu \Lri_{\len-\nu}$ rather than $\pi^\nu
\varsigma(\Lri_{\len-\nu})$.  These conventions are also applied to
matrices.

The next theorem gives a complete description of the similarity
classes in $\gu_3(\lri_\len)$ and their unitary shadows; it is the
counterpart of Theorem~\ref{thm:irredundant.list.for.M3x3} for
anti-hermitian matrices.

\begin{thm}\label{thm:representatives.for.GU3.action.gu3}
  The set
  \[
  \mathcal{R}_3'(\Lri_\len) = \{ A \in \mathcal{R}_3(\Lri_\len) \mid
  \text{$A$ is $\GL_3(\Lri_\len)$-conjugate to an anti-hermitian
    matrix} \},
  \]
  parametrising $\cQ^{\gu_3}_{\lri,\len}$, consists of the following
  matrices:
  \begin{enumerate}
  \item[(i)] $d\Id_3$, where $d \in \gu_1(\lri_\len)$; the associated
    unitary shadow has type~$\Gsha$.
  \item[(ii)] $d\Id_3 + \pi^{i} D(a,0,0)$, where $0 \leq i<\len$, $\;
    d \in \gu_1(\lri_\len)$ and $a \in \Lri_{\len-i}^\times$ with
    $a^\circ + a = 0$; the associated unitary shadow has type~$\Lsha$.
  \item[(iii)]  $\phantom{x}$ \vspace*{-0.7\baselineskip}
    \[
    d\Id_3 + \pi^i D(a,0,0) +\pi^j
    \begin{bmatrix}
      0 & 0 \\
      0 & C
    \end{bmatrix},
    \]
    where $0 \leq i < j < \len$, $\; d \in
    \varsigma(\gu_1(\lri_{j}))$, $\; a \in \Lri_{\len-i}^\times$ with
    $a^\circ + a = 0$, and $C \in \gl_2(\Lri_{\len-j})$ a companion
    matrix with characteristic polynomial $t^2 + b_1 t + b_0$ such
    that $b_1^{\, \circ} = -b_1$ and $b_0^{\, \circ} = b_0$; the
    associated unitary shadows have types $\Tasha$, $\Tbsha$ or
    $\Msha$, depending on~$C$.
  \item[(iv)] $d\Id_3+\pi^iC$, where $0 \leq i < \len$, $\; d \in
    \varsigma(\gu_1(\lri_{i}))$ and $C \in \gl_3(\Lri_{\len-i})$ a
    companion matrix with characteristic polynomial $t^3 + c_2 t^2 +
    c_1 t + c_0$ such that $c_k^\circ + (-1)^k c_k = 0$ for $k \in
    \{0,1,2\}$; the associated unitary shadows have types $\Tasha$,
    $\Tbsha$, $\Tcsha$, $\Msha$ or $\Nsha$, depending on $C$.
  \item[(v)] $d' \Id_3 +\pi^i E$, where $0 \leq i < \len$, $\; d' \in
    \varsigma(\gu_1(\Lri_{\len-i}))$ and $E \in \gl_3(\Lri_{\len-i})$
    is one of the following matrices:
    \begin{enumerate}
    \item[(I)] $E(\len-i,0,0,c,d)$, where $c,d \in
      \gu_1(\lri_{\len-i})$ with $v(c)>0$,
    \item[(II)] $E(\mu,a,b,c,d)$, where $1 \leq \mu < \len-i$, $\; a,b
      \in \Lri_{\len-i}$ with $\mu = v(b) \leq v(a)$ and $a^\circ + a
      = b^\circ - b = 0$, $\; c \in \gu_1(\lri_{\len-i})$ with $v(c)>0$
      and $d \in \varsigma(\gu_1(\lri_\mu))$,
    \item[($\text{III}_1$)] $E(\mu,a,b,c,d)$, where $1 \leq \mu <
      \len-i$, $\; a,b \in \Lri_{\len-i}$ with $\mu = v(a) < v(b)$ and
      $a^\circ + a = b^\circ - b$, $\; c \in \gu_1(\lri_{\len-i})$
      with $v(c)>0$ and $d \in \varsigma(\gu_1(\lri_\mu))$;
    \end{enumerate}
    the associated unitary shadows in these subcases have types
    $\Jsha$, $\Msha$ and~$\Nsha$.
  \end{enumerate}
\end{thm}

\begin{proof}
  The proof is based on Theorem~\ref{thm:irredundant.list.for.M3x3}:
  we go through the cases described there and keep track of which
  similarity classes in $\gl_3(\Lri_\len)$ intersect
  $\gu_3(\lri_\len)$ non-trivially.  The latter is achieved by using
  the criterion provided by Proposition~\ref{Dieudonne2}.  On the way
  we pin down in each case the Lie centraliser shadow of $A \in
  \mathcal{R}_3'(\Lri_\len)$, with respect to a suitable hermitian
  form, to determine the unitary shadow associated to~$A$;
  see~Lemma~\ref{lem:unitary-shadow-well-defined}.

  Let~$A \in \mathcal{R}_3(\Lri_\len)$ be one of the matrices
  representing a similarity class in~$\gl_3(\Lri_\len)$.
  \begin{list}{}{\setlength{\leftmargin}{5pt}
      \setlength{\itemsep}{2pt} \setlength{\parsep}{1pt}}
  \item [(i)] The similarity class $\{ A \}$ of $A =d\Id_3 \in
    \gl_3(\Lri_\len)$ intersects $\gu_3(\lri_\len)$ if and only
    $d^\circ+d=0$, that is $d \in \gu_1(\lri_\len)$.  In this case the
    associated unitary shadow has type~$\Gsha$.
  \item [(ii)] Let $A = d\Id_3 + \pi^i D(a,0,0) \in \gl_3(\Lri_\len)$,
    with $0 \leq i < \len$, $\; d \in \Lri_\len$ and $a \in
    \varsigma(\Lri_{\len-i}^\times)$.  Then $A \in
    \mathcal{R}_3'(\Lri_\len)$ if and only if its eigenvalues
    $\lambda_1 = d+\pi^i a$ and $\lambda_2=\lambda_3=d$ are
    anti-hermitian, that is, $d + d^\circ = \pi^i(a+a^\circ) = 0$.
    Indeed, if $A \in \mathcal{R}_3'(\Lri_\len)$ then its eigenvalues
    satisfy: $\lambda_1^{\, \circ}, \lambda_2^{\, \circ},
    \lambda_3^{\, \circ}$ are equal to $-\lambda_1, -\lambda_2,
    -\lambda_3$, up to a permutation.  As $\lambda_1 \ne \lambda_2 =
    \lambda_3$, we deduce that $\lambda_k^{\, \circ} = -\lambda_k$ for
    $k \in \{1,2,3\}$.  For $A \in \mathcal{R}_3'(\Lri_\len)$ the
    associated unitary shadow has type~$\Lsha$.
  \item [(iii)] Let
    \[
    A = d \Id_3 + \pi^{i} D(a,0,0) + \pi^{j}
    \begin{bmatrix}
      0 & 0 \\
      0 & C
    \end{bmatrix}
    \in \gl_3(\Lri_\len),
    \]
    where $0 \leq i < j < \len$, $\; d \in \varsigma(\Lri_{j})$, $\; a
    \in \Lri_{\len-i}^\times$ and $C \in \gl_2(\Lri_{\len-j})$ a
    companion matrix.  Applying~(ii) to $A$ modulo~$\fP^j$, we see
    that $d \in \varsigma(\gu_1(\lri_j))$ whenever $A \in
    \mathcal{R}_3'(\Lri_\len)$.  Let us assume that this condition is
    satisfied.

    Proposition~\ref{Dieudonne2} shows that $A \in
    \mathcal{R}_3'(\Lri_\len)$ if and only if there exist $\Gamma \in
    \GL_3(\Lri_\len)$ such that $\Gamma^\circ = \Gamma$ and
    $A^\circ\Gamma+\Gamma A=0$. Applying
    Lemma~\ref{lem:diagonal-form-shadow}, for $A' = -A^\circ$ and
    $X=\Gamma$, one may further demand that $\Gamma$ is block diagonal
    for blocks of sizes $1\times1$ and $2\times2$.  Consequently
    $a^\circ + a = 0$ is a necessary condition and, as $C$ is a
    companion matrix, the full assertion follows from
    Lemma~\ref{cyclic.u}.

    Suppose that $A \in \mathcal{R}_3'(\Lri_\len)$.  Using
    Lemma~\ref{lem:diagonal-form-shadow}, the associated Lie
    centraliser shadow is built from the Lie centraliser shadow of
    $a$, i.e.\ $\gu_1(\kk)$, and the Lie centraliser shadow of the
    reduction of $C$ modulo~$\fP$; cf.~\cite[Proof of
    Corollary~7.7]{AKOV1}.  Table~\ref{tab:shadows.iso.GU} shows that
    the resulting unitary shadow has type $\Tasha$, $\Tbsha$ or
    $\Msha$, depending on $C$.
  \item [(iv)] Let $A = d\Id_3 + \pi^i C \in \gl_3(\Lri_\len)$, where
    $0\leq i<\len$, $\; d \in \varsigma(\Lri_i)$ and $C \in
    \gl_3(\Lri_{\len-i})$ is a companion matrix with characteristic
    polynomial $t^3 + c_2t^2 + c_1t + c_0$.  Applying~(i) to $A$
    modulo $\fP^i$, we see that a $d \in \varsigma(\gu_1(\lri_i))$ is
    a necessary condition for $A \in \mathcal{R}_3'(\Lri_\len)$. Let
    us assume that this condition is satisfied.  Then $A \in
    \mathcal{R}_3'(\Lri_\len)$ if and only if $C$ is similar to an
    anti-hermitian matrix.  Furthermore, $C$ is similar to an
    anti-hermitian matrix if and only if $c_k^\circ + (-1)^kc_k = 0$
    for $k \in \{0,1,2\}$, by Lemma~\ref{cyclic.u}.

    Suppose that $A \in \mathcal{R}_3'(\Lri_\len)$ so that $C$ is
    similar to an anti-hermitian matrix.  The associated Lie
    centraliser shadow is equal to the centraliser of the reduction of
    $C$ modulo~$\fP$; cf.~\cite[Proof of Corollary~7.7]{AKOV1}.
    Inspection of Table~\ref{tab:shadows.iso.GU} shows that the
    unitary shadows occurring are of type $\Tasha$, $\Tbsha$,
    $\Tcsha$, $\Msha$, $\Nsha$, corresponding to different kinds of
    minimal polynomials of degree~$3$.
 \item [(v)] Consider
    \[
    A= d\Id_3 + \pi^i E \in \gl_3(\Lri_\len),
    \]
    where $0 \le i<\len$, $\; d \in \Lri_\len$ and
    \[
    E = E(m,a,b,c,0) =
    \begin{bmatrix}
      0  & \pi^m & 0 \\
      0 & 0 & 1 \\
      a & b & c
    \end{bmatrix}
    \in \gl_3(\Lri_{\len - i})
    \]
    with $1 \leq m \leq \len - i$ and $a,b,c \in \Lri_{\len-i}$ such
    that $v(a),v(b),v(c) > 0$.  Writing $\mu = \mu_{\len-i}(m,a,b) =
    \min\{m, v(a),v(b),\len-i\}$, we may assume further that $d \in
    \varsigma(\Lri_{i+\mu})$; see
    Theorem~\ref{thm:irredundant.list.for.M3x3}.  According to
    Proposition~\ref{Dieudonne2}, one has $A \in
    \mathcal{R}_3'(\Lri_\len)$ if and only if there exists
    \[
    \Gamma =
    \begin{bmatrix}
      x & y & z \\
      y^\circ & u & w \\
      z^\circ & w^\circ & r
    \end{bmatrix}
    \in \GL_3(\Lri_\len), \quad \text{with $ x,u,r \in \lri_\len$,
      hence $\Gamma=\Gamma^\circ$,}
    \]
    such that $A^\circ \Gamma + \Gamma A = 0$.  Comparing matrix
    entries, we obtain the following equivalent system of equations
    over $\Lri_\len$:
    \begin{equation}\label{nightmare}
      \begin{split}
        (1,1) \qquad &0=(d^\circ+d)x+\pi^i(a^\circ z^\circ + az), \\
        (1,2) \qquad &0=(d^\circ+d)y+\pi^i(a^\circ w^\circ + \pi^m x + bz), \\
        (2,2) \qquad &0=(d^\circ+d)u+\pi^i(\pi^m y + \pi^m y^\circ +
        b^\circ w^\circ + bw), \\
        (1,3) \qquad &0=(d^\circ+d)z+\pi^i(a^\circ r + y + cz), \\
        (2,3) \qquad &0=(d^\circ+d)w+\pi^i(\pi^m z+b^\circ r + u + cw), \\
        (3,3) \qquad &0=(d^\circ+d)r+\pi^i (w + w^\circ + cr + c^\circ
        r).
      \end{split}
    \end{equation}
    Assume that such a matrix $\Gamma$ exists. Since $\Gamma$ is
    invertible, at least one of $x,y,u$ is invertible.  Reducing
    equations (1,1), (1,2), (2,2) modulo $\fP^{i+\mu}$, we deduce that
    $d^\circ+d = 0$.  This leads to the following observation.  By
    reducing equations (1,3), (2,3) and (3,3) modulo $\fP^{i+1}$ we
    deduce that $y$, $u$, and $w + w^\circ$ are $0$ modulo $\fP$.
    Consequently, $x$ is invertible, and $w,w^\circ$ are invertible;
    otherwise the second column of $\Gamma$ would be congruent to $0$
    modulo~$\fP$.

    Since $d^\circ + d = 0$, we deduce that $A \in
    \mathcal{R}_3'(\Lri_\len)$ if and only if $E$ is similar to an
    anti-hermitian matrix.  As the characteristic polynomial of $E$ is
    equal to $t^3 - ct^2 + bt +\pi^m a \in \Lri_{\len - i}[t]$,
    Lemma~\ref{cyclic.u} supplies necessary conditions for $E$ being
    similar to an anti-hermitian matrix:
    \begin{equation} \label{equ:nec.con.for.B} \pi^m (a^\circ + a) =
      b^\circ - b = c^\circ + c = 0.
    \end{equation}
    With these at hand, \eqref{nightmare} reduces to the following
    system of equations over~$\Lri_{\len-i}$:
    \begin{equation}\label{smallnightmare}
      \begin{array}{llll}
        (1,1) \quad & 0=a^\circ z^\circ + az, & \quad (1,3) \quad
        & 0=a^\circ r + y + cz, \\
        (1,2) \quad & 0=a^\circ w^\circ + \pi^m x + bz, & \quad (2,3)
        \quad & 0 =\pi^m z+b r + u + cw, \\
        (2,2) \quad & 0=\pi^m( y + y^\circ), & \quad (3,3) \qquad
        &0 = w + w^\circ. \\
      \end{array}
    \end{equation}
    Recall also that $x,u,r \in \lri_\len$ are $\circ$-invariant.
    From now on, all computations will be carried out
    over~$\Lri_{\len-i}$.  We strengthen the first of the necessary
    conditions~\eqref{equ:nec.con.for.B} to:
    \begin{equation} \label{equ:nec.con.for.B.2}
      a^\circ + a = 0 \qquad \text{if $v(b) \geq m$.}
    \end{equation}
    Indeed, from the equalities $(2,3)$ and $(3,3)$
    in~\eqref{smallnightmare} we deduce that $\pi^m z^\circ = \pi^m
    z$, hence $(bz)^\circ = bz$ if $v(b) \geq m$.  But then equalities
    $(1,2)$ and $(3,3)$ in~\eqref{smallnightmare}, together with $w
    \in \Lri_{\len-i}^\times$ imply $a^\circ w^\circ = a w$ and hence
    $a^\circ = -a$.  Below we will show that one can always arrange
    $v(b) \geq m$ so that the conclusion holds, in effect, unconditionally.

    It is time to pin down not only necessary but, also sufficient
    conditions for $E$ to be similar to an anti-hermitian matrix.
    From Proposition~\ref{prop:centraliser.of.E.new} we adapt the
    notation
    \[
    F_{m,a,b,c}(t_1,t_2,s_1,s_2,s_3) = \Mat {t_1}{\pi^m {s_3}
      -c{s_1}}{s_1}{s_2}{t_2}{s_3}{a{s_3}}{\pi^m{s_2}+b{s_3}}{{t_2}+c{s_3}}
    \in\gl_3(\Lri_{\len-i})
    \]
    for $t_1,t_2,s_1,s_2,s_3 \in \Lri_{\len-i}$.  Recall further that
    $\Lri = \lri[\delta]$.  Similar to the situation described in
    Theorem~\ref{thm:irredundant.list.for.M3x3}, we distinguish three
    subcases (I), (II), (III).

    \smallskip

    \begin{enumerate}
    \item [(\vI)] Suppose that $\mu = \len-i$, that is $\pi^{m} = a =
      b=0$.  Subject to the necessary condition $c^\circ + c = 0$ that
      we identified above, the matrix
      \[
      \Gamma_0 =
      \begin{bmatrix}
        1 & 0 & 0 \\
        0 & - c \delta & \delta \\
        0 & -\delta & 0
      \end{bmatrix}
      \in \GL_3(\Lri_{\len-i})
      \]
      is hermitian and satisfies~$E^\circ \Gamma + \Gamma E=0$.
      Proposition~\ref{Dieudonne2} shows that $E$ is similar to an
      anti-hermitian matrix.

      To determine the unitary shadow we compute the Lie centraliser
      of $E$ in
      \[
      \gu_3(\lri_{\len-i}; \Gamma_0) \coloneqq \{ Y \in \gl_n(\Lri) \mid
      Y^\circ \Gamma_0 + \Gamma_0 Y = 0 \},
      \]
      the unitary $\lri$-Lie lattice with respect to $\Gamma_0$, and
      subsequently reduce modulo~$\fP$.  By
      Proposition~\ref{prop:centraliser.of.E.new}, the centraliser
      $\Cen_{\gl_3(\Lri_{\len-i})}(E)$ consists of all matrices of the
      form
      \[
      Y = F_{\infty, 0,0,c}(t_1,t_2,s_1,s_2,s_3) = \Mat
      {t_1}{-c{s_1}}{s_1}{s_2}{t_2}{s_3}{0}{0}{{t_2}+c{s_3}}.
      \]
      The intersection $\mathfrak{u} = \gu_3(\lri_{\len-i}; \Gamma_0)
      \cap \Cen_{\gl_3(\Lri_{\len-i})}(E)$ is easily
      determined:
      \[
      \mathfrak{u} = \{F_{\infty,0,0,c}(t_1,t_2,s_1,s_2,s_3) \mid
      t_1^\circ + t_1 = t_2^\circ + t_2 = s_3^\circ + s_3 = s_1 +
      \delta s_2^\circ = 0\}.
      \]
      The reduction $\overline{\mathfrak{u}}$ modulo $\fP$ has
      dimension $5$ and thus coincides with the centraliser in
      $\gu_3(\kk;\overline{\Gamma_0})$ of the reduction $\overline{E}$
      modulo~$\fP$; see Table~\ref{tab:shadows.iso.GU}.  Moreover, one
      easily computes the isomorphism type of
      $\overline{\mathfrak{u}}$ and deduces that the associated
      unitary shadow has type~$\Jsha$; cf.\ \cite[Appendix~C]{AKOV1}.
    \item [(\vII)] Suppose that $1 \leq \mu < \len-i$ and $\mu = m =
      v(b) \le v(a)$.  Recalling the necessary
      conditions~\eqref{equ:nec.con.for.B} and
      \eqref{equ:nec.con.for.B.2}, we choose $\alpha \in
      \Lri_{\len-i}$ and $\beta, \gamma \in \Lri_{\len-i}^\times$ with
      $\alpha^\circ + \alpha = \beta^\circ - \beta = \gamma^\circ +
      \gamma = 0$ such that $a = \pi^\mu \alpha$, $b = \pi^\mu \beta$
      and $c = \pi^{v(c)} \gamma$.  Furthermore, we choose $e \in
      \lri_{\len-i}$ such that
      \[
      x_e = \beta f - (1+ \pi^\mu) \delta \alpha \in
      \lri_{\len-i}^\times, \qquad \text{where $f = e + \delta c \in
        \lri_{\len-i}$.}
      \]
      In particular, this implies $\pi^m x_e = bf - (1+\pi^\mu) \delta
      a$, and the matrix
      \[
      \Gamma_0 =
      \begin{bmatrix}
        x_e & c f & -f \\
        -c f & (1 + \pi^\mu) e - f & (1+\pi^\mu) \delta \\
        -f & - (1+\pi^\mu) \delta & 0
      \end{bmatrix}
      \in \GL_3(\Lri_{\len-i})
      \]
      is hermitian and satisfies~$E^\circ \Gamma_0 + \Gamma_0 E = 0$.
      Thus Proposition~\ref{Dieudonne2} shows that $E$ is similar to
      an anti-hermitian matrix.

      Next we determine the unitary shadow, using the same strategy as
      in case~(\vI).  By Proposition~\ref{prop:centraliser.of.E.new},
      the centraliser of $E$ in $\gl_3(\Lri_{\len-i})$ is given by
      \begin{multline}\label{equ:cen}
        \Cen_{\gl_3(\Lri_{\len-i})}(E) = \\
         \{F_{m,a,b,c}(t_1,t_2,s_1,s_2,s_3) \mid a s_1 = \pi^m s_2, b
        s_1 = \pi^m (t_2 - t_1), b s_2 = a (t_2 - t_1).\}
      \end{multline}
      We are interested in those $Y \in
      \Cen_{\gl_3(\Lri_{\len-i})}(E)$ that satisfy the
      $\lri_{\len-i}$-linear equation $Y^\circ \Gamma_0 + \Gamma_0 Y=
      0$.  A straightforward computation shows that parameters
      $(t_1,t_2,s_1,s_2,s_3)$ leading to such $Y$ must satisfy the
      following congruences modulo~$\fP$,
      \[
      t_1^\circ + t_1 \equiv_\fP t_2^\circ + t_2 \equiv_\fP 0, \qquad s_1
      \equiv_\fP \beta^{-1} (t_2-t_1), \qquad s_2 \equiv_\fP \alpha
      \beta^{-1} (t_2 - t_1).
      \]
      We deduce that the Lie centraliser shadow of $E$
      in~$\gu_3(\kk;\overline{\Gamma_0})$ is at most $3$-dimensional.

      Conversely, $E$ itself, the scalar matrix $\delta \Id_3$, and
      the matrix
      \[
      Y_0 = F_{m,a,b,c}(0,\beta \delta,\delta,\alpha \delta,0) =
      \begin{bmatrix}
        0 & - c \delta & \delta \\
        \alpha \delta & \beta \delta & 0 \\
        0 & a \delta & \beta \delta
      \end{bmatrix}
      \]
      centralise $E$ and satisfy the condition $Y^\circ \Gamma_0 +
      \Gamma_0 Y= 0$; to verify the assertion for $Y_0$, observe that
      $\pi^\mu c \delta^{-1} Y_0 = E^3 - (b+c^2) E - \pi^\mu a \Id_3$.
      Hence the Lie centraliser shadow of $E$
      in~$\gu_3(\kk;\overline{\Gamma_0})$ has dimension~$3$, and by
      inspection of Table~\ref{tab:shadows.iso.GU} we deduce that the
      unitary shadow associated to $E$ has type~$\Msha$.
    \item [(\vIII)] Suppose that $1 \leq \mu < \len-i$ and $\mu= \min
      \{m,v(a)\} < v(b) \le \len$.  Recall the necessary
      conditions~\eqref{equ:nec.con.for.B} and
      \eqref{equ:nec.con.for.B.2}.  In
      Theorem~\ref{thm:irredundant.list.for.M3x3} there is a
      subdivision into three cases
      \[
      \mathrm{(III_1)} \;\; \mu = m = v(a), \qquad \mathrm{(III_0)}
      \;\; \mu = m < v(a), \qquad \mathrm{(III_\infty)} \;\; \mu =
      v(a) < m.
      \]
      We claim that matrices corresponding to $\mathrm{(III_0)}$ and
      $\mathrm{(III_\infty)}$ are not similar to anti-hermitian
      matrices.  Indeed, equality $(1,2)$ in~\eqref{smallnightmare}
      yields $a^\circ w^\circ \equiv_{\fP^{\mu+1}} -\pi^m x$.  Since
      both $w^\circ$ and $x$ are already required to be invertible, a
      necessary condition for the solubility of this congruence is
      $\mu = m = v(a^\circ) = v(a)$.

      We now focus on $\mathrm{(III_1)}$; the procedure is similar to
      case~(\vII).  We choose $\alpha, \gamma \in
      \Lri_{\len-i}^\times$ and $\beta \in \Lri_{\len-i}$ with
      $\alpha^\circ + \alpha = \beta^\circ - \beta = \gamma^\circ +
      \gamma = 0$ such that $a = \pi^\mu \alpha$, $b = \pi^\mu \beta$
      and $c = \pi^{v(c)} \gamma$.  Furthermore, we put $\xi =
      \alpha^{-1}$ so that $\xi^\circ + \xi = 0$.  Subject to the
      necessary conditions collected above, the matrix
      \[
      \Gamma_0 =
      \begin{bmatrix}
        1 & -c & 1 \\
        c & \xi (1+\beta) c - \pi^\mu & -\xi(1+\beta) \\
        1 & \xi (1+\beta) & 0
      \end{bmatrix}
      \in \GL_3(\Lri_{\len-i})
      \]
      is hermitian and satisfies $E^\circ \Gamma_0 + \Gamma_0 E =
      0$. Thus Proposition~\ref{Dieudonne2} shows that $E$ is similar
      to an anti-hermitian matrix.

      To determine the unitary shadow, we look for solutions $Y \in
      \Cen_{\gl_3(\Lri_{\len-i})}(E)$, see \eqref{equ:cen}, to the
      equation $Y^\circ \Gamma_0 + \Gamma_0Y=0$.  A straightforward
      computation shows that parameters $(t_1,t_2,s_1,s_2,s_3)$
      leading to such $Y$ must satisfy the following congruences
      modulo $\fP$,
      \[
      t_1+t_1^\circ \equiv_\fP t_2-t_1 \equiv_\fP 0, \qquad
      s_1+s_1^\circ \equiv_\fP s_3-s_3^\circ \equiv_\fP 0, \qquad s_1
      \equiv_\fP \xi s_2.
      \]

      As in the case~(\vII) one shows that the Lie centraliser shadow
      of $E$ in~$\gu_3(\kk;\overline{\Gamma_0})$ is
      $3$\nobreakdash-dimensional.
%
      By inspection of Table~\ref{tab:shadows.iso.GU} we deduce that
      the unitary shadow associated to $B$ has type~$\Nsha$.
    \end{enumerate}
    \vspace*{-\baselineskip}
  \end{list}
\end{proof}

\subsection{Proof of
  Theorem~\ref{thm:U.shad.graph}}\label{subsec:proof.main.SU}
Part~(1) of Theorem~\ref{thm:U.shad.graph} follows from collecting the
types in Theorem~\ref{thm:representatives.for.GU3.action.gu3}.  To
prove part~(2), we proceed along the same lines as in the proof of
part~(2) of Theorem~\ref{thm:G.shad.graph}.  Interestingly, we get the
same polynomials $a_{\sigma,\tau}$ as in
Theorem~\ref{thm:G.shad.graph}, with the exception of types $\Kasha$
and $\Kbsha$ which do not occur in the present setting.

Let $A_\len \in \gl_3(\Lri_\len)$ be of one of the matrices specified
in Theorem~\ref{thm:representatives.for.GU3.action.gu3} and $\cC$ the
intersection of its similarity class with $\gu_3(\lri_\len)$.
Starting from the shadow $\sigma$ of $\cC$, we determine the shadows
$\tau$ associated to $\cCtilde$, the intersections with
$\gu_3(\lri_\len)$ of the similarity classes of lifts of $A_\len$ to
matrices $\wt{A}_{\len+1} \in \gl_3(\Lri_{\len+1})$.  We also
keep track of the multiplicities of such lifts: these depend only on
the shadows involved and the non-zero values $a_{\sigma,\tau}(q)$ are
as listed in Table~\ref{tab:branch.rules.A2}.

\begin{list}{}{\setlength{\leftmargin}{0pt}
    \setlength{\labelwidth}{-10pt} \setlength{\itemsep}{2pt}
    \setlength{\parsep}{1pt}}
\item[($\Gsha$)] Suppose that $\sigma$ has type~$\Gsha$.  Then $A_\len
  = d \Id_3$ with $d \in \gu_1(\lri)$.  Consider
  $\wt{A}_{\len+1} = d\Id_3 +\pi^\len X \in
  \gl_3(\Lri_{\len+1})$, with $X \in \gl_3(\kk_2)$.  Without loss of
  generality we may assume that $X \in \gu_3(\kk_2)$.  Then
  $\shU(\wt{A}_{\len+1}) = \shU(X)$ and
  $\shu(\wt{A}_{\len+1}) = \shu(X)$; furthermore, these shadows
  can be classified according to the shape of the minimal polynomial
  of $X$ as listed in Table~\ref{tab:shadows.iso.GU}.  The number
  $a_{\sigma,\tau}(q)$ of distinct lifts with shadow $\tau$ is the
  number of distinct minimal polynomials of the shape given in
  Table~\ref{tab:shadows.iso.GU}, paired with a compatible
  characteristic polynomial for type $\Lsha$.  Explicit formulae for
  the $a_{\sigma,\tau}(q)$ are given in
  Table~\ref{tab:branch.rules.A2}.
\item[($\Lsha$)] Suppose that $\sigma$ has type~$\Lsha$.  By
  Theorem~\ref{thm:representatives.for.GU3.action.gu3} we may assume
  that $A_\len = d\Id_3 + \pi^i D(a,0,0)$, where $0 \leq i < \len$,
  $\; d \in \gu_1(\lri_\len)$ and $a \in \Lri_{\len-i}^\times$ with
  $a^\circ + a = 0$.  Any lift of $A_\len$ that is conjugate to an
  anti-hermitian matrix is conjugate to a matrix of the form
  \[
  \wt{A}_{\len+1} =d\Id_3 + \pi^i D(a,0,0) + \pi^{\len}
  \begin{bmatrix}
    f & 0 \\
    0 & F
  \end{bmatrix} \quad \text{with $f \in \gu_1(\kk)$, $F \in
    \gl_2(\kk_2)$},
  \]
  where $F$ scalar corresponds to case (ii) and $F$ a companion matrix
  corresponds to case (iii).  By the analysis of cases (ii) and (iii)
  in the proof of
  Theorem~\ref{thm:representatives.for.GU3.action.gu3}, we classify
  the similarity classes depending on the minimal polynomial of~$F$.
  The shadow $\tau = \shUC(\cCtilde)$ has one of four types:
  \begin{itemize} \renewcommand{\labelitemi}{$\circ$}
  \item $\tau$ has type $\Lsha$ if and only if $F$ is an
    anti-hermitian scalar matrix.  There are $a_{\sigma,\tau}(q) =
    q^2$ choices for $(f,F)$.
  \item $\tau$ has type $\Tasha$ if and only if $F$ has a reducible
    separable minimal polynomial over~$\kk_2$ with anti-hermitian
    roots $c_1,c_2$.  There are $a_{\sigma,\tau}(q) = \hlf (q-1)q^2$
    choices for $(f,F)$.
  \item $\tau$ has type $\Tbsha$ if and only if $F$ has a reducible
    separable minimal polynomial over~$\kk_2$ with roots satisfying
    $c_1+c_2^\circ=0$.  There are $a_{\sigma,\tau}(q) = \hlf (q-1)q^2$
    choices for $(f,F)$.
  \item $\tau$ has type $\Msha$ if and only if $F$ has minimal
    polynomial $(x-\alpha)^2$ for some $\alpha \in \gu_1(\kk)$.  There
    are $a_{\sigma,\tau}(q) =q^2$ choices for $(f,F)$.
  \end{itemize}
\item[($\Jsha$)] Suppose that $\sigma$ has type~$\Jsha$. In this case
  we may assume, by
  Theorem~\ref{thm:representatives.for.GU3.action.gu3}, that
  \[
  A_\len = d \Id_3 + \pi^i
  \begin{bmatrix}
    0 & 0 & 0 \\
    0 & 0 & 1 \\
    0 & 0 & c
  \end{bmatrix},
  \]
  where $0 \leq i < \len$, $\; d \in \gu_1(\Lri_\len)$ and $c \in
  \gu_1(\lri_{\len-i})$ with $v(c)>0$.
  Theorem~\ref{thm:representatives.for.GU3.action.gu3} yields a
  complete parametrisation for the intersections $\cCtilde$ of
  $\gu_3(\lri_{\len+1})$ with similarity classes of matrices lying
  above $A_\len$ in $\gl_3(\Lri_{\len+1})$.  The shadow $\tau =
  \shGC(\cCtilde)$ has one of three types:
  \begin{itemize} \renewcommand{\labelitemi}{$\circ$}
  \item $\tau$ is of type $\Jsha$ if and only if the lift of $A_\len$
    is conjugate to
    \[
    \wt{A}_{\len+1} = d'\Id_3 + \pi^i
    \begin{bmatrix}
      0 & 0 & 0 \\
      0 & 0 & 1\\
      0 & 0 & c'
    \end{bmatrix},
    \]
    where $c'$ and $d'$ are arbitrary anti-hermitian lifts of $c$ and
    $d$, respectively.  Consequently there are $a_{\sigma, \tau}(q) =
    q^2$ choices.
  \item $\tau$ is of type $\Msha$ if and only if the lift of
    $A_\len$ is conjugate to
    \[
    \wt{A}_{\len+1} = \varsigma(d) \Id_3 + \pi^i
    \begin{bmatrix}
      0 & \pi^{\len-i} & 0 \\
      0 & 0 & 1 \\
      a' \pi^{\len-i} & b'\pi^{\len-i} & c'
    \end{bmatrix},
    \]
    where $a' \in \gu_1(\kk)$, $b' \in \kk^\times$ and $c' \in
    \gu_1(\lri_{\len-i+1})$ is a lift of $c$.  Therefore, there are
    $a_{\sigma, \tau}(q) = (q-1)q^2$ choices.
  \item $\tau$ is of type~$\Nsha$ if and only if the lift of
    $A_\len$ is conjugate to
    \[
    \wt{A}_{\len+1} = \varsigma(d) \Id_3 + \pi^i
    \begin{bmatrix}
      0 & \pi^{\len-i} & 0 \\
      0 & 0 & 1 \\
      a'\pi^{\len-i} & 0 & c'
    \end{bmatrix},
    \]
    where $a' \in \gu_1(\kk)$ and $c' \in \gu_1(\lri_{\len-i+1})$ is a
    lift of $c$. There are $a_{\sigma, \tau}(q) = (q-1)q$ choices.
  \end{itemize}
\item[($\Tasha,\Tbsha,\Tcsha,\Msha,\Nsha$)] Suppose that $\sigma$ has
  type equal to one of $\Tasha$, $\Tbsha$, $\Tcsha$, $\Msha$, $\Nsha$.
  From Table~\ref{tab:shadows.iso.GU} we observe that all these cases
  are minimal in the sense that $\shu(A_\len)$ cannot properly contain
  the Lie centraliser shadow of any other type.  This implies that the
  shadow associated to any lift $\wt{A}_{\len+1}$ of the matrix
  $A_\len$ satisfies $\shUC(\cCtilde) = \shUC(\cC)$.  Therefore, in
  all the cases under consideration
  Proposition~\ref{pro:class.quot.GU} and Definition~\ref{def:b.GU}
  yield
  \[
  a_{\sigma,\sigma}(q)=q^{\dim \gu_3} / b^{(-1)}_{\sigma,\sigma}(q) =
  q^{\dim(\sigma)}=q^3.
  \]
\end{list}


\section{Similarity class zeta functions} \label{sec:sim.zeta} Let
$\lri$ be a compact discrete valuation ring with valuation ideal $\fp$
and finite residue field $\kk$ of cardinality~$q$.  In the context of
anti-hermitian matrices we assume that $\cha(\kk) \neq 2$.  There is
no other restriction on the characteristic of $\lri$ or $\kk$.  In
this section we define similarity class zeta functions of the finite
spaces $\gl_n(\lri_\len)$ and $\gu_n(\lri_\len)$ for $\len \in \N_0$,
and suitable limit objects as $\len\rightarrow\infty$. We employ the
results from Sections~\ref{sec:sim.gl} and \ref{sec:sim.gu} to
compute, in Section~\ref{subsec:sim.zeta.A2}, all of these functions
for~$n=3$.  From these we deduce Theorem~\ref{thm:sim.zeta.local} and
Corollary~\ref{cor:F}.

\subsection{Similarity class zeta functions and shadow graphs}\label{subsec:sim.class.zeta.graphs}
Let $n\in\N$.  The two cases
\[
\ag = \gl_n \text{ and } \aG = \GL_n, \qquad \ag = \gu_n \text{ and }
\aG = \GU_n
\]
of pairs of $\lri$-schemes are very similar and we treat them in
parallel.  Fix $\len \in \N_0$.  We write $\cQ_\len$ for the finite
set of similarity classes $\cQ_{\lri,\len}^\ag = \Ad(\aG(\lri))
\backslash \ag(\lri_\len)$, introduced in
Definitions~\ref{def:graph.Qgl} and~\ref{def:graph.Qgu}, and let $\Sh$
stand for the shadow set $\Sh_{\aG(\lri)}$; cf.\ \eqref{equ:Sh.GL} and
\eqref{equ:Sh.GU}.

\begin{defn}\label{def:sim.class.zeta}
  The \emph{similarity class zeta function} of $\ag(\lri_\len)$ is the
  Dirichlet polynomial
  \[
  \gamma_\len(s) \coloneqq \sum_{\cC \in \cQ_\len} \lvert \cC
  \rvert^{-s}.
  \]
  For $\sigma\in \Sh$, we set
  \begin{equation}\label{def:gamma.sigma.fin}
    \gamma^\sigma_\len(s) = \sum_{\substack{\cC \in \cQ_\len \\ \sh(\cC)
        = \sigma}} \lvert \cC \rvert^{-s},
  \end{equation}
  yielding the natural decomposition $\gamma_\len(s) = \sum_{\sigma
    \in \Sh} \gamma_\len^\sigma(s)$.
\end{defn}

\begin{prop}\label{pro:limit.gamma}
  Let $\sigma\in \Sh$ and write $\gamma^\sigma_\len(s) =
  \sum_{m=1}^\infty c^\sigma_{\len,m} m^{-s}$.  For each $m \in \N$,
  the sequence $(q^{-\len}c^\sigma_{\len,m})_{\len\in\N_0}$ is
  eventually constant.  In particular, the normalised Dirichlet
  polynomials $q^{-\len}\gamma^\sigma_\len(s)$ converge
  coefficientwise to a Dirichlet series
  \[
  \gamma^\sigma(s) \coloneqq
  \lim_{\len\rightarrow\infty}q^{-\len}\gamma_\len^\sigma(s)
  \]
  with non-negative rational coefficients. If $\cha(\kk)$ does not
  divide $n$, then the numbers $c^\sigma_{\len,m}$ are all divisible
  by $q^\len$, whence $\gamma^\sigma(s)$ has integral
  coefficients.
\end{prop}

\begin{proof}
  
  In principle, the coefficients $c^\sigma_{\len,m}$ can be computed
  by induction on $\len$; this requires consideration of other shadows
  $\tau$ and class sizes $m' \le m$.  (In the special case $n=3$ we
  can carry out the procedure effectively; see
  Lemma~\ref{recurrence.gamma} and
  Proposition~\ref{pro:sim.shadow.fin} below.)  To prove that
  $q^{-\len} c^\sigma_{\len,m}$ becomes constant as $\len \to \infty$,
  we observe that the $q^\len$ scalar matrices are the only matrices
  in $\ag(\lri_\len)$ whose lifts to matrices in $\ag(\lri_{\len+1})$
  do not all give rise to larger similarity classes.  Indeed,
  Propositions \ref{pro:class.quot.GL} and \ref{pro:class.quot.GU}
  imply that if $\cCtilde \subset \ag(\lri_{\len+1})$ is a class with
  $\sh_\aG(\cCtilde) = \wt{\tau}$ which lies above a class $\cC
  \subset \ag(\lri_\len)$ with $\sh_\aG(\cC) = \tau$, then the
  quotient $\lvert \cCtilde \rvert / \lvert \cC \rvert$ is greater
  than $1$ unless $\tau$ is of type~$\Gsha$, i.e.\ a scalar matrix.
  The branching process by which one arrives from a scalar matrix in
  $\ag(\lri_\len)$ to similarity classes of the given size $m$ and
  shadow $\sigma$ in $\ag(\lri_{\len'})$ for $\len' > \len$ is
  independent of~$\len$.  Due to the normalisation, the numbers
  $q^{-\len} c^\sigma_{\len,m}$ thus stabilise as $\len \to \infty$.

  Finally we argue for the integrality.  Let $m \in \N$.  For every
  similarity class $\cC \in \cQ_\len$ of size $m$ and shadow~$\sigma$,
  the scalar shifts $d \Id_n + \cC$, where $d \Id_n \in
  \ag(\lri_\len)$, form $q^\len$ similarity classes in $\cQ_\len$,
  each of size $m$ and shadow~$\sigma$. They are all distinct since
  the traces of any two such shifts by $d \Id_n $ and $d' \Id_n$, say,
  differ by $n(d-d')$ and $n$ is invertible in $\lri_\len$. Thus each
  coefficient $c^\sigma_{\len,m}$ is divisible by~$q^\len$.
\end{proof}

\begin{remark}
  Corollary~\ref{cor:infinite.gamma} below shows that the assumption
  $\cha(\kk)\nmid n$ may be necessary and that the limit functions
  $\gamma^\sigma(s)$, $\sigma \in \Sh$, are rational functions
  for~$n=3$. Proposition~\ref{pro:gamma.sigma.fin.A1} implies the
  analogous fact for $n=2$. It is an interesting question whether
  rationality also holds for $n \geq 4$. It is perceivable that the
  methods developed in the paper \cite{HMRC}, which proves rationality
  of zeta functions enumerating classes of certain definable
  equivalence relations, are applicable in this context.
\end{remark}

From now on let $n=3$ so that $\ag, \aG$ are either $\gl_3,\GL_3$ or
$\gu_3,\GU_3$.  It is convenient to use the parameter $\ee = \ee_\ag =
\ee_\aG \in \{1,-1\}$ to distinguish between the non-unitary and the
unitary setting; see~\eqref{equ:def.epsilon}.  Recall also the
definition of shadow types $\T^{(\ee)}$; see~\eqref{equ:types-def}.
For $3\times3$ matrices, it turned out that the similarity class trees
$\cQ^{\gl_3}_\lri$ and $\cQ^{\gu_3}_\lri$ have a structure that can be
described uniformly for different choices of~$\lri$.  This motivates
the following definition.

\begin{defn}\label{def:shadow.graph.A2}
  The \emph{shadow graph} associated to the scheme pair $(\ag, \aG)$
  is the finite directed graph $\Gamma= \Gamma^{(\ee)}$ with the
  following vertex and edge sets
  \begin{align*}
    V(\Gamma) & = \T^{(\ee)}, \\
    E(\Gamma) & = \{ (\cS,\cT) \in \T^{(\ee)} \times \T^{(\ee)} \mid
    \exists \sigma, \tau \in \Sh \text{ of types } \cS,\cT:
    a_{\sigma,\tau} \ne 0\}.
  \end{align*}
  In the following it is convenient to refer to $\dot{V}(\Gamma)
  \coloneqq \Sh$ and $\dot{E}(\Gamma) \coloneqq \{ (\sigma,\tau) \in
  \Sh \times \Sh \mid a_{\sigma,\tau} \ne 0\}$, suppressing the
  implicit dependency on $\lri$.
\end{defn}

\begin{remark}
  Recall that the polynomials $a_{\sigma,\tau}\in\Z[\sixth][t]$ are
  defined in Theorems~\ref{thm:G.shad.graph} and
  \ref{thm:U.shad.graph}, and tabulated in
  Table~\ref{tab:branch.rules.A2}.  The results in
  Sections~\ref{sec:sim.gl} and~\ref{sec:sim.gu} imply that $\Gamma$
  is naturally isomorphic to the directed graph $\dot{\Gamma}$ with
  vertex set $V(\dot{\Gamma}) = \dot{V}(\Gamma)$ and edge set
  $E(\dot{\Gamma}) = \dot{E}(\Gamma)$.  The graph $\dot{\Gamma}$ in
  turn is nothing but the quotient graph of the rooted tree $\cQ
  \coloneqq \cQ_{\lri}^\ag$ (cf.\ Definitions~\ref{def:graph.Qgl}
  and~\ref{def:graph.Qgu}) induced by the map $V(\cQ) \rightarrow
  \Sh$, $\cC \mapsto \sh(\cC)$.
\end{remark}

By Theorems~\ref{thm:G.shad.graph} and \ref{thm:U.shad.graph}, the
shadow graph $\Gamma$ and the data $a_\xi(q)$, $\xi \in
\dot{E}(\Gamma)$, determine the tree $\cQ$ completely, whereas the
graph $\Gamma$ and the data $b^{(\ee)}_\xi(q)$, $\xi \in
\dot{E}(\Gamma)$, determine the sizes of the similarity classes, which
correspond to the vertices of~$\cQ$.  Indeed, if $\cC \in \cQ_\len$,
then Propositions~\ref{pro:class.quot.GL} and~\ref{pro:class.quot.GU}
show that
\begin{equation*}\label{equ:sise.of.sim.class}
  \lvert \cC \rvert =  \prod_{i=0}^{\len-1}
  b^{(\ee)}_{\sh(\cC_i), \sh(\cC_{i+1})}(q),
\end{equation*}
where $\cC_i$ is the reduction of $\cC$ modulo~$\fp^i$.
Figure~\ref{fig:shadow.graph.A2} displays the shadow graph $\Gamma$
associated to the scheme pair $(\ag,\aG)$.  The edge labels in
Figure~\ref{fig:shadow.graph.A2} match with the row numbers in
Table~\ref{tab:branch.rules.A2}.

\begin{figure}[htb!]
 \centering
 \caption{The shadow graph $\Gamma^{(\ee)}$ for $\gl_3, \GL_3$ ($\ee=1$)
   and $\gu_3, \GU_3$ ($\ee = -1$)${}^*$}
 \label{fig:shadow.graph.A2}
 \begin{displaymath}
   \xymatrix@+10pt{ \boldsymbol{\Tcsha} \ar@(ul,ur)[]|{18} & &
     \boldsymbol{\Gsha} \ar@(ul,ur)[]|{1}  \ar[ddl]|{3}
     \ar[ll]|{6} \ar[dr]|{4} \ar[ddr]|{5} \ar[dddr]|{7} \ar[rr]|{2}
     \ar[dll]|{8} & & \boldsymbol{\Lsha} \ar@(ul,ur)|{9}\ar[dl]|{10}
     \ar[ddl]|{11} \ar[dddl]|{12} \\
     \boldsymbol{\Nsha} \ar@(dl,ul)[]|{18} & & & \boldsymbol{\Tasha}
     \ar@(ul,ur)[]|{18} \\
     \boldsymbol{\Kasha} \ar@(dl,ul)@{.>}[]|{18}& \boldsymbol{\Jsha}
     \ar@{.>}[dl]|{17} \ar@{.>}[l]|-{16} \ar[ul]|{15} \ar[drr]|{14}
     \ar@(dl,dr)[]|{13}
     &  & \boldsymbol{\Tbsha} \ar@(dl,dr)[]|{18}\\
     \boldsymbol{\Kbsha} \ar@(dl,dr)@{.>}[]|{18} & & &
     \boldsymbol{\Msha} \ar@(dl,dr)[]|{18} }
 \end{displaymath}
 \text{${}^*$ For $\ee = -1$ the vertices $\Kasha$ and $\Kbsha$ with
   the incident edges are to be omitted.}
\end{figure}

In particular, the graph $\Gamma$ together with the data $a_\xi(q)$
and $b^{(\ee)}_\xi(q)$, $\xi \in \dot{E}(\Gamma)$, provides recursive
formulae for the similarity class zeta functions of $\ag(\lri_\len)$,
$\len \in \N_0$.

\begin{lem}\label{recurrence.gamma}
  For $\len \in \N_0$ and $\tau \in \Sh$ we have
  \begin{equation*}\label{equ:gamma.rec}
    \gamma_{\len+1}^\tau(s) = \sum_{(\sigma,\tau) \in \dot{E}(\Gamma^{(\ee)})}
    a_{\sigma,\tau}(q) \, b^{(\ee)}_{\sigma,\tau}(q)^{-s} \, \gamma^\sigma_\len(s).
  \end{equation*}
\end{lem}

\begin{proof}
  For $\ee = 1$, the claim follows from
  Proposition~\ref{pro:class.quot.GL} and
  Theorem~\ref{thm:G.shad.graph}; for $\ee=-1$, from
  Proposition~\ref{pro:class.quot.GU} and
  Theorem~\ref{thm:U.shad.graph}:
  \begin{equation*}
    \gamma^\tau_{\len+1}(s) = \sum_{\substack{\cCtilde \in
        \cQ_{\len+1} \\ \sh(\cCtilde) = \tau}} \lvert \cCtilde
    \rvert^{-s} = \sum_{\substack{\cC \in \cQ_\len \\
        \sh(\cC) = \sigma}} \, \sum_{\substack{(\sigma,\tau) \in \\
        \dot{E}(\Gamma^{(\ee)})}} a_{\sigma,\tau}(q) \, \big(
    b^{(\ee)}_{\sigma,\tau}(q) \lvert \cC \rvert \big)^{-s} =
    \sum_{\substack{(\sigma,\tau) \in \\ \dot{E}(\Gamma^{(\ee)})}} a_{\sigma,\tau}(q)
    b^{(\ee)}_{\sigma,\tau}(q)^{-s} \gamma^\sigma_\len(s).  \qedhere
  \end{equation*}
\end{proof}

\subsection{Explicit formulae for similarity class zeta functions for
  type $\mathsf{A}_2$}\label{subsec:sim.zeta.A2}
In order to state explicit formulae for the Dirichlet generating
function $\gamma^\sigma_\len(s)$ we define, for $r\in\N$, the
auxiliary polynomials
\begin{equation}\label{def:aux.count}
  f^{r}_\len(a_1,\ldots,a_r) \coloneqq \sum_{\substack{(j_1,\ldots,j_r)
      \in \N_0^r \\ \sum_{i=1}^r j_i \leq \len-r}} a_1^{j_1} \cdots
  a_r^{j_r} \in \Z[a_1,\ldots,a_r].
\end{equation}
We shall only make use of $f^1_\len$ and $f^2_\len$. Note that, as
rational functions in $\Q(a_1,\ldots,a_r)$,
\[
f^1_\len(a_1) = \frac{1-a_1^{\len}}{1-a_1} \qquad \text{and} \qquad
f^2_\len(a_1,a_2) =\frac{a_1a_2^{\len} - a_1^{\len}a_2 + a_1^{\len}
  - a_2^{\len} + a_2 - a_1}{(a_2-a_1)(1-a_1)(1-a_2)}.
\]
We set
\begin{equation}\label{equ:aux.ABC}
  \UU =f_\len^1(q^{1-4s}), \quad
  \VV =f_\len^1(q^{2-6s}), \quad
  \WW =f_\len^2(q^{1-4s},q^{2-6s}).
\end{equation}

\begin{prop} \label{pro:sim.shadow.fin} For $\sigma\in\Sh$ of type
  $\cS\in\T^{(\ee)}$ and $\len\in\N_0$,
  \[
  \gamma_\len^\sigma(s) = q^\len
  \Gamma^{\cS}_{\ee,q,\len}(s),
  \]
  where the function $\Gamma^{\cS}_{\ee,q,\len}(s) \coloneqq
  \Gamma^{\cS}_{\mathsf{A}_2,\ee,q,\len}(s)$ is defined as
  \begin{equation*}
   \begin{array}{ll}
     1 & \text{if $\cS = \Gsha$,} \\ 
     (q-1) \left((q^2+\ee q + 1)q^2\right)^{-s} \UU & \text{if
       $\cS = \Lsha$,} \\ 
     \left( (q - \ee)^3(q + \ee) \right)^{-s} \UU & \text{if $\cS =
       \Jsha$,} \\ 
     \sixth  (q-1) \left( (q+\ee)(q^2+\ee q + 1)q^3) \right)^{-s}
     \left[ (q-2) \VV + 3 (q-1) q^{1-4s} \WW \right] & \text{if $\cS =
       \Tasha$,} \\ 
     \hlf  (q-1) \left( (q^3 - \ee)q^3 \right)^{-s} \left[ q \VV +
       (q-1) q^{1-4s} \WW \right] & \text{if $\cS = \Tbsha$,} \\ 
     \third (q^2-1) \left( (q + \ee)(q - \ee)^2q^3 \right)^{-s} \VV &
     \text{if $\cS = \Tcsha$,} \\ 
     (q-1) \left( (q - \ee)^3(q + \ee)q^2 \right)^{-s} \left[ \VV +
       2q^{1-4s}\WW \right] & \text{if $\cS = \Msha$,} \\ 
     \left( (q^2 - 1)(q^3 - \ee)q \right)^{-s}
     \left[ \VV +(1-q^{-1}) q^{1-4s} \WW \right] & \text{if
       $\cS = \Nsha$,} \\ 
     \left( (q^2 - 1)(q^3 - \ee)q^5 \right)^{-s} \, \WW &
     \text{if $\cS \in \{ \Kasha, \Kbsha \}$.} 
    \end{array}
  \end{equation*}
\end{prop}

\begin{proof}
  The proof is a straightforward induction on $\len$, using:
  Lemma~\ref{recurrence.gamma}, the explicit formulae for the
  polynomials $a_{\sigma,\tau}(q)$ and $b_{\sigma,\tau}^{(\ee)}(q)$
  from Table~\ref{tab:branch.rules.A2}, and the
  definitions~\eqref{equ:aux.ABC}.

  The case $\len=0$ is clear as $\Gamma^{\cS}_{\ee,q,0}(s) = 0$ unless
  $\cS=\Gsha$.  For the induction step we assume that the proposition
  is proved for $\len\in\N_0$.  We give exemplary proofs for the types
  $\Gsha$, $\Lsha$, and $\Tasha$, which are representative of the
  shadow graph's local complexities; cf.\
  Figure~\ref{fig:shadow.graph.A2}. The computations for the other
  types are similar.

  Let $\sigma, \tau, \upsilon \in \Sh$ be shadows of types $\Gsha,
  \Lsha, \Tasha$, respectively.  For $\sigma$ we have
  \[
  \gamma_{\len+1}^\sigma(s) = a_{\sigma,\sigma}(q) \,
  b^{(\ee)}_{\sigma,\sigma}(q)^{-s} \, \gamma_\len^\sigma(s) = q
  \gamma_\len^\sigma(s) = q^{\len+1}
  \]
  as claimed.  For $\tau$ we have
  \begin{align*}
    \gamma_{\len+1}^\tau(s) & = a_{\sigma,\tau}(q) \,
    b^{(\ee)}_{\sigma,\tau}(q)^{-s} \, \gamma_\len^\sigma(s) +
    a_{\tau,\tau}(q) \, b^{(\ee)}_{\tau,\tau}(q)^{-s} \,
    \gamma_\len^\tau(s) \\
    & = q^{\len+1}(q-1) \left( (q^2+\ee q + 1)q^2 \right)^{-s} +
    q^{1-4s}q^{\len+1} (q-1) \left((q^2+\ee q + 1)q^2 \right)^{-s}
    \sum_{j=0}^{\len-1} q^{(1-4s)j} \\
    & = q^{\len+1} (q-1) \left((q^2+\ee q + 1)q^2 \right)^{-s}
    \sum_{j=0}^{\len} q^{(1-4s)j}
  \end{align*}
  as claimed. We now argue for the shadow $\upsilon$. Set
  \begin{align*}
    u_{q,\len}(s) & = \sixth q^{\len}(q-1)(q-2) \left( (q^2+\ee
      q+1)(q+\ee)q^3 \right)^{-s} \sum_{j=0}^{\len-1}q^{(2-6s)j}, \\
    v_{q,\len}(s) & = \hlf q^{\len+1}(q-1)^2 \left( (q^2+\ee
      q+1)(q+\ee)q^7 \right)^{-s} \sum_{\substack{(j_1,j_2)\in\N_0^{\,2}\\
        j_1+j_2\leq\len-2}} q^{(1-4s)j_1 + (2-6s)j_2}
  \end{align*}
  so that
  \[
  \gamma_\len^\upsilon(s) = u_{q,\len}(s) + v_{q,\len}(s).
  \]
  Then
  \begin{align*}
    \gamma^\upsilon_{\len+1}(s) & = a_{\sigma,\upsilon}(q) \,
    b^{(\ee)}_{\sigma,\upsilon}(q)^{-s} \, \gamma_\len^\sigma(s) +
    a_{\tau,\upsilon}(q) \, b^{(\ee)}_{\tau,\upsilon}(q)^{-s} \,
    \gamma_\len^\tau(s) + a_{\upsilon,\upsilon}(q) \,
    b^{(\ee)}_{\upsilon,\upsilon}(q)^{-s}
    \, \gamma_\len^\upsilon(s) \\
    & = \sixth q(q-1)(q-2) \left( (q^2+\ee q+1)(q+\ee)q^3
    \right)^{-s} \gamma_\len^\sigma(s) \\
    & \quad +\hlf q^2(q-1) \left( (q+\ee)q^5 \right)^{-s}
    \gamma_\len^\tau(s) + q^{3-6s} \left(
      u_{q,\len}(s) + v_{q,\len}(s) \right) \\
    & = \sixth q^{\len+1}(q-1)(q-2) \left( (q^2+\ee q+1)(q+\ee)q^3
    \right)^{-s} + q^{3-6s} u_{q,\len}(s) \\
    & \quad + \hlf q^2(q-1) \left( (q+\ee)q^5 \right)^{-s} q^\len(q-1)
    \left( (q^2+\ee q+1)q^2 \right)^{-s} \sum_{j=0}^{\len-1}
    q^{(1-4s)j} + q^{3-6s} v_{q,\len}(s) \\
    & = u_{q,\len+1}(s) + v_{q,\len+1}(s).  \qedhere
  \end{align*}
\end{proof}

\begin{cor}\label{cor:infinite.gamma}
  For $\sigma \in \Sh$ of type $\cS\in\T^{(\ee)}$,
  \[
  \gamma^\sigma(s) =
  \lim_{\len\rightarrow\infty}q^{-\len}\gamma_\len^\sigma(s) =
  \Gamma^{\cS}_{\ee,q}(s), 
  \]
  where the function $\Gamma^{\cS}_{\ee,q}(s) \coloneqq
  \Gamma^{\cS}_{\mathsf{A}_2,\ee,q}(s)$ is given by
  \begin{equation*}
    \begin{array}{ll}
      1 & \text{if $\cS = \Gsha$,} \\ 
      (q-1) \left((q^2+\ee q + 1)q^2\right)^{-s} (1-q^{1-4s})^{-1} & \text{if
        $\cS = \Lsha$,} \\ 
      \left( (q - \ee)^3(q + \ee) \right)^{-s} (1-q^{1-4s})^{-1} &
      \text{if $\cS = \Jsha$,} \\ 
      \sixth  (q-1) \left( (q+\ee)(q^2+\ee q + 1)q^3) \right)^{-s}
      \frac{q-2+2q^{2-4s}-q^{1-4s}}{(1-q^{1-4s})(1-q^{2-6s})} & \text{if $\cS =
        \Tasha$,} \\ 
      \hlf  (q-1) \left( (q^3 - \ee)q^3 \right)^{-s} (q-q^{1-4s})\left( (1-q^{1-4s})(1-q^{2-6s}) \right)^{-1} &
      \text{if $\cS = \Tbsha$,} \\ 
      \third (q^2-1) \left( (q + \ee)(q - \ee)^2q^3 \right)^{-s}
      (1-q^{2-6s})^{-1} & 
      \text{if $\cS = \Tcsha$,} \\ 
      (q-1) \left( (q - \ee)^3(q + \ee)q^2 \right)^{-s}
      (1+q^{1-4s}) \left((1-q^{1-4s})(1-q^{2-6s}) \right)^{-1} & \text{if $\cS =
        \Msha$,} \\ 
      \left( (q^2 - 1)(q^3 - \ee)q \right)^{-s}
      (1-q^{-4s}) \left( (1-q^{1-4s})(1-q^{2-6s}) \right)^{-1} & \text{if
        $\cS = \Nsha$,} \\ 
      \left( (q^2 - 1)(q^3 - \ee)q^5 \right)^{-s}
      \left( (1-q^{1-4s})(1-q^{2-6s}) \right)^{-1} & 
      \hspace*{-1.25cm}\text{if $\cS \in \{ \Kasha, \Kbsha \}$.} 
     \end{array}
  \end{equation*}
\end{cor}

We conclude this section with the proofs of two main results stated in
the introduction.

\begin{proof}[Proof of Theorem~\ref{thm:sim.zeta.local}]
  The claimed formula is a direct consequence of
  Proposition~\ref{pro:sim.shadow.fin}, upon noting that
  $s_\len(\ag(\lri)) = \gamma_\len(0) = \sum_{\sigma \in \Sh}
  \gamma_\len^\sigma(0)$ for $\len \in \N_0$.  Using
  \[
  \UU \to \frac{1}{1-q}, \quad \VV \to \frac{1}{1-q^2}, \quad \WW \to
  \frac{1}{(1-q)(1-q^2)} \qquad \text{as $\len \to \infty$, $s \to
    0$}
  \]
  the computation becomes routine.
\end{proof}

\begin{proof}[Proof of Corollary~\ref{cor:F}]
  The corollary is formulated in such a way that, given a place $v
  \not\in S$ of the number field~$k$, the Euler factor
  $\zeta^{\mathrm{sc}}_{\fg(\Gri_v)}(s)$ of
  $\zeta^{\mathrm{sc}}_{\fg(\Gri_S)}(s)$ in
  \eqref{equ:euler.zeta.sc} is equal to
  $\zeta^{\mathrm{sc}}_{\gu_3(\Gri_v)}(s)$ if $\ee_\mathbf{G} = -1$
  and $v$ is not decomposed in the quadratic extension of $K \,\vert\, k$
  defining~$\mathbf{G} = \GU_3(K,f)$; in all other cases
  $\zeta^{\mathrm{sc}}_{\fg(\Gri_v)}(s) =
  \zeta^{\mathrm{sc}}_{\gl_3(\Gri_v)}(s)$.  The claimed formula thus
  follow from \eqref{equ:sim.local} via the Euler product
  decomposition of the Dedekind zeta function~$\zeta_k(s)$, the fact
  that the abscissa of convergence of $\zeta_k(s)$ is~$1$, and the
  Tauberian Theorem~\ref{thm:tauber} stated in
  Section~\ref{sec:abscissa}.
\end{proof}


\part{Representation zeta functions of groups of type
  $\mathsf{A}_2$}\label{part:2}


\section{The Kirillov orbit method and Clifford
  theory} \label{sec:kom.clifford}

Let $\lri$ be a compact discrete valuation ring of residue
characteristic~$p$.  For the main part, we focus in this section on
the case $\cha(\lri) = 0$; we also exhibit analogues of some results
in positive characteristic.  Fix $n \in \N_{\geq 2}$ and let $\aG$ be
one of the four $\lri$-group schemes $\GL_n$, $\GU_n$, $\SL_n$,
$\SU_n$, assuming $p>2$ in the unitary cases.  Write $G = \aG(\lri)$
and $N = \aG^1(\lri)$, the $1$st principal congruence subgroup.  We
develop techniques to study the irreducible complex characters of~$G$,
in relation to the irreducible complex characters of~$N$.  Given a
character $\chi \in \Irr(N)$, we write $S_\chi = \textrm{I}_G(\chi)$
for the inertia group of $\chi$ and we denote by $R_\chi$ the maximal
normal pro-$p$ subgroup of~$S_\chi$.  Observe that $N \triangleleft
R_\chi \triangleleft S_\chi \leq G$.

Under suitable assumptions, relating $p$ to $n$ and $e =
e(\lri,\Z_p)$, the pro-$p$ groups $N$ and $R_\chi$ are guaranteed to
belong to the class of saturable and potent groups; see
Sections~\ref{subsec:sat} and~\ref{subsec:appl.to.matrices}.  This
makes them amenable to the Kirillov orbit method, a machinery to
describe the irreducible complex characters in terms of co-adjoint
orbits; cf.\ Section~\ref{subsec:kom}.  In
Section~\ref{subsec:basic.princ} we provide the setup to apply the
Kirillov orbit method to the principal congruence
subgroups~$\aG^m(\lri)$.  The transition from similarity class zeta
functions (see Section~\ref{sec:sim.zeta}) to representation zeta
functions of groups of the form $\aG^m(\lri)$ and their finite
quotients~$\aG^m(\lri)/\aG^\len(\lri)$ is set out in
Section~\ref{subsec:sim.irrep}.  In Section~\ref{subsec:ext.kom} we
discuss circumstances under which the character $\chi$ extends from
$N$ to the pro-$p$ group~$R_\chi$; in Section~\ref{subsec:ext.coh} we
provide cohomological criteria for the character to extend further,
from $R_\chi$ to~$S_\chi$.  By Clifford theory, any extension of
$\chi$ to its inertia group $S_\chi$ induces irreducibly to~$G$,
completing the transition from $\Irr(N)$ to $\Irr(G)$.

To a certain degree the procedure works also over compact discrete
valuation rings of positive characteristic; we will indicate the
necessary modifications on the way.  The corresponding starred remarks
can be skipped if one wants to focus on the main situation.

\subsection{Saturable $\Z_p$-Lie lattices and pro-$p$
  groups}\label{subsec:sat} We recall some results from
$p$\nobreakdash-adic Lie theory.  The notions we require originate
from Lazard's pioneering work~\cite{La65} and were put into a group
theoretic setting by Lubotzky and Mann; see~\cite{DDMS/99}.  They were
developed further and refined in~\cite{Kl05,Gonzalez2}.

Let $\fr$ be a $\Z_p$-Lie lattice.  A Lie sublattice $\fn$ of $\fr$ is
said to be \emph{PF-embedded} in $\fr$ if there exists a \emph{potent
  filtration} starting at $\fn$, i.e.\ a descending series of Lie
sublattices $\fn = \fn_1 \supseteq \fn_2 \supseteq \ldots$ with
$\bigcap_i \fn_i = 0$ such that (i) $[\fr,\fn_i] \subset \fn_{i+1}$
and (ii) $[\fr \,\, {}_{(p-1)}, \fn_i] \subset p \fn_{i+1}$ for all $i
\geq 1$.  Here and in the following we use right-normed Lie brackets
so that $[X \,\, {}_{(n)}, Y] = [X,[ \cdots [ X,[X,Y]] \cdots] ]$,
with $X$ occurring $n$ times.  Observe that every PF-embedded Lie
sublattice is a Lie ideal in~$\fr$.  By \cite[Theorem~4.1]{Gonzalez2},
the Lie lattice $\fr$ is \emph{saturable} in the sense of Lazard if
and only if it is PF-embedded in itself.  Given a saturable Lie
lattice $\fr$ one introduces, by means of the Hausdorff series
\[
\Phi_\text{Hd}(X,Y) = X + Y + \tfrac{1}{2} [X,Y] + \tfrac{1}{12} \big(
[X,[X,Y]] + [Y,[Y,X]] \big) + \ldots \in \Q \langle\!\langle X,Y
\rangle\!\rangle
\]
a group multiplication on the set~$\fr$.  This yields a
\emph{saturable} pro-$p$ group $R = \exp(\fr)$, again in the sense of
Lazard.  Moreover, the map $\fr \mapsto \exp(\fr)$ yields an
isomorphism between the category of saturable $\Z_p$-Lie lattices and
saturable pro-$p$ groups.  (Sometimes the term `saturable' is applied
to groups that are not necessarily finitely generated.  In this paper,
we agree that saturable pro-$p$ groups are by definition finitely
generated.)  We denote the inverse isomorphism by $\log$, writing $\fr
= \log(R)$.  We write $e^X = \exp(X)$ for $X \in \fr$ and $\log(x)$
for $x \in R$ to denote the corresponding elements in the associated
structure.  Conjugation in $R$ is linked to the \emph{adjoint action}
of $R$ on $\fr$ via
\[
\Ad(x) Y = \log(x e^Y x^{-1}) \qquad \text{for $x \in R$ and $Y \in \fr$.}
\]

\begin{lem} \label{lem:Hausdorff-expansion} Let $\fr$ be a saturable
  $\Z_p$-Lie lattice.  Let $\fn$ be a PF-embedded Lie ideal of~$\fr$,
  with potent filtration $\fn = \fn_1 \supseteq \fn_2 \supseteq
  \ldots$.  Let $X \in \fr$, $i \in \N$ and $Y \in \fn_i$.  Then there
  exists $Z \in \fn_{i+1}$ such that
  \[
  \Ad(e^X) Y = Y + \ad(X) Y + \ad(X) Z.
  \]
\end{lem}

\begin{proof}
  The proof is similar to that of~\cite[Lemma~2.3(4)]{Gonzalez1}.  We
  have
  \[
  \Ad(e^X) Y = Y + \ad(X) Y + \sum_{j=2}^\infty \frac{\ad(X)^j Y}{j!}.
  \]
  Hence it suffices to check that $Z = \sum_{j=2}^\infty \ad(X)^{j-1}
  Y / j!$ is an element of~$\fn_{i+1}$.  Let $j \geq 2$.  Writing $j-1
  = (p-1)k + l$ with $k \geq 0$ and $0 \leq l \leq p-2$, we conclude
  that
  \[
  \ad(X)^{j-1} Y = \ad(X)^{(p-1)k} (\ad(X)^l Y) \in [\fr \,\,
  {}_{(p-1)k}, \fn_{i+l}] \subset p^k \fn_{i+k+l}.
  \]
  On the other hand, the $p$-valuation of $j!$ is at most $\lfloor
  (j-1)/(p-1) \rfloor = k$.  Thus $\ad(X)^{j-1} y / j! \in
  \fn_{i+k+l}$.  Since $k+l \geq 1$ tends to infinity with~$j$, the
  claim follows.
\end{proof}

\begin{lem}
  Let $\fr$ be a saturable $\Z_p$-Lie lattice with a PF-embedded Lie
  ideal~$\fn$.  Writing $N = \exp(\fn)$, we have $\log(e^X N) = X +
  \fn$ for every $X \in \fr$.
\end{lem}

\begin{proof}
  Let $\fn = \fn_1 \supseteq \fn_2 \supseteq \ldots$ be a potent
  filtration starting at~$\fn$, and fix $X \in \fr$.  We observe that
  $\log(e^X N) = \{ \Phi_\text{Hd}(X,Y) \mid Y \in \fn \} \subset X
  + \fn$.  To obtain the reverse inclusion, it suffices to show that
  $X + \fn \subset \{ \Phi_\text{Hd}(X,Y) \mid Y \in \fn \} + \fn_i$
  for all $i \in \N$.  We argue by induction.  Clearly, the claim is
  true for $i=1$, as $\Phi_\text{Hd}(X,0)=X$.  For the induction step,
  let $i \geq 2$ and consider an arbitrary element $Z \in \{
  \Phi_\text{Hd}(X,Y) \mid Y \in \fn \} + \fn_{i-1}$, that is $Z =
  \Phi_\text{Hd}(X,Y) + U$ with $Y \in \fn$ and $U \in \fn_{i-1}$.
  Then $Z = \Phi_\text{Hd}(\Phi_\text{Hd}(X,Y),U) + U' =
  \Phi_\text{Hd}(X,Y') + U'$ for suitable $U' \in [\fr,\fn_{i-1}]
  \subset \fn_i$ and $Y' = \Phi_\text{Hd}(Y,Z) \in \fn$.
\end{proof}

\begin{cor} \label{lem:grp.Lie.decomp} Let $\fr$ be a saturable
  $\Z_p$-Lie lattice with a PF-embedded Lie ideal $\fn$ and a
  saturable Lie sublattice~$\fh$.  Write $R = \exp(\fr)$, $N =
  \exp(\fn)$, and $H = \exp(\fh)$ for the corresponding saturable
  pro-$p$ groups.  Then $\fr = \fn + \fh$ if and only if $R = NH$.
\end{cor}

A pro-$p$ group $R$ is called \emph{potent} if~$\gamma_{p-1}(R)
\subset R^p$.  This notion is closely linked to saturability: if $R$
is finitely generated, torsion-free and potent then $R$ is saturable;
see~\cite[Corollary~5.4]{Gonzalez2}.  Conversely, a saturable pro-$p$
group need not be potent.

\subsection{Application to pro-$p$ subgroups of matrix
  groups} \label{subsec:appl.to.matrices}
Let $\lri$ be a compact discrete valuation ring of residue
characteristic~$p$.  Fix a uniformiser $\pi$ so that the valuation
ideal of $\lri$ takes the form $\fp = \pi \lri$, and let $\Lri$ be an
unramified quadratic extension of~$\lri$ with valuation ideal~$\fP =
\pi \Lri$.  For $p>2$ we write $\Lri = \lri[\delta]$, where $\delta =
\sqrt{\rho}$ for an element $\rho \in \lri$ whose reduction modulo
$\fp$ is a non-square in the residue field~$\lri/\fp$.

Let $n \in \N_{\geq 2}$.  In this section we consider the Sylow
pro-$p$ subgroups of $\GL_n(\lri)$ and $\SL_n(\lri)$, respectively
$\GU_n(\lri)$ and $\SU_n(\lri)$, as well as the corresponding
$\lri$-Lie lattices.  As before, we assume throughout that $p>2$ in
the unitary setting.  In particular, we are interested in the lower
central series of these groups and Lie lattices.

In order to arrive at a uniform description, it is convenient to work
with versions of the unitary groups and the unitary Lie lattices that are
different from those used in Section~\ref{sec:sim.gu}.  Let $W =
(w_{ij}) \in \GL_n(\Lri)$ denote the matrix corresponding to the
longest element in the Weyl group of permutation matrices, i.e.\ let
$w_{ij} = \delta_{i,n+1-j}$ using the Kronecker delta.  We equip the
$\Lri$-algebra $\mathsf{Mat}_n(\Lri)$ with the
$(\Lri,\lri)$-involution
\[
A^\star = W A^\circ W^{-1} \qquad \text{for $A \in
  \mathsf{Mat}_n(\Lri)$,}
\]
where $\circ$ denotes the standard $(\Lri,\lri)$-involution `conjugate
transpose' as in~\eqref{equ:def-std-inv}.  Then
\[
\GU_n^\star(\lri)=\left\{A \in \GL_n(\Lri) \mid A^\star A =\Id_n
\right\} \quad \text{and} \quad \gu_n^\star(\lri)=\left\{A \in
  \gl_n(\Lri) \mid A^\star + A =0 \right\}
\]
are isomorphic to $\GU_n(\lri)$ and $\gu_n(\lri)$; cf.\
Lemma~\ref{Dieudonne1}.  Similarly, we are interested in
\[
\SU_n^\star(\lri) = \GU_n^\star(\lri) \cap \SL_n(\Lri) \quad
\text{and} \quad \su_n^\star(\lri) = \gu_n^\star(\lri) \cap
\fsl_n(\Lri),
\]
which are isomorphic to $\SU_n(\lri)$ and $\su_n(\lri)$.

Observe that $A = (a_{ij}) \in \gl_n(\Lri)$ belongs to
$\gu_n^\star(\lri)$ if and only if its entries satisfy the conditions
$a_{ij} + a_{n+1-j,n+1-i}^\circ = 0$ for $1 \leq i,j \leq n$.  From
this, one constructs natural $\lri$-bases for the Lie lattices
$\gl_n(\lri)$ and~$\gu_n(\lri)$.  Denoting by $E_{ij}$ the elementary
$n \times n$ matrix with entry $1$ in the $(i,j)$-position and entries
$0$ elsewhere, we define
\begin{equation} \label{equ:def.e.epsilon.i.j} 
  \begin{split}
    E_{ij}^{(1)} & = E_{ij}, \\
    E_{ij}^{(-1)} & =
    \begin{cases}
      E_{ij} - E_{n+1-j,n+1-i} & \text{if $i+j < n+1$}, \\
      \delta E_{ij} & \text{if $i+j = n+1$}, \\
      \delta E_{ij} + \delta E_{n+1-j,n+1-i}& \text{if $i+j > n+1$}.
    \end{cases}
  \end{split}
\end{equation}
Using the parameter $\ee \in \{1, -1\}$ to distinguish between the
general linear and unitary settings, and defining
\begin{equation} \label{equ:G-H-epsilon}
  \begin{split}
    & \aG = \aG^{(\ee)} =
    \begin{cases}
      \GL_n& \\ \GU_n^\star&
    \end{cases}
    \text{and} \quad \ag = \ag^{(\ee)} =
    \begin{cases}
      \gl_n & \quad \text{for $\ee =1$,} \\
      \gu_n^\star & \quad \text{for $\ee =-1$,}
    \end{cases} \\
    & \aH = \aH^{(\ee)} =
    \begin{cases}
      \SL_n & \\ \SU_n^\star&
    \end{cases}
    \text{and} \quad \ah = \ah^{(\ee)} =
    \begin{cases}
      \fsl_n & \quad \text{for $\ee =1$,} \\
      \su_n^\star & \quad \text{for $\ee = -1$,}
    \end{cases}
  \end{split}
\end{equation} 
we see that the matrices $E_{ij}^{(\ee)}$, $1 \leq i,j \leq n$, form a
basis for the $\lri$-Lie lattice scheme~$\ag$.

For $m \in \Z$, the $\lri$-submodule schemes
\begin{align*}
  \ab_m = \ab_m^{(\ee)} & = \ah^{(\ee)} \cap \mathrm{span} \langle
  E^{(\ee)}_{i,j} \mid j-i  \ge m \rangle \\
  & = \{ (x_{ij}) \in \ah^{(\ee)} \mid x_{ij} = 0 \text{ for } j-i < m
  \}
\end{align*}
form a filtration 
\[
\{ 0 \} = \ldots = \ab_n \subset \ab_{n-1} \subset \ldots
\subset \ab_{1-n}= \ldots = \ah.
\]
Moreover, identifying $\Lri \otimes_\lri \gl_n(\lri)$ with
$\gl_n(\Lri)$ via the basis $E_{ij}$, $1 \leq i,j \leq n$, one checks
easily that, for each $m \in \Z$, the $\Lri$-submodule schemes $\Lri
\otimes_\lri \ab^{(1)}_m$ and $\Lri \otimes_\lri \ab^{(-1)}_m$ are
equal as subschemes of the $\Lri$-Lie lattice scheme $\fsl_n = \Lri
\otimes_\lri \fsl_n = \Lri \otimes_\lri \su_n^\star$.  Consequently,
the resulting $\lri$-submodule filtrations of $\ah^{(1)}(\lri)$ and
$\ah^{(-1)}(\lri)$ produce, under extension of scalars, the same
$\Lri$-submodule filtration $\ab_m^{(1)}(\Lri) = \ab_m^{(-1)}(\Lri)$,
$m \in \Z$, of $\fsl_n(\Lri) = \ah^{(1)}(\Lri) = \ah^{(-1)}(\Lri)$.
This means that properties that are `stable' under extension of
scalars can be derived uniformly for both cases, $\ee=1$ and $\ee=-1$.

Next we record an auxiliary lemma describing certain types of
commutators of terms of the filtration described above, based on
explicit matrix identities.

\begin{lem} \label{lem:filtration} With the above notation, the
  following hold:
  \begin{enumerate}
  \item $\ab_m \cdot \ab_{m'} \subset \ab_{m + m'}$, in particular
    $[\ab_m,\ab_{m'}] \subset \ab_{m+m'}$ for all $m,m' \in \Z$;
  \item if $p>2$ or $n \ge 3$ then $[\ab_1,\ab_m]=\ab_{m+1}$ for $m
    \in \Z$ with $m \ge 1-n$;
  \item if $p>2$ or $n \ge 3$ then $[\ab_0,\ah] = \ah$.
  \end{enumerate}
\end{lem}

\begin{proof} (1) Let $X=(x_{ij}) \in \ab_m$ and $Y = (y_{ij}) \in
  \ab_{m'}$ for $m, m' \in \Z$.  Suppose that the $(i,k)$-entry of
  $XY$ is non-zero, i.e.\ $\sum_{j=1}^{n} x_{ij} y_{jk} \ne 0$.  Then
  $x_{ij} \ne 0$ and $y_{jk} \ne 0$ for some index $j$, and
  consequently $k-i = (j-i)+(k-j) \geq m+m'$.

  (2) Let $m \in \Z$ with $m \geq 1-n$.  By (1), it remains to show
  that $\ab_{m+1} \subset [\ab_1,\ab_m]$.  This can be checked
  modulo~$\pi$, and since extension of scalars preserves the dimension
  of vector spaces, it is enough to show that $\ab_{m+1}(\Lri) \subset
  [\ab_1(\Lri),\ab_m(\Lri)]$.  Thus we may assume without loss of
  generality that~$\ee=1$; see the remark preceding the lemma.

  The $\lri$-lattice $\ab_{m+1}$ is spanned by elements of the form
  \begin{enumerate}
  \item[(i)] $E_{ij}$, where $1 \leq i,j \leq n$ with $i \ne j$ and
    $j-i \geq m+1$,
  \item[(ii)] $E_{ii} - E_{i+1,i+1}$, where $1 \le i < n$, if $m \le
    -1$.
  \end{enumerate}
  First consider $E_{ij}$ of type~(i).  If $j-i \geq 2$ or $j-i < 0$,
  we use the identities
  \[
  E_{ij} =
  \begin{cases}
    [E_{i,i+1},E_{i+1,j}] & \text{if $i<i+1<j$ or $j < i< i+1 \leq n$}, \\
    [E_{i,j-1},E_{j-1,j}] & \text{if $1 \le j-1<j<i$}
  \end{cases}
  \]
  to deduce that $E_{ij} \in [\ab_1,\ab_m]$; the case $(i,j) = (n,1)$
  does not arise, as $j-i \ge m+1 \ge 2-n$.  It remains to consider
  the case $j-i=1$, i.e.\ $j=i+1$. Then $m \le 0$ and we use the
  identities
  \begin{equation} \label{equ:elem-mat-identity}
  E_{ij} =
  \begin{cases}
    \hlf [E_{ii} - E_{jj},E_{ij}] &
    \text{if $p > 2$}, \\
    [E_{ii}-E_{kk},E_{ij}] & \text{if $k \not \in \{i,j\}$}
    \end{cases}
  \end{equation}
  for $j = i+1$ to deduce that $E_{i,i+1} \in [\ab_1,\ab_m]$.

  Finally, for $m \leq -1$ and $1 \leq i < n$ we see that
  $E_{ii} - E_{i+1,i+1} = [E_{i,i+1},E_{i+1,i}] \in [\ab_1,\ab_m]$.

  (3) Similar to part (2) we may assume without loss of generality
  that $\ee=1$ and, clearly, it suffices to show that $\ah \subset
  [\ab_0,\ah]$.  For $1 \leq i,j \leq n$ with $i \ne j$ we use the
  identities~\eqref{equ:elem-mat-identity} to see that $E_{ij} \in
  [\ab_0,\ah]$.  For $1 \leq i < n$ we have $E_{ii} - E_{i+1,i+1} =
  [E_{i,i+1},E_{i+1,i}] \in [\ab_0,\ah]$.
\end{proof}

We are interested in the $\lri$-Lie lattice scheme
\begin{equation}\label{equ:def.of.sylow}
\as = \as^{(\ee)}=\pi \ah + \ab_1^{(\ee)},
\end{equation}
which is defined so that, under suitable assumptions detailed below,
$\as(\lri)$ is a saturable Lie lattice yielding a Sylow pro-$p$
subgroup $\exp(\as(\lri))$ of $\aH(\lri)$.

\begin{lem} \label{p-sylow-Lie} Suppose that $p>2$ or $n \geq 3$.  The
  terms of the lower central series of the Lie lattice scheme $\as$
  are:
  \[ 
  \gamma_{i+jn}(\as) = \pi^j \gamma_{i}(\as) = \pi^{j+2}\ah +
  \pi^{j+1} \ab_{i-n} + \pi^j \ab_i,
  \]
  where $1 \leq i \leq n$ and $j \geq 0$.
\end{lem}

\begin{proof} 
  Note that $\ah = \ab_{1-n}$.  For $(i,j) = (1,1)$, the formula on
  the right-hand side equals
  \[
  \pi^{1+2} \ah + \pi^{1+1} \ab_{1-n} + \pi^1 \ab_1 = \pi (\pi \ah +
  \ab_1) = \pi \as.
  \]
  Hence it suffices to prove the formula up to the $(n+1)$th term of
  the lower central series.  Clearly, $\gamma_1(\as) = \pi^2 \ah + \pi
  \ab_{1-n} + \ab_1 = \as$ holds true.  Now suppose that $1 \leq i
  \leq n$.  By induction and using Lemma~\ref{lem:filtration}~(1), we
  have
  \[
  \begin{split}
    \gamma_{i+1}(\as) = [\as,\gamma_i(\as)] & = [\pi \ah +
    \ab_1,\pi^2 \ah + \pi \ab_{i-n} + \ab_i] \\
    & = \pi^3 [\ah,\ah] + \pi^2 [\ah,\ab_1 + \ab_{i-n}] + \pi
    ([\ah,\ab_i] + [\ab_1,\ab_{i-n}])
    + [\ab_1,\ab_i] \\
    & \subset \pi^2 \ah + \pi \ab_{i+1-n} + \ab_{i+1}.
  \end{split}
  \]
  Moreover, for $i = n$, we note that the last term indeed equals
  $\pi^2 \ah + \pi \ab_1 = \pi \as$, as $\ab_{n+1} = \{0\}$.

  The required reverse inclusions
  \[
  [\ah,\ab_1 + \ab_{i-n}] \supset \ah, \quad [\ah,\ab_i] +
  [\ab_1,\ab_{i-n}] \supset \ab_{i+1-n}, \quad [\ab_1,\ab_i] \supset
  \ab_{i+1}.
  \] 
  are obtained from Lemma~\ref{lem:filtration}~(2) and (3), upon noting
  that $\ab_0 \subset \ab_1 + \ab_{i-n}$ and $i-n,i \ge 1-n$.
\end{proof}

\begin{prop} \label{prop:p.gt.n.e+1} Let $\lri$ be a compact discrete
  valuation ring with $\cha(\lri) = 0$ and residue characteristic~$p$.
  Let $n \in \N_{\geq 2}$ and put $e = e(\lri,\Z_p)$.  Suppose that $p
  > en+n$. Let $\aG$ be $\GL_n$ or $\GU_n$, and accordingly let
  $\aH$ be $\SL_n$ or $\SU_n$.  Put $G = \aG(\lri)$ and
  $N=\aG^1(\lri)$, or $G = \aH(\lri)$ and $N=\aH^1(\lri)$.  Let $R$ be
  any pro-$p$ subgroup of $G$ containing $N$.

  Then $N$ and $R$ are potent and saturable, $\fn = \log(N)$ is
  PF-embedded in~$\fr = \log(R)$, and $\fn$ is naturally isomorphic to
  $\ag^1(\lri)$ or $\ah^1(\lri)$, respectively.
\end{prop}

\begin{proof} 
  In the unitary setting, it is convenient to work with $\aG =
  \GU_n^\star$ and $\aH = \SU_n^\star$.  Hence let $\aG, \ag$ and
  $\aH, \ah$ be as in~\eqref{equ:G-H-epsilon}, parametrised implicitly
  by $\ee \in \{1,-1\}$.  Without loss of generality we may assume
  that $R$ is contained in a Sylow pro-$p$ subgroup of our choice.
  Observe that
  \[
  \tilde{S} = \left\{ A \in \GL_n(\Lri) \mid A \text{ is upper
      uni-triangular modulo $\pi$} \right\}
  \]
  is a Sylow pro-$p$ subgroup of $\GL_n(\Lri)$ and that $S = \tilde{S}
  \cap G$ is a Sylow pro-$p$ subgroup of~$G$.  The inequality $p>
  en+1$ guarantees that $\tilde{S}$ is saturable; moreover,
  $\tilde{S}$ embeds naturally into the associative algebra
  $\mathsf{Mat}_n(\Lri')$, where $\Lri'$ is a finite extension
  of~$\Lri$, such that the Lie lattice $\tilde{\fs} = \log(\tilde{S})$
  can be identified with a Lie sublattice of $\gl_n(\Lri')$ and the
  $\exp$-$\log$ correspondence is achieved by applying the $p$-adic
  exponential and logarithm series to matrices over $\Lri'$;
  see~\cite[III (3.2.7)]{La65}.  In~\cite[Proposition~2.5]{Kl05}, the
  argument is extended to show that Sylow pro-$p$ subgroups of
  automorphism groups of $p$-adic vector spaces equipped with a
  bilinear form are saturable.  The proof given there covers, mutatis
  mutandis, also hermitian forms, and we conclude that, as $p > en+1$,
  in all cases considered here, the group $S$ is saturable.
  Furthermore the corresponding Lie lattice $\log(S)$ can be
  identified with the matrix Lie lattice
  \[
  \fs = \{ X \in \fg \mid X \text{ is strictly upper triangular modulo
    $\pi$} \},
  \]
  where $\fg = \ag(\lri)$ or $\fg = \ah(\lri)$.  Moreover applying the
  $p$-adic exponential and logarithm maps, defined by the series
  $\mathrm{Exp}(Z) = \sum_{j=0}^\infty Z^j/j!$ and $\mathrm{Log}(1+Z)
  = \sum_{j=1}^\infty (-1)^{j+1} Z^j/j$, one can translate between $S$
  and $\fs$ in place of the Hausdorff series construction.

  It is convenient to deal first with the case $G = \aH(\lri)$ and
  $\fg = \ah(\lri)$.  We observe that $\fs = \as(\lri)$, where $\as$
  is as in~\eqref{equ:def.of.sylow}.  We claim that the Lie sublattice
  $\fn \coloneqq \ah^1(\lri)$ is PF-embedded in~$\fs$.  More
  precisely, we claim that
  \[
  \fn_i \coloneqq \gamma_i(\fs) \cap \fn, \quad i \in \N,
  \]
  forms a potent filtration of $\fn$ in~$\fs$.  Indeed, from
  Lemma~\ref{p-sylow-Lie} we obtain $\gamma_n(\fs) \subset \fn$ and
  $\gamma_{n+en}(\fs) \subset p \gamma_n(\fs)$.  In particular, $\fn_i
  = \gamma_i(\fs)$ for $i \geq n$, and it suffices to observe that $p
  > en+n$ implies
  \[
  [\fs \,\, {}_{(p-1)},\fn] \subset [\fs \,\, {}_{(en)},[\fs \,\,
  {}_{(n-1)},\fn]] \subset [\fs \,\, {}_{(en)}, \gamma_n(\fs)] =
  \gamma_{n+en}(\fs) = p\gamma_n(\fs) \subset p\fn.
  \]

  Furthermore, $\mathrm{Exp}(\fn) \subset N$ and, comparing Haar
  measures, we deduce that $\mathrm{Exp}(\fn) = N$ so that $\fn$ is
  naturally identified with the Lie lattice $\log(N)$ corresponding to
  the saturable group~$N$.

  By~\cite[Theorem~B]{Gonzalez2}, the terms of the lower central
  series of $S$ and $\fs$ correspond to one another via the
  $\exp$-$\log$ correspondence.  From this we observe that
  \[
  \gamma_{n+en}(S) = \mathrm{Exp}(\gamma_{n+ne}(\fs)) = \mathrm{Exp}(p
  \gamma_n(\fs)) \subset \mathrm{Exp}(p \fn) = N^p.
  \]
  Since $p > en+n$, we deduce that for the given subgroup $R$ of
  $S$,
  \[
  \gamma_{p-1}(R) \subset \gamma_{n+ne}(S) \subset N^p \subset R^p.
  \]
  Thus $R$ is finitely generated, torsion-free, and potent, hence
  saturable.

  It remains to treat the case $G = \aG(\lri)$ and $\fg = \ag(\lri)$.
  From $p>n$ we see that $\fg = \ah(\lri) + \mathfrak{z}$, where
  $\mathfrak{z}$ denotes the centre of~$\fg$, and hence $\fs =
  \as(\lri) + \pi \mathfrak{z}$.  This implies that $\gamma_i(\fs) =
  \gamma_i(\as(\lri))$ for $i \geq 2$.  Using this observation, it is
  easy to extend the arguments provided for~$\aH(\lri)$ and
  $\ah(\lri)$ to conclude the proof.
\end{proof}

\begin{remark}
  For the purpose of extending characters from $N$ to~$G$,
  Proposition~\ref{prop:p.gt.n.e+1} is only applied to pro-$p$
  subgroups $R = R_\chi$ that arise as maximal normal pro-$p$
  subgroups of inertia subgroups $S_\chi = \textrm{I}_G(\chi)$, where
  $\chi \in \Irr(N)$.  A priori this specific situation requires
  control over much fewer groups $R$ and it is possible that with
  extra work the restrictions on~$p$ can be eased.
\end{remark}

\begin{remark.star} \label{rem:char-p-comments-1} Let $n, \len \in \N$
  with $n \geq 2$.  Suppose now that $\lri$ has arbitrary
  characteristic and residue characteristic~$p \geq n\len$.  As in the
  proposition, let $G = \aG(\lri)$ and $N=\aG^1(\lri)$, or $G =
  \aH(\lri)$ and $N=\aH^1(\lri)$.  In addition, we write $M =
  \aG^\len(\lri)$ or $M = \aH^\len(\lri)$.  Let $R$ be any pro-$p$
  subgroup of $G$ containing~$N$.  Then $R/M$ has nilpotency class at
  most $p-1$, as $p \geq n\len$; cf.~\cite[Appendix~A]{Kl05}.

  Consequently, we may argue similarly as
  in~\cite[Section~6.4]{Gonzalez1}. There is a surjective homomorphism
  $\eta$ from the free nilpotent pro-$p$ group $\wt{R}$ of class $p-1$
  on a certain number of generators onto~$R/M$.  We observe that both
  $\wt{R}$ and the pre-images $\wt{N} \coloneqq \eta^{-1}(N)$ and
  $\wt{M} \coloneqq \eta^{-1}(M)$ are potent and saturable.  Moreover,
  $\log(\wt{N})$ is PF-embedded in $\log(\wt{R})$.  Finally, the
  finite Lie ring $\log(\wt{N}) / \log(\wt{M})$ is naturally
  isomorphic to $\ag^1(\lri) / \ag^\len(\lri)$ or~$\ah^1(\lri) /
  \ah^\len(\lri)$.
\end{remark.star}

\subsection{The Kirillov orbit method}\label{subsec:kom}
The Kirillov orbit method, as described below, applies to potent
saturable pro-$p$ groups and yields a description of their irreducible
complex characters in terms of co-adjoint orbits; for details
see~\cite{Gonzalez1}.

Recall that the \emph{Pontryagin dual} $\mathfrak{a}^\vee$ of a
locally compact, abelian group~$\mathfrak{a}$ consists of all
continuous homomorphisms from $\mathfrak{a}$ to the circle group $\{z
\in \C \mid \lvert z \rvert = 1\}$.  Let $G$ be a saturable pro-$p$
subgroup and $\fg = \log(G)$ the corresponding $\Z_p$-Lie lattice.
The adjoint action of $G$ on $\fg$ induces an action of $G$ on the
Pontryagin dual $\fg^{\vee}$ of (the additive group)~$\fg$.  We call
this action the \emph{co-adjoint action} and denote it by~$\Ad^*$.  In
concrete terms, this action is given by
\[
(\Ad^*(g) \omega)(X) = \omega (\Ad(g^{-1}) X) = \omega (\log(g^{-1}
e^X g)) \quad \text{for $g \in G$, $\omega \in \fg^\vee$, and $X \in
  \fg$}.
\]

\begin{thm} \label{thm:kom} Let $G$ be a potent
  saturable pro-$p$ group and let $\fg = \log(G)$.  Then there is a
  one-to-one correspondence $\Ad^*(G) \backslash \fg^\vee \rightarrow
  \Irr(G)$, $\Omega \mapsto \chi_\Omega$ between $\Ad^*(G)$-orbits in
  $\fg^{\vee}$ and irreducible characters of~$G$.  Furthermore, the
  following hold.
  \begin{enumerate}
  \item \label{cl:Kirillov.1} For every $\Ad^*(G)$-orbit $\Omega$ the
    character $\chi_\Omega$ is given by
    \[
    \chi_\Omega(g) = \frac{1}{\lvert \Omega \rvert^{1/2}} \sum_{\omega
      \in \Omega} \omega(\log (g)) \qquad \text{for $g \in G$.}
    \]
    In particular, the degree of the character $\chi_\Omega$ is equal
    to~$\lvert \Omega \rvert^{1/2}$.
  \item \label{cl:Kirillov.2} Suppose that $H$ is a potent open
    subgroup of $G$ and let $\fh = \log(H)$.  Let $\Omega = \Ad^*(H)
    \omega$ and $\Theta = \Ad^*(G) \theta$ be co-adjoint orbits of
    $\omega \in \fh^{\vee}$ and $\theta \in \fg^{\vee}$.  Then
    $\chi_\Omega$ is a constituent of $\Res^G_H(\chi_\Theta)$ if and
    only if there exists $g \in G$ such that $\omega = (\Ad^*(g)
    \theta) \vert_\fh$.
  \end{enumerate}
\end{thm}

\begin{proof}
  The proof of the first half of the theorem is given
  in~\cite{Gonzalez1}.  We now justify part~\eqref{cl:Kirillov.2}.
  Observe that $H$ is saturable by~\cite[Corollary~5.4]{Gonzalez2}.
  The multiplicity of $\chi_\Omega$ in $\Res_H^G(\chi_\theta)$ is
  given by the inner product of the two characters.  Since the map
  $\exp \colon \fh \to H$ is a measure-preserving bijection, we deduce
  from part~\eqref{cl:Kirillov.1} that
  \begin{align*}
    \langle \chi_\Omega, \Res_H^G(\chi_\Theta) \rangle & = \int_H
    \chi_\Omega (h) \cdot \overline{\Res_H^G(\chi_\Theta)(h)} \, d\mu(h) \\
    & = \frac{1}{\lvert \Omega \rvert^{1/2} \lvert \Theta
      \rvert^{1/2}} \sum _{\omega' \in \Omega} \sum _{\theta' \in
      \Theta} \int_{\fh} \omega'(X) \cdot \overline{\theta'
      \vert_\fh(X)} \, d\mu(X).
  \end{align*}
  All the terms $\omega'$ and $\theta' \vert_\fh$ in the above sum
  represent $1$-dimensional characters of~$\fh$.  Hence by the
  orthogonality of characters we deduce that
  \[
  \int _\fh \omega'(X) \cdot \overline{\theta' \vert_\fh (X)} \,
  d\mu(X) =
  \begin{cases}
    1 & \textrm{if $\omega' = \theta' \vert_\fh$,} \\
    0 & \textrm{if $\omega' \neq \theta' \vert_\fh$.}
  \end{cases}
  \]
  The claim follows immediately from this.
\end{proof}

\begin{cor} \label{cor:kom} Let $G$ be a potent
  saturable pro-$p$ group and let $N$ be a potent open normal subgroup
  of~$G$.  Let $\fg = \log(G)$ and $\fn = \log(N)$.  Then the Kirillov
  orbit map induces a one-to-one correspondence $\Ad^*(G/N) \backslash
  (\fg/\fn)^\vee \rightarrow \Irr(G/N)$, $\Omega \mapsto \chi_\Omega$.
\end{cor}

\begin{proof}
  By Theorem~\ref{thm:kom} the irreducible
  characters of $G$ are of the form $\chi_\Omega$, where $\Omega$ runs
  through the $\Ad^*(G)$-orbits in $\fg^\vee$.  Irreducible characters
  of $G/N$ correspond to characters $\chi_\Omega$ with $\chi_\Omega(g)
  = \chi_\Omega(1) = \lvert \Omega \rvert^{1/2}$ for $g \in N$.  For
  $\omega \in \Omega$ this condition is equivalent to $\omega(\log(g))
  = 1$ for $g \in N$, i.e.\ $\omega(X) = 1$ for all $X \in \fn$.
\end{proof}

\begin{remark.star} \label{rem:char-p-comments-2} Continuing in the setup
  and with the notation of Remark~\ref{rem:char-p-comments-1}, we
  observe that the characters of the finite $p$-groups $R/M$ and $N/M$
  can be described in terms of the Kirillov orbit method applied to
  the potent saturable pro-$p$ groups $\wt{R}$, $\wt{N}$, and
  $\wt{M}$, using Corollary~\ref{cor:kom}.  This relies on the fact
  that the natural isomorphism between the finite group and Lie
  lattice sections is equivariant under the adjoint actions.
  Theorem~\ref{thm:kom} can be applied mutatis
  mutandis.
\end{remark.star}

\subsection{Principal congruence subgroups}\label{subsec:basic.princ}
Let $\lri$ be a compact discrete valuation ring of characteristic $0$,
with residue field $\kk$ of cardinality $q$, and put $p = \cha(\kk)$
and $e = e(\lri,\Z_p)$.  Let $\pi$ be a uniformiser of~$\lri$ and let $\Lri \supset \lri$ be an unramified
quadratic extension and fix $n \in \N$.  Let $\aG$ be one of the
$\lri$-group schemes $\GL_n, \GU_n$, assuming $p>2$ in the unitary
case, and let $\ag$ denote the corresponding $\lri$-Lie lattice scheme
$\gl_n, \gu_n$.  Write $\Sh$ for the shadow set~$\Sh_{\aG(\lri)}$; see
Definitions~\ref{def.of.shadows} and~\ref{def.of.unitary.shadows}.

Let $\len, m \in \N$ with~$\len \geq m$.  Let $G=\aG(\lri)$ and let
$G^m = \aG^m(\lri)$ denote its $m$th principal congruence subgroup.
We write $G^m_\len$ for the quotient~$G^m/G^\len$.  Put $\fg =
\ag(\lri)$, and let $\fg^m = \ag^m(\lri)$ denote the $m$th principal
congruence Lie sublattice.  Write $\fg_\len = \fg / \fg^\len$ and
$\fg^m_\len = \ag^m / \ag^\len$.  

To ensure that the group $G^m$ is potent and saturable, we assume that
\[
p>2 \quad \text{and} \quad m \ge e/(p-2);
\]
cf.\ \cite[Proposition~2.3]{AKOV1}.  Saturability gives a Lie
correspondence between $G^m$ and $\fg^m \simeq \log(G^m)$.
Let $\Lfi$ denote the fraction field of $\Lri$.  We fix a non-trivial
character of the additive group of $\Lfi$
\begin{equation} \label{equ:Tate-character} \phi \colon \Lfi
  \xrightarrow{\mathrm{Tr}_{\Lfi \,\vert\, \Q_p}} \Q_p \longrightarrow
  \Q_p/\Z_p \overset{\simeq}{\longrightarrow} \mu_{p^\infty} \subset
  \C^\times.
\end{equation}
One obtains an isomorphism
\[
\Lfi \rightarrow \Lfi^\vee, \quad x \mapsto \phi_x, \qquad
\text{where} \quad \phi_x \colon \Lfi \rightarrow \C^\times, \,
\phi_x(y) = \phi(\pi^{-\nu}xy).
\]
Here $\nu$ is the valuation of a generator of the different $\mathfrak{D}_{\mathfrak{F}|\Q_p}$, which is introduced in order to maintain the self-duality upon descent to finite quotients. For $A \in \gl_n(\Lri_\len)$ we
consider the character
\[
\omega_A \colon \gl_n(\Lri_\len) \to \C^\times, \quad \omega_A(X)=\phi
(\pi^{-\len}\tr(AX) )
\]
and its restriction to $\fg_\len$ which we also denote by~$\omega_A$
for simplicity.

\begin{lem}\label{lem:duality.finite}
  The map $\fg_\len \rightarrow \fg_\len^{\vee}$, $A \mapsto \omega_A$
  is an isomorphism of finite abelian groups that is $G$-equivariant
  with respect to the adjoint action $\Ad$ on $\fg_\len$ and the
  co-adjoint action $\Ad^*$ on $\fg_\len^{\, \vee}$.
\end{lem}

\begin{proof}
  The essential observation is that the trace induces a non-degenerate
  $\Lri_\len$-bilinear form
  \[
  \beta \colon \gl_n(\Lri_\len) \times \gl_n(\Lri_\len)\to\Lri_\len,
  \quad (A,B) \mapsto \tr(AB).
  \] 
  Indeed, for every $A\in\gl_n(\Lri_\len)$, the $\Lri_\len$-linear map
  $B \mapsto \tr(AB)$ is the zero map if and only if~$A=0$ as one
  sees, for instance, by evaluating it on elementary
  matrices. Restriction of the form $\beta$ to $\gl_n(\lri_\len)
  \times \gl_n(\lri_\len)$ establishes the isomorphism
  $\gl_n(\lri_\len) \simeq \gl_n(\lri_\len)^{\vee}$.

  Similarly, restriction of $\beta$ to $\gu_n(\lri_\len) \times
  \gu_n(\lri_\len)$ establishes the desired isomorphism in the unitary
  case.  For this one notes that the image of the restriction lies in
  $\lri_\len$ as $\tr(AB)^\circ=\tr(AB)$ for $A,B \in
  \gu_n(\lri_\len)$.  Furthermore, evaluation of the
  $\Lri_\len$-linear map $B \mapsto \tr(AB)$ on matrices of the form
  $E_{ij}^{(-1)}$ as described in~\eqref{equ:def.e.epsilon.i.j} --
  which form a generating set of $\gu_n(\lri_\len)$ as an
  $\lri_\len$-module -- shows that this map is the zero map if and
  only if~$A=0$. 

  In both cases, the $G$-equivariance is immediate.
\end{proof}

We also use the isomorphism
\begin{equation} \label{equ:pont.dual} \fg_{\len - m}
  \rightarrow (\fg^m_\len)^{\vee} , \quad A \mapsto \omega_A^m,
  \qquad \text{where} \quad \omega_A^m \colon \fg^m_\len \rightarrow
  \C^\times, \, \omega_A^m(X) = \omega_A(\pi^{-m}X).
\end{equation}

As $G^m$ is potent, the Kirillov orbit method sets up a correspondence
between $\Irr(G^m)$ and $\Ad^*(G^m) \backslash (\fg^m)^\vee$.  As
$(\fg^m)^\vee = \lim_{\len \to \infty} (\fg_\len^m)^\vee$, the
isomorphisms~\eqref{equ:pont.dual} lead us to consider the finite
orbit spaces $\Ad(G^m) \backslash \fg_{\len-m}$ for $\len > m$.  The
analysis of these spaces in the next section utilises the similarity
class zeta functions from Section~\ref{sec:sim.zeta}.

\begin{remark.star} \label{rem:char-p-comments-3} The setup above can
  be adapted to study the characters of the groups $G^m_\len$ in the
  case that $\cha(\lri)$ is arbitrary, provided $p \geq \len/m$.  The
  latter condition ensures that the nilpotency class of the finite
  $p$-group $G^m_\len$ is at most~$p-1$.  Write $N = G^m$ and $M =
  G^\len$.  Similar to Remark~\ref{rem:char-p-comments-1}, there is a
  homomorphism $\eta$ from a free nilpotent pro-$p$ group $\wt{N}$ of
  class $p-1$ on a certain number of generators onto~$N/M$.  Both
  $\wt{N}$ and $\wt{M} \coloneqq \eta^{-1}(M)$ are potent and
  saturable.  The finite Lie ring $\log(\wt{N})/\log(\wt{M})$ is
  naturally isomorphic to~$\fg^m_\len$.  If $\lri$ has positive
  characteristic~$p$, the non-trivial character
  \[
  \phi \colon \Lfi \xrightarrow{\mathrm{Tr}_{\Lfi \,\vert\,
      \F_p(\!(t)\!)}} \F_p(\!(t)\!) \xrightarrow{\mathrm{Res}_0} \F_p
  \overset{\simeq}{\longrightarrow} \mu_{p} \subset \C^\times,
  \]
  replaces the character described in~\eqref{equ:Tate-character}, with $\nu=0$ as the different in this case is trivial.
  Here the residue map $\mathrm{Res}_0$ picks out the coefficient
  of~$t^{-1}$.
\end{remark.star}

\subsection{From similarity class zeta functions to representation
  zeta functions} \label{subsec:sim.irrep} We continue to use the
notation set up in Section~\ref{subsec:basic.princ}.  In addition, we
write $\aH$ for the $\lri$-group scheme $\SL_n$ or~$\SU_n$, according
to whether $\aG$ is $\GL_n$ or $\GU_n$.  Put $H = \aH(\lri)$ and, for
$m \in \N$, let $H^m = \aH^m(\lri)$ denote its $m$th principal
congruence subgroup.  For $\len \in \N$ with~$\len \geq m$ write
$H^m_\len = H^m/H^\len$.  In
Section~\ref{subsec:sim.class.zeta.graphs} we introduced the
similarity class zeta functions $\gamma_\len^\sigma(s)$, $\sigma \in
\Sh$, and discussed the limit of their normalizations
$q^{-\len}\gamma^\sigma_\len(s)$ as $\len \rightarrow \infty$;
cf.\ Proposition~\ref{pro:limit.gamma}.  The following variants of
these functions play a central role in our derivation of formulae for
representation zeta functions.

\begin{defn}\label{def:xi}
 Let $\sigma \in \Sh$.  For $\len\in\N_0$ we set
 \[
 \xi_\len^\sigma(s) = [\aG(\kk) : \sigma(\kk)]^{1+s/2}q^{-\len}
 \gamma^\sigma_\len(s/2),
 \]
 and we define $\xi^\sigma(s) = \lim_{\len\rightarrow\infty}
 \xi^\sigma_\len(s) = [\aG(\kk) : \sigma(\kk)]^{1+s/2}
 \gamma^\sigma(s/2)$; see Proposition~\ref{pro:limit.gamma}.
\end{defn}

In Proposition~\ref{pro:zeta.H^m_l} below we provide formulae for the
zeta functions of groups of the form $G^m_\len$, $H^m_\len$, and $H^m$
in terms of the functions $\xi^\sigma_\len$ and $\xi^\sigma$.  For
this we require the following lemma.

\begin{lem} \label{lem:Omega.sum} Let $\len, m \in \N$ with $\len
  \geq m$.  Then
  \begin{equation*} \sum_{\Omega \in \Ad^*(G^m)
      \backslash ({\fg}^m_\len)^\vee} \lvert \Omega \rvert^{-s/2} =
    \begin{cases}
      q^{(\len-m) \dim \aG} & \text{if $\len \leq 2m$,} \\
      q^{(m-1)\dim \aG} \sum_{\Omega \in \Ad^*(G^1) \backslash
        ({\fg}^1_{\len-2m+2})^\vee} \lvert \Omega \rvert^{-s/2} &
      \text{if $\len > 2m$.}
    \end{cases}
  \end{equation*}
\end{lem}

\begin{proof}
  If $\len \leq 2m$, the co-adjoint action of $G^m$ on
  $(\fg^m_\len)^\vee$ is trivial and the sum over the singletons
  equals $\lvert \fg^m_\len \rvert = q^{(\len-m) \dim \aG}$, as
  claimed.  Now suppose that $\len > 2m$.  Using the $G$-equivariant
  isomorphism~\eqref{equ:pont.dual} between $\fg_{\len -m}$
  and $(\fg^m_\len)^\vee$, and similarly between $\fg_{\len -2m+1}$
  and $(\fg^1_{\len-2m+2})^\vee$, it suffices to work with
  $\Ad(G^m)$-orbits in $\fg_{\len -m}$ on the left hand side and
  $\Ad(G^1)$-orbits in $\fg_{\len-2m+1}$ on the right hand side.

  Suppose that $A \in \fg$, and for $k \in \N$ let $A_k \in \fg_k$ be
  the image of $A$ under the natural projection.  Consider the orbit
  $\mathcal{A} = \Ad(G^m) A_{\len-m} \in \Ad(G^m) \backslash
  \fg_{\len-m}$.  If $\aG = \GL_n$, then
  \begin{equation} \label{equ:silly.map} G^m \rightarrow \fg^m, \quad
    \Id_3 + \pi^m X \mapsto \pi^m X
  \end{equation}
  is a measure-preserving bijection, mapping the stabiliser
  $\Stab_{G^m}(A_k)$ onto the Lie centraliser $\Cen_{\fg^m}(A_k)$ for
  every $k \in \N$.  Furthermore, as $\len > 2m$, the map
  \[
  \fg^m \rightarrow \fg^1, \quad \pi^m X \mapsto \pi X
  \]
  is an isomorphism of abelian groups, mapping
  $\Cen_{\fg^m}(A_{\len-m})$ onto
  $\Cen_{\fg^1}(A_{\len-2m+1})$.  If $\aG = \GU_n$, we reach the
  same conclusions by using the Cayley map $\cay$ (cf.\
  Definition~\ref{def:cayley.map}) in place of \eqref{equ:silly.map};
  compare with the proof of Proposition~\ref{pro:class.quot.GU}.  Thus
  in each case we deduce that
  \[
  \lvert \mathcal{A} \rvert = \lvert G^m : \Stab_{G^m}(A_{\len-m})
  \rvert = \lvert \fg^m : \Cen_{\fg^m}(A_{\len-m}) \rvert =
  \lvert \fg^1 : \Cen_{\fg^1}(A_{\len-2m+1}) \rvert,
  \]
  and, using~\eqref{equ:silly.map}, we get
  \begin{align*}
    \sum_{\mathcal{A} \in \Ad(G^m) \backslash \fg_{\len-m}} \lvert
    \mathcal{A} \rvert^{-s/2} & = \sum_{A_{\len-m} \in \fg_{\len-m}}
    \lvert \fg^1 : \Cen_{\fg^1}(A_{\len-2m+1}) \rvert^{-1-s/2} \\
    & = q^{(m-1) \dim \fg} \sum_{A_{\len-2m+1} \in \fg_{\len-2m+1} }
    \lvert G^1 : \Stab_{G^1}(A_{\len-2m+1})
    \rvert^{-1-s/2} \\
    & = q^{(m-1) \dim \aG} \sum_{\mathcal{A} \in \Ad(G^1) \backslash
      \fg_{\len-2m+1}} \lvert \mathcal{A} \rvert^{-s/2}. \qedhere
  \end{align*}
\end{proof}

\begin{prop}\label{pro:zeta.H^m_l}
  Let $\len, m \in \N$ with $\len \geq m$, and suppose that $p \nmid
  2n$ and $m \geq e/(p-2)$.  Then
  \begin{equation}\label{equ:zeta.H^m_l}
    \zeta_{H^m_\len}(s) =    \frac{\zeta_{G_\len^m}(s)}{q^{\len-m}} =
    \begin{cases}
      q^{(\len-m) \dim \aH} & \text{if $\len \leq 2m$,} \\
      q^{(m-1) \dim \aH} \sum_{\sigma \in \Sh}
      \xi_{\len-2m+1}^\sigma(s) & \text{if $\len > 2m$.}
    \end{cases}
  \end{equation}
  Moreover,
  \begin{equation}\label{eqn:zeta.H^m}
    \zeta_{H^m}(s) = q^{(m-1) \dim \aH} \sum_{\sigma \in
      \Sh}\xi^\sigma(s).
  \end{equation}
\end{prop}

\begin{proof}
  Observe that $H^m_\len$ is the kernel of the determinant map on
  $G^m_\len$.  Since $p \nmid n$, every element in the $p$-group
  $\det(G^m_\len)$ admits an $n$th root.  Thus the central subgroup
  $\textup{S}(G^m_\len)$ of scalar matrices maps onto
  $\det(G^m_\len)$, and $G^m_\len$ decomposes as a direct product
  $G^m_\len = H^m_\len \times \textup{S}(G^m_\len)$. This yields the
  first equality in~\eqref{equ:zeta.H^m_l}.

  The assumption $m \geq e/(p-2)$ guarantees that the pro-$p$ group
  $G^m$ is potent and saturable (see \cite[Proposition~2.3]{AKOV1}) so
  that the orbit method can be used to parametrise irreducible
  characters of the finite quotient~$G^m_\len$; see
  Corollary~\ref{cor:kom}.  For $\len \leq 2m$ the
  second equality in~\eqref{equ:zeta.H^m_l} follows directly from the
  observation that $G^m_\len$ is abelian and $\dim \aG = \dim \aH +1$.
  For $\len > 2m$, we apply the orbit method,
  Lemma~\ref{lem:Omega.sum}, the
  correspondence~\eqref{equ:pont.dual}, and the fact
  \begin{equation}\label{equ:orbit.ratio}
  [ G : \Stab_G(A) ] = [\aG(\kk) : \sigma(\kk)] [ G^1 : \Stab_{G^1}(A)
  ] \quad \text{for $A \in \fg_{\len-2m+1}$ with $\sh_{\aG}(A) =
    \sigma$}
  \end{equation}
  to obtain
  \begin{align*}
    \zeta_{G^m_\len}(s) & = \sum_{\Omega \in \Ad^*(G^m) \backslash
      (\fg^m_\len)^\vee} \lvert \Omega \rvert^{-s/2} \\
    & = q^{(m-1)\dim \aG} \sum_{\Omega \in \Ad^*(G^1)
      \backslash(\fg^1_{\len-2m+2})^\vee} \lvert \Omega
    \rvert^{-s/2} \\
    & = q^{(m-1) \dim \aG} \sum_{\sigma\in\Sh} \,\, \sum_{
      \substack{\mathcal{A} \in \Ad(G^1) \backslash \fg_{\len -2m+1} \\
        \forall A \in \mathcal{A}: \; \sh_{\aG}(A)=\sigma}} \lvert
    \mathcal{A} \rvert^{-s/2}\\
    & = q^{(m-1)\dim \aG} \sum_{\sigma\in\Sh} \,\,
    \sum_{\substack{\mathcal{A} \in \Ad(G) \backslash \fg_{\len-2m+1}
        \\ \forall A \in \mathcal{A}: \; \sh_{\aG}(A) = \sigma}}
    \lvert \mathcal{A} \rvert^{-s/2} \, [\aG(\kk) : \sigma(\kk)]^{1+s/2} \\
    & = q^{\len-m} q^{(m-1)\dim \aH} \sum_{\sigma\in \Sh}
    \xi^\sigma_{\len-2m+1}(s).
  \end{align*}
  As for equation~\eqref{eqn:zeta.H^m}, every continuous character of
  $H^m$ factors through $H^m_\len$ for sufficiently large~$\len$.
  Therefore
  \[
  \zeta_{H^m}(s)=\lim_{\len \to \infty} \zeta_{H^m_\len}(s) = q^{(m-1)
    \dim \aH} \sum_{\sigma \in \Sh}\xi^\sigma(s).  \qedhere
  \]
\end{proof}

\begin{remark} \label{rem:also-for-l=2m} In fact, the second formula
  on the right-hand side of~\eqref{equ:zeta.H^m_l} also holds for
  $\len = 2m$, because $\sum_{\sigma \in \Sh} \xi_1^\sigma(s) =
  q^{\dim \aH}$.
\end{remark}

We record, as a byproduct of the proof of
Proposition~\ref{pro:zeta.H^m_l}, an explicit formula for the degrees
of irreducible characters of the principal congruence subgroup
quotients $G^m_\len$.

\begin{cor} \label{cor:dim.as.func.shadows} Let $\len,m \in \N$ with
  $\len \geq 2m$.  Suppose that $A \in \fg_{\len-m}$ corresponds, via
  the isomorphism~\eqref{equ:pont.dual}, to a character $\omega_A^m
  \in (\fg_\len^m)^\vee$.  Set $\cA = \Ad(G) A \in \Ad(G) \backslash
  \fg_{\len-m}$ and, for $1 \leq i \leq \len-2m+1$, denote by $\cA_i$
  the reduction of $\cA$ modulo~$\pi^i$.  Then the degree of the
  character $\chi_{\Ad^*(G^m) \omega_A^m} \in \Irr(G^m_\len)$
  associated to the co-adjoint orbit of $\omega_A^m$ is equal
  to
  \begin{equation*}
    \chi_{\Ad^*(G^m) \omega_A^m}(1) = q^{\frac{1}{2} \left( (\len - 2m) \dim\aG
        - \sum_{i=1}^{\len-2m} \dim(\sh_\aG(\cA_i) )\right)}.
  \end{equation*}
\end{cor}

\begin{proof}
  For $\len = 2m$ the degree is indeed~$1$.  Suppose that $\len > 2m$
  and set $\sigma = \sh_\aG(\cA_{\len-2m+1})$. Then
  \begin{align*}
    \chi_{\Ad^*(G^m) \omega_A^m}(1)^2 & = [ \aG(\kk):\sigma(\kk)]^{-1}
    \lvert \cA_{\len-2m+1} \rvert \\
    & = [\aG(\kk):\sigma(\kk)]^{-1} \prod_{i=1}^{\len-2m} q^{\dim\aG -
      \dim(\sh_\aG(\cA_i))} \frac{\| \sh_\aG(\cA_{i}) \|}{\|
      \sh_\aG(\cA_{i+1})\|} \\
    & = \prod_{i=1}^{\len-2m}q^{\dim\aG - \dim(\sh_\aG(\cA_i))} \\
    & = q^{(\len - 2m )\dim\aG -
      \sum_{i=1}^{\len-2m}\dim(\sh_\aG(\cA_i))}.  \qedhere
  \end{align*}
\end{proof}

\begin{remark.star} \label{rem:char-p-comments-4} As in the previous
  sections we comment on how to modify the results for $G^m_\len$ and
  $H^m_\len$ in case of arbitrary characteristic, provided that $p \geq
  \len/m$.  Definition~\ref{def:xi} does not change.
  Lemma~\ref{lem:Omega.sum} and the first equality
  in~\eqref{equ:zeta.H^m_l} hold as stated, also in positive
  characteristic.  To obtain the second equality
  in~\eqref{equ:zeta.H^m_l} and
  Corollary~\ref{cor:dim.as.func.shadows}, we use the Kirillov orbit
  method for the finite $p$-groups $G^m_\len$ and $H^m_\len$ as
  indicated in Remark~\ref{rem:char-p-comments-3}.  This relies on the
  condition $p \geq \len/m$.
 \end{remark.star}

 \subsection{Extensions of characters via the orbit
   method}\label{subsec:ext.kom}
 Let $G$ be a topological group, and let $N$ be a normal open subgroup
 of~$G$.  The group $G$ acts on $N$ by conjugation, and hence acts on
 the set $\Irr(N)$ of irreducible complex characters of~$N$:
 \[
 ^g \chi(h) = \chi(g^{-1} h g) \qquad \text{for $g \in G$, $\chi \in
   \Irr(N)$ and $h \in N$.}
 \]
 The \emph{inertia group} $\In_G(\chi)$ of $\chi \in \Irr(N)$ in $G$ is
 the stabiliser of $\chi$ in~$G$.

 Suppose further that $N$ is a potent saturable pro-$p$ group, and
 write and $\fn = \log(N)$.  Fix $\omega \in \fn^\vee$ with co-adjoint
 orbit $\Omega = \Ad^*(N) \omega$, and recall that $\chi_\Omega \in
 \Irr(N)$ denotes the character corresponding to $\Omega$ via the orbit
 method.

 \begin{lem}\label{lemma:R.and.G1} With the notation as above, the
   following hold:
   \begin{enumerate}
   \item [(a)] $\In_G(\chi_\Omega) = \Stab_G(\Omega) = N
     \Stab_G(\omega)$,
   \item [(b)] $\Ad^*(\In_G(\chi_\Omega)) \omega = \Omega$.
   \end{enumerate}
 \end{lem}

 \begin{proof}
   From Theorem~\ref{thm:kom} we see that
   $\In_G(\chi_\Omega) = \Stab_G(\Omega)$.  Clearly, one has
   $\Stab_G(\Omega) \supseteq N \Stab_G(\omega)$.  For the reverse
   inclusion, assume that $g \in \Stab_G(\Omega)$.  Then there exists
   $h \in N$ such that $\Ad^*(g) \omega = \Ad^*(h) \omega$ so that
   $h^{-1} g \in \Stab_G(\omega)$.  Thus $g \in N \Stab_G(\omega)$.
   This proves part~(a), and (b) is a direct consequence of~(a).
 \end{proof}

 In addition to the notation fixed already, suppose that $N \leq R \leq
 G$, where $R$ is a potent saturable pro-$p$ group, and set $\fr =
 \log(R)$.

 \begin{lem}\label{lemma:reps.2.char} With the notation as above, the
   character $\chi_\Omega$ extends from $N$ to~$R$ if and only if there
   exists $\theta \in \fr^\vee$, with co-adjoint orbit $\Theta =
   \Ad^*(R) \theta$, such that $\theta \vert_\fn = \omega$ and $\lvert
   \Theta \rvert = \lvert \Omega \rvert$; in this case $\chi_\Theta \in
   \Irr(R)$ is an extension of $\chi_\Omega$.
 \end{lem}

 \begin{proof}
   Let $\Theta \subset \fr^\vee$ be an $\Ad^*(R)$-orbit.  Then
   $\chi_\Theta$ extends $\chi_\Omega$ if and only if $\chi_\Omega$ is
   a constituent of $\Res^R_N(\chi_\Theta)$ and $\chi_\Theta(1) =
   \chi_\Omega(1)$.  By Theorem~\ref{thm:kom}, these
   conditions are equivalent to: there exists $\theta \in \Theta$ such
   that $\theta \vert_\fn = \omega$ and $\lvert \Theta \rvert = \lvert
   \Omega \rvert$.
 \end{proof}

 \begin{cor}\label{cor:criterion.for.extendability}
   With the notation as above, suppose that $R \subset
   \In_G(\chi_\Omega)$.  Then the character $\chi_\Omega$ extends
   to~$R$ if and only if $\omega$ extends to $\theta \in \fr^\vee$ such
   that $\lvert \Ad^*(R) \omega \rvert = \lvert \Theta \rvert$, where
   $\Theta = \Ad^*(R) \theta$.  In this case $\chi_\Theta \in \Irr(R)$
   is an extension of $\chi_\Omega$.
 \end{cor}

 \begin{proof}
   The claim follows directly from the previous two lemmata.
 \end{proof}

 Next we develop a criterion for applying
 Corollary~\ref{cor:criterion.for.extendability}.  Let $\fr$ be a
 $\Z_p$-Lie lattice, let $\fn$ be a Lie ideal of $\fr$ and let $\omega
 \in \fn^\vee$.  The \emph{radical} of $\omega$ in $\fr$ is
 \[
 \rad_\fr(\omega) \coloneqq \{ X \in \fr \mid \forall Y \in \fn :
 \omega([X,Y]) = 1 \};
 \]
 see~\cite{Gonzalez1}, but note the difference in notation.  We observe
 that $\rad_\fr(\omega)$ is equal to
 \[
 \stab_\fr (\omega) = \{ X \in \fr \mid \ad^*(X) \omega = 1 \},
 \]
 the \emph{stabiliser} of $\omega$ under the co-adjoint action given by
 \[
 (\ad^*(X) \omega)(Y) = \omega(-\ad(X)Y) \qquad \text{for $X \in \fr$,
   $\omega \in \fn^\vee$, and $Y \in \fn$.}
 \]

 \begin{lem} \label{lem:exp.stab.=.Stab} Let $\fr$ be a saturable
   $\Z_p$-Lie lattice, and let $R = \exp(\fr)$.  Let $\fn$ be an open
   Lie ideal that is PF-embedded in~$\fr$.  Let $\omega \in \fn^\vee$.
   Then $\stab_\fr(\omega)$ is saturable and
   \[
   \exp(\stab_\fr(\omega)) = \Stab_R(\omega).
   \]
 \end{lem}

 \begin{proof}
   Suppose that $X \in \fr$ such that $e^X \in \Stab_R(\omega)$, i.e.\
   such that
   \[
   \omega(\Ad(e^X) Y) = \omega(Y) \qquad \text{for all $Y \in \fn$.}
   \]
   Let $\fn = \fn_1 \supseteq \fn_2 \supseteq \ldots$ be a potent
   filtration for $\fn$.  We claim that $\omega(\ad(X) Y) \in
   \omega(\ad(X) \fn_j)$ for all $j \in \N$.  Since $\omega(\ad(X)
   \fn_j) = 1$ for sufficiently large $j$, this will imply that $X \in
   \stab_\fr(\omega)$.  Clearly, one has $\omega(\ad(X) Y) \in
   \omega(\ad(X) \fn_1)$.  Now let $j \geq 2$ and suppose inductively
   that $\omega(\ad(X) Y) = \omega(\ad(X) Y_{j-1})$ for some $Y_{j-1}
   \in \fn_{j-1}$.  Then by Lemma~\ref{lem:Hausdorff-expansion} we have
   \[
   \Ad(e^X) Y_{j-1} = Y_{j-1} + \ad(X) Y_{j-1} - \ad(X) Y_j,
   \]
   where $Y_j \in \fn_j$.  This yields
   \[
   \omega(\ad(X) Y) = \omega(\ad(X) Y_{j-1}) = \omega(\Ad(e^X) Y_{j-1}
   - Y_{j-1}) \omega(\ad(X) Y_j) = \omega(\ad(X) Y_j).
   \]

   Conversely, suppose that $X \in \stab_\fr(\omega)$, i.e.\ that
   $\omega(\ad(X) Y) = 1$ for all $Y \in \fn$.  Then by
   Lemma~\ref{lem:Hausdorff-expansion} we have
   \[
   \omega(\Ad(e^X) Y) = \omega(Y) \omega(\ad(X) (Y+Z)) = \omega(Y)
   \]
   for some $Z \in \fn$.  Hence $e^X \in \Stab_R(\omega)$.

   It follows that $\stab_\fr(\omega)$ is saturable and
   $\exp(\stab_\fr(\omega)) = \Stab_R(\omega)$.
 \end{proof}

 \begin{prop} \label{pro:character.extends} Let $\fr$ be a potent
   saturable $\Z_p$-Lie lattice with a potent open Lie ideal $\fn$ that
   is PF-embedded in~$\fr$.  Put $R = \exp(\fr)$, $N = \exp(\fn)$.  Let
   $\theta \in \fr^\vee$, $\Theta = \Ad^*(R) \theta$ and $\omega =
   \theta \vert_\fn \in \fn^\vee$, $\Omega = \Ad^*(N) \omega$.  Suppose
   that $\fr = \fn + \stab_\fr(\theta)$.  Then
   \begin{enumerate}
   \item $\stab_\fr(\theta) = \stab_\fr(\omega)$;
   \item $\Stab_R(\theta) = \Stab_R(\omega)$;
   \item $\chi_\Omega \in \Irr(N)$ extends to $\chi_\Theta \in
     \Irr(R)$.
   \end{enumerate}
 \end{prop}

 \begin{proof}
   (1) Clearly, $\stab_\fr(\theta) \subset \stab_\fr(\omega)$.  For the
   reverse inclusion, let $X \in \stab_\fr(\omega)$.  As $\fr = \fn +
   \stab_\fr(\theta)$ it suffices to show that $\theta(-\ad(X)Y)= 1$ for
   $Y \in \fn \cup \stab_\fr(\theta)$.  If $Y \in \fn$ then
   $\theta(-\ad(X)Y) = \omega(-\ad(X)Y) = 1$.  If $Y \in
   \stab_\fr(\theta)$ then $\theta(-\ad(X)Y) = \theta(\ad(Y)X) = 1$.

   (2) The claim follows immediately from (1) and
   Lemma~\ref{lem:exp.stab.=.Stab}.

   (3) From $\fr = \fn + \stab_\fr(\theta)$,
   Lemma~\ref{lem:exp.stab.=.Stab} and
   Corollary~\ref{lem:grp.Lie.decomp} we deduce that $R = N
   \Stab_R(\theta)$.  From (2) we deduce that $\In_R(\chi_\Omega) = R$,
   and (2) with Corollary~\ref{cor:criterion.for.extendability} shows
   that $\chi_\Theta$ extends $\chi_\Omega$.
 \end{proof}

 \subsection{Cohomological criteria for extendability}\label{subsec:ext.coh}
 Let $G$ be a group, $N$ a finite-index normal subgroup of $G$ and
 $\chi \in \Irr(N)$.  Define $\Irr(G \,\vert\, \chi)$ to be the set of
 irreducible characters $\psi$ of $G$ such that $\chi$ is an
 irreducible constituent of $\Res^G_N(\psi)$.  The \emph{relative
   representation zeta function} of $G$ with respect to $\chi$ is
 defined as
 \begin{equation}\label{def:relative.zeta.function}
   \zeta_{G \,\vert\, \chi}(s) = \sum_{\psi \in \Irr(G \,\vert\, \chi)}
   \left( \psi(1) / \chi(1) \right)^{-s}.
 \end{equation}
 In the notation introduced in~\cite[Section~7.2.1]{AKOV1}, we have
 $\zeta_{G \,\vert\, \chi}(s) = [ G : \In_G(\chi) ]^{-s}
 \zeta_{G,\chi}(s)$.  Clifford theory yields the following proposition;
 see~\cite[Corollary~6.17]{Is}.

 \begin{prop}
   Let $G$ be a profinite group, $N$ an open normal subgroup of~$G$,
   and $\chi \in \Irr(N)$.  Suppose that $\chi$ extends to a character
   of $\In_G(\chi)$.  Then
   \[
   \zeta_{G \,\vert\, \chi}(s) = [G : \In_G(\chi)]^{-s} \,
   \zeta_{\In_G(\chi)/N}(s).
   \]
 \end{prop}

 \begin{cor} \label{cor:gen.formula} Using the same notation as in
   the proposition and supposing that every $\chi \in \Irr(N)$ extends
   to a character of $\In_G(\chi)$, we have
   \begin{equation}\label{equ:gen.sum}
     \zeta_G(s) = \sum_{\chi \in \Irr(N)} \lvert G : \In_G(\chi)
     \rvert^{-1-s} \, \zeta_{\In_G(\chi)/N}(s) \, \chi(1)^{-s}.
   \end{equation}
 \end{cor}

 Whether a character extends can be studied in the framework of the
 second cohomology group of the inertia quotient.  Let $S$ be a group
 with a finite-index normal subgroup $R \triangleleft S$, and let
 $\chi \in \Irr(R)$.  Clearly, a necessary condition for the
 extendability of $\chi$ to $S$ is that $S$ fixes the character~$\chi$,
 i.e.\ that $\In_S(\chi) = S$.  Assuming this, one constructs an
 element in the second cohomology group $\textup{H}^2(S/R,\C^\times)$,
 also known as the Schur multiplier, as follows;
 see~\cite[Chapter~11]{Is}.

 Let $M$ be a left $R$-module affording the character~$\chi$.  Choose a
 left transversal $T$ for $R$ inside $S$ such that $1 \in T$.  For
 every $t \in T$, the $R$-modules $M$ and $tM$ are isomorphic, because
 $t \in \In_S(\chi)$.  For each $t \in T$, we choose an isomorphism
 $P_t \colon M \rightarrow tM$, selecting the identity $P_1 = \Id$
 for~$t=1$.  Every element of $S$ can be written uniquely as $th$,
 where $t \in T$ and $h \in R$.  We define $P_{th} \colon M \rightarrow
 tM$ by $P_{th}(m) = P_t(h \cdot m)$.  It can be easily checked that
 for every pair $g_1,g_2 \in S$, the operator
 \[
 P_{g_1 g_2}^{-1} \circ P_{g_1} \circ P_{g_2} \colon M \rightarrow M
 \]
 is a non-zero endomorphism of the $R$-module~$M$; because $M$ is
 irreducible, this morphism is multiplication by a scalar
 $\alpha(g_1,g_2) \in \C^\times$, say.  Note that the value of
 $\alpha(g_1,g_2)$ depends only on the cosets $g_1R$ and $g_2R$.  The
 function $\alpha$ is a $2$-cocycle and, although it generally depends
 on the particular choices for $T$ and $P_t$, the cohomology class
 $\beta \in \textup{H}^2(S/R,\C^\times)$ that it represents does not.
 By \cite[Theorem~11.7]{Is}, the character $\chi$ extends to $S$ if and
 only if $\beta$ is trivial.

 In our setting $S$ is a profinite group and $R$ an open normal pro-$p$
 subgroup of~$S$.  In this case all finite-dimensional representations
 of $R$ factor through finite $p$-groups.

 \begin{prop} \label{pro:extension.p.group} Let $S$ be a profinite
   group, $R \triangleleft S$ be an open normal pro-$p$ subgroup of
   $S$ and $\chi \in \Irr(R)$ such that $\In_S(\chi) = S$.  If $\beta
   \in \textup{H}^2(S/R,\C^\times)$ is the cohomology class attached to
   $(S,R,\chi)$, then $\beta$ is a $p$-element in
   $\textup{H}^2(S/R,\C^\times)$.
 \end{prop}

 \begin{proof}
   We continue to use the notation set up above.  Let $\rho$ denote a
   representation associated with the $R$-module~$M$.  Fix volume forms
   on the $R$-modules $tM$, $t \in T$.  Let $t_1, t_2 \in T$ and
   suppose that $h \in R$ is such that $t_1 t_2 h \in T$.  Taking
   determinants, we get
   \[
   \alpha(t_1,t_2)^{\dim(M)} = \det(P_{t_1t_2h})^{-1} \det(P_{t_1})
   \det(P_{t_2}) \det(\rho(h)).
   \]
   We may assume that the isomorphisms $P_t$, $t \in T$, have
   determinant~$1$.  Since $\rho$ is an irreducible representation of
   the pro-$p$ group~$R$, its dimension $\dim(M)$ is a power of $p$,
   and $\det \rho(h)$ is a $p^n$th root of unity for some~$n \in \N$.
   Therefore $\alpha(t_1,t_2)$ is a $p^m$-root of unity for some~$m \in
   \N$.  It follows that the order of $\beta$ is a power of~$p$.
 \end{proof}

 \subsection{Extension of characters from $N$ to their stabiliser in $G$}
 From now on consider again $G = \aG(\lri)$, where $\aG$ denotes one
 of the $\lri$-group schemes $\GL_n, \GU_n, \SL_n, \SU_n$, with $1$st
 principal congruence subgroup~$N = G^1 = \aG^1(\lri)$ and $\cha(\lri)
 = 0$, as discussed at the beginning of
 Section~\ref{sec:kom.clifford}.  Let $\chi \in \Irr(N)$ with
 stabiliser $S_\chi = \textrm{I}_G(\chi)$, and let $R_\chi$ denote a
 maximal normal pro-$p$ subgroup of~$S_\chi$.  Assume that $p > en+n$,
 and let $\fn = \log(N)$, $\fr = \log(R)$ denote the $\Z_p$-Lie
 lattice associated to~$N$, $R$; see
 Proposition~\ref{prop:p.gt.n.e+1}.  Writing $\ag$ for the $\lri$-Lie
 lattice scheme $\gl_n, \gu_n, \fsl_n$, or $\su_n$, according to our
 choice of~$\aG$, we have $\fn = \ag^1(\lri)$; furthermore we put $\fg
 = \ag(\lri)$.  For $\len,m \in \N$ with $m \leq \len$, we write
 $\fg_\len = \fg / \fg^\len$ and $\fg^m_\len = \fg^m / \fg^\len$.  We
 also put $\fn_\len = \fg^1_{\len +1}$.  Similarly to the situation in
 Section~\ref{subsec:basic.princ}, the character $\phi$
 in~\eqref{equ:Tate-character} induces an isomorphism
 \begin{equation}\label{Pont.duality.for.h0}
   \fg_\len \rightarrow \fg_\len^{\vee}, \quad A \mapsto \omega_A,
   \quad \text{where} \quad \omega_A \colon \fg_\len \rightarrow \C^\times, \,
   \omega_A(X) = \phi \left(\pi^{-\len}\tr(AX)\right),
 \end{equation}
 of finite abelian groups which is $G$-equivariant with respect to the
 adjoint action $\Ad$ on $\fg_\len$ and the co-adjoint action $\Ad^*$
 on $\fg_\len^\vee$. For the general linear and unitary cases this is
 Lemma~\ref{lem:duality.finite}. The special linear/unitary cases
 follow along the same lines noting that $p$ does not divide $n$.

 In
 addition there is, similar to \eqref{equ:pont.dual}, an
 isomorphism
 \begin{equation*}
   \fg_\len \rightarrow (\fn_\len)^\vee, \quad A \mapsto
   \omega_A^1, \quad \text{where} \quad \omega_A^1 \colon
   (\fn_\len)^\vee \rightarrow \C^\times, \, \omega_A^1(X) = \omega_A(\pi^{-1}X).
 \end{equation*}

 \begin{defn}\label{def:shad.pres.lift} Let $G$ and $\fg$ be as above.
   Let $A\in\fg_\len$ and $\wt{A} \in \fg_{\len+1}$.  We say
   that $\wt{A}$ is a \emph{shadow-preserving lift} of $A$ if $A
   \equiv_{\pi^{\len}} \wt{A}$ and $\overline{\Cen_G(A)} =
   \overline{\Cen_{G}(\wt{A})}$.
 \end{defn}

 \begin{thm} \label{thm:extension.to.radical} Let $G$, $N$ and $\fg$
   be as above.  Suppose that $p > en+n$. Let $\chi \in \Irr(N)$,
   corresponding to the co-adjoint orbit $\Ad^*(N) \omega_A^1$ for $A
   \in \fg_\len$, and put $S_\chi = \In_G(\chi)$ with maximal normal
   pro-$p$ subgroup~$R_\chi$.  If
   \begin{itemize} \renewcommand{\labelitemi}{$\circ$}
   \item there exists a shadow-preserving lift
     $\wt{A}\in\fg_{\len+1}$ of $A$ and
   \item $\textup{H}^2(S_\chi / R_\chi, \C^\times) = 1$,
   \end{itemize}
   then $\chi$ extends to $S_\chi$.
 \end{thm}

 \begin{proof}
   As $\wt{A}$ is a shadow-preserving lift of $A$, we have that
   \[
   S_\chi / N = \Stab_G(A) N / N = \Stab_G(\wt{A}) N / N.
   \]
   As $N \leq R_\chi \leq S_\chi$, this implies that
   \[
   R_\chi = \Stab_{R_\chi}(A) \, N = \Stab_{R_\chi}(\wt{A}) \, N,
   \]
   and thus
   \[
   R_\chi = \Stab_{R_\chi}(\omega_A^1) \, N = \Stab_{R_\chi}
   (\omega_{\wt{A}}) \, N.
   \]
   Setting $\theta = \omega_{\wt{A}}\vert_{\fr}$, we apply
   Lemma~\ref{lem:grp.Lie.decomp} to deduce that $\fr = \fn +
   \stab_{\fr}(\theta)$.  We conclude from
   Proposition~\ref{pro:character.extends} that the character
   $\chi_\Theta \in \Irr(R)$ associated to $\Theta = \Ad^*(R) \theta$
   extends $\chi$.

   To show that $\chi_\Theta$ extends to $S_\chi$, we first show that
   $\In_{S_\chi}(\chi_\Theta) = \In_{S_\chi}(\chi)$.  The inclusion
   $\In_{S_\chi}(\chi_\Theta) \subset \In_{S_\chi}(\chi)$ is clear
   and the reverse inclusion follows from
   \begin{equation*}
     \In_{S_\chi} (\chi) = \Stab_{S_\chi}(\omega_A^1) N =
     \Stab_{S_\chi}(A) N = \Stab_{S_\chi}({\wt{A}}) N \subset
     \Stab_{S_\chi}(\theta) N \subset
     \In_{S_\chi}(\chi_\Theta).
   \end{equation*}
   Since $\textup{H}^2(S_\chi/R_\chi,\C^\times) = 1$, the character
   $\chi_\Theta$ extends to an irreducible character of~$S_\chi$; cf.\
   Section~\ref{subsec:ext.coh}.
 \end{proof}

 \begin{remark.star} \label{rem:char-p-comments-5} We indicate with
   what changes the results, in particular
   Theorem~\ref{thm:extension.to.radical}, remain true in case of
   arbitrary characteristic, provided that $p$ is sufficiently large.
   Definition~\ref{def:shad.pres.lift} remains the same in positive
   characteristic.  Using all the previous remarks in
   Section~\ref{sec:kom.clifford}, one sees that the conclusion of
   Theorem~\ref{thm:extension.to.radical}, for $\chi \in \Irr(N)$,
   holds true if the level $\len-1$ of $\chi$ satisfies $p \geq
   n\len$.  Here, the level $\len-1$ of $\chi$ is the minimal value of
   $\len-1$ for $\len \in \N$ such that $\chi$ is trivial on the
   $\len$-th principal congruence subgroup;
   cf.\ Definition~\ref{def:level}.
 \end{remark.star}


\section{Zeta functions of groups of type
  $\mathsf{A}_2$}\label{sec:zeta.A2}

In this section we apply some of the machinery developed in previous
sections in the special case of groups of type $\mathsf{A}_2$. In
particular, we prove Theorem~\ref{thm:J} in
Section~\ref{subsec:principal}, Theorem~\ref{thm:C} in
Section~\ref{subsec:van.coh.gps}, Corollary~\ref{cor:D} in
Section~\ref{subsec:zeta.shadows}, and Theorems~\ref{thm:H} and \ref{thm:I} in
Section~\ref{subsec:ennola.duality}.

We use the notation introduced in Section~\ref{sec:kom.clifford},
focusing now on the special case $n=3$.  In summary, $\lri$ denotes a
compact discrete valuation ring with valuation ideal $\fp$ and finite
residue field~$\kk$, where $p \coloneqq \cha(\kk)$ and $q \coloneqq
\lvert \kk \rvert$.  The letter $\aG$, accordingly $\aH$, stands for
one of the $\lri$-group schemes $\GL_3, \GU_3$, accordingly $\SL_3,
\SU_3$, assuming $p>2$ in the unitary cases.  Write $G = \aG(\lri)$,
$H = \aH(\lri)$, and $G^m, H^m, G_\len^m, H_\len^m$ for the principal
congruence subgroups and finite subquotients.  The corresponding
$\lri$-Lie lattice schemes, respectively Lie lattices, are denoted by
$\ag, \ah$, $\fg = \ag(\lri), \fh = \ah(\lri)$, $\fg_\len^m,
\fh_\len^m$ et cetera.  The parameter $\ee = \ee_\aG = \ee_\aH \in
\{1,-1\}$ facilitates the parallel treatment of the (general) linear
and (general) unitary setting; compare~\eqref{equ:def.epsilon}.  For
$\len \in \N_0$, we use the abbreviated notation $\cQ_\len \coloneqq
\cQ_{\lri,\len}^\ag$ (cf.\ Definitions~\eqref{def:graph.Qgl}
and~\eqref{def:graph.Qgu}) and $\Sh \coloneqq \Sh_{\aG(\lri)}$ (cf.\
\eqref{equ:Sh.GL} and \eqref{equ:Sh.GU}).

\subsection{Principal congruence subgroups} \label{subsec:principal}
Proposition~\ref{pro:zeta.H^m_l} provides general formulae for the
representation zeta functions of (i)~the finite groups $G^m_\len,
H^m_\len$ in terms of the functions $\xi^\sigma_\len$ and (ii)~the
infinite groups $H^m$ in terms of the limit functions $\xi^\sigma =
\lim_{\len\rightarrow \infty}\xi^\sigma_\len$;
cf.\ Definition~\ref{def:xi}.  The definition of the functions
$\xi^\sigma_\len$ has two ingredients: firstly the indices
$[\aG(\kk):\sigma(\kk)]$ and secondly the partial similarity class
zeta function $\gamma_\len^\sigma(s)$;
cf.\ Definition~\ref{def:sim.class.zeta}.  The former are tabulated in
Table~\ref{tab:shadows.GLGU} while explicit formulae for the latter
are provided in Proposition~\ref{pro:sim.shadow.fin}.  We record
expressions for the functions~$\xi^\sigma_\len$, using the auxiliary
functions $\UU, \VV, \WW$ introduced in~\eqref{equ:aux.ABC}.

\begin{prop}\label{pro:xi.sigma.len}
  For $\sigma \in \Sh$ of type $\cS\in\T^{(\ee)}$ and $\len \in
  \N_0$,
  \[
  \xi^\sigma_\len(s) = \Xi^{\cS}_{\ee,q,\len}(s),
  \]
  where the function $\Xi^{\cS}_{\ee,q,\len}(s) \coloneqq
  \Xi^{\cS}_{\mathsf{A}_2,\ee,q,\len}(s)$ is defined as
 \begin{equation*}
    \begin{array}{ll}
      1 & \text{if $\cS = \Gsha$,} \\
      (q-1)(q^2 + \ee q + 1) q^2 A_{q,\len}(s/2) & \text{if
        $\cS = \Lsha$,} \\
      (q^3-\ee)(q+\ee) A_{q,\len}(s/2) & \text{if $\cS = \Jsha$,} \\
      \sixth (q-1)(q^2 + \ee q + 1)(q + \ee)q^3
      \left[ (q-2)B_{q,\len}(s/2) +3(q-1)q^{1-2s}C_{q,\len}(s/2)\right] &
      \text{if $\cS = \Tasha$,} \\
      \hlf (q-1)(q^3-\ee)q^3 \left[ q B_{q,\len}(s/2) + (q-1) q^{1-2s}
        C_{q,\len}(s/2) \right] & \text{if $\cS = \Tbsha$,} \\
      \third (q^2-1)(q+\ee)(q-\ee)^2q^3 B_{q,\len}(s/2) & \text{if
        $\cS = \Tcsha$,} \\
      (q-1)(q^3-\ee)(q+\ee)q^2 \left[ B_{q,\len}(s/2) + 2q^{1-2s}
        C_{q,\len}(s/2) \right] & \text{if $\cS = \Msha$,} \\
      (q^3-\ee)(q^2-1) \left[ qB_{q,\len}(s/2) + (q-1) q^{1-2s}
        C_{q,\len}(s/2) \right] & \text{if $\cS = \Nsha$,}\\
      (q^3-\ee)(q^2-1) q^{1-2s}C_{q,\len}(s/2) &
      \hspace*{-1.2cm}\text{if $\cS \in \{
        \Kasha, \Kbsha \}$.}
    \end{array}
  \end{equation*}
\end{prop}

\begin{cor} \label{cor:infinite.xi} For $\sigma \in \Sh$ of type
  $\cS\in\T^{(\ee)}$,
  \[
  \xi^\sigma(s) = \Xi^{\cS}_{\ee,q}(s),
  \]
  where the function $\Xi^{\cS}_{\ee,q}(s) \coloneqq
  \Xi^{\cS}_{\mathsf{A}_2,\ee,q}(s)$ is given by
  \begin{equation*}
    \begin{array}{ll}
      1 & \text{if $\cS = \Gsha$,} \\ 
      (q-1) (q^2 + \ee q + 1) q^2 \; (1 - q^{1-2s})^{-1} & \text{if
        $\cS = \Lsha$,} \\ 
      (q^3 - \ee)(q + \ee) \; (1 - q^{1-2s})^{-1} & \text{if $\cS =
        \Jsha$,} \\ 
      \tfrac{1}{6} (q-1) (q^2 + \ee q + 1) (q + \ee)
      q^3  \; (q-2+2q^{2-2s} - q^{1-2s}) \left((1 - q^{1-2s}) (1 -
        q^{2-3s})\right)^{-1} & \text{if $\cS = \Tasha$,} \\ 
      \hlf (q-1) (q^3 - \ee) (1 - q^{-2s}) q^4 \; \left((1 - q^{1-2s})(1 -
        q^{2-3s})\right)^{-1} & \text{if $\cS = \Tbsha$,} \\ 
      \third (q^2 -1) (q + \ee) (q-\ee)^2 q^3 \; (1 - q^{2-3s})^{-1} &
      \text{if $\cS = \Tcsha$,} \\ 
      (q^3 - \ee) (q-1) (q + \ee) q^2 \; (1 + q^{1-2s}) \left( (1 -
        q^{1-2s}) (1 - q^{2-3s}) \right)^{-1} & \text{if $\cS =
        \Msha$,} \\ 
      (q^3 - \ee) (q^3 - q) \; (1 - q^{-2s}) \left((1 - q^{1-2s})(1 -
        q^{2-3s}) \right)^{-1} & \text{if $\cS = \Nsha$,} \\ 
      (q^3 - 1) (q^2 - 1) q^{1-2s} \; \left((1 - q^{1-2s}) (1 -
        q^{2-3s}) \right)^{-1} & \hspace*{-1.2cm}\text{if $\cS \in \{
        \Kasha, \Kbsha \}$.} 
    \end{array}
  \end{equation*}
\end{cor}

\begin{proof}
  We may regard $\xi^\sigma(s)$ as a limit of $\xi^\sigma_\len(s)$ as
  $\len \to \infty$. Taking the formal limit
  \[
  \Xi^{\cS}_{\ee,q}(s) = \lim_{\len \to
    \infty}\Xi^{\cS}_{\ee,q,\len}(s)
  \]
  amounts to employing the substitutions
  \[
  A_{q,\len}(s/2) \to \frac{1}{1-q^{1-2s}}, \quad B_{q,\len}(s/2) \to
  \frac{1}{1-q^{2-3s}}, \quad C_{q,\len}(s/2) \to
  \frac{1}{(1-q^{1-2s}) (1-q^{2-3s})}
  \]
  to the formulae provided by the proposition.
\end{proof}

\begin{proof}[Proof of Theorem~\ref{thm:J}]
  The explicit formulae for the representation zeta functions of the
  finite groups $G^m_\len,H^m_\len$ as well as the infinite groups
  $H^m$ follow directly from Proposition~\ref{pro:zeta.H^m_l},
  Remark~\ref{rem:char-p-comments-4},
  Proposition~\ref{pro:xi.sigma.len}, and
  Corollary~\ref{cor:infinite.xi}.
\end{proof}

We now give an alternative proof of
Corollary~\ref{cor:infinite.xi}. It is based on
Lemma~\ref{recurrence.gamma} and finite recursion equations, bypassing
the computation of the functions~$\xi^\sigma_\len$. Recall the
Definition~\ref{def:shadow.graph.A2} of the shadow graph $\Gamma =
\Gamma^{(\ee)}$.

\begin{prop}\label{pro:recurrence.zeta}
  Let $\tau\in\Sh$ of type $\cS\in\T^{(\ee)}$.  Then $\xi^{\tau}(s)
  \equiv 1$ if $\cS=\Gsha$. Otherwise,
  \begin{equation*}
    \xi^{\tau}(s) = \sum_{\substack{(\sigma,\tau)\in \dot{E}(\Gamma)
        \\ \sigma \ne \tau}} \frac{ q^{-1} \, [\sigma(\kk):\tau(\kk)]
      \, a_{\sigma,\tau}(q) \, q^{-\frac{1}{2}(\dim \aG
        -\dim(\sigma))s} \, \xi^\sigma(s)}{1 - q^{-1} \,
      a_{\tau,\tau}(q) \, b^{(\ee)}_{\tau,\tau}(q)^{-s/2}}.
  \end{equation*}
\end{prop}

\begin{proof}
  Note that the only edge in the shadow graph $\Gamma =
  \Gamma^{(\ee)}$ with target $\Gsha$ is~$(\Gsha,\Gsha)$.  Applying
  Lemma~\ref{recurrence.gamma} to the special case $\tau$ of type
  $\Gsha$ and consulting Table~\ref{tab:branch.rules.A2}, we use
  induction to find that $\gamma^\tau_\len(s) = q^\len$, and hence
  $\xi^\tau(s) \equiv 1$ for $\tau$ of type~$\Gsha$.

  Now let $\tau$ be a shadow of type different from $\Gsha$.
  Lemma~\ref{recurrence.gamma} shows that for $\len \in \N_0$ we have
  \[
  \frac{\gamma_{\len+1}^\tau(s/2)}{q^{\len+1}} = q^{-1} \,
  a_{\tau,\tau}(q) \, b^{(\ee)}_{\tau,\tau}(q)^{-s/2} \,
  \frac{\gamma^\tau_\len(s/2)}{q^\len} +
  \sum_{\substack{(\sigma,\tau)\in \dot{E}(\Gamma) \\ \sigma \ne
      \tau}} q^{-1} \, a_{\sigma,\tau}(q) \,
  b^{(\ee)}_{\sigma,\tau}(q)^{-s/2} \,
  \frac{\gamma^\sigma_\len(s/2)}{q^\len}.
  \]
  Multiplying both sides by $[\aG(\kk):\tau(\kk)]^{1+s/2}$ and taking
  the limit as $\len \to \infty$, we obtain
  \[
  \xi^\tau(s) = q^{-1} \, a_{\tau,\tau}(q) \,
  b^{(\ee)}_{\tau,\tau}(q)^{-s/2} \, \xi^\tau(s) +
  \sum_{\substack{(\sigma,\tau)\in \dot{E}(\Gamma) \\ \sigma \ne
      \tau}} q^{-1} \, a_{\sigma,\tau}(q) \,
  b^{(\ee)}_{\sigma,\tau}(q)^{-s/2}
  \frac{[\aG(\kk):\tau(\kk)]^{1+s/2}}{[\aG(\kk):\sigma(\kk)]^{1+s/2}}
  \xi^\sigma(s).
  \]
  Substituting in the defining value for $b^{(\ee)}_{\sigma,\tau}(q)$
  from \eqref{equ:def.b.GL} resp.~\eqref{equ:def.b.GU}, using
  $[\sigma(\kk):\tau(\kk)] = \Vert \sigma \Vert / \Vert \tau \Vert$,
  and solving for $\xi^\tau(s)$ yields the desired formula.
\end{proof}

\subsection{Expressions for zeta functions of groups of type
  $\mathsf{A}_2$ } \label{subsec:van.coh.gps} For groups of type
$\mathsf{A}_2$ we are in a position to apply the sufficient criteria
developed in Section~\ref{sec:kom.clifford} for the extension of
characters from the principal congruence subgroup $G^1$, respectively
$H^1$, to $G$, respectively~$H$.  Indeed, the next two lemmata
establish the existence of shadow-preserving lifts and the vanishing
of the relevant cohomology groups.

\begin{lem} \label{lem:good.pair} Let $\len\in\N_0$. Every $A \in
  \fg_\len$ has a shadow-preserving lift $\wt{A} \in \fg_{\len+1}$.
  Furthermore, if $p \ne 3$ then every $A \in \fh_\len$ has a
  shadow-preserving lift $\wt{A} \in \fh_{\len+1}$.
\end{lem}

\begin{proof}
  The assertion for $\fg_\len$ is equivalent to the claim that in the
  shadow graph $\Gamma = \Gamma^{(\ee)}$, see
  Figure~\ref{fig:shadow.graph.A2}, there is a loop on every vertex
  or, equivalently, that $a_{\sigma,\sigma}(q)\neq 0$ for all
  $\sigma\in\Sh$. That this is the case follows from
  Theorems~\ref{thm:G.shad.graph} and~\ref{thm:U.shad.graph}; cf.\
  Table~\ref{tab:branch.rules.A2}.

  Assume now that $p\neq3$ and consider $A \in \fh_\len \leq
  \fg_\len$.  Let $\wt{B} \in \fg_{\len+1}$ be a shadow-preserving
  lift of $A$ within $\fg_{\len+1}$, i.e.\ suppose that $A
  \equiv_{\pi^{\len}} \wt{B}$ and $G^1 \Cen_G(A) = G^1
  \Cen_G(\wt{B})$.  Put $\wt{A} = \wt{B} -
  \frac{\tr(\wt{B})}{3}\Id_3$.  Then $\wt{A} \in \fh_{\len+1}$ and
  $G^1 \Cen_G (A) = G^1 \Cen_G(\wt{A})$.

  We claim that $H^1 \Cen_{H}(A) = H^1 \Cen_{H}(\wt{A})$.  The
  inclusion $\supset$ is obvious.  For the reverse inclusion, let $h
  \in \Cen_{H}(A)$.  There is $g \in G^1$ and $y \in \Cen_G(\wt{A})$
  such that $h = g y$.  From $\det h = 1$ and $\det g \equiv_\pi 1$ we
  deduce $\det y \equiv_\pi 1$.  As $p\neq3$, the pro-$p$ group
  $\GL_1^1(\lri)$ for $\ee = 1$, or $\GU_1^1(\lri)$ for $\ee = -1$,
  contains a cube root of $\det y$.  Hence there is a scalar matrix $s
  \in G^1$ such that $s^{-1} y \in \Cen_H(\wt{A})$ and $g s \in G^1
  \cap H = H^1$.  It follows that $h = (g s) (s^{-1} y) \in H^1
  \Cen_H(\wt{A})$.
\end{proof}

\begin{remark}\label{rem:hensel.lifts}
  The proof of the existence of shadow-preserving lifts of matrices in
  type $\mathsf{A}_2$ in Lemma~\ref{lem:good.pair} resorts to the
  shadow graph $\Gamma$. We sketch here a geometric point of view on
  the existence of shadow-preserving lifts, which also pertains to
  other `semisimple' Lie rings, say of type $\mathsf{A}_{n-1}$, $n\geq
  4$, or other classical types, where we currently do not know of an
  equally uniform description of the lifting behaviour of similarity
  classes.

  In \cite[Section~3.2]{AKOV1} we presented zeta functions of groups
  such as $\SL_n^m(\lri)$ in terms of $\fp$-adic integrals. The latter
  are, in general, defined in terms of polynomial ideals defining the
  rank varieties of certain matrices of linear forms. In
  \cite[Section~5]{AKOV1} we described a link between these rank
  varieties and stratifications of the (complexification of the)
  associated semisimple Lie algebras by quasi-affine varieties
  comprising elements of constant centraliser dimension,
  $\cV_{i}\setminus\cV_{i+1}$ in the parlance
  of~\cite{AKOV1}.

  It can be shown that shadow-preserving lifts exist if and only if,
  for all~$\len\in\N$ and any such variety, every point modulo
  $\fp^\len$ has a lift to a point modulo~$\fp^{\len+1}$. The latter
  holds -- essentially by Hensel's Lemma -- if the relevant varieties
  are all smooth. In general, they are disjoint unions of finitely
  many subvarieties called \emph{sheets}; see~\cite{Borho/82}.  The
  complex Lie algebra $\fsl_3(\C)$, for instance, is the union of
  three sheets, consisting of elements of centraliser dimensions $8$
  (the null sheet), $6$ (the subregular sheet), and $4$ (the regular
  sheet), respectively. But even if the sheets are all smooth, as they
  are in type $\mathsf{A}_{n-1}$, some of their unions might not
  be. In $\fsl_4(\C)$ and $\fsl_5(\C)$, the varieties of elements of
  constant centraliser dimension each consist of a single sheet; hence
  they are smooth, and shadow-preserving lifts exist.  Already in
  $\fsl_6(\C)$, however, the varieties of elements of centraliser
  dimension $17$ and $11$ both are the union of two sheets,
  respectively, of different dimensions. For a further discussion,
  also regarding the dimensions of sheets in other semisimple Lie
  algebras, see \cite{Moreau/08}.
\end{remark}

\begin{lem} \label{lem:cohomology} Suppose that $p>3$ is prime and
  that $q$ is a power of~$p$.  Then
  \begin{enumerate}
  \item[(a)] If $T \leq \GL_3(\F_q)$ or $T \leq \GU_3(\F_q)$ is
    contained in a torus, then $\lvert \textup{H}^2(T,\C^\times)
    \rvert$ is prime to $p$,
  \item[(b)] $\textup{H}^2(\SL_3(\F_q),\C^\times) = 1$ and
    $\textup{H}^2(\SL_2(\F_q),\C^\times) = 1$,
  \item[(c)] $\textup{H}^2(\SU_3(\F_q),\C^\times) = 1$ and
    $\textup{H}^2(\SU_2(\F_q),\C^\times) = 1$,
  \item[(d)] $\textup{H}^2(\GL_2(\F_q),\C^\times) = 1$,
  \item[(e)] $\textup{H}^2(\GU_2(\F_q),\C^\times) = 1$.
  \end{enumerate}
\end{lem}

\begin{proof}
 The assertion (a) holds because every prime that divides the order of
 $\textup{H}^2(T,\C^\times)$ must also divide the order of $T$, which
 is not divisible by $p$.  Assertions (b) and (c) are well known;
 cf.\ \cite{Griess/73, Steinberg/81} as well as
 \cite[p.~246]{Karpilovsky/87} and~\cite[Table~5]{Atlas/85}.

  Next we consider (d).  From the Lyndon--Hochschild--Serre spectral
  sequence
  \[
  \textup{H}^a(G/N,\textup{H}^b(N,M)) \Rightarrow \textup{H}^{a+b}(G,M),
  \]
  applied to $a+b=2$, $G = \GL_2(\F_q)$, $N = \SL_2(\F_q)$ and $M =
  \C^\times$, we deduce that the order of
  $\textup{H}^2(\GL_2(\F_q),\C^\times)$ divides
  \[
  \lvert \textup{H}^0(\F_q^\times,\textup{H}^2(\SL_2(\F_q),\C^\times))
  \rvert\cdot \lvert
  \textup{H}^1(\F_q^\times,\textup{H}^1(\SL_2(\F_q),\C^\times)) \rvert
  \cdot \lvert
  \textup{H}^2(\F_q^\times,\textup{H}^0(\SL_2(\F_q),\C^\times))
  \rvert.
  \]
  From (b) we see that $\textup{H}^2(\SL_2(\F_q),\C^\times)$ is
  trivial and hence the first factor is $1$.  Since $\SL_2(\F_q)$ is
  perfect for $q>3$, also $\textup{H}^1(\SL_2(\F_q),\C^\times)$ is
  trivial and the second factor is $1$.  Finally, the third factor
  equals $ \lvert \textup{H}^2(\F_q^\times,\C^\times) \rvert = 1$,
  because $\F_q^\times$ is cyclic.

  Noting that $\SU_2(\F_q) \simeq \SL_2(\F_q)$ (e.g., see
  \cite[II.8.8]{Huppert/67}), a similar argument yields~(e).
\end{proof}

We have collected all the necessary information to deal with
characters of~$G$.  For $H$ we prove an additional lemma, enabling us
to employ the theory of shadows also in this context.

\begin{lem} \label{lem:shadow.cap.H}
  With the notation as above, suppose that the Kirillov orbit method
  is applicable to~$H^1$ and that~$p\neq 3$.  Let $\chi \in \Irr(H^1)$
  and $\omega \in \fh^{\vee}$ a representative of the corresponding
  co-adjoint orbit.  Let $S_\chi = \textrm{I}_H(\chi)$ be the inertia
  subgroup. Then
  \[
  S_\chi / H^1 = \Stab_H(\omega)H^1/H^1 \simeq
  \left(\Stab_{G}(\omega)G^1 / G^1 \right) \cap \aH(\kk),
  \]
  where $\Stab_{G}(\omega)G^1 / G^1$ is identified with its image in
  $G/ G^1 \simeq \aG(\kk)$.
\end{lem}

\begin{proof}
  The equality follows from Lemma~\ref{lemma:R.and.G1}, and we argue
  for the isomorphism.  It suffices to show that every coset on the
  right-hand side is the image of a coset on the left-hand side under
  the natural inclusion map.  Let $g \in \Stab_G(\omega)$ such that
  the reduction of $g$ modulo $G^1$ is unimodular.  Since $p \ne 3$,
  there is a scalar matrix $u \in G^1$ such that $h \coloneqq g u \in
  \Stab_H(\omega)$; compare the proof of Lemma~\ref{lem:good.pair}.
  It follows that $h H^1$ maps to $gG^1$.
\end{proof}

\begin{proof}[Proof of Theorem~\ref{thm:C}]
  Consider one of the groups $\GL_3(\lri), \GU_3(\lri), \SL_3(\lri),
  \SU_3(\lri)$, and its $1$st principal congruence subgroup~$N$.
  Based upon $p \geq 3e+3$, we check the hypotheses in
  Theorem~\ref{thm:extension.to.radical} regarding characters $\chi
  \in \Irr(N)$.  Lemma~\ref{lem:good.pair} guarantees the existence of
  shadow-preserving lifts.  Lemma~\ref{lem:cohomology}, in conjunction
  with Tables~\ref{tab:shadows.GLGU} and \ref{tab:shadows.SLSU} as
  well as Lemma~\ref{lem:shadow.cap.H}, shows that the relevant
  cohomology groups vanish.  Thus
  Theorem~\ref{thm:extension.to.radical} implies that every $\chi \in
  \Irr(N)$ extends to its inertia group $S_\chi$ and
  Corollary~\ref{cor:gen.formula} is applicable.

  Regarding the infinite groups $\SL_3(\lri), \SU_3(\lri)$,
  formula~\eqref{equ:zeta.H.infin} follows now directly by collecting
  the summands in~\eqref{equ:gen.sum} according to shadows.

  Regarding the finite groups $\GL_3(\lri_\len), \GU_3(\lri_\len),
  \SL_3(\lri_\len), \SU_3(\lri_\len)$, formulae~\eqref{equ:zeta.G.fin}
  and \eqref{equ:zeta.H.fin} are obtained by restricting the relevant
  sums to characters factoring over the $\len$th principal congruence
  subgroup.  Once more, Lemma~\ref{lem:shadow.cap.H} is used to deal
  with the special linear/unitary groups.
\end{proof}

\subsection{Zeta functions of the shadows}\label{subsec:zeta.shadows}
In order to derive from Theorem~\ref{thm:C} explicit formulae such as
the one in Corollary~\ref{cor:D} we need, in addition to the functions
$\xi^{\sigma}_{\len-1} = \Xi^{\cS}_{\ee,q,\len-1}$ and $\xi^\sigma =
\Xi^{\cS}_{\ee,q}$ given in Proposition~\ref{pro:xi.sigma.len} and
Corollary~\ref{cor:infinite.xi}, the respective shadows' indices and
zeta functions.  In fact, the former may be expressed in terms of the
latter, because the order of any finite group is equal to the value of
its zeta function at~$s=-2$.  The isomorphism classes of the groups
whose zeta functions we need to compute are listed in
Tables~\ref{tab:shadows.GLGU} and~\ref{tab:shadows.SLSU}. As before,
$\kk_m$ denotes a degree $m$ extension of~$\kk$.  For $m \mid n$, we
write $N_{\kk_n \,\vert\, \kk_m} \colon \kk_n^\times \rightarrow
\kk_m^\times$ for the norm map.  Furthermore, $\ee=1$ for $\aG =
\GL_3, \aH = \SL_3$ and $\ee=-1$ for $\aG = \GU_3, \aH = \SU_3$; we
write $q = \lvert \kk \rvert$ and $\iota(\ee,q) =
\gcd(q-\ee,3)$.

\begin{table}[htb!]
  \centering
  \caption{Shadows in $\GL_3(\kk)$, for $\ee=1$, and $\GU_3(\kk)$, for
    $\ee=-1$}
  \label{tab:shadows.GLGU}
  \begin{tabular}{|c||l|l|l|}
    \hline
    Type & $\sigma(\kk) \subset \GL_3(\kk)$ &
    $\sigma(\kk)\subset\GU_3(\kk)$ & Order $\lvert \sigma(\kk) \rvert$
    \\ \hline $\Gsha$ & $\GL_3(\kk)$ & $\GU_3(\kk)$ &
    $(q-\ee)(q^2-1)(q^3-\ee) q^3$ \\
    $\Lsha$& $\GL_1(\kk) \times \GL_2(\kk)$ & $\GU_1(\kk) \times
    \GU_2(\kk)$  & $(q^2-1)(q-\ee)^2q$\\
    $\Jsha$  &  $\Heis(\kk) \rtimes \GL_1(\kk)^{2} $ &
    $\Heis(\kk) \rtimes \GU_1(\kk)^2 $ & $(q-\ee)^2q^3$ \\
    $\Tasha$ &   $\GL_1(\kk)^3$   & $\GU_1(\kk)^3 $ & $(q-\ee)^3$\\
    $\Tbsha$ &   $\GL_1(\kk_2) \times \GL_1(\kk)$ &  $\GL_1(\kk_2)
    \times \GU_1(\kk)$ &  $(q^2-1)(q-\ee)$\\
    $\Tcsha$ &   $\GL_1(\kk_3)$     &  $\GU_1(\kk_3)$ & $q^3-\ee$\\
    $\Msha$ &  $\GL_1(\kk[x]/(x^2)) \times \GL_1(\kk)$ &
    $\GU_1(\kk)\times\GU_1(\kk) \times \aG_a(\kk)$ & $(q-\ee)^2q$\\
    $\Nsha$ &  $\GL_1(\kk[x]/x^3)$ & $\GU_1(\kk) \times \aG_a(\kk)\times \aG_a(\kk)$
    & $(q-\ee)q^2$\\
    \hline
    $\Kasha, \Kbsha^*$ & $\GL_1(\kk) \times \aG_a(\kk) \times \aG_a(\kk)$ & --- &
    $(q-\ee)q^2$\\
    \hline
    \multicolumn{4}{@{} c @{}}{${}^*$ 
Only applies if $\ee=1$.}
  \end{tabular}
\end{table}

\begin{defn}\label{def:shadow.rep}
  Let $\sigma \in \Sh_{\aG(\lri)}$.  Recalling that $\sigma(\kk)$
  denotes a subgroup of $\aG(\kk)$ representing the shadow $\sigma$,
  we put
  \[
  \sigma'(\kk) \coloneqq \sigma(\kk) \cap \aH(\kk).
  \]
\end{defn}
\begin{table}[htb!]
  \centering
  \caption{Shadows in $\SL_3(\kk)$, for $\ee=1$, and $\SU_3(\kk)$, for
    $\ee=-1$}
  \label{tab:shadows.SLSU}
  \begin{tabular}{|c||l|l|l|}
    \hline
    Type & $\sigma'(\kk) = \sigma(\kk) \cap \SL_3(\kk)$ &
    $\sigma'(\kk) = \sigma(\kk) \cap \SU_3(\kk)$ & Order
    $|\sigma'(\kk)|$ \\
    \hline
    $\Gsha$ & $\SL_3(\kk)$ & $\SU_3(\kk)$ &
    $(q^2-1)(q^3-\ee)q^3$ \\
    $\Lsha$& $\GL_2(\kk)$ & $\GU_2(\kk)$  & $(q^2-1)(q-\ee)q$\\
    $\Jsha$  &  $\Heis(\kk) \rtimes \GL_1(\kk) $ &
    $\Heis(\kk) \rtimes \GU_1(\kk) $ & $(q-\ee)q^3$ \\
    $\Tasha$ &   $\GL_1(\kk) \times \GL_1(\kk)  $   & $\GU_1(\kk) \times \GU_1(\kk) $ & $(q-\ee)^2$\\
    $\Tbsha$ &   $\GL_1(\kk_2)$ &  $\GL_1(\kk_2)$ &  $q^2-1$\\
    $\Tcsha$ &   $\ker(N_{\kk_3 \,\vert\, \kk})$     &  $\ker(N_{\kk_6
      \,\vert\, \kk_3})
    \cap \ker(N_{\kk_6 \,\vert\, \kk_2})$ & $q^2+\ee q+1$\\
    $\Msha$ &  $\GL_1(\kk[x]/(x^2))$ &  $\GU_1(\kk) \times \aG_a(\kk)$ & $(q-\ee)q$\\
    $\Nsha$ &  $\mu_3(\kk) \times \aG_a(\kk) \times \aG_a(\kk)$ & $(\mu_3(\kk_2) \cap \ker
    (N_{\kk_2 \,\vert\, \kk})) \times (\aG_a(\kk))^2$  & $\iota(\ee,q)\, q^2$\\
    \hline
    $\Kasha, \Kbsha^*$ & $\mu_3(\kk) \times \aG_a(\kk) \times \aG_a(\kk)$ & ---&
    $\iota(\ee,q)\, q^2$\\ 
    \hline
    \multicolumn{4}{@{} c @{}}{${}^*$ 
Only applies if $\ee=1$.}
  \end{tabular}
\end{table}

\begin{prop}\label{pro:zeta.shadows}
  Let $\sigma \in \Sh$ be of type $\cS\in\T^{(\ee)}$.  Then
  \begin{align*}
    \zeta_{\sigma(\kk)}(s) & = (q-\ee) Z^{\cS}_{\ee,1,q}(s),
    \\ \zeta_{\sigma'(\kk)}(s) & = Z^{\cS}_{\ee,\iota(\ee,q),q}(s),
  \end{align*}
  where $Z^{\cS}_{\ee,i,q} \coloneqq Z^{\cS}_{\mathsf{A}_2,\ee,i,q}$
  for $i \in \{1,3\}$ is given by
  \begin{equation*}
    \begin{split}
      Z^{\Gsha}_{\ee,i,q}(s) & = 1 + (q^2+\ee q)^{-s} + (q-1-\ee)
      (q^2 + \ee q+1)^{-s} \\
      & \quad + \hlf (q^2-q-1+\ee) (q^3-\ee)^{-s} + q^{-3s} +
      (q-1-\ee) (q^3+\ee q^2+q)^{-s} \\
      & \quad + \third (q^2+\ee q -2) ((q+\ee)(q-\ee)^2)^{-s}+
      \genfrac{}{}{0.1pt}{1}{2}{3}i^2 ((q+\ee)(q-\ee)^2/i)^{-s} \\
      & \quad + \sixth (q-\ee)(q-3-\ee) ((q^2+\ee q+1)(q+
      \ee))^{-s}+\third i^2 ((q^2 + \ee q+1)(q+\ee)/i)^{-s},
    \end{split}
  \end{equation*}
  if $\cS = \Gsha$, and in the remaining cases $\cS \ne \Gsha$ defined
  as
  \begin{equation*}
    \begin{array}{ll}
      (q-\ee) \left( 1+q^{-s} + \hlf (q-2) (q+1)^{-s} + \hlf
        q(q-1)^{-s} \right) & \text{if $\cS = \Lsha$,} \\
        (q-\ee) + (q+\ee) i^{2} ((q-\ee)/i)^{-s}
        +(q-1)(q-\ee) q^{-s} & \text{if $\cS = \Jsha$,}\\
        (q-\ee)^2 & \text{if $\cS = \Tasha$,} \\
        q^2-1 & \text{if $\cS = \Tbsha$,} \\
        q^2+\ee q+1,  & \text{if $\cS = \Tcsha$,} \\
        q(q-\ee)  & \text{if $\cS = \Msha$,} \\
        i q^2  & \text{if $\cS \in \{\Nsha, \Kasha, \Kbsha \}$.}
    \end{array}
  \end{equation*}
\end{prop}

\begin{proof}
  The isomorphism types of the groups in question appear in
  Tables~\ref{tab:shadows.GLGU}~and~\ref{tab:shadows.SLSU}.  The
  formula for $\zeta_{\aG(\kk)}(s)$ is extracted from the character
  tables in~\cite{Steinberg}, for $\ee=1$, and~\cite[\S 7]{Ennola/63},
  for~$\ee=-1$.  The formulae for $\zeta_{\aH(\kk)}(s)$ are obtained
  from the data provided in~\cite{Frame-Simpson}.  For groups of type
  $\Lsha$ the formula follows, for example, from
  \cite[\S~15.9]{DigneMichel}, for~$\ee=1$, and \cite[\S
  6]{Ennola/63}, for~$\ee=-1$.  It remains to discuss shadows of type
  $\Jsha$, the only other non-abelian cases.

  Groups in the shadow $\sigma \in \Sh_{\aG(\lri)}$ of type $\Jsha$
  are isomorphic to $J_\ee = E_\ee \rtimes D_\ee$, where $E_\ee \simeq
  \Heis(\kk)$ is given explicitly by
  \[
  E_1 \coloneqq \left\{
    \begin{bmatrix}
      1  & s_1& z \\
      0 & 1 & s_2 \\
      0 & 0 & 1
    \end{bmatrix}
    \mid s_1,s_2,z\in \kk \right\} \quad \text{and} \quad E_{-1}
  \coloneqq \left\{
    \begin{bmatrix}
      1  & s & z \\
      0 & 1 & s^\circ \\
      0 & 0 & 1
    \end{bmatrix}
    \mid s,z \in \kk_2, s^\circ s =z+z^\circ \right\}.
  \]
  This representation of the shadow $\Jsha$ uses the centraliser of
  the elementary matrix $e_{13}$ that has its non-zero entry in the
  $(1,3)$-position, highlighting the appearance of the Heisenberg
  group.  The convention in the rest of the present paper, using the
  centraliser of $e_{23}$, is consistent with~\cite{APOV}; the
  centraliser of $e_{12}$ is used in~\cite{AKOV1}.  The groups $D_\ee$
  are given explicitly by
\begin{align*}
  D_1 \coloneqq \{\diag(u,v,u) \mid u, v \in \kk^\times\} &\simeq \GL_1(\kk) \times \GL_1(\kk), 
  \\ D_{-1} \coloneqq \{\diag(u,v,u) \mid u, v \in
  \ker(N_{\kk_2 \,\vert\, \kk})\} & \simeq \GU_1(\kk) \times \GU_1(\kk).
\end{align*}
Furthermore, the intersection of $J_\ee$ with $\aH(\kk)$ is equal to
$J_\ee' = E_\ee \rtimes D'_\ee$, where
  \[
  D_\ee' \coloneqq D_\ee \cap \aH(\kk)=\{(u,v,u) \in D_\ee \mid
  v=u^{-2}\} \simeq 
\begin{cases}
      \GL_1(\kk) & \text{if $\ee =1$,} \\
     \GU_1(\kk) & \text{if $\ee=-1$.}
    \end{cases}  
  \]
  Let
  \[
  Z_\ee \coloneqq \left\{ \left[
    \begin{smallmatrix}
      1 & 0 & z \\
      0 & 1 & 0 \\
      0 & 0 & 1
    \end{smallmatrix}
    \right] \mid z\in \kk \right\}
  \]
  denote the centre of~$E_\ee$.  The group $E_\ee \simeq \Heis(\kk)$
  has $q-1$ irreducible characters of degree~$q$ which correspond
  bijectively to the non-trivial characters of the centre, and $q^2$
  linear characters factoring through its abelianisation~$Q_\ee
  \coloneqq E_\ee/Z_\ee \simeq \kk \times \kk$.

  The group $D_\ee$ acts trivially on $Z_\ee$ and hence stabilises all
  the $q$-dimensional irreducible characters of~$E_\ee$.  As $q$ is
  prime to $\lvert D_\ee \rvert = (q-\ee)^2$, they all extend to
  irreducible characters of~$J_\ee$.  We get $(q-\ee)^2(q-1)$ distinct
  $q$-dimensional irreducible characters of~$J_\ee$, and similarly
  $(q-\ee)(q-1)$ such characters of~$J_\ee'$.

  The remaining irreducible characters of $J_\ee$,
  respectively~$J_\ee'$, factor through its quotient by~$Z_\ee$, viz.\
  $Q_\ee \rtimes D_\ee$, respectively $Q_\ee \rtimes D_\ee'$.  We
  consider separately the cases $\ee =1$ and $\ee =-1$.

  \smallskip

  First suppose that $\ee =1$.  It is convenient to identify $Q_1$ and
  its dual $Q_1^{\,\vee}$ with the additive group~$\kk \times \kk$.  With this
  identification, the action of $\diag(u,v,u) \in D_1$ on
  $Q_1^{\,\vee}$ is given by $(s_1,s_2) \mapsto
  (u^{-1}v{s_1},uv^{-1}{s_2})$.  We use Mackey's method for
  semi-direct products; cf.~\cite[Section~8.2]{Serre-LinRep}.  The
  orbits of $D_1$, respectively $D_1'$, on $Q_1^{\,\vee} \simeq \kk \times \kk$
  are classified as follows.  For $D_1$ we obtain
  \begin{center}
    \begin{tabular}{l|l|l}
      {Orbit} & Parameter& {Stabiliser in }$D_1$\\ \hline
      $[0,0]$      & --- & $\GL_1(\kk)\times \GL_1(\kk)$  \\
      $[0,\kk^{\times}]$   & --- & $\GL_1(\kk)$  \\
      $[\kk^{\times}, 0]$   & --- & $\GL_1(\kk)$ \\
      $\kk^{\times}\cdot [s_1,s_2]$  &$(s_1,s_2) \in \kk^{\times}
      \times_{\kk^{\times}} \kk^{\times}$  &  $\GL_1(\kk)$
    \end{tabular}
  \end{center}
  yielding $\lvert \GL_1(\kk) \times \GL_1(\kk) \rvert = (q-1)^2$
  linear characters of~$J_1$ lying above the trivial orbit $[0,0]$ and
  $q+1$ irreducible characters of degree $q-1$ lying above the
  remaining orbits.  Similarly, for $D_1'$ we obtain
  \begin{center}
    \begin{tabular}{l|l|l}
      {Orbit} & Parameter& {Stabiliser in }$D_1'$\\ \hline
      $[0,0]$       & --- & $\GL_1(\kk)$ \\
      $[0,(\kk^{\times})^3\cdot s_2]$ & $s_2 \in
      \kk^{\times}/(\kk^{\times})^3$ & $\mu_3(\kk)$ \\
      $[(\kk^{\times})^3\cdot s_1, 0]$ & $s_1 \in
      \kk^{\times}/(\kk^{\times})^3$ & $\mu_3(\kk)$ \\
      $(\kk^{\times})^3 \cdot [s_1,s_2]$& $(s_1,s_2)
      \in \kk^{\times} \times_{(\kk^{\times})^3} \kk^{\times}$ &
      $\mu_3(\kk)$
    \end{tabular}
  \end{center}
  yielding $\lvert \GL_1(\kk) \rvert=q-1$ linear characters and $(
  \lvert \kk \rvert+1) \lvert \kk^\times/(\kk^\times)^3 \rvert \lvert
  \mu_3(\kk) \rvert = (q+1) \iota(\ee,q)^2$ irreducible characters of
  degree $\lvert \kk^{\times}/\mu_3(\kk) \rvert =(q-1)/\iota(\ee,q) $
  of~$J_1'$.

  \smallskip

  Now suppose that $\ee=-1$.  In this case we identify
  $Q_{-1}$ and its dual $Q_{-1}^{\,\vee}$ with the additive
  group~$\kk_2$.  The action of $\diag(u,v,u) \in D_{-1}$ is given by
  $s \mapsto u^{-1}vs$.  To use Mackey's method for semi-direct
  products we classify the orbits of $D_{-1}$, respectively $D_{-1}'$,
  on $Q_{-1}^{\,\vee} \simeq \kk_2$.  For $D_1$ we obtain
  \begin{center}
    \begin{tabular}{l|l|l} {Orbit} & Parameter& Stabiliser in
      $D_{-1}$\\\hline
      $[0]$     & --- & $\GU_1(\kk) \times \GU_1(\kk)$ \\
      $[s]$ &$s \in \kk_2^{\times}/\GU_1(\kk)$ & $\GU_1(\kk)$
    \end{tabular}
  \end{center}
  yielding $ \lvert \GU_1(\kk) \times \GU_1(\kk) \rvert =(q+1)^2$
  linear characters and $\lvert \kk_2^{\times}/\GU_1(\kk) \rvert
  \lvert \GU_1(\kk) \rvert=(q^2-1)$ irreducible characters of degree
  $\lvert \GU_1(\kk)\rvert=(q+1)$ of~$J_{-1}$.  Similarly for $D_1'$
  we obtain
  \begin{center}
    \begin{tabular}{l|l|l} {Orbit} & Parameter& Stabiliser in
      $D_{-1}'$\\\hline
      $[0]$       & --- & $\GU_1(\kk)$ \\
      $[s]$ & $s \in \kk_2^{\times}/\GU_1(\kk)^3$ &$\mu_3(\kk_2) \cap
      \GU_1(\kk)$
    \end{tabular}
  \end{center}
  yielding $\lvert \GU_1(\kk) \rvert=q+1$ linear characters and
  $\lvert \kk_2^{\times}/\GU_1(\kk)^3 \rvert \lvert \mu_3(\kk_2) \cap
  \GU_1(\kk) \rvert=(q-1)\iota(\ee,q)^2$ irreducible characters of
  degree $\lvert \GU_1(\kk)/\mu_3(\kk_2) \cap \GU_1(\kk)
  \rvert=(q+1)/\iota(\ee,q)$ of~$J_{-1}'$.

  \smallskip

  In summary, for $\sigma$ of type $\Jsha$ we showed that
  \[
  \begin{split}
    \zeta_{\sigma(\kk)}(s) & = (q-\ee) \left( (q-\ee) +
      (q+\ee)(q-\ee)^{-s} +
      (q-1)(q-\ee)q^{-s} \right), \\
    \zeta_{\sigma'(\kk)}(s) & = (q-\ee) + (q+\ee) \iota(\ee,q)^2 \left(
      (q-\ee) / \iota(\ee,q) \right)^{-s} +(q-1)(q-\ee) q^{-s}.  \qedhere
  \end{split}
  \]
\end{proof}

\begin{proof}[Proof of Corollary~\ref{cor:D}]
  The corollary is obtained from formula~\eqref{equ:zeta.H.infin} in
  Theorem~\ref{thm:C} and the explicit formulae provided in
  Table~\ref{tab:shadows.SLSU}, Proposition~\ref{pro:zeta.shadows},
  and Corollary~\ref{cor:infinite.xi}.
\end{proof}

\subsection{Character degrees and Ennola
  duality} \label{subsec:ennola.duality} In this section we prove
Theorems~\ref{thm:H}~and~\ref{thm:I}.  Consider the finite groups
$G_\len = \aG(\lri_\len)$ and $H_\len = \aH(\lri_\len)$, for $\len \in
\N$. The conditions on $p$ in the two theorems ensure that the
Kirillov orbit method is available to describe the characters of the
finite principal congruence subgroups $G_\len^1$ and $H_\len^1$ and
that these characters extend to their respective stabilizers in
$G_\len$ and $H_\len$; see Theorem~\ref{thm:extension.to.radical},
Remark~\ref{rem:char-p-comments-5}, and compare with the proof of
Theorem~\ref{thm:C}.  An irreducible character $\chi$ of $G_\len$,
respectively~$H_\len$, therefore determines, and is determined by, the
following data:
\begin{itemize} \renewcommand{\labelitemi}{$\circ$}
\item a shadow $\sigma$ of type $\cS \in \T^{(\ee)}$,
\item a $G_\len$-orbit, respectively $H_\len$-orbit, of an
  irreducible character $\varphi_\sigma$ of $G^1_\len$, respectively
  $H^1_\len$, whose inertia subgroup in $G_\len$, respectively
  $H_\len$, gives rise to the shadow $\sigma$,
\item a choice of an extension $\hat \varphi_\sigma$ to its inertia
  subgroup in~$G_\len$, respectively $H_\len$, that we will not
  mention further,
\item an irreducible character $\psi_\sigma$ of $\sigma(\kk)$,
  respectively $\sigma'(\kk)$.
\end{itemize}
Moreover, to any such $\varphi_\sigma$ one associates
\begin{itemize} \renewcommand{\labelitemi}{$\circ$}
\item a unique path $\Delta(\varphi_\sigma) \in
  \mathrm{Path}^{\len-1}(\Gsha,\cS)$ of length $\len-1$ in the shadow
  graph $\Gamma^{(\ee)}$, see Figure~\ref{fig:shadow.graph.A2},
  starting at $\Gsha$ and ending at $\cS$.
\end{itemize}
By Corollary~\ref{cor:dim.as.func.shadows}, the degree of
$\varphi_\sigma$ is determined by the path
$\Delta=\Delta(\varphi_\sigma)$:
\begin{equation}\label{equ:dimension.path}
  \varphi_\sigma(1)=\prod_{(\tau,\upsilon) \in \Delta} q^{\frac{1}{2}
    \left(\dim \aG - \dim(\tau)\right)},
\end{equation}
feeding into the degree formula (cf.\ \eqref{equ:gen.sum})
\begin{equation}\label{equ:dimension.product}
  \chi(1) = \varphi_\sigma(1) \psi_\sigma(1) [\aG(\kk):\sigma(\kk)], \qquad
  \text{respectively} \quad
  \chi(1) = \varphi_{\sigma}(1) \psi_{\sigma}(1) [\aH(\kk):\sigma'(\kk)].
\end{equation}

\begin{proof}[Proof of Theorem \ref{thm:H}]
  We are required to give an Ennola-type description of the $p'$-part
  of the character degrees of $G_\len$.  Let $\sigma \in \Sh$ be of
  type~$\cS$.  We use~\eqref{equ:dimension.product} to control the
  character degrees of $\chi \in \Irr(G_\len)$ associated
  to~$\sigma$.

  Formula~\eqref{equ:dimension.path} shows that the contribution
  $\varphi_\sigma(1)$ to $\chi(1)$ is a $q$-power, depending only on
  $\Delta = \Delta(\varphi_\sigma)$.   

  The shadow graphs $\Gamma^{(1)}$ and $\Gamma^{(-1)}$ are almost
  identical: there is a natural correspondence between paths in
  $\Gamma^{(1)}$ not ending in $\Kasha,\Kbsha$ and paths
  in~$\Gamma^{(-1)}$.  Moreover, paths in $\Gamma^{(1)}$ of the same
  length $\len-1$ and containing at the same position one of the edges
  $(\Jsha,\Kasha), (\Jsha,\Kbsha), (\Jsha,\Nsha)$ lead to the same
  character degrees.  Finally, shadows $\sigma$ of types
  $\Kasha,\Kbsha,\Nsha$ yield isomorphic groups~$\sigma(\kk)$.  Thus,
  for our purposes, we may simply ignore $\Kasha, \Kbsha$.

  By Proposition~\ref{pro:zeta.shadows} the $p'$-parts of character
  degrees of $\sigma(\kk)$ are of the form $g(q)$, for polynomials $g
  \in \mathbb{Z}[t]$, involving $\ee$ as a parameter in such a way
  that the Ennola transform $g(t) \mapsto (-1)^{\deg g}g(-t)$
  translates between the cases $\aG = \GL_3$ and $\aG = \GU_3$.  From
  Table~\ref{tab:shadows.GLGU} we see that the same holds for the
  indices $[\aG(\kk):\sigma(\kk)]$.  Thus \eqref{equ:ennola} follows
  from~\eqref{equ:dimension.product} and the fact that the Ennola
  transform is multiplicative. 

  The explicit descriptions of the sets $\cd(\aG(\lri_{\len}))_{p'}$
  are easily obtained, using e.g.\ Proposition~\ref{pro:zeta.shadows}.
  We note that the two additional terms for $\len>1$ are owed to
  shadows of types $\Jsha$, $\Msha$, and $\Nsha$.
\end{proof}

\begin{defn}\label{def:level}
  The \emph{level} of a character $\chi$ of $H = \aH(\lri)$ or $H^1 =
  \aH^1(\lri)$ is equal to $\len-1$, where $\len \in \N$ is minimal
  such that $\chi$ is trivial, i.e.\ equal to the constant function
  $\chi(1)$, on the principal congruence subgroup~$H^\len$.  The
  terminology extends in a natural way to characters of~$H_\len$,
  respectively~$H^1_\len$, by implicitly lifting them to $H$,
  respectively~$H^1$.
\end{defn}

\begin{proof}[Proof of Theorem \ref{thm:I}]
  In the special case $\len=1$,
  Proposition~\ref{pro:zeta.shadows} provides the necessary
  information about character degrees of the group~$\aH(\kk)$.  We
  thus focus on the case $\len \geq 2$.
  
  As explained above, a character $\chi \in \Irr(H)$ of level $\len-1
  \geq 1$ can be connected with a shadow $\sigma$ of type
  $\cS\in\T^{(\ee)}$, a character $\phi_\sigma \in \Irr(H_\len)$ of
  level~$\len-1$ and a path $\Delta = \Delta(\varphi_\sigma) \in
  \mathrm{Path}^{\len-1}(\Gsha,\cS)$ of length $\len-1$ in the shadow
  graph~$\Gamma^{(\ee)}$.  Observe that $\Delta$ does not begin with a
  loop~$(\Gsha,\Gsha)$; in particular, $\cS \ne \Gsha$.  Furthermore,
  we have
  \[
  \chi(1) \ge \min_{\substack{\cS \in \T^{(\ee)},\\
      \text{$\sigma$ of type $\cS$}}} \; \min_{\substack{
      \varphi_\sigma \text{ such that } \\
      (\Gsha,\Gsha) \not\in \Delta(\varphi_\sigma) \in
      \text{Path}^{\len-1}(\Gsha,\cS)}} \; \min_{\psi_\sigma \in
    \Irr\left(\sigma'(\kk)\right)} \quad \phi_{\sigma}(1) \,
  \psi_\sigma(1) \,[\aH(\kk):\sigma'(\kk)],
  \]
  and similarly
  \[
  \chi(1) \le \max_{\substack{\cS \in \T^{(\ee)},\\
      \text{$\sigma$ of type $\cS$}}} \; \max_{\substack{
      \varphi_\sigma \text{ such that } \\
      (\Gsha,\Gsha) \not\in \Delta(\varphi_\sigma) \in
      \text{Path}^{\len-1}(\Gsha,\cS)}} \; \max_{\psi_\sigma \in
    \Irr\left(\sigma'(\kk)\right)} \quad \phi_{\sigma}(1) \,
  \psi_\sigma(1) \,[\aH(\kk):\sigma'(\kk)];
  \]
  cf.~\eqref{equ:dimension.product}.  To control the degree
  $\varphi_{\sigma}(1)$, given by \eqref{equ:dimension.path}, we argue
  as follow.  From~\eqref{equ:zeta.H^m_l} and
  Remark~\ref{rem:also-for-l=2m} we see that the Dirichlet polynomial
  \[
  \partial\xi^\sigma_{\len-1}(s) \coloneqq
  \xi^\sigma_{\len-1}(s) - \xi^\sigma_{\len-2}(s)
  \]
  enumerates the irreducible characters of $H^1_{\len}$ of level
  $\len-1$ and shadow~$\sigma$.  We set
  \[
  \XX = q^{(1-2s)(\len-2)}f^1_{\len-2}(q^{1-s});
  \]
  cf.\ \eqref{def:aux.count}.  Proposition~\ref{pro:xi.sigma.len}
  shows that, for $\sigma \in \Sh$ of type $\cS$, the function
  $\partial\xi^\sigma_{\len-1}(s)$ equals
  \begin{equation*}
    \begin{array}{ll}
      (q-1)(q^2+\ee q +1) \, q^{(1-2s)(\len-2)+2} & \text{if $\cS = \Lsha$,} \\
      (q^3-\ee)(q+\ee) \, q^{(1-2s)(\len-2)} & \text{if
        $\cS = \Jsha$,} \\
      \sixth (q-1)(q^2+\ee q +1)(q+\ee)q^3
      \left[ (q-2)q^{(2-3s)(\len-2)} + 3 (q-1) \XX \right] &
      \text{if $\cS = \Tasha$,} \\
      \hlf (q-1)(q^3-\ee)q^3 \left[q^{(2-3s)(\len-2)+1}  +(q-1)
        \XX \right] & \text{if $\cS = \Tbsha$,} \\
      \third(q^2-1)(q+\ee)(q-\ee)^2 \, q^{(2-3s)(\len-2)+3} &
      \text{if $\cS = \Tcsha$,} \\
      (q-1)(q^3-\ee)(q+\ee)q^2 \left[ q^{(2-3s)(\len-2)} +
        2 \XX \right] & \text{if $\cS = \Msha$,} \\
      (q^2-1)(q^3-\ee) \left[ q^{(2-3s)(\len-2)+1}  +(q-1)
        \XX \right] & \text{if $\cS = \Nsha$,} \\
      (q^2-1)(q^3-\ee) \, \XX & \text{if $\cS \in \{
        \Kasha, \Kbsha \}$.}
    \end{array}
  \end{equation*}
  These functions being Dirichlet polynomials in~$q^{-s}$, we define
  \[
  P_{\len-1}^\sigma = \{ m \in \N \mid \text{the coefficient of
    $q^{-ms}$ in $\partial\xi^\sigma_{\len-1}(s)$ is non-zero} \}.
  \]
  Clearly,
  \begin{equation*}
    \begin{array}{ll}
      P_{\len-1}^\sigma=\{2\len-4\} & \text{if $\cS \in \{\Lsha, \Jsha\}$,} \\
      P_{\len-1}^\sigma=\{3\len-6\} & \text{if $\cS = \Tcsha$,} \\
      P_{\len-1}^\sigma=\{2\len-4, 2\len-3, \ldots,3\len-7,
      3\len-6 \} &
      \text{if $\cS \in  \{\Tasha, \Tbsha, \Msha,\Nsha\}$,} \\
      P_{\len-1}^\sigma=\{2\len-4, 2\len-3, \ldots, 3\len-7\} &
      \text{if $\cS \in \{\Kasha, \Kbsha\}$},
    \end{array}
  \end{equation*} 
  Setting
  \[
  C_\sigma(q) \coloneqq \frac{[\aH(\kk):\sigma'(\kk)]}{q^{\dim \aG -
      \dim(\sigma)}}
  \]
  we see, using \eqref{equ:dimension.product}, that 
  \begin{equation}\label{equ:ineq}
    C_\sigma(q) \cdot q^{2\len-4 + \dim \aG -\dim(\sigma)} \le \chi(1)
    \le C_\sigma(q) \cdot q^{3\len-6+\dim \aG -\dim (\sigma)}
    \max_{\psi_\sigma \in \Irr \left(\sigma'(\kk)\right)} \psi_\sigma(1),
  \end{equation}
  Table~\ref{tab:shadows.SLSU} allows us to write $C_\sigma(q)$
  explicitly in terms of $\ee$ and~$q$; in particular, we see that
  $C_\sigma(q) = 1 + \mathrm{o}(q^{-1})$.  To obtain the bounds for
  $\chi(1)$ given in the theorem, it thus suffices to bound the
  remaining factors in~\eqref{equ:ineq}.  The minimum on the left-hand
  side is $q^{2\len}$, attained for shadows $\sigma$ of type $\Lsha$
  and $\Jsha$.  Inspecting the explicit formulae for the shadow zeta
  functions $\zeta_{\sigma'(\kk)}(s)$ given in
  Proposition~\ref{pro:zeta.shadows}, one deduces easily that the
  maximum on the right-hand side is $q^{3\len}$ and occurs, for
  example, for $\sigma$ of type $\Tasha$ for $\ee=1$ and $\Tcsha$ for
  $\ee=-1$, both necessarily with $\psi_\sigma(1) = 1$ as the
  respective groups $\sigma'(\kk)$ are abelian.
\end{proof}


\section{Ad\`elic zeta functions for type $\mathsf{A}_2$ and
  their analytic properties}\label{sec:abscissa}
\newcommand{\bruch}{5/6}

Theorem~\ref{thm:A} and Corollary~\ref{cor:B} are established in
Section~\ref{subsec:rep.global}.  Theorem~\ref{thm:G} is proved in
Section~\ref{subsec:sim.global}.

\subsection{Zeta functions of ad\`elic and arithmetic
  groups}\label{subsec:rep.global}
Let $\mathbf{H}(\widehat{\Gri_S})$ be an ad\`elic profinite group as
in Theorem~\ref{thm:A}.  This means that $\mathbf{H}$ is a connected,
simply-connected absolutely almost simple algebraic group of type
$\mathsf{A}_2$ defined over a number field~$k$, with
$S$-integers~$\Gri_S$ for a finite set $S \subset \cV_k$ of places
including all the archimedean ones.  Here $\cV_k$ denotes the
collection of all places of $k$, and we write $\cV_k^\infty$ for the
set of archimedean places.  The starting point for our study of the
analytic properties of the zeta function
$\zeta_{\mathbf{H}(\widehat{\Gri_S})}(s)$ is the Euler product
\begin{equation}\label{equ:euler.A2}
  \zeta_{\mathbf{H}(\widehat{\Gri_S})}(s) =  \prod_{v \in
    \cV_k \smallsetminus
    S} \zeta_{\mathbf{H}(\Gri_v)}(s),
\end{equation}
arising from the isomorphism $\mathbf{H}(\widehat{\Gri_S}) \simeq
\prod_{v \in \cV_k \smallsetminus S} \mathbf{H}(\Gri_v)$.

The classification of absolutely almost simple algebraic groups over
number fields implies that $\mathbf{H}$ is either an inner form, i.e.\
of type ${}^1\!\mathsf{A}_2$, arising from a matrix algebra over a
central division algebra over $k$, or an outer form, i.e.\ of type
${}^2\!\mathsf{A}_2$, arising from a matrix algebra over a central
division algebra over a quadratic extension $K$ of~$k$, equipped with
an involution and with reference to a suitable hermitian form; see
\cite[Propositions~2.17 and 2.18]{PlatonovRapinchuk/94} and the
summary in~\cite[Appendix~A]{AKOV1}.  The crucial point for us is that
there is a finite set $T \subset \cV_k$ with $S \subset T$
such that, for all $v$ in
\[
\cV_0 \coloneqq \cV_k \smallsetminus T,
\]
the completion $\mathbf{H}(\Gri_v)$ featuring in \eqref{equ:euler.A2},
is of the form $\SL_3(\Gri_v)$ or $\SU_3(\Gri_v)$ and, in the latter
case, $v$ is not dyadic and does not divide the (relative)
discriminant $\Delta_{K \,\vert\, k}$ of~$K \,\vert\, k$.  Set
\begin{equation}\label{def:V}
  \cV_{\SL} = \{v\in \cV_0 \mid \mathbf{H}(\Gri_v) \simeq
  \SL_3(\Gri_v)\} \quad \textrm{and} \quad \cV_{\SU}= \{v \in
  \cV_0 \mid \mathbf{H}(\Gri_v) \simeq \SU_3(\Gri_v)\}.
\end{equation}
We know, e.g.~from \cite[Theorem~B]{AKOV1}, that all of the finitely
many `exceptional' factors of \eqref{equ:euler.A2}, indexed by the
non-archimedean places in~$T$, converge to a holomorphic function on
the half-plane $\{s\in\C \mid \real(s) > 2/3\}$ without any zeros.
Hence the abscissa of convergence of
\begin{equation}\label{equ:euler.proof}
  Z(s) \coloneqq  \prod_{v \in \cV_0} \zeta_{\mathbf{H}(\Gri_v)}(s) = \prod_{v\in \cV_\SL} \zeta_{\SL_3(\Gri_v)}(s) \cdot
  \prod_{v\in \cV_\SU}\zeta_{\SU_3(\Gri_v)}(s),
\end{equation}
is equal to the abscissa of convergence of
$\zeta_{\mathbf{H}(\widehat{\Gri_S})}(s)$, which is known to be~$1$;
cf.~\cite[Theorem~C]{AKOV1}.  Moreover, it suffices to prove the first
statement of Theorem~\ref{thm:A} for~$Z(s)$ instead of
$\zeta_{\mathbf{H}(\widehat{\Gri_S})}(s)$.

The set $\cV_{\SU}$ is finite if and only if $\mathbf{H}$ is an inner
form.  If $\mathbf{H}$ is an outer form, then $\cV_{\SU}$ has positive
analytic density; see~\cite[Lemma~A.1]{AKOV1}.  In this case, the
distinction whether $\mathbf{H}(\Gri_v) \simeq \SL_3(\Gri_v)$ or
$\mathbf{H}(\Gri_v) \simeq \SU_3(\Gri_v)$ is, for all $v \in \cV_0$,
dictated by the decomposition behaviour of the prime ideal $\fp_v$ of
$\Gri$ associated to $v$ in the ring of integers~$\Gri_K$ of~$K$.
This behaviour, in turn, is described by the \emph{Artin symbol} of
the quadratic extension~$K \,\vert\, k$.  Indeed, the value of the
Artin symbol at a place $v \in \cV_k \smallsetminus \cV_k^\infty$ not
dividing the discriminant $\Delta_{K \,\vert\, k}$ is given by
\[
\ee(v) = \artin{K}{k}{v}=
\begin{cases}
  1  & \textrm{ if $\fp_v$ is decomposed in $\Gri_K$,} \\
  -1 & \textrm{ if $\fp_v$ is inert in $\Gri_K$};
\end{cases}
\]
cf., for instance, \cite[Chapter~VI, \S~7]{Neukirch/99}.  The Artin
symbol thus defines the key parameter~\eqref{equ:def.epsilon} in a
global setting.  For $v \in \cV_k \smallsetminus
\cV_k^\infty$, with residue field $\kk_v$ of cardinality
$q_v$, we write $\iota(v) \coloneqq \gcd(q_v -1,3) \in \{1,3\}$ for
the number of roots of unity in $\kk_v$, as in~\eqref{def:iota}.

Equation~\eqref{equ:zeta.H.infin} presents each factor
$\zeta_{\mathbf{H}(\Gri_v)}(s)$ of \eqref{equ:euler.proof} as a finite
sum of rational functions, indexed by shadow types and each depending
on the parameters $q_v$, $\ee(v)$, and~$\iota(v)$.  Furthermore,
$(1-q_v^{1-2s}) (1-q_v^{2-3s})$ is a common denominator for these
summands.  Informally speaking, we will show that clearing this common
denominator strictly improves the abscissa of convergence of the Euler
product defining $Z(s)$, from $1$ to at least $\bruch$.  More
precisely, we claim that
\begin{align*}
  \eta(s) & \coloneqq Z(s) \prod_{v \in \cV_0}
  (1-q_v^{\,1-2s})(1-q_v^{\,2-3s}) \\
  & \phantom{:}= \prod_{v\in \cV_\SL} (1-q_v^{\,1-2s})(1-q_v^{\,2-3s})
  \zeta_{\SL_3(\Gri_v)}(s) \cdot \prod_{v\in \cV_\SU}
  (1-q_v^{\,1-2s})(1-q_v^{\,2-3s}) \zeta_{\SU_3(\Gri_v)}(s)
\end{align*}
converges and does not vanish on the half-plane $\{s\in\C \mid
\real(s) > \bruch\}$.  As the Dedekind zeta function $\zeta_k(s) =
\prod_{v\in \cV_k \smallsetminus \cV_k^\infty}(1-q_v^{-s})^{-1}$
converges on $\{s\in\C \mid \real(s) > 1\}$, this yields a new proof
of the fact that the abscissa of convergence of
$\zeta_{\mathbf{H}(\widehat{\Gri_S})}(s)$ is equal to~$1$.  It also
shows that $\zeta_{\mathbf{H}(\widehat{\Gri_S})}(s)$ has meromorphic
continuation to at least $\{s \in \C \mid \real(s) > \bruch\}$ and
that the continued function has a unique singularity in this domain,
namely a double pole at~$s=1$; cf.~\cite[Chapter~VII,
Corollary~5.11]{Neukirch/99}. This will establish the first part of
Theorem~\ref{thm:A}.

\smallskip

We now prove the claim about the convergence of $\eta(s)$, using the
following well-known lemma.
\begin{lem}\label{lem:poly}
  Let $\mathcal{W} \subset \cV_k \smallsetminus \cV_k^\infty$.  Let
  $I$ be a finite index set and $f_i,g_i \in\Q[t]$, $i\in I$, be
  polynomials of degrees $\deg(f_i) \geq 0$ and $\deg(g_i) \geq 1$,
  respectively.  Then the Euler product
  \begin{equation}\label{equ:euler.poly}
    \prod_{v \in \mathcal{W}}\left( 1 + \sum_{i\in I}f_i(q_v) g_i(q_v)^{-s} \right)
  \end{equation}
  converges on $\left\{s\in\C \mid \real(s) > \max_{i \in I}
    \frac{\deg(f_i)+1}{\deg(g_i)} \right\}$.
\end{lem}
Note that we make no assumption on the set $\mathcal{W}$ of places and
no statement about the precise value of the abscissa of convergence of
the product~\eqref{equ:euler.poly}. We refer to
$\frac{\deg(f_i)+1}{\deg(g_i)}$ as the \emph{degree ratio} of the
expression~$f_i(q_v) g_i(q_v)^{-s}$, for $i \in I$.  Fixing $\ee \in
\{1,-1\}$ and $\iota \in \{1,3\}$, we set
\[
\mathcal{W} = \mathcal{W}_{\ee,\iota} = \{v\in \cV_0 \mid
\ee(v)=\ee, \iota(v)=\iota\}.
\]
Let~$v\in \mathcal{W}$.  The factor of $\eta(s)$ indexed by $v$ is a
sum of terms of the form
\begin{equation}\label{equ:cleared}
  [\mathbf{H}(\kk_v):{\sigma'}(\kk_v)]^{-1-s} \,
  \zeta_{\sigma'(\kk_v)}(s) \, \xi^\sigma(s) \,
  (1-q_v^{\,1-2s})(1-q_v^{\,2-3s}),
\end{equation}
where $\sigma$ ranges over the shadow set $\Sh_{\aG(\Gri_v)}$ for $\aG
= \GL_3$ or $\aG = \GU_3$ according to~$\ee$;
cf.~\eqref{equ:zeta.H.infin}.  By construction of $\mathcal{W}$, we
may write these terms as sums of polynomial expressions in $q_v$ with
constant coefficients as in the factors of the Euler
product~\eqref{equ:euler.poly}. More precisely, there exist a finite
index set $I$ and non-constant polynomials $f_i, g_i \in \Q[t]$ such
that the sum over the expressions in~\eqref{equ:cleared} is of the
form $1 + \sum_{i\in I} f_i(q_v) g_i(q_v)^{-s}$.  By
Lemma~\ref{lem:poly}, it remains to analyse the degree ratios
occurring for each shadow $\sigma$ of type $\cS$, say, and to verify
that they are all bounded above by~$\bruch$. In the sequel we
occasionally write $q$ instead of~$q_v$.

If $\cS=\Gsha$, then \eqref{equ:cleared} equals
\[
\zeta_{\sigma'(\kk_v)}(s) \, (1-q_v^{\,1-2s})(1-q_v^{\,2-3s}),
\]
where $\sigma'(\kk_v)$ is the finite group of Lie type $\SL_3(\kk_v)$
if $\ee=1$ or $\SU_3(\kk_v)$ if $\ee=-1$.  Inspection of the formulae
for these zeta functions, given in~\eqref{equ:zeta.Lie.type.H}, shows
that
\begin{multline}\label{equ:1}
  \zeta_{\sigma'(\kk_v)}(s) \, (1-q^{\,1-2s})(1-q^{\,2-3s}) = 1 +
  \left( q(q^{\,2} + \ee q+1)^{-s} - q^{\,1-2s} \right)
  \\
  + \left( \tfrac{1}{6} q^{\,2} \left((q^{\,2} + \ee q +1)(q
      +\ee) \right)^{-s} + \tfrac{1}{2} q^{\,2}(q^{\,3} - \ee)^{-s}
    + \tfrac{1}{3} q^{\,2} ((q + \ee)(q - \ee)^2)^{-s} -
    q^{\,2-3s} \right) \\
  + \left(\textrm{terms of degree ratios at most }4/5\right).
\end{multline}
For $s\in\R_{>0}$, the binomial series expansion implies that, for a
suitable constant $C_1 \in \R_{>0}$ and $q$ sufficiently large,
\[
\left\vert q(q^{2} + \ee q+1)^{-s}-q^{1-2s} \right\vert \leq
C_1q^{-2s}
\]
so that the relevant term on the right-hand side of~\eqref{equ:1} may
be replaced by a polynomial expression of degree ratio $1/2$ without
worsening the abscissa of convergence of the Euler product
defining~$\eta(s)$. A similar argument shows that, for $s \in
\R_{>0}$, there is a constant $C_2\in\R_{>0}$ such that, for all
sufficiently large $q$,
\begin{equation}\label{equ:est.cancel}
  \left| \tfrac{1}{6}q^2\left((q^2+\ee q +1)(q+\ee)\right)^{-s} + \tfrac{1}{2}q^2(q^3-\ee)^{-s} +
    \tfrac{1}{3}q^2 \left( (q+\ee)(q-\ee)^2 \right)^{-s} -
    q^{2-3s} 
  \right| \leq C_2 q^{1-3s},
\end{equation}
so that the relevant term on the right-hand side of~\eqref{equ:1} may
be replaced by a polynomial expression of degree ratio~$2/3$ without
worsening the abscissa of convergence of the Euler product.

If $\cS=\Lsha$ then~\eqref{equ:cleared} takes the form
\[
\hlf (q-1)(q-\ee) \left(2 + 2q^{\,-s} + (q-2)(q+1)^{-s} + q
  (q-1)^{-s} \right) \left(q^{\,2} (q^{\,2} + \ee q + 1) \right)^{-s}
(1-q^{\,2-3s}).
\]
The degree ratios occurring are at most~$4/5$.  The forms taken by the
summand~\eqref{equ:cleared} in the remaining cases $\cS \in \{ \Jsha,
\Tasha, \Tbsha, \Tcsha, \Msha, \Nsha, \Kasha, \Kbsha \}$, together
with upper bounds for the occurring degree ratios, are listed in
Table~\ref{tab:deg-ratios}. Overall, the degree ratios occurring in
these cases are bounded above by~$\bruch$. 

\begin{table}[htb!]
  \centering
  \caption{Bounds for various degree ratios, where $q = q_v$}
  \label{tab:deg-ratios}
  \begin{tabular}{|c||l|c|}
    \hline
    Type &  Summand \eqref{equ:cleared} & Deg.\ ratios \\\hline
    $\Jsha$ & $\left( (q^3-\ee)(q+\ee) \right)^{-s} (1-q^{2-3s})$ & \\
    & \hfill $\cdot  \left( (q-\ee) + (q+\ee)\iota^{2}((q-\ee)/\iota)^{-s} +
      (q-1)(q-\ee)q^{-s} \right) $&  $\leq 5/8$ \\
    $\Tasha$ & $\frac{1}{6} \left( q^3(q^2+\ee q +
      1)(q+\ee)\right)^{-s}(q-\ee)^2(q-1) (q-2+2q^{2-2s} - q^{1-2s})$
    & $\leq 5/6$\\
    $\Tbsha$ & $  \frac{1}{2} \left( q^3(q^3-\ee)
    \right)^{-s}q(q-1)(q^2-1)(1-q^{-2s})$ & $\leq 5/6$\\
    $\Tcsha$ & $ \frac{1}{3} \left( q^3(q + \ee)(q-\ee)^2 \right)^{-s}
    (q^2+\ee q +1)(q^2-1)(1-q^{1-2s})$ & $\leq 5/6$\\
    $\Msha$ & $ \left( q^2(q^3 - \ee)(q+\ee) \right)^{-s}
    q(q-\ee)(q-1)(1+q^{1-2s})$ & $\leq 2/3$\\
    $\Nsha$ & $  \iota^2 \left(q(q^3-\ee)(q^2-1)\iota \right)^{-s} q^2
    (1-q^{-2s})$ & $\leq 1/2$ \\
    \hline
    $\Kasha, \Kbsha^*$ & $  \iota^2
    \left(q(q^3-\ee)(q^2-1)/\iota\right)^{-s} q^{2-2s}$ & $\leq
    3/8$\\
    \hline
    \multicolumn{3}{@{} c @{}}{${}^*$ Only applies if $\ee = 1$.}
  \end{tabular}
\end{table}

This establishes the claim about the convergence of $\eta(s)$, and
hence the first part of Theorem~\ref{thm:A}.  To prove the second part
we recall the following Tauberian theorem.

\begin{thm}[{\cite[Theorem~4.20]{duSG/00}}] \label{thm:tauber} Let the
  Dirichlet series $f(s) = \sum_{n=1}^\infty a_n n^{-s}$ with
  non-negative real coefficients be convergent for $\real(s) > \alpha
  > 0$.  Assume that in a neighbourhood of~$\alpha$, one has $f(s) =
  g(s)(s-\alpha)^{-\beta} + h(s)$, where $g(s), h(s)$ are holomorphic
  functions, $g(\alpha) \ne 0$ and $\beta > 0$.  Assume also that
  $f(s)$ can be holomorphically continued to the line $\real(s) =
  \alpha$ except for the pole at $s = \alpha$. Then 
  \[
  \frac{g(\alpha)}{\alpha \mathsf{\Gamma}(\beta)} = \lim_{N\rightarrow
    \infty} \frac{\sum_{n=1}^N a_n}{N^\alpha(\log N)^{\beta-1}}.
  \] 
\end{thm}
Here, $\mathsf{\Gamma}$ denotes the $\Gamma$-function. The second
claim of Theorem~\ref{thm:A} follows from the first, with $\alpha=1$,
$\beta=2$, and $c(\mathbf{H}(\widehat{\Gri_S})) = g(\alpha)/(\alpha
\mathsf{\Gamma}(\beta)) = g(1)$ for some holomorphic function $g$ as
in Theorem~\ref{thm:tauber}.

\begin{proof}[Proof of Corollary~\ref{cor:B}] The group
  $\mathbf{H}(\Gri_S)$ contains a subgroup $\Gamma$ of finite index
  such that $\wh{\Gamma} \simeq \prod_{v\not\in S} \Gamma_v$, where
  $\Gamma_v$ is an open subgroup of $\mathbf{H}(\Gri_v)$ for all
  places $v$, with equality for all but finitely many~$v$; if
  $\mathbf{H}(\Gri_S)$ has the sCSP we may take $\Gamma =
  \mathbf{H}(\Gri_S)$.  It follows that $\zeta_{\Gamma}(s) =
  \zeta_{\mathbf{H}(\mathbb{C})}(s)^{\lvert k : \Q \rvert}
  \prod_{v\not\in S} \zeta_{\Gamma_v}(s)$; cf.~\cite[Theorem~3.3]{LL}.
  The corollary is deduced from Theorem~\ref{thm:A}; its proof shows
  how to deal with the finitely many `exceptional' non-archimedean
  factors for which $\Gamma_v \ne \mathbf{H}(\Gri_v)$, and we only
  need to accommodate for the additional archimedean factors.  In
  fact, they can be dealt with in a similar way:
  by~\cite[Theorem~5.1]{LL}, each factor $\zeta_{\mathbf{H}(\C)}(s)$
  converges and does not vanish on the complex half-plane~$\{s\in\C
  \mid \real(s) > 2/3\}$.
\end{proof}

We add some remarks concerning the constant $c(\mathbf{H}(\Gri_S))$ in
Corollary~\ref{cor:B} in case $\mathbf{H}(\Gri_S)$ has the sCSP;
similar comments apply to $c(\mathbf{H}(\widehat{\Gri_S}))$ in
Theorem~\ref{thm:A}.  The invariant $c(\mathbf{H}(\Gri_S))$ is a
rational multiple of the product of the following factors:
\begin{itemize} \renewcommand{\labelitemi}{$\circ$}
\item the $\lvert k:\Q \rvert$-th power of the special value
  $\zeta_{\SL_3(\C)}(1)$,
\item the square of the residue $\lim_{s\rightarrow 1}(s-1)\zeta_k(s)$
  at $s=1$ of the Dedekind zeta function $\zeta_k(s)$,
\item an Euler product $\prod_{v\in \cV'}
  (1-q_v^{-1})^2 \zeta_{\mathbf{H}(\Gri_v)}(1)$ for a cofinite subset
  $\cV' \subset \cV_k \smallsetminus S$.
\end{itemize}
The residue $\lim_{s\rightarrow 1}(s-1)\zeta_k(s)$ is, of course, well
known and given by the classical class number formula;
see~\cite[Chapter~VII, Corollary~5.11]{Neukirch/99}.

For $\mathbf{H}(\Gri_S) = \SL_3(\Z)$, for instance, we obtain
\begin{equation}\label{equ:euler.constant}
  c(\SL_3(\Z)) =  \zeta_{\SL_3(\C)}(1) \prod_{p \textrm{ prime}} \left(
    (1-p^{-1})^2 \zeta_{\SL_3(\Z_p)}(1) \right).
\end{equation}
It is known that $\zeta_{\SL_3(\C)}(s)$ is equal to the
`Mordell-Tornheim double series'
\[
\zeta_{\textrm{MT},2}(s) = \sum_{(m_1,m_2)\in\N^2}
(m_1m_2(m_1+m_2))^{-s};
\]
see, for instance, \cite[p.~359]{KMT/11}. In \cite[p.~369]{M58},
Mordell shows that
\[
\zeta_{\SL_3(\C)}(1) = \zeta_{\textrm{MT},2}(1) = 2\zeta(3),
\]
where $\zeta(s) = \zeta_\Q(s)$ denotes the Riemann zeta function.
Furthermore, for $k = \Q$, the residue $\lim_{s\rightarrow
  1}(s-1)\zeta(s) = 1$, and it remains to deal with the third factor
listed above.  Unfortunately, we are unable to determine
$c(\SL_3(\Z))$ completely as we do not know $\zeta_{\SL_3(\Z_3)}(s)$
-- or even just the special value $\zeta_{\SL_3(\Z_3)}(1)$ --
explicitly; see~\cite[Theorem~1.4]{AKOVIII/11} for a formula for the
related zeta function $\zeta_{\SL_3^m(\lri)}(s)$ for unramified
extensions $\lri$ of~$\mathbb{Z}_3$.  But we arrive at the following
numerical fact regarding the third factors listed above.

\begin{prop}
  Suppose that $\lri$ is a compact discrete valuation ring of
  characteristic zero, satisfying the conditions hypotheses of
  Corollary~\ref{cor:D}, and that $\aH(\lri)$ is either $\SL_3(\lri)$,
  for $\ee=1$, or $\SU_3(\lri)$, for $\ee=-1$, and put $\iota =
  \gcd(q-\ee,3)$.  Then
  \begin{equation*}
    (1-q^{-1})^2 \zeta_{\aH(\lri)}(1) =
    \frac{W_{\ee,\iota}(q)}{(q^3-\ee)(q^2-1)q^5},
  \end{equation*}
  where 
  \begin{multline*}
    W_{\ee,\iota}(q) = q^{10} - (2\ee + 1) q^8 + (\iota^3 - \ee + 2)
    q^7 + 4\ee q^6 + ((\ee + 1)\iota^3 - (2\ee + 3)) q^5 \\
    - (2\iota^3-3) q^4 + (\iota^3 + \ee - 3) q^3 + (\ee + 1) q^2 -
    \ee(2q - 1).
  \end{multline*}
\end{prop}

\begin{proof}
  The claim follows by inspection of the explicit formulae given in
  Theorem~\ref{thm:C}.
\end{proof}

\subsection{Ad\`elic similarity class zeta functions}\label{subsec:sim.global}
We recall the setup of Theorem~\ref{thm:G}.  Let $k$ be a number field
and $\mathbf{G}$ one of the $k$-algebraic groups $\GL_3$ or
$\GU_3(K,f)$, where the unitary group is defined with respect to the
standard hermitian form $f$ associated to the non-trivial Galois
automorphism of a quadratic extension $K$ of~$k$.  We denote by $\fg$
the corresponding Lie algebra scheme $\gl_3$ or~$\gu(K,f)$, and we use
$\ee = \ee_\mathbf{G} \in \{1,-1\}$ to distinguish between the general
linear and unitary cases.

Let $S \subset \cV_k$ be a finite set of places, including all the
archimedean ones, and, if $\ee = -1$, suppose that $S$ includes all
dyadic places as well as those places that ramify in the quadratic
extension~$K \,\vert\, k$.  Put $\cV_0 = \cV_k \smallsetminus S$ and
consider the Euler product
\[
Z_{\fg(\Gri_S)}(s) = \prod_{v\in \cV_0} Z_{\fg(\Gri_v)}(s),
\qquad \text{where } Z_{\fg(\lri)}(s) \coloneqq \lim_{\len \rightarrow
  \infty} q^{\,-\len} \gamma_\len(s)
\]
for a compact discrete valuation ring $\lri$ of residue
cardinality~$q$; compare Definition~\ref{def:sim.class.zeta} and
Proposition~\ref{pro:sim.shadow.fin}.  In analogy with~\eqref{def:V},
we set
\[
\cV_{\GL} = \{v \in \cV_0 \mid
\mathbf{G}(\Gri_v)\simeq \GL_3(\Gri_v)\}\quad \textrm{ and } \quad
\cV_{\GU} = \{v\in \cV_0 \mid \mathbf{G}(\Gri_v)\simeq
\GU_3(\Gri_v)\}
\]
so that $\cV_0 = \cV_{\GL} \cup \cV_{\GU}$ and
hence
\[
Z_{\fg(\Gri_S)}(s) = \prod_{v\in \cV_{\GL}} Z_{\gl_3(\Gri_v)}(s)
\prod_{v\in \cV_{\GU}} Z_{\gu_3(\Gri_v)}(s).
\]

Similar to the proof of Theorem~\ref{thm:A}, it suffices to show that
the function
\begin{equation}\label{equ:eta.simi}
  \begin{split}
    \eta_{\simi}(s) & \coloneqq Z_{\fg(\lri)}(s) \prod_{v\in
      \cV_0}(1-q_v^{\,1-4s})(1-q_v^{\,2-6s}) \\
    & \phantom{:}= \prod_{v\in \cV_\GL}
    (1-q_v^{\,1-4s})(1-q_v^{\,2-6s}) Z_{\gl(\Gri_v)}(s) \cdot
    \prod_{v\in \cV_\GU} (1-q_v^{\,1-4s})(1-q_v^{\,2-6s})
    \zeta_{\gu_3(\Gri_v)}(s)
  \end{split}
\end{equation}
converges and does not vanish on the half-plane $\{s\in\C \mid
\real(s) > 2/5\}$.  This will establish the first two parts of
Theorem~\ref{thm:G}; the third part follows via the Tauberian
Theorem~\ref{thm:tauber}.

For each $v\in \cV_0$, we have
\begin{equation}\label{equ:Z} (1-q_v^{\,1-4s})(1-q_v^{\,2-6s})Z_{\fg(\Gri_v)}(s) =
  \sum_{\cS\in\T^{(\ee)}} (1-q_v^{\,1-4s})(1-q_v^{\,2-6s})\Gamma^{\cS}_{\ee(v),q_v}(s),
\end{equation}
where the functions $\Gamma^{\cS}_{\ee(v),q_v}(s)$ are given in
Corollary~\ref{cor:infinite.gamma}. Fix $v\in \cV_0$.  We analyse the
degree ratios occurring in each summand of~\eqref{equ:Z}. As in
Section~\ref{subsec:rep.global} we write $q$ for $q_v$ when we do not
want to stress the dependence on the place.

If $\cS=\Gsha$, then the relevant summand in \eqref{equ:Z} is just
\begin{equation}\label{equ:S=G} (1-q_v^{\,1-4s})(1-q_v^{\,2-6s}) = 1 -
  q_v^{\,1-4s} - q_v^{\,2-6s} + q_v^{\,3-10s}.
\end{equation}
The term $q_v^{\,3-10s}$ has degree ratio $2/5$.  We will show that
this is the maximal degree ratio occurring. In particular, the terms
$-q_v^{\,1-4s}$ and $-q_v^{\,2-6s}$ ``cancel'' with terms occurring in
types $\Lsha$ and $\Tasha$, $\Tbsha$, $\Tcsha$, respectively, in a way
we shall explain.

Indeed, if $\cS=\Lsha$, then the relevant summand in \eqref{equ:Z} is
\begin{multline*}
  (q-1) \left( (q^{\,2} + \ee q + 1) q^{\,2} \right)^{-s} (1-
  q^{\,2-6s}) \\ = q \left( (q^{\,2} + \ee q + 1) q^2
  \right)^{-s} + (\textrm{terms of degree ratios at most }3/10).
\end{multline*}
By arguments akin to those used in the proof of Theorem~\ref{thm:A},
for $s\in \R_{>0}$, there exists a constant $C_1 \in \R_{>0}$ such that
for sufficiently large $q$,
\[
\left\vert q \left( (q^2 + \ee q + 1)q^2 \right)^{-s} - q^{1-4s}
\right \vert \leq C_1 q^{-4s}.
\]
This shows that the term $-q^{1-4s}$, occurring for $\cS=\Gsha$, and
the term $q \left( (q^2 + \ee q + 1 ) q^2 \right)^{-s}$, occurring for
$\cS=\Lsha$, may be replaced by a polynomial expression of degree
ratio $1/4$, without worsening the abscissa of convergence of the
Euler product~\eqref{equ:eta.simi}.

If $\cS\in\{\Tasha,\Tbsha,\Tcsha\}$, then the summands in
\eqref{equ:Z} indexed by the relevant shadows are
\begin{align*}
  (\Tasha) \quad & \sixth (q-1) \left( (q+\ee)(q^2+\ee q + 1)q^3
  \right)^{-s} (q+2q^{2-4s}-2-q^{1-4s}) \\
  (\Tbsha) \quad &\hlf (q-1) \left( (q^3-\ee)q^3 \right)^{-s}(q-q^{1-4s}) \\
  (\Tcsha) \quad &\third (q^2-1) \left( (q+\ee)(q-\ee)^2q^3
  \right)^{-s}(1-q^{1-4s}).
\end{align*}
Modulo terms of degree ratios at most $3/10$, these read
\begin{align*}
  (\Tasha) \quad & \sixth q^2 \left( (q+\ee)(q^2+\ee q + 1)q^3 \right)^{-s}\\
  (\Tbsha) \quad &\hlf q^2 \left( (q^3-\ee)q^3 \right)^{-s} \\
  (\Tcsha) \quad &\third q^2 \left( (q+\ee)(q-\ee)^2q^3 \right)^{-s}.
\end{align*}
Similar to the proof of Theorem~\ref{thm:A}, we deduce that, for $s\in
\R_{>0}$, there exists a constant $C_2\in\R_{>0}$ such that for
sufficiently large $q$,
\begin{multline*}
  \left\vert \sixth q^2 \left((q+\ee)(q^2+\ee q + 1)q^3 \right)^{-s} +
    \hlf q^2 \left((q^3-\ee)q^3 \right)^{-s} + \third q^2 \left(
      (q+\ee)(q-\ee)^2q^3 \right)^{-s} -
    q^{2-6s} \right\vert \\
  \leq C_2 q^{1-6s}.
\end{multline*}
This ``cancels'' the term $-q^{2-6s}$ from~\eqref{equ:S=G}.  The forms
taken by the relevant summand in~\eqref{equ:Z} in the remaining cases
$\cS \in \{ \Jsha, \Msha, \Nsha, \Kasha, \Kbsha \}$, together with an
upper bound for the occurring degree ratios, is given in
Table~\ref{tab:deg-ratios-2}.
\begin{table}[htb!]
  \centering
  \caption{Bounds for various degree ratios, where $q = q_v$}
  \label{tab:deg-ratios-2}
  \begin{tabular}{|c||l|c|}
    \hline
    Type &  Summand in \eqref{equ:Z} & Degree ratios \\\hline
    $\Jsha$ & $\left( (q-\ee)^3(q+\ee) \right)^{-s} (1-q^{2-6s}) $&
    $\leq 3/10$ \\
    $\Msha$ & $(q-1) \left( (q-\ee)^3(q+\ee)q^2 \right)^{-s} (1 +
    q^{1-4s})$ & $\leq 1/3\phantom{0}$ \\
    $\Nsha$ & $\left( (q^2-1)(q^3-\ee)q \right)^{-s} (1 - q^{-4s})$ &
    $\leq 1/6\phantom{0}$\\\hline
    $\Kasha, \Kbsha^*$ & $\left( (q^2-1)(q^3-\ee)q^5 \right)^{-s}$ &
    $\leq 1/10$\\
    \hline
    \multicolumn{3}{@{} c @{}}{${}^*$ Only applies if $\ee = 1$.}
  \end{tabular}
\end{table}
This concludes the proof of Theorem~\ref{thm:G}.


\appendix
\part*{Appendix and References}


\section{A model version: groups of type $\mathsf{A}_1$}\label{sec:A1}

The main ideas of this paper may be applied to groups of
type~$\mathsf{A}_1$, such as groups of the form $\GL_2(\lri)$ or
$\GU_2(\lri)$, where $\lri$ is a compact discrete valuation ring, and
various subquotients of these groups. We record here -- mainly without
(detailed) proofs -- results on similarity classes and associated zeta
functions, as well as representation zeta functions of such
groups. This leads, on the one hand, to new, unified computations for
the known zeta functions of groups of the form $\SL_2(\lri)$ and
$\GL_2(\lri_\len)$; cf.\ \cite{JZ,Onn}, respectively. It also allows
us to compute new zeta functions, such as the ones of groups of the
form~$\GU_2(\lri_\len)$. Throughout we assume that $\lri$ is a compact
discrete valuation ring with residue field $\kk$ of cardinality $q$
and residue characteristic~$p$. Let $\aG$ be one of the $\lri$-group
schemes $\GL_2$ and $\GU_2$
and \begin{equation} \label{equ:def.epsilon.A1} \ee = \ee_\aG
  =
  \begin{cases}
    +1 & \text{if $\aG = \GL_2$}, \\
    -1 & \text{if $\aG = \GU_2$},
  \end{cases}
\end{equation}
analogous to \eqref{equ:def.epsilon}. The $\lri$-group scheme $\GU_2$
is defined with respect to the unramified quadratic extension of
$\lri$; see the discussion at the beginning of
Section~\ref{sec:sim.gu} for details. We exclude $p=2$ from our
considerations in the unitary case.  We write $\ag\in\{\gl_2,\gu_2\}$
for the $\lri$-Lie lattice scheme associated to $\aG$ and $\Sh$ for
the respective shadow set $\Sh_{\GL_2(\lri)}$ or~$\Sh_{\GU_2(\lri)}$.

\subsection{Similarity classes and their shadows}
As in type $\mathsf{A}_2$, similarity classes in $\ag(\lri_\len)$ are
controlled by shadows and branching rules. The -- rather simple --
classification of similarity classes in $\Ad(\GL_2(\lri)) \backslash
\gl_2(\lri_\len)$ is described in~\cite[Section~2.1]{APOV}. The
unitary case is analogous. It turns out that -- as in the
$\mathsf{A}_2$-case -- similarity classes and their lifting behaviour
are governed by branching rules and shadow graphs. The following
result is analogous to Theorems~\ref{thm:G.shad.graph}
and~\ref{thm:U.shad.graph}, formulated uniformly for both values of
the parameter~$\ee$.

\begin{thm}[Classification of shadows and branching
  rules] \label{thm:shadow.graph.A1}\quad
  \begin{enumerate}
  \item The shadow set $\Sh$ consists of four elements, classified by
    the types
    \[
    \Gsha', \,  \Tasha', \, \Tbsha', \, \Nsha'
    \]
    described in Table~\textup{\ref{tab:shadows.GLGU.A1}}.
  \item For all $\sigma, \tau \in \Sh$ there exists a polynomial
    $a_{\sigma, \tau} \in \Z[\hlf][t]$ such that the following holds:
    for every $\len \in \N$ and every $\cC \in \cQ_{\lri,\len}^{\ag}$
    with $\shG(\cC) = \sigma$ the number of classes $\cCtilde \in
    \cQ_{\lri,\len+1}^{\ag}$ with $\shG(\cCtilde) = \tau$ lying above
    $\cC$ is equal to $a_{\sigma, \tau}(q)$.
  \end{enumerate}
\end{thm}

Set $\T_{\mathsf{A}_1}= \{\Gsha',\Tasha',\Tbsha',\Nsha'\}$. As in type
$\mathsf{A}_2$, it is remarkable that there are $\kk$-forms of
algebraic groups $\aI^{\cS}_{\mathsf{A}_1,\ee}$, for
$\cS\in\T_{\mathsf{A}_1}$, whose $\kk$-rational points
$\aI^{\cS}_{\mathsf{A}_1,\ee}(\kk)$ represent the shadow sets
$\Sh_{\aG(\lri)}$. Similarly to the $\mathsf{A}_2$-situation, given
$\sigma\in\Sh_{\aG(\lri)}$ we set
$\sigma(\kk)=\aI^{\cS}_{\mathsf{A}_1,\ee}(\kk)$ and $\sigma'(\kk)=
\sigma(\kk)\cap\SL_2(\kk)$.

It is noteworthy that the quantities $b^{(\ee)}_{\sigma,\tau}$ defined
in Definitions~\ref{def:b.GL} and \ref{def:b.GU} are in fact
polynomials in~$q$. Together with the polynomials $a_{\sigma,\tau}$
they determine recursively the numbers and sizes of similarity classes
in $\ag(\lri_\len)$ for all $\len\in\N$. These \emph{branching rules}
in type $\mathsf{A}_1$ are collected in
Table~\ref{tab:branch.rules.A1}. In analogy with the
$\mathsf{A}_2$-situation (cf.\ Definition~\ref{def:shadow.graph.A2})
one may associate a \emph{shadow graph} with each of the scheme pairs
$(\ag,\aG)$; cf.\ Figure~\ref{fig:shadow.graph.A1}. In contrast to the
$\mathsf{A}_2$-situation, they coincide for $\ee=1$ and~$\ee=-1$.

\begin{proof}[Proof of Theorem~\ref{thm:shadow.graph.A1} (sketch)]
  Instead of giving a proof from scratch, we indicate how the shadows
  and the polynomials $a_{\sigma,\tau}$ -- and, in fact,
  $b_{\sigma,\tau}^{(\ee)}$ -- can be extracted from the
  $\mathsf{A}_2$-case. Indeed, shadow type $\Lsha$ corresponds to the
  groups $\GL_2 \times \GL_1$ and $\GU_2 \times \GU_1$. It follows
  that the shadow graph for groups of type $\mathsf{A}_1$ is the
  subgraph of the $\mathsf{A}_2$-shadow graph in
  Figure~\ref{fig:shadow.graph.A2} consisting of vertices $\Lsha$,
  $\Tasha,\Tbsha$, and $\Nsha$ together with edges $9$, $10$, $11$,
  $12$, and with a loop around each vertex. The transition quantities
  are given by dividing the data $a_{\sigma,\tau}(q)$ of
  Table~\ref{tab:branch.rules.A2} by $q$ to cancel the redundant
  $\GL_1$ or $\GU_1$ factor, and by dividing the data
  $b_{\sigma,\tau}^{(\ee)}(q)$ in that table by $q^4$ to get the
  correct dimension. In the theorem's statement we used the labels
  $\Gsha'$, $\Tasha'$, $\Tbsha'$, and $\Nsha'$ for the respective
  $\mathsf{A}_1$-analogues of $\Lsha$, $\Tasha$, $\Tbsha$, and
  $\Nsha$.
\end{proof}

\begin{table}[htb!]
  \centering
  \caption{Branching rules for $\cQ_\lri^{\gl_2}$ ($\ee=1$) and
    for $\cQ_\lri^{\gu_2}$ ($\ee=-1$)}
  \label{tab:branch.rules.A1}
  \begin{tabular}{|c|c|c||l|l|}
    \hline \#& Type of $\sigma$ & Type of $\tau$ & $a_{\sigma,\tau}(q)$
    & $b^{(\epsilon)}_{\sigma,\tau}(q)$ \\ \hline $1$ & $\Gsha'$ &
    $\Gsha'$ & $q$ & $1$ \\ $2$ & $\Gsha'$ & $\Tasha'$ & $\hlf q(q-1)$ &
    $(q+\ee)q$ \\ $3$ & $\Gsha'$& $\Tbsha'$ & $\hlf q(q-1)$ & $(q-\ee)q $
    \\ $4$ & $\Gsha'$& $\Nsha'$ & $q$ & $q^2-1$ \\ \hline\hline $5$ & other &
    same as $\sigma$ & $q^2$ & $q^2$ \\ \hline
  \end{tabular}
\end{table}

\begin{figure}[htb!]
\centering
 \caption{The shadow graph $\Gamma$ for
   $(\mathsf{gl}_2,\mathsf{GL}_2)$ and
   $(\mathsf{gu}_2,\mathsf{GU}_2)$}
  \label{fig:shadow.graph.A1}
  \begin{displaymath}
    \xymatrix{ & & \\ \Tasha' \ar@(ul,dl)[]|{5}& \Gsha'
      \ar@(ul,ur)[]|{1} \ar[l]|{2} \ar[r]|{3} \ar[d]|{4} & \Tbsha' \ar@(ur,dr)[]|{5}\\ & \Nsha'
      \ar@(dl,dr)[]|{5} & }
  \end{displaymath}

\end{figure}

\subsection{Similarity class zeta functions}
As in the $\mathsf{A}_2$-case, the classification of shadows and the
associated branching rules allows us to compute various similarity
class and representation zeta functions. Recall the definitions of the
similarity class zeta functions $\gamma^\sigma_\len(s)$
in~\eqref{def:gamma.sigma.fin} and the finite geometric progressions
$A_{q,\len}(s)$ in~\eqref{equ:aux.ABC}.

\begin{prop}\label{pro:gamma.sigma.fin.A1}
  For $\sigma\in\Sh$ of type $\cS\in\mathbb{T}_{\mathsf{A}_1}$ and
  $\len\in\N_0$,
  \[
  \gamma_\len^\sigma(s) = q^\len
  \Gamma^{\cS}_{\mathsf{A}_1,\epsilon,q,\len}(s),
  \]
  where the function $\Gamma^{\cS}_{\mathsf{A}_1,\epsilon,q,\len}(s)$ is
  defined as
  \begin{equation*}
    \begin{array}{ll}
      1 & \text{if $\cS = \Gsha'$,}
      \\ \hlf
      (q-1)\left(q(q+\ee)\right)^{-s}A_{q,\len}(s/2) & \text{if $\cS =
        \Tasha'$,} \\ \hlf
      (q-1)\left(q(q-\ee)\right)^{-s}A_{q,\len}(s/2) & \text{if $\cS =
        \Tbsha'$,} \\
      (q^2-1)^{-s}A_{q,\len}(s/2) & \text{if $\cS = \Nsha'$.}
    \end{array}
  \end{equation*}
\end{prop}

\begin{proof} Analogous to Proposition \ref{pro:sim.shadow.fin}.
\end{proof}

Recall further Definition~\ref{def:xi} of the functions
$\xi^\sigma_\len(s)$ for $\sigma\in\Sh$ and their limits $\xi^\sigma$
as $\len\rightarrow\infty$.

\begin{prop}
  For $\sigma\in\Sh$ of type $\cS\in\T_{\mathsf{A}_1}$ and
  $\len\in\N_0$,
  \[
  \xi^\sigma_\len(s) = \Xi^{\cS}_{\mathsf{A}_1,\epsilon,q,\len}(s),
  \]
  where the function $\Xi^{\cS}_{\mathsf{A}_1,\epsilon,q,\len}(s)$ is
  defined as
  \begin{equation*}
    \begin{array}{ll}
      1 & \text{if $\cS = \Gsha'$,} \\
      \hlf q(q-1)(q+\ee)A_{q,\len}(s/4) &
      \text{if $\cS = \Tasha'$,} \\\hlf
      q(q-1)(q-\ee)A_{q,\len}(s/4) & \text{if $\cS = \Tbsha'$,}
      \\ (q^2-1)A_{q,\len}(s/4) & \text{if $\cS =
        \Nsha'$.}
    \end{array}
  \end{equation*}
\end{prop}

\begin{proof}
 Straightforward from the data collected in
 Tables~\ref{tab:shadows.GLGU.A1} and \ref{tab:shadows.SLSU.A2}.
\end{proof}

\begin{cor}
  For $\sigma\in\Sh$ of type $\cS\in\T_{\mathsf{A}_1}$,
  \[
  \xi^\sigma(s) = \Xi^{\cS}_{\mathsf{A}_1,\epsilon,q}(s),
  \]
  where the function $\Xi^{\cS}_{\mathsf{A}_1,\epsilon,q}(s)$ is
  defined as
  \begin{equation*}
    \begin{array}{ll}
      1 & \text{if $\cS = \Gsha'$,} \\
      \hlf q(q-1)(q+\ee)(1-q^{1-s})^{-1} &
      \text{if $\cS = \Tasha'$,} \\\hlf
      q(q-1)(q-\ee)(1-q^{1-s})^{-1} & \text{if $\cS = \Tbsha'$,}
      \\ (q^2-1)(1-q^{1-s})^{-1} & \text{if $\cS =
        \Nsha'$.}
    \end{array}
  \end{equation*}
\end{cor}

\begin{table}[htb!]
  \centering
  \caption{Shadows in $\GL_2(\kk)$ for $\epsilon=1$, and $\GU_2(\kk)$,
    for $\epsilon=-1$}
  \label{tab:shadows.GLGU.A1}
  \begin{tabular}{|c||l|l|l|l|l|}
    \hline Type & $\sigma(\kk) \subset \GL_2(\kk)$ &
    $\sigma(\kk)\subset\GU_2(\kk)$ & Order $\lvert \sigma(\kk)
    \rvert$ \\ \hline $\Gsha'$ & $\GL_2(\kk)$ & $\GU_2(\kk)$ &
    $q(q-\ee)(q^2-1)$ \\ $\Tasha'$& $\GL_1(\kk) \times \GL_1(\kk)$ &
    $\GU_1(\kk) \times \GU_1(\kk)$ & $(q-\ee)^2$\\ $\Tbsha'$ &
    $\GL_1(\kk_2) $ & $\GL_1(\kk_2)$ & $q^2-1$\\ $\Nsha'$ &
    $\GL_1(\kk) \times \aG_a(\kk)$ & $\GU_1(\kk) \times \aG_a(\kk)$ &
    $ q(q-\ee)$\\ \hline
  \end{tabular}
\end{table}

\begin{table}[htb!]
  \centering
  \caption{Shadows in $\SL_2(\kk)$ for $\epsilon=1$, and $\SU_2(\kk)$,
    for $\epsilon=-1$}
  \label{tab:shadows.SLSU.A2}
  \begin{tabular}{|c||l|l|l|l|l|}
    \hline Type & $\sigma'(\kk) = \sigma(\kk)\cap\SL_2(\kk)$ &
    $\sigma'(\kk) = \sigma(\kk) \cap \SU_2(\kk)$ & Order $\lvert
    \sigma(\kk) \rvert$ \\ \hline $\Gsha'$ & $\SL_2(\kk)$ &
    $\SU_2(\kk)$ & $q(q^2-1)$ \\ $\Tasha'$& $\GL_1(\kk)$ &
    $\GU_1(\kk)$ & $q-\ee$\\ $\Tbsha'$ & $\{a \in \kk_2^\times \mid
    a^\circ a =1 \}$ & $\{a \in \kk_2^\times \mid a^\circ = a \}$ &
    $q+\ee$\\ $\Nsha'$ & $\mathbb{Z}/2\mathbb{Z} \times \aG_a(\kk)$ &
    $\mathbb{Z}/2\mathbb{Z} \times \aG_a(\kk)$ & $ 2q$\\ \hline
  \end{tabular}
\end{table}

 For $\len \in \mathbb{N}_0$, let
\[
s_{\len}(\ag(\lri)) \coloneqq \gamma_\len(0) = \lvert \Ad \aG(\lri)
\backslash \ag(\lri_\len) \rvert
\]
denote the number of similarity classes in $\ag(\lri_\len)$. In
analogy with Theorem~\ref{thm:sim.zeta.local} we obtain the following
from our formulae for the functions $\gamma_\len(s)$.

\begin{thm}
 Let $\lri$, $\aG$, $\ag$ and $\ee = \ee_\aG$ be as above; if $\ee=-1$
 suppose that $\lri$ has odd residue characteristic.  Then
\begin{equation}\label{equ:sim.A1} \sum_{\len=0}^{\infty}
  s_{\len}(\ag(\lri)) t^\len = \frac{1}{(1-qt)(1-q^2t)}.
\end{equation}
\end{thm}

\begin{remark}
 For $\epsilon=1$ equation~\eqref{equ:sim.A1} was already computed in
 \cite[Section~2]{APOV}. It is remarkable that the same formula covers
 the case $\epsilon=-1$.
\end{remark}

As in the $\mathsf{A}_2$-case, these results may be put in an ad\`elic
respectively global context as follows. Let $k$ be a number field with
ring of integers $\Gri$.  Let $\mathbf{G}$ be one of the $k$-algebraic
groups $\GL_2$ or $\GU_2(K,f)$, where the unitary group $\GU_2(K,f)$
is defined with respect to the standard hermitian form $f$ associated
to the non-trivial Galois automorphism of a quadratic extension $K$
of~$k$.  Accordingly, let $\fg$ be one of the Lie algebra schemes
$\gl_2$ or $\gu_2(K,f)$.  Put $\ee_\mathbf{G} = 1$ if $\mathbf{G} =
\GL_2$, and $\ee_\mathbf{G} = -1$ if $\mathbf{G}$ is unitary.  For the
ring of $S$-integers $\Gri_S$, where $S$ is a finite set of places of
$k$ including all the archimedean ones, we consider the Dirichlet
series $\zeta^{\textup{sc}}_{\ag(\Gri_S)}(s)$ defined
in~\eqref{def:sim.global}.

\begin{cor}
 Let $\Gri_S \subset k$ and $\mathbf{G}$, $\fg$, $\ee_\mathbf{G}$ be
 as above; if $\ee_\mathbf{G} = -1$ suppose that $S$ includes all
 dyadic places of $k$ as well as those places which ramify in the
 quadratic extension $K$ of $k$ defining $\mathbf{G} = \GU_2(K,f)$.
 Then
 \begin{equation*}
  \zeta^{\textup{sc}}_{\ag(\Gri_S)}(s) = \zeta_{k,S}(s-1)\zeta_{k,S}(s-2).
\end{equation*}
In particular, there exists an invariant $\delta_1(\Gri_S)\in\R_{>0}$ such
that
\[
\delta_1(\Gri_S) = \lim_{N\rightarrow\infty} \frac{\sum_{n=1}^\N
  s_n(\ag(\Gri_S))}{N^3}.
\]
\end{cor}

We define $Z_{\fg(\Gri_S)}(s) = \sum_{n=1}^\infty \simi_n(\fg(\Gri_S))
n^{-s} \coloneqq \lim_{I \triangleleft \Gri_S} Z_{\fg(\Gri_S/I)}(s)$
as in~\eqref{equ:def.Z.global}. The formulae for the functions
$\gamma^\sigma_\len(s)$ provided in
Proposition~\ref{pro:gamma.sigma.fin.A1} yield the following result.

\begin{thm}
  Let $\Gri_S \subset k$ and $\mathbf{G}$, $\fg$, $\ee_\mathbf{G}$
  be as above; if $\ee_\mathbf{G} = -1$ suppose that $S$ includes all
  dyadic places of $k$ as well as those places which ramify in the
  quadratic extension $K$ of $k$ defining $\mathbf{G} = \GU_2(K,f)$.
  Then
  \[
  Z_{\fg(\Gri_S)}(s) = \prod_{v\not\in S} \left(1 +
    \frac{\hlf(q_v-1)((q_v(q_v+\epsilon_v))^{-s} +
      (q_v(q_v-\epsilon_v))^{-s})+(q_v^2-1)^{-s}}{1-q_v^{1-2s}}\right),
  \]
  where $q_v$ is the residue cardinality at $v$ and $\ee_v=-1$ if
  $\ee=-1$ and $\fg(\Gri_v)\simeq\gu_2(\Gri_v)$, and $\ee_v=1$
  otherwise. Consequently, the following hold.
  \begin{enumerate}
  \item The abscissa of convergence of $Z_{\fg(\Gri_S)}(s)$ is equal
    to $1$.
  \item The function $Z_{\fg(\Gri_S)}(s)$ has meromorphic continuation
    to $\{s\in\C \mid \real(s) > 1/2\}$. The only pole of
    $Z_{\fg(\Gri_S)}(s)$ in this domain is a single pole at $s=1$.
  \item There exists an invariant $\delta_2(\Gri_S)\in\R_{>0}$ such
    that
    \[
    \delta_2(\Gri_S) = \lim_{N\rightarrow\infty}\frac{\sum_{n=1}^N
      \simi_n(\fg(\Gri_S))}{N}.
    \]
  \end{enumerate}
\end{thm}

\subsection{Zeta functions of the shadows}
Recall that, for a shadow $\sigma\in\Sh$, we set
$\sigma'(\kk)\coloneqq \sigma(\kk)\cap \SL_2(\kk)$. See
Table~\ref{tab:shadows.SLSU.A2} for details on the groups occurring.

\begin{prop}
 Let $\sigma \in \Sh$ be of type $\cS\in\mathbb{T}_{\mathsf{A}_1}$.
 Then
\begin{align*}
  \zeta_{\sigma(\kk)}(s) &= (q-\epsilon)
  Z^{\cS}_{\mathsf{A}_1,\ee,1,q}(s) \text{ and } \\
  \zeta_{\sigma'(\kk)}(s) & = Z^{\cS}_{\mathsf{A}_1,\ee,2,q}(s),
\end{align*}
where, for $i\in\{1,2\}$,
\begin{equation*}
  Z^{\Gsha'}_{\mathsf{A}_1,\epsilon,i,q}(s) = 1+q^{-s}+\hlf (q-3)
  (q+1)^{-s}+\hlf (q-1)(q-1)^{-s} + {i^2}/{2}
  \left({(q+1)}/{i}\right)^{-s} + i^2/2((q-1)/i)^{-s}.
\end{equation*}
In the remaining cases, the function
$Z^{\cS}_{\mathsf{A}_1,\ee,i,q}(s)$ is defined as
\begin{equation*}
 \begin{array}{ll}
  q-\ee & \text{if $\cS = \Tasha'$,} \\
  q+\ee & \text{if $\cS = \Tbsha'$,} \\
  iq & \text{if $\cS = \Nsha'$.}
 \end{array}
\end{equation*}
\end{prop}

\subsection{Zeta functions of groups of type $\mathsf{A}_1$}
Let $\aH$ denote the $\lri$-group scheme~$\SL_2$ or $\SU_2$, according
to $\ee=\ee_\aH\in\{-1,1\}$ as above. Recall that the ramification
index of a compact discrete valuation ring $\lri$ of characteristic
$0$ is denoted by~$e=e(\lri,\mathbb{Z}_p)$.

\begin{thm}
  Let $\lri$ be a compact discrete valuation ring of residue
  characteristic $p=\cha(\kk)$. Let $\aG$, $\aH$ be either $\GL_2$,
  $\SL_2$ or $\GU_2$, $\SU_2$ as above and $\len\in\N$. Assume that $p
  \geq \min\{2\len,2e+2\}$ if $\cha(\lri) = 0$, and $p \geq 2\len$ if
  $\cha(\lri) = p$.  Then the following hold:
  \begin{align}
    \zeta_{\aG(\lri_\len)}(s) &= q^{\len-1}\sum_{\cS \in
      \T_{\mathsf{A}_1}}[\aG(\kk) :
      \aI^\cS_{\mathsf{A}_1,\epsilon}(\kk)]^{-1-s} \zeta_{
      \aI^\cS_{\mathsf{A}_1,\epsilon}(\kk)}(s)
    \Xi^\cS_{\mathsf{A}_1,\ee,q,\len-1}(s),\label{equ:zeta.G.fin.A1}\\ \zeta_{\aH(\lri_\len)}(s)
    &= \sum_{\cS\in\T_{\mathsf{A}_1}} [\aH(\kk) : (\aH(\kk)\cap
      \aI^{\cS}_{\mathsf{A}_1,\epsilon}(\kk))]^{-1-s}
    \zeta_{\aH(\kk)\cap \aI^{\cS}_{\mathsf{A}_1,\epsilon}(\kk)}(s)
    \Xi_{\mathsf{A}_1,\epsilon,q,\len-1}^{\cS}(s)\label{equ:zeta.H.fin.A1}.
  \end{align}

  Moreover, if $\cha(\lri)=0$ and $p > 2e+2$, then
  \begin{equation}\label{equ:zeta.H.infin.A1}
    \zeta_{\aH(\lri)}(s) =
    \sum_{\cS\in\T_{\mathsf{A}_1}} [\aH(\kk) : (\aH(\kk)\cap
    \aI^{\cS}_{\mathsf{A}_1,\epsilon}(\kk))]^{-1-s} \zeta_{\aH(\kk)\cap
      \aI^{\cS}_{\mathsf{A}_1,\epsilon}(\kk)}(s)
    \Xi_{\mathsf{A}_1,\epsilon,q}^{\cS}(s).
  \end{equation}
\end{thm}

\begin{remark}
  Equation~\eqref{equ:zeta.H.infin.A1} confirms -- in the cases where
  it is applicable -- Jaikin-Zapirain's formula for the representation
  zeta function $\zeta_{\SL_2(\lri)}(s)$.  Recall the notational
  convention $\sigma(\kk) = \aI^{\cS}_{\mathsf{A}_1,\epsilon}(\kk)$
  and that the relevant information about these groups is recorded in
  Table~\ref{tab:shadows.GLGU.A1}.
\end{remark}

\begin{remark}
It is noteworthy that the special values of the zeta functions
$\zeta_{\aG(\lri_\len)}(s)$ -- at least as far as they are given by
\eqref{equ:zeta.G.fin.A1} -- at $s=-1$, i.e.\ the sums of character
degrees of the groups~$\aG(\lri_\len)$, coincide with the numbers of
symmetric matrices in $\aG(\lri_\len)$, viz.\
$$ \zeta_{\aG(\lri_\len)}(-1) = (1-\ee q^{-1}) q^{3\len}.$$ This is in
contrast to the situation in type $A_2$;
cf.\ Remark~\ref{rem:zeros.special.A2}.
\end{remark}

\begin{thm}
  Let $\lri$ and $\aG$, $\aH$, $\ee = \ee_\aG = \ee_\aH$ be as above.
  Let $\len,m \in \N$ with $\len \geq m$. Suppose that $p>2$; suppose
  further that $m \geq \min \{\len/p, e/(p-2) \}$ if $\cha(\lri) = 0$,
  and $m\geq \len/p$ if $\cha(\lri) = p$.
  \begin{align*}
    \zeta_{\aH^m(\lri)/\aH^\len(\lri)}(s) &= \begin{cases}
      q^{3(\len-m)} & \text{ if } \len \leq 2m, \\ q^{3(m-1)}\left( 1
        + (q^3-1)
        \frac{1-\left(q^{1-s}\right)^{\len-2m+1}}{1-q^{1-s}}\right) &
      \text{ if } \len > 2m. \end{cases}\\
    \zeta_{\aG^m(\lri)/\aG^\len(\lri)}(s) & =
    q^{\len-m}\zeta_{\aH^m(\lri)/\aH^\len(\lri)}(s).
  \end{align*}
  Moreover, if $\cha(\lri)=0$ then
  \begin{equation}\label{equ:SL2.princ}
    \zeta_{\aH^m(\lri)}(s) = q^{3(m-1)}
    \sum_{\cS \in \mathbb{T}_{\mathsf{A}_1}}\Xi^{\cS}_{\mathsf{A}_1,\epsilon,q}(s)
    =q^{3m}\frac{1-q^{-2-s}}{1-q^{1-s}}.
  \end{equation}
 \end{thm}

\begin{remark}
  Equation~\eqref{equ:SL2.princ} confirms
  \cite[Theorem~1.2]{AKOVIII/11}.  Despite appearance in
  \eqref{equ:zeta.H.fin.A1}, \eqref{equ:zeta.H.infin.A1}, and
  \eqref{equ:SL2.princ}, and in contrast to their analogues in
  type~$A_2$, the zeta functions $\zeta_{\aH(\lri_\len)}(s)$,
  $\zeta_{\aH(\lri)}(s)$, and $\zeta_{\aH^m(\lri)}(s)$ are independent
  of~$\ee$. This reflects the fact that the isomorphism $\SL_2(\F_q)
  \simeq \SU_2(\F_q)$ (cf.\ \cite[II.8.8]{Huppert/67}) generalises to
  $\SL_2(\lri) \simeq \SU_2(\lri)$.
\end{remark}






\def\cprime{$'$} \def\cprime{$'$}
\providecommand{\bysame}{\leavevmode\hbox to3em{\hrulefill}\thinspace}
\providecommand{\MR}{\relax\ifhmode\unskip\space\fi MR }
\providecommand{\MRhref}[2]{%
  \href{http://www.ams.org/mathscinet-getitem?mr=#1}{#2}
}
\providecommand{\href}[2]{#2}

\end{document}